\documentclass[leqno,11pt]{amsart}
\usepackage{amssymb, amsmath,latexsym,amsfonts,amsbsy, amsthm,mathtools,graphicx,CJKutf8,CJKnumb,CJKulem,color}
\usepackage{float}
\usepackage{hyperref}

\numberwithin{equation}{section}

\allowdisplaybreaks

\let\f=\frac

\let\om=\omega

\let\Om=\Omega

\let\na=\nabla

\let\pa=\partial

\def\R{\mathbb R}
\def\T{\mathbb T}
\def\Z{\mathbb Z}

\newcommand{\tp}{\mathfrak{p}}

\newcommand{\beq}{\begin{equation}}
\newcommand{\eeq}{\end{equation}}
\newcommand{\ben}{\begin{eqnarray}}
\newcommand{\een}{\end{eqnarray}}
\newcommand{\beno}{\begin{eqnarray*}}
\newcommand{\eeno}{\end{eqnarray*}}

\newtheorem{theorem}{Theorem}[section]
\newtheorem{definition}[theorem]{Definition}
\newtheorem{lemma}[theorem]{Lemma}

\newtheorem{Theorem}{Theorem}[section]

\newtheorem{Proposition}[Theorem]{Proposition}

\newtheorem{Remark}[Theorem]{Remark}

\begin{document}

\title[Asymptotic stability of the Kolmogorov flow]
{Asymptotic stability of the Kolmogorov flow at high Reynolds numbers}  

\author[Chen]{Qi Chen}
\address{School of Mathematical Sciences, Zhejiang University, Zhejiang,  P. R. China}
\email{chenqi123@zju.edu.cn}

\author[Jia]{Hao Jia}
\address{School of Mathematics, University of Minnesota, Minneapolis, MN, USA}
\email{jia@umn.edu}

\author[Wei]{Dongyi Wei}
\address{School of Mathematical Sciences, Peking University, Beijing,  P. R. China}
\email{jnwdyi@pku.edu.cn}

\author[Zhang]{Zhifei Zhang}
\address{School of Mathematical Sciences, Peking University, Beijing, P. R. China}
\email{zfzhang@pku.edu.cn}

\date{\today}

\begin{abstract}
In this paper we prove the asymptotic stability of the Kolmogorov flow on a non-square torus for perturbations $\omega_0$ satisfying $\|\omega_0\|_{H^3}\ll\nu^{1/3}$, where $0<\nu\ll1$ is the viscosity. Kolmogorov flows are important metastable states to the two dimensional incompressible Navier Stokes equations in the high Reynolds number regime. 
Our result shows that the perturbed solution will rapidly converge to a shear flow close to the Kolmogorov flow, before settling down to the Kolmogorov flow and slowly decaying to $0$ as $t\to\infty$. In fact, our analysis reveals several interesting time scales and rich dynamical behavior of the perturbation in the transition period $0<t\leq 1/\nu$. 

The threshold $\nu^{1/3}$, which is the same as that for the Couette flow, is quite surprising since one of the key stability mechanisms, enhanced dissipation, becomes considerably weaker in the case of Kolmogorov flows due to the presence of critical points. To overcome this essential new difficulty, we establish sharp vorticity depletion estimates near the critical points to obtain improved decay rates for the vorticity and velocity fields that are comparable with those for Couette flows, at least for our purposes. We then combine these estimates (enhanced dissipation, inviscid damping and vorticity depletion) with a quasilinear approximation scheme and a multiple-timescale analysis naturally adapted to the dynamics of the perturbation, to obtain the $\nu^{1/3}$ threshold for dynamic stability of Kolmogorov flows. The threshold is expected to be sharp when the perturbation is considered in Sobolev spaces. This appears to be the first result that applies vorticity depletion estimates to improve thresholds for nonlinear asymptotic stability in incompressible fluid equations. Our method is very robust and seems to be useful for the analysis of hydrodynamic stability problems in many other physically relevant settings. 
\end{abstract}

\maketitle

\tableofcontents

\section{Introduction and main results}

In two-dimensional turbulence, it has been observed that certain solutions rapidly approach long-lived quasi-stationary states on timescales much shorter than the viscous timescale (see e.g. \cite{Mc1, Mc2}). These states then dominate the fluid dynamics for very long periods of time. Such quasi-stationary states are referred to as metastable states. To understand this dynamical phenomenon,  we consider the 2-D Navier-Stokes equations on a non-square torus $\mathbb{T}_{\mathfrak{p}}\times\mathbb{T}_{2\pi}$ with $\mathfrak{p}=2\pi \kappa \,\, (\kappa<1)$ in the high Reynolds number ($Re$) regime:
\begin{equation}\label{eq:NS}\left\{
  \begin{aligned}
  &\partial_tU-\nu\Delta U+U\cdot\nabla U+\nabla P=0,\\
  &\nabla\cdot U=0,\\
  &U|_{t=0}=U_0(x,y),
  \end{aligned}\right.
\end{equation}
where $U=(U^x,U^y)$ is the velocity, $P$ is the pressure, and the viscosity coefficient $\nu=1/Re\in (0,1)$.  Let $\Omega=\partial_xU^y-\partial_yU^x$ be the vorticity, which satisfies 
 \begin{align}\label{eq:NS-vorticity}
 \partial_t\Omega-\nu\Delta \Omega+U\cdot\nabla \Omega=0,\quad \Om|_{t=0}=\Om_0(x,y).
\end{align}

The long-lived quasi-stationary states of \eqref{eq:NS} are related to the stationary solutions of the Euler equations. It has been well documented in physics literature and numerical simulations that certain maximum entropy solutions of the Euler equations are the most probable quasi-stationary states that one would observe after an intermediate stage of evolution, see e.g. Yin, Montgomery and Clercx \cite{Yin} for a recent study and references therein for more related works. These states are referred to as bar states and dipole states. 
Bar states are described by the Kolmogorov flows, which may be represented, without loss of generality, by a simple cosine function in one spatial direction:
\beno
U_{bar}=\big(\mathrm{e}^{-\nu t}\cos y, 0).
\eeno
Kolmogorov flows are distinguished from other shear flows (which could also attract general solutions in the slightly viscous two dimensional flows for a long time), since they have the lowest frequency  and therefore decay most slowly under viscous effects.

These states are particularly relevant in systems where the aspect ratio of the domain is not equal to one ($\kappa\neq1$). When $\kappa=1$ the dipole states (also known as Taylor-Green vortices) are more commonly observed. Dipole states are characterized by a pair of counter-rotating vortices and can be represented by a combination of cosine functions in both spatial directions:
\beno
U_{dipole}=\mathrm{e}^{-\nu t}(\cos y, -\cos x).
\eeno

\subsection{Statement of our main result}
In this paper, we focus on the precise dynamics near Kolmogorov flows, when $0<\nu\ll1$. An important question, in view of numerical observations on the rapid convergence to bar states, is to understand the stability mechanisms near the Kolmogorov flows and determine the size of their ``basin of attraction". In an early important work, Beck and Wayne \cite{BW-SIAM} discovered for the  Burgers equation that each long-time asymptotic state corresponds to a one-dimensional, invariant manifold of viscous N-waves. These manifolds attract nearby solutions at a much faster rate than the viscous timescale and govern the behavior of solutions for a very long time before they reach their ultimate asymptotic state. Beck and Wayne also studied the linearized Navier Stokes equation around the Kolmogorov flow in \cite{BW} and obtained some significant results that we will describe below. However, there are many significant difficulties in the study of even the linearized equation, and the theory on the asymptotic stability of the Kolmogorov flow remains essentially incomplete. 

The goal of this paper is to essentially improve our understanding of this problem by providing a precise description of the metastable behavior and asymptotic stability of the Kolmogorov flow for the 2-D Navier-Stokes equations on a non-square torus. 

The main result of this paper is stated as follows.

\begin{theorem}\label{thm:main}
There exist $\epsilon,\epsilon_1>0$ such that for all $0 < \nu\leq \epsilon_1$, if the initial vorticity $\Omega_0$ satisfies 
\begin{equation}\label{th:eq0}
\|\Omega_0-\sin y\|_{H^3(\mathbb{T}_{\mathfrak{p}}\times\mathbb{T}_{2\pi})}\leq \epsilon_1\nu^{\f13},
\end{equation}
then the solution $\Omega(t)$ of \eqref{eq:NS-vorticity} is global in time and satisfies the following asymptotic stability estimates: for $t\geq 0$,
\begin{align}
 \label{th:eq1}  & \|\Omega_{\neq}(t)\|_{L^2(\mathbb{T}_{\mathfrak{p}}\times\mathbb{T}_{2\pi})}\leq C\big(\nu^{\f13}+\langle \nu t^3\rangle^{-1}\big)
   \mathrm{e}^{-\epsilon\nu^{\f12}\min(t,\nu^{-1})-\nu t}\|P_{\neq}\Omega_{0}\|_{H^3(\mathbb{T}_{\mathfrak{p}}\times\mathbb{T}_{2\pi})},\quad \\
\label{th:eq2}   &\|U_{\neq}(t)\|_{L^2(\mathbb{T}_{\mathfrak{p}}\times\mathbb{T}_{2\pi})}\leq C\big(\nu^{\f13}+\langle t\rangle^{-1}\big)
   \mathrm{e}^{-\epsilon\nu^{\f12}\min(t,\nu^{-1})-\nu t}\|P_{\neq}\Omega_{0}\|_{H^3(\mathbb{T}_{\mathfrak{p}}\times\mathbb{T}_{2\pi})},\\
    &\|\Omega(t)-\mathrm{e}^{-\nu t}\sin y\|_{L^2(\mathbb{T}_{\mathfrak{p}}\times\mathbb{T}_{2\pi})}\leq C\mathrm{e}^{-\nu t/2}\|\Omega_0-\sin y\|_{H^3(\mathbb{T}_{\mathfrak{p}}\times\mathbb{T}_{2\pi})}.
    \end{align}
In the above, $\langle a\rangle:=\sqrt{1+|a|^2}$ is the usual Japanese bracket, and 
$$h_{\neq}:= h(x,y)-\frac{1}{\tp}\int_{\T_{\tp}}h(x',y)dx'$$ is the projection of $h$ off the zero mode in $x$. 
\end{theorem}

 The threshold $\nu^{1/3}$ in \eqref{th:eq0}, the same as that of Couette flow, is surprising. Indeed, in view of the rate of enhanced dissipation $\mathrm{e}^{-c\nu^{1/2}t}$ for non-monotonic shear flows (rather than $\mathrm{e}^{-c\nu^{\f13}t}$ which holds for the Couette flow), it seems natural, at first glance, to conjecture that the optimal threshold should be $\nu^{1/2}$, which would be considerably smaller than the threshold we obtained. Interestingly, as far as we know, even the threshold $\nu^{1/2}$ for non-monotonic shear flows remained an open problem before our work.
 
 A detailed discussion on the proof can be found in section \ref{sec:ideas}, which requires sharp estimates on all three known stabilizing mechanisms in our setting: \textit{inviscid damping}, \textit{enhanced dissipation} and \textit{vorticity depletion}. We only highlight here that the (optimal) decay factor $\langle\nu t^3\rangle^{-1}$ in \eqref{th:eq1} implies strong decay for the vorticity field well before the usual enhanced dissipation plays any role. This factor is highly nontrivial and is a consequence of the deep interaction between vorticity depletion and enhanced dissipation phenomena.
 Estimates suggesting that the perturbed {vorticity} field could decay over the time scale $\nu^{-1/3}$ near the Kolmogorov flow first appeared in \cite{WZZ-AM}. More precisely it was proved in \cite{WZZ-AM} that for the linearized Navier Stokes equation near the Kolmogorov flow, the {velocity} field decays with a rate $\langle\nu t^3\rangle^{-1/2}$, in addition to the usual enhanced dissipation estimates. Upgrading this factor to the sharp $\langle\nu t^3\rangle^{-1}$ for the vorticity, although still far from being sufficient by itself, is an important ingredient in establishing the unexpected threshold $\nu^{1/3}$. 

We have the following further comments on Theorem \ref{thm:main}. 

\begin{itemize}

\item[1.] Our result shows that the Kolmogorov flow attracts nearby solutions at a much faster rate than that determined by diffusion alone, up to the viscous timescale $t\lesssim \f1 \nu$.  When time exceeds this scale, the $x$--dependent perturbations become exceedingly small and the viscous effect will dominate the dynamics of the solution. The decay estimates \eqref{th:eq1}-\eqref{th:eq2} are rich in their dynamical implications, and are a consequence of three essential stability mechanisms near the Kolmogorov flows: enhanced dissipation, inviscid damping and vorticity depletion, see subsection \ref{sec:ideas} for an in-depth discussion.


\item[2.] We derive uniform-in-$\nu$ estimates for both enhanced dissipation and inviscid damping to the nonlinear Navier-Stokes system, which are almost identical to those obtained for the linearized Navier-Stokes system in \cite{WZ-SCM}. Although vorticity depletion effect is not obvious in the bounds \eqref{th:eq1}-\eqref{th:eq2}, we shall see below that it plays a decisive role in obtaining the extra decay factor $\langle\nu t^3\rangle^{-1}$.

\item[3.] The threshold $\nu^{1/3}$ for asymptotic stability we obtained is likely sharp. On the technical level, this is quite clear from our proof, since in many of the steps we need estimates which are completely sharp. In addition, for the simpler case of Couette flows, where much more precise control on the solution can be achieved, see Masmoudi and Zhao \cite{MZ-AIHP}, the threshold is also $\nu^{1/3}$. Convincing arguments were given in \cite{MZ-AIHP} to suggest that $\nu^{1/3}$ should be optimal, as long as one works with Sobolev space perturbations. In Gevrey spaces, the situation might be different and it is not clear to us if one can establish better thresholds or even $\nu$-independent threshold as in nonlinear inviscid damping \cite{BM-IHES, IJ-Acta, MZ-AM}. It is also worth noting that the Sobolev space $H^3$ for the perturbation is of quite low regularity and seems optimal as well.

\item[4.] Our strategy to achieve the sharp threshold is very robust. The general method should be applicable to studying many other mathematical problems in the field of hydrodynamic stability at high Reynolds numbers.

\end{itemize}

\subsection{Background and related work} In this subsection we briefly introduce important background for our result, related work and the three main stability mechanisms. Due to the rapid developments in this area, we will not attempt to provide a comprehensive survey, instead, we will focus mostly on results directly related to our work.

\subsubsection{Transition threshold problems} The Kolmogorov flow was introduced by Kolmogorov in 1959 as a fundamental model for investigating the onset of turbulence (see, for example, \cite{AM, CP, VG}). The flow is globally stable on the torus $\mathbb{T}_{2\mathfrak{\kappa}\pi}\times\mathbb{T}_{2\pi}$ with $\kappa\le 1$ \cite{Mar, MS}. When $\kappa>1$, the flow 
is linearly unstable \cite{Yud}.  The main goal of this paper is to investigate the {\bf asymptotic stability threshold problem} for the Kolmogorov flow on $\mathbb{T}_{\mathfrak{2\kappa\pi}}\times\mathbb{T}_{2\pi}$ with $\kappa<1$. More precisely, we address the following problem, which was first formulated by Masmoudi and Zhao in the context of the 2D Couette flow \cite{MZ-AIHP}:  \smallskip

{\it Given a norm $\|\cdot\|_X\ (X\subset L^2)$, find a $\beta=\beta(X)\ge0$ as small as possible so that for the initial vorticity $\|\Omega_0-\sin y\|_X\ll \nu^{\beta} $ and for $t\lesssim 1/\nu$,
\beno
&&\|\Omega_{\neq}(t)\|_{L^2}\le C\mathrm{e}^{-c\nu^{1/2}t}\|\Omega_0-\sin y\|_X,\\
&&\|U_{\neq}(t)\|_{L^2}\le C(1+t)^{-1}\mathrm{e}^{-c\nu^{1/2}t}\|\Omega_0-\sin y\|_X.
\eeno
}

The smaller the value of $\beta$, the more crucial and challenging it becomes to control the energy cascade between different modes across various timescales.
Resolving this problem will therefore yield a comprehensive mathematical understanding of the metastable behavior of Kolmogorov flow and two-dimensional turbulence. \smallskip
  
Our result is also closely related to the larger effort in establishing the transition threshold in the dynamics of incompressible fluid equations for both two and three dimensions. The transition mechanism from laminar flow to turbulent flow is a classical yet unsolved problem in fluid mechanics \cite{Tre, Chap, BGM-BAMS, WZ-ICM}. To understand this transition mechanism, Trefethen et al. \cite{Tre} first formulated the {\bf transition threshold problem}: namely, to quantify the magnitude of perturbations required to trigger instability and to determine their scaling with the Reynolds number. More recently, Bedrossian, Germain, and Masmoudi \cite{BGM-BAMS} provided a rigorous mathematical framework for this problem, which is stated as follows.\smallskip\\
{\it Given a norm $\|\cdot\|_{X}$, find a $\beta=\beta(X)$ so that
	\begin{align*}
&\|u_0\|_{X} \leq Re^{-\beta} \to \textrm{stability},\\
&\|u_0\|_{X} \geq Re^{-\beta}\to \textrm{instability}.
		\end{align*}
}

This exponent $\beta$ is referred to as the transition threshold. Recent mathematical results show that for the 3-D Navier-Stokes equations, many physical effects, such as  lift-up effect, inviscid damping, enhanced dissipation, and boundary layer effect, play crucial roles in determining the transition threshold (usually $\beta\ge 1$), see \cite{BGM-AM, BGM-MAMS1, BGM-MAMS2, WZ-CPAM, CWZ-MAMS, LWZ-CPAM, CDLZ}.  Furthermore, these mathematical results also provide important insights into the transition process, especially the secondary instability mechanism induced by the transient energy growth of the streamwise streak due to the lift-up effect.    

\subsubsection{Related results and Methodology}

Let us  review the following important progress on the asymptotic stability threshold of the 2-D Couette flow.
 For the domain without a physical boundary, $\mathbb{T}\times \R$,  
\begin{itemize}

\item if $X$ is the Gevrey class $2+$, then $\beta=0$ \cite{BMV};

\item if $X$ is a Sobolev space, then $\beta\le \f12$ \cite{BWV, MZ-CPDE};

\item if $X$ is a Sobolev space, then $\beta\le \f13$ \cite{MZ-AIHP, WZ-TJM};

\item if $X$ is the Gevrey class $\f 1s$, $ s\in[0,\f12]$, then $\beta\le \f{1-2s}{3(1-s)}$ \cite{LMZ}.
\end{itemize}

The proof of these results is based on the Fourier multiplier method. The multiplier is constructed both to capture the enhanced dissipation and inviscid-damping effects and to control the worst energy cascade between different modes over various timescales. When applicable, the Fourier method can provide very precise control on the solution. See also \cite{AB, BH1, BH2, LMZ-CPAM} for recent relevant results. 

The Fourier multiplier method fails in the presence of a physical boundary.  For the 2-D Navier-Stokes equations in a finite channel $\mathbb{T}\times [-1,1]$, it was proved in \cite{CLWZ-ARMA} that $\beta\le \f12$ under the non-slip boundary condition by developing the space-time estimate method, which relies on resolvent estimates for the linearized NS operator. Under Navier-slip boundary condition, in \cite{WZ-pre} by the last two authors, the stability threshold $\beta\le \f13$ in the Sobolev space was proved via a quasilinear approximation method, which relies on the sharp inviscid damping estimates for the linearized Euler equations and resolvent estimates for the linearized Navier-Stokes equations. The work \cite{WZ-pre} not only provides an alternative proof of the results in Masmoudi and Zhao \cite{MZ-AIHP}, but more importantly, the general approach introduced there also plays an essential role in our current work.

The Fourier multiplier method also fails for non-monotone flows such as the Kolmogorov flow. By using the space-time estimate method, it was  proved in \cite{WZZ-AM} that $\beta\le \f23+$ for 2-D Kolmogorov flow, while \cite{DL-SCM} proved $\beta\le \f23$ for 2-D plane Poiseuille flow under Navier-silp boundary condition. We believe the result in \cite{DL-SCM} could be improved to $\beta\le \f13$ by applying the method in this paper.
 See also \cite{BeekieHe, CEW, Del, CL-JDE, CL-Non, LX, Chen} for more relevant results on transition thresholds in various settings.  

Finally, we mention some related results for the kinetic equations. Chaturvedi, Luk and Nguyen \cite{CLN} proved that for the Vlasov-Poisson-Landau system in the weakly collisional limit, if the initial data satisfies a certain smallness condition (i.e., the perturbation in the Sobolev space is of the order $O(\nu^{\f13})$), there exists a unique global smooth solution that converges to the global Maxwell distribution. See \cite{Bed, BZZ}  for the Vlasov–Poisson–Fokker–Planck system and \cite{BCD} for the non-cutoff Boltzmann equation in the weakly collisional limit.

\subsubsection{Enhanced dissipation, Inviscid damping and Vorticity depletion}

Let us recall the three key physical mechanisms underlying the metastable behavior of the Kolmogorov flow.

In \cite{BW}, Beck and Wayne proved that the solution to the linearized Navier-Stokes equation around the Kolmogorov flow but without the nonlocal term  will rapidly converge to zero. 
More precisely, consider the following linearized equation
\beno
\pa_t\Om-\nu\Delta \Om+\mathrm{e}^{-\nu t}\cos y \pa_x\Om=0,  \quad \Om(0,x,y)=\Om_0(x,y).
\eeno
For a class of initial data, by using the hypocercivity method developed in \cite{Vill}, they proved that  for $t\lesssim 1/\nu$,
\beno
\|\Om_{\neq}(t)\|_{L^2}\lesssim \mathrm{e}^{-\epsilon \nu^{\f 12}t}.
\eeno
Here and in what follows, we denote {the zero mode and nonzero mode as}
\beno
P_0f={\mathfrak{p}}^{-1}\int_{\mathbb{T}_{\mathfrak{p}}}f\mathrm{d}x,\quad P_{\neq}f=f_{\neq}=f-P_{0}f.
\eeno
Note that the decay rate $\nu^{1/2}$ is much faster than the viscous decay rate $\nu$. This phenomenon is called the {\bf enhanced dissipation}, which is due to the vorticity mixing induced by
the shear flow.  In \cite{BW}, Beck and Wayne also provided numerical evidence that the same result holds for the full linearized Navier-Stokes equation, namely, 
\beno
\pa_t\Om-\nu\Delta \Om+\mathrm{e}^{-\nu t}\cos y\big(1+(\Delta)^{-1}\big)\pa_x\Om=0.
\eeno
The introduction of the non-local term $\cos y\pa_x\Delta^{-1}$ poses significant challenges for mathematical analysis. In a series of works \cite{WZZ-AM, LWZ-CPAM, WZ-SCM}, the last two authors and their collaborators developed the wave operator method, the resolvent estimate method, and the time-weighted hypocoercivity method to address the Beck-Wayne conjecture (see \cite{IMM} for a different method). Specifically, the following enhanced dissipation result was established: for $\kappa<1$, if $P_0\Om_0=0$, then there exists $\epsilon>0$ independent of $\nu$ such that for $0\le t \lesssim \f 1\nu$,
\beno
\|\Om(t)\|_{L^2}\le C\mathrm{e}^{-\epsilon \nu^{\f 12}t}\|\Om_0\|_{L^2};
\eeno
and for $\kappa=1$, there holds 
\beno
\|(1-P_1)\Om(t)\|_{L^2}\le C\mathrm{e}^{-\epsilon \nu^{\f 12}t}\|(1-P_1)\Om_0\|_{L^2},
\eeno
where $P_1$ is the orthogonal projection to the space spanned by $\{\sin x, \cos x\}$. 

The enhanced dissipation phenomenon also occurs in general shear flows. For monotone shear flows such as Couette flow $(y,0)$,  the enhanced dissipation rate is $\nu^{\f13}$ \cite{CLWZ-ARMA, CWZ-CMP}; for non-monotone shear flows with non-degenerate critical points such as plane Poiseuille flow $(1-y^2,0)$, 
the enhanced dissipation rate is $\nu^\f12$ \cite{CEW, DL-JDE}.  Let us mention that the enhanced dissipation properties of dipole states on the square torus remain an important open question.

The second crucial physical mechanism is {\bf inviscid damping}, analogous to Landau damping in plasma physics \cite{MV}, which also arises from phase mixing.  This mechanism was first observed for Couette flow by Orr \cite{Orr} and later extended to monotone flows by Case \cite{Case}. For the linearized 2-D Euler equation around a class of monotone shear flows, Zhao and the last two authors \cite{WZZ-CPAM} established sharp inviscid damping rates:
\beno
\|U^x_{\neq}(t)\|_{L^2}\le C(1+t)^{-1},\quad  \|U^y(t)\|_{L^2}\le C(1+t)^{-2}.
\eeno 
The same decay rates were later obtained for Kolmogorov flow in \cite{WZZ-AM} and for general non-monotone flows in \cite{IIJ-2}. Viewed purely as a passive transport equation driven by Kolmogorov flow, these rates are surprising. For non-monotone profiles, a new phenomenon of {\bf vorticity depletion} at the stationary streamlines (first observed by Bouchet and Morita \cite{Bou} and subsequently proved by Zhao and the last two authors \cite{WZZ-APDE}) becomes decisive. More precisely, for Kolmogorov flow, we have
\beno
|\Om(t,x,y_c)|\le C(1+t)^{-1}\quad \text{whenever}\quad \sin y_c=0. 
\eeno

For more relevant results on inviscid damping, we refer to \cite{BM-IHES, DM, IJ-CMP, IJ-CPAM, IJ-Acta, MZ-AM, CWZZ-JEMS, Zhao, Z1, Z2, BCV, WZZ-CMP, IJ-ARMA, GNR, Ren, IIJ-1}. In particular, in a breakthrough work \cite{BM-IHES}, Bedrossian and Masmoudi established nonlinear inviscid damping for the 2-D Couette flow, see also \cite{IJ-CMP}. More recently, Ionescu  and the third author \cite{IJ-Acta}, and Masmoudi and Zhao \cite{MZ-AM} independently proved nonlinear inviscid damping for stable monotone shear flows.

\subsection{Main ideas and outline of the proof}\label{sec:ideas}
In this section we explain our main ideas and outline the key points in the proof of Theorem \ref{thm:main}. 

The fundamental new difficulty in proving asymptotic stability for Kolmogorov flows, in comparison with the case of monotone shear flows, is the significantly weaker rate of enhanced dissipation for the linearized Navier Stokes equation near the Kolmogorov flow. In particular, for a very long time $0<t<\nu^{-1/2}$ the usual enhanced dissipation provides no useful decay for the vorticity field. One might hope to leverage the inviscid damping estimates for the velocity fields to control the solution. However, by now it is well known that nonlinear inviscid damping without the help of diffusion requires the control on the ``profile" of the vorticity field in Gevrey spaces. It is not clear if one can obtain such strong control even for the linearized equation near the Kolmogorov flow. 

To overcome this essential difficulty, we need to establish  a large family of sharp linear and nonlinear estimates, and the argument is rather involved. In order to focus on the main ideas, we will make the following simplifying assumptions. \textit{Firstly}, we shall consider only the time scale $0<t<1/\nu$. This is the dynamic regime where inviscid damping, enhanced damping and vorticity depletion play a decisive role. For $t>1/\nu$, the $x$-dependent components of the solution is already exceedingly small, and the flow is essentially dominated by heat equation. 
\textit{Secondly}, we assume that the initial perturbation does not contain shear components, i.e., the average of $\Omega_0-\sin y$ in $x$ is zero for all $y\in \T_{2\pi}$. In general, the main background flow would be a small perturbation of the Kolmogorov flow and an additional (quite interesting but technical) argument is needed to transfer estimates for Kolmogorov flow to this perturbed flow. \textit{Thirdly}, we ignore the slow decay factor $\mathrm{e}^{-\nu t}$ of the Kolmogorov flow $\mathrm{e}^{-\nu t}(\cos y, 0)$. In general, we need another perturbation argument to incorporate the time-dependence of the coefficients in the linearized Navier Stokes equations. 
It is important to treat all the issues that we ignored in the above assumptions for applications to the full perturbed Navier Stokes equation, and some of the approximation schemes involved in resolving these issues are in fact interesting by themselves. 

Therefore, in this section we will work with the following simplified model for the perturbation velocity and vorticity fields $u$ and $\omega$: 
\begin{align}\label{eq:intro1}
\left\{
\begin{aligned}
   & \partial_t\omega-\nu\Delta \omega+\cos y\,\partial_x\omega+\cos y\,\partial_x\phi+u\cdot\nabla \omega=0,\\
   &u=(-\partial_y,\partial_x)\phi,\quad \phi=\Delta^{-1}\omega,\\
   &\om|_{t=0}=\om_0.
   \end{aligned}\right.
\end{align}
As discussed above, we assume that $\frac{1}{\tp}\int_{\T_{\tp}}\omega_0(x,y)\,dx\equiv0$ for $y\in\T_{2\pi}$. For initial data $\|\omega_0\|_{H^3}\ll \nu^{1/3}$, we need to obtain similar bounds as in \eqref{th:eq1}-\eqref{th:eq2} up to time $1/\nu$. More precisely,
for $0\leq t<1/\nu$,
\begin{align*}
   & \|\omega_{\neq}(t)\|_{L^2}\leq C\big(\nu^{\f13}+\langle \nu t^3\rangle^{-1}\big)
   \mathrm{e}^{-\epsilon\nu^{\f12}t}\|\omega_{0}\|_{H^3},\quad \\
   &\|u_{\neq}(t)\|_{L^2}\leq C\big(\nu^{\f13}+\langle t\rangle^{-1}\big)
   \mathrm{e}^{-\epsilon\nu^{\f12}t}\|\omega_{0}\|_{H^3}.
\end{align*}



\subsubsection{Quasilinear approximation method}
In the context of transition threshold problems, quasilinear approximation method was first introduced by the last two authors in \cite{WZ-pre}. We decompose the solution as 
\begin{equation}\label{eq:intro2}
    \omega = \omega_L + \omega_e, \quad u = u_L + u_e,\quad \phi = \phi_L + \phi_e,
\end{equation}
where $\omega_L$ (together with its associated velocity and stream functions $u_L, \phi_L$) captures the main linear part of the solution which can be controlled quite precisely, and $\omega_e$ is the nonlinear part (hence of higher order $O(\nu^{2/3})$). The task is then reduced to bound $\omega_e$ and the associated velocity field $u_e$, where $\omega_e$ solves the equation
\begin{equation}\label{eq:intro3}
  \left\{\begin{aligned}
  &\partial_t\omega_e-\nu\Delta\omega_e+\cos y\,\partial_x\omega_e+\cos y\,\partial_x\phi_e+ u\cdot\nabla \omega_e+u_e\cdot\nabla \omega_L+Er=0,\\
  &u_e=(-\partial_y,\partial_x)\phi_e,\quad \phi_e=\Delta^{-1}\omega_e,
  \quad\omega_e|_{t=0}=\om_0-\om_{L}(0).
  \end{aligned}
  \right.
\end{equation}
In the above, $Er$ is the error of the approximation for $\omega_L$ as a solution to \eqref{eq:intro1}. It turns out that we can construct $\omega_L$ such that, roughly speaking, 
\begin{equation}\label{eq:intro4}
    \|Er\|_{L^1_tL^2}\lesssim \nu^{2/3}. 
\end{equation}
See Proposition \ref{prop:error} for the full bounds on the error of the approximation. 

The key advantage of this approach is that the nonlinear terms are much smaller and consequently it is easier to use dissipation to close various estimates. 

\subsubsection{Construction of the approximate solution}
A natural choice for constructing $\omega_L$ is to solve the linearized equation of \eqref{eq:intro1} with initial data $\omega_0$. Very recently, uniform-in-viscosity (as $\nu\to0$) estimates have been obtained in various settings, see e.g. \cite{GNR,CWZ-CMP,BCJ2024}. In particular, uniform-in-viscosity inviscid damping and vorticity depletion estimates were proved in \cite{BCJ2024}, for a class of periodic shear flows including the Kolmogorov flow. However, the proof is quite complicated and needs significant adaptations in our setting. More crucially, the results still miss the endpoint cases which are essential for applications to get the sharp threshold. 

In our paper, we take a direct approach and construct the approximate solution $\omega_L$ through modifying the solution to the linearized Euler equation near the Kolmogorov flow. More precisely, let $\omega_E$ solve the linearized Euler equation with initial data $\omega_0$, and we then set 
\begin{equation}\label{eq:intro5}
    (\omega_L)_k(t,y) = \mathrm{e}^{-\nu k^2(\sin y)^2t^3/3}(\omega_E)_k(t,y),
\end{equation}
where $(h)_k$ denotes the $k-$th Fourier mode in $x$ of $h$, $k\in \frac{2\pi}{\tp}\Z$. This construction is motivated by the expectation that $(\omega_E)_k$ oscillates with the factor $\mathrm{e}^{-ikt\cos y }$, and as a result 
$$\nu\partial_y^2(\omega_E)_k\sim -\nu k^2 t^2(\sin y)^2(\omega_E)_k$$
at the leading order. 

In the proof, we need to take into consideration several nontrivial complications, including the time dependence of the background flow and the fact that the background flow has to be taken as a small perturbation of the Kolmogorov flow since the initial perturbation could contain zero mode. See section \ref{sec:quasi} for the full construction.

\subsubsection{Sharp estimates for the linearized Euler equation around Kolmogorov flows}
To get sharp estimates on $\omega_L$, we need to fully leverage the following inviscid damping estimates and vorticity depletion estimates (in the pointwise sense) for the linearized Euler equation around the Kolmogorov flow:
\begin{equation}\label{eq:LE}
    \partial_t\omega_k+\mathrm{i}k \cos y(\omega_k+\psi_k)=0,\quad \Delta_k\psi_k=\omega_k,\quad 
     \omega_k|_{t=0}=\omega_{0,k}.
\end{equation}
where $\Delta_k=\partial_y^2-k^2$ and  $ k\in\alpha \mathbb{Z}\setminus\{0\},\ \alpha:=2\pi/\mathfrak{p}>1$. {We define $\|f\|_{H_k^s}=\|(-\Delta_k)^{s}f\|_{L^2}$ for $s\in\R$. Then $\|f\|_{H_k^s}=\|(\partial_y,k)^{s}f\|_{L^2} $ for 
$s\in \Z_+$ and $\|f\|_{H_k^0}=\|f\|_{L^2}$.}

\begin{theorem}\label{thm:LE}
 Let $(\psi_k, \omega_k)$ solve \eqref{eq:LE}. Then it holds that
  \begin{align*}
     &|\psi_{k}(t,y)|\lesssim\langle t\rangle^{-2}|k|^{-5/2}M_k,\\
     &|\partial_y\psi_{k}(t,y)|\lesssim \big(\langle t\rangle^{-3/2}|k|^{-2}+\langle t\rangle^{-1}|b'||k|^{-3/2}\big)M_k,\\
     &|\omega_{k}(t,y)|\lesssim\min\big(\langle t\rangle^{-1}|k|^{-3/2}+|b'|^2|k|^{-1/2},|k|^{-5/2}\big)M_k,\\
     &|\partial_y(\mathrm{e}^{\mathrm{i}ktb}\omega_{k})(t,y)|\lesssim\min\big(\langle t\rangle^{-1/2}|k|^{-1}+|b'||k|^{-1/2},|k|^{-3/2}\big)M_k,\\
     &|\partial_y(\mathrm{e}^{\mathrm{i}ktb}\psi_{k})(t,y)|\lesssim \langle t\rangle^{-1.2}|k|^{-2.2}M_k,\\
     &\|(\partial_y,k)^2(\mathrm{e}^{\mathrm{i}ktb}\omega_{k})\|_{L^2}\lesssim|k|^{-1}M_k.
  \end{align*}
 Here $M_k=\|\omega_{0,k}\|_{H_k^3}$ and $b(y)=\cos y$.
\end{theorem}
All the estimates in Theorem \ref{thm:LE} except for the fifth inequality are sharp, both in terms of the decay we obtained and the regularity we needed. See section \ref{sec:dimension} for a heuristic explanation. The fifth inequality which implies $\langle t\rangle^{-1.2}$ decay of one derivative of the ``profile" of the stream function is perhaps the most difficult, and goes well beyond all known linear estimates for the Kolmogorov flow. We conjecture that the sharp rate should be close to $t^{-3/2}$ instead of $t^{-1.2}$. 

From \eqref{eq:intro5} and the third inequality in Theorem \ref{thm:main}, it follows that $\omega_L$ enjoys the crucial decay rate $\langle\nu t^3\rangle^{-1}$, in view of the pointwise inequality 
\begin{equation}
    (\sin y)^2 \mathrm{e}^{-\nu k^2(\sin y)^2t^3/3}\lesssim \langle k^2\nu t^3\rangle^{-1}. 
\end{equation}

The proof of Theorem \ref{thm:LE} is one of the main innovations of our paper. The key is the discovery of the \textit{almost conserved} quantity 
\begin{equation}\label{eq:intro6}
    \omega_{1k}:=\big(\Delta_k+\mathrm{i}kt[\Delta_k,B]-k^2t^2(1-B^2)\big)\omega_k,
\end{equation}
and the remarkable coercive property
\begin{equation}\label{eq:intro7}
\begin{split}
      &\Lambda_1(-\Delta_k)^{s}\Lambda_1+2(1-B^2)(-\Delta_k)^{1+s}+2(-\Delta_k)^{1+s}(1-B^2)\\
      &\quad\geq 4k^2A(-\Delta_k)^{s}A+\delta(s)(-\Delta_k)^{s},
      \end{split}
   \end{equation}
  where $\delta(s)=4-\max((8s^2+7s+2)s/2,8s^2+8s-1)\in[0.52,4]$ and
  \beno
  \Delta_k=\pa_y^2-k^2,\quad  A=\sin y\,(1+\Delta_k^{-1}),\quad B= \cos y\,(1+\Delta_k^{-1}), \quad \Lambda_1=[\Delta_k, B].
    \eeno
 The quantity $\omega_{1k}$ which contains second order derivatives of $\omega_k$, in combination with \eqref{eq:intro7}, provides powerful control on the solution. See Lemma \ref{lem:t2Lamom-om1} for the details. 

 The fact that $\omega_{1k}$ is an almost conserved quantity can be seen from the equation
 \begin{equation}\label{eq:intro8}
     (\partial_t+ikB)\omega_{1k}=-4tk^4\Lambda_3\omega_k,
 \end{equation}
 where
 \begin{equation}\label{eq:intro9}
     \Lambda_3=(1-\Delta_k^{-1})\Delta_{k,-1}^{-1}\Delta_{k,1}^{-1},\quad \Delta_{k,s}=\mathrm{e}^{-\mathrm{i}s y}\Delta_k\mathrm{e}^{\mathrm{i}s y}= (\partial_y+\mathrm{i}s)^2-k^2. 
 \end{equation}
 Indeed, analogous to inviscid damping estimates (decay of velocity and stream functions), it can expected that $\Lambda_3\omega_k$ enjoys relatively fast decay, so that the right hand side of \eqref{eq:intro8} becomes integrable over time which implies the global boundedness of $\omega_{1k}$. 

 The proof of \eqref{eq:intro7} is quite involved, and is based on the observation that through a crucial decomposition (see Lemma \ref{lem8}) the coercive inequality can be transformed into a form amenable to Fourier analysis, which can then be deduced from a real analysis lemma on the Fourier coefficients, see Lemma \ref{lem9}. The proof of Lemma \ref{lem9}, which are essentially inequalities about two explicitly defined positive sequences, is nontrivial and interesting by itself. However, it seems worthwhile to further explore a more conceptual understanding of \eqref{eq:intro7}.
 

\subsubsection{Spacetime estimates}
Thanks to Theorem \ref{thm:LE} and \eqref{eq:intro5}, we have precise control on the approximate solution $\omega_L$. It remains to control $\omega_e$ which satisfies the equation \eqref{eq:intro3}. 
To control $\omega_e$, our primary tool is the following spacetime estimates for the linearized Navier Stokes system:
\begin{align}\label{eq:intro11}
  \partial_tf-\nu \Delta f+\cos y\,\partial_x(f+\Delta^{-1}f)=\partial_y g_1+\partial_xg_2+g_3+g_4.
\end{align}

For an interval $I$, we introduce the space-time norms as
 \begin{align}\label{eq:intro12}
    \|f\|_{X_I}=&\|f\|_{L^\infty(I;L^2)}+\nu^{\f12}\|\nabla f\|_{L^2(I;L^2)}+ \||D_x|^{\f12}\nabla \Delta^{-1}f\|_{L^2(I;L^2)} \\ \notag&+ \||\sin y|^{\f12}\partial_x\nabla \Delta^{-1}f\|_{L^2(I;L^2)}.
 \end{align}
Then we have  for $P_0g_3=0$ and $I= [t_1,t_2]$, 
\begin{equation}\label{eq:intro13}
\|f\|_{X_I}\leq C\big(\|f(t_1)\|_{L^2}+\nu^{-\f12}\|(g_1,g_2)\|_{L^2(I;L^2)}+\||D_x|^{-\f12}\nabla g_3\|_{L^2(I;L^2)}+\|g_4\|_{L^1(I;L^2)}\big).
\end{equation}
See Proposition \ref{prop:LNS-sp1} for the fuller version when the background shear is a small perturbation of Kolmogorov and is time dependent. 

The bounds \eqref{eq:intro13} provide sharp and robust estimates to treat the nonlinear terms in the equation \eqref{eq:intro3} for $\omega_e$ (with the important exception of the \textit{reaction term}, see discussions below). We remark that the norm $X_I$ is exactly at the energy level regularity, which combines effectively with energy estimates that also play a crucial role especially in an initial stage of evolution.  

The proof of \eqref{eq:intro13} is based on resolvent estimates. The connection between spacetime estimates and resolvent bounds may be seen most directly when we consider only $g_3$ term (while setting initial data as well as all other inhomogeneous terms as identically zero) and focus on the the last two components of the $X_I$ norm \eqref{eq:intro12}.  Take a Fourier transform in $t$ for the equation \eqref{eq:intro11}. Denoting $\widehat{f}(\tau), \widehat{\,g_3}(\tau)$ as the Fourier transform of $f$ and $g_3$ respectively,  then we have the equation on $\T_{2\pi}$,
\begin{equation}\label{eq:intro13.5}
  (i\tau   -\nu\Delta)\widehat{f}(\tau)+\cos y\,\partial_x (1+\Delta^{-1})\widehat{f}(\tau)=\widehat{\,g_3}(\tau).
\end{equation}

By the Plancherel identity, we see that in this special case \eqref{eq:intro13} is reduced to the following \textit{resolvent bounds} for all $\tau\in\R$,
\begin{equation}\label{eq:intro14}
    \begin{split}
&\||D_x|^{\f12}\nabla\Delta^{-1}\widehat{f}(\tau)\|_{L^2}+\||\sin y|^{\f12}\partial_x\nabla\Delta^{-1}\widehat{f}\|_{L^2}\leq   C\||D_x|^{-\f12}\nabla \widehat{\,g_3}(\tau)\|_{L^2}.
    \end{split}
\end{equation}
The proof of \eqref{eq:intro14} is  based on energy estimates and the limiting absorption principle. We remark that for \eqref{eq:intro14} to hold, the nonlocal term $\cos y\, \partial_x\Delta^{-1}\widehat{f}(\tau)$ is essential since it plays a decisive role in the vorticity depletion phenomenon. 

\subsubsection{The three time scales}
Our proof distinguishes three timescales $T_0=\nu^{-1/6},\,T_1=\nu^{-4/9}$ and $T_2=\nu^{-1}$, which are naturally adapted to the dynamics of the solution and are motivated by the following considerations. 

To treat the nonlinear terms in \eqref{eq:intro3} we use energy estimates and spacetime bounds \eqref{eq:intro13}. Notice that in applying \eqref{eq:intro13}, we treat all nonlinear terms essentially as perturbations, including the transport term $u_L\cdot\nabla \omega_e$ which involves one derivative of $\omega_e$. However, the transport effect is still significant at the initial stage of evolution for $t\in[0, T_0]$ and we need to take advantage of the transport structure via energy estimates. After $t=T_0$, due to the decay of $u_L$ and the viscous regularization effect, the term $u_L\cdot\nabla\omega_e$ becomes perturbative and can be bounded using the spacetime estimates \eqref{eq:intro13}. 

Over the time interval $t\in[T_0, T_1]$ we apply the spacetime estimates \eqref{eq:intro13} to control the evolution of $\omega_e$. The main difficulties are that there are many terms which require us to apply different components of the inequality \eqref{eq:intro13}, and we need to take advantage of several crucial cancellations between the nonlinear terms, including a delicate treatment of the reaction term. At $t=T_1$,  the solution reaches a \textit{milestone}, in that the approximate solution $\omega_L$ has decayed to the same order of magnitude as the nonlinear part $\omega_e$. Indeed, we have
$$\langle\nu t^3\rangle^{-1}\sim \nu^{1/3}, \quad{\rm when}\,\,t=T_1=\nu^{-4/9}.$$
Afterwards it is no longer meaningful to distinguish the two parts. 

For $t\in[T_1, T_2]$, as discussed above, we do not have to perform the quasilinear approximation and instead can use the spacetime estimates \eqref{eq:intro13} to directly bound $\omega$. An important new aspect of the analysis is to explicitly track the enhanced dissipation factor $\mathrm{e}^{-c\nu^{1/2}t}$ for $\omega$ which becomes significant for $t\ge\nu^{-1/2}$. For $t\ge T_2$, the background shear flow starts to substantially decay with the rate $\mathrm{e}^{-c\nu t}$, and the stabilizing effects (inviscid damping, vorticity depletion and enhanced dissipation) gradually become weaker. We note that the non-shear components of our solution is exceedingly small, less than $\mathrm{e}^{-c\nu^{-1/2}}$ after $t=T_2$. As a consequence, the flow is dominated by the slow viscous decay of the shear flow, which is attracted to the Kolmogorov flow over a time scale $t\sim \nu^{-1}$. 

\subsubsection{The reaction term}
While the spacetime estimates \eqref{eq:intro13} are sufficient to bound most of the nonlinear terms in equation \eqref{eq:intro3}, there is one important exception, namely the reaction term $t\sin y\,u_e^y\partial_x\omega_L$, which needs a considerably more subtle treatment and is a key difficulty of our argument. The same difficulty also appeared in \cite{WZ-pre} by the last two authors for the Couette flow and we follow the same strategy here, but the argument becomes significantly more complicated technically.

The reaction term $t\sin y \,u_e^y\partial_x\omega_L$ comes from the nonlinearity $u^y_e\partial_y\omega_L$ through the identity
\begin{equation}\label{eq:intro15}
\begin{split}
  u^y_e\partial_y\omega_L&=u_e^y(\partial_y-t\sin y\,\partial_x)\omega_L+t\sin y \,u_e^y\partial_x\omega_L. 
\end{split}    
\end{equation}
As is well known, for equation \eqref{eq:intro3} the derivative $\partial_y-t\sin y\,\partial_x$, which is adapted to the transport structure, can be considered as a ``good derivative" and is more manageable. In particular, for our purposes, $u_e^y(\partial_y-t\sin y\,\partial_x)\omega_L$ can be bounded through the spacetime estimates \eqref{eq:intro13}.  

The main difficulty is that $t\sin y\,u_e^y\partial_x\omega_L$ is growing over time, and spacetime estimates \eqref{eq:intro13} are not sufficient to control it. To overcome this difficulty, we take Fourier transform in $x$ in the linearized equation with the reaction term as the right hand side, and obtain that
\begin{equation}\label{eq:intro16}
    \partial_tf_k-\nu\Delta_kf_k+\mathrm{i}k\cos y\,(1+\Delta_k^{-1})f_k=\sum_{l\in\frac{2\pi}{\tp}\Z}(k-l)l t\sin y\,\phi_{e,l}\omega_{L,k-l}.
\end{equation}
The first important observation is that we can write
\begin{equation}\label{eq:intro17}
    \phi_{e,l}\omega_{L,k-l}(t,y):=g_{k,l}(t,y)\mathrm{e}^{-\mathrm{i}t(k-l)\cos y},
\end{equation}
where the ``profile" $g_{k,l}$ has bounded first derivative, see  Theorem \ref{thm:LE}. 

The heart of the matter is that the forcing term in \eqref{eq:intro16} is always \textit{non-resonant} with the intrinsic oscillation of the equation which is given by $\mathrm{e}^{-\mathrm{i}kt\cos y}$, since $l\neq0$. To see how this non-resonance property can be used to improve the estimates on $f_k$, let us consider the model equation for $t\ge0, y\in\R$,
\begin{equation}\label{eq:intro18}
    \partial_tf_k+\mathrm{i}kyf_k=g(t,y)\mathrm{e}^{-\mathrm{i}(k-l)yt},\quad f_k|_{t=0}\equiv0.
\end{equation}
For $l\neq0$, using transform in $y$, it follows from straightforward calculation that
\begin{equation}\label{eq:intro19}
    \begin{split}
        \|f_k(t,y)\|_{L^2(y\in\R)}&\lesssim \|\widehat{\,f_k}(t,\xi-kt)\|_{L^2(\xi\in\R)}\lesssim \Big\|\int_0^t\widehat{\,g\,\,}(s,\xi-ls)\,ds\Big\|_{L^2(\xi\in\R)}\\
        &\lesssim |l|^{-1/2}\|\langle\xi-ls\rangle\widehat{\,g\,\,}(s,\xi-ls)\|_{L^2([0,t]\times\R)}\lesssim \|(g,\partial_yg)\|_{L^2([0,t]\times\R)}. 
    \end{split}
\end{equation}
The key point is that we can use $L^2_t$ instead of $L^1_t$ type of norm on $g$ to bound $f_k$, albeit at the expense of using one extra derivative in $y$ of $g$ (which we can control).

To implement the above idea in our case where the main oscillatory factor $\cos y$ is a non-monotonic function, we need to take a slightly generalized form of Fourier transform (corresponding to the usual Fourier transform in a new coordinate where $\cos y$ becomes linear). The fact that this transform is ``degenerate" at the critical points of $\cos y$ causes many technical difficulties, which we are able to overcome using also the vanishing property of the reaction term coming from the vorticity depletion effect of $\omega_L$. We also need to show that the viscous term $\nu\Delta$ and the nonlocal term $\mathrm{i}k\cos y \,\Delta_k^{-1}$ can be bounded perturbatively for our purposes. We refer to section \ref{sec:reactionpart} for the details.

\subsubsection{Dimension counting, anisotropy and time dependent coefficients}\label{sec:dimension}
Our proof, although conceptually quite simple and robust, involves a large number of nontrivial estimates and significant amount of calculations. In the hope of making it more accessible, we share the following (non-rigorous) heuristics and additional comments on some important technical aspects not discussed above.

\begin{itemize}
    \item We note the following heuristics (when close to the critical point $y=0$) in terms of regularity
    \begin{equation}\label{eq:intro20}
        \partial_y\sim \partial_x\sim k\sim kt y, \quad \nu \partial_x^2\sim \partial_t. 
    \end{equation}
The first two and the last comparisons are natural; the third comparison is motivated by the fact that $\partial_y-\mathrm{i}kt\sin y\sim \partial_y-\mathrm{i}kty$ is a ``good derivative". Inspired by \eqref{eq:intro20}, we introduce the following dimension counting
\begin{equation}\label{eq:intro21}
    x\sim L,\quad y\sim L, \quad k\sim L^{-1},\quad t\sim L^{-1}, \quad \nu\sim L^3. 
\end{equation}
In the above, we may regard $L$ as certain base (length) unit. Adopting \eqref{eq:intro21}, it is easy to check that all the spacetime estimates \eqref{eq:intro13} and the linear estimates (except the fifth inequality) in Theorem \ref{thm:LE} (at the highest regularity of $M_k$) are  dimensionally consistent, which is an indication of their sharpness. 

\item From linear inviscid damping estimates where $u^x$ and $u^y$ have different rates of decay, we can expect certain anisotropy in the $x$ and $y$ directions, despite them having the same dimension. This anisotropy manifests itself in many of our estimates, including the linear estimates in Theorem \ref{thm:LE} and the bounds \eqref{eq:intro13}.

\item For the sake of clarity and simplicity, our discussion above have ignored two important complications in the linear estimates: (I) the background shear flow is not exactly the Kolmogorov flow due to the perturbation, and (II) the linearized
equation has coefficients that are time dependent. These issues need to be addressed to extend the spacetime estimates to the full linearized equation. Our main idea is, as expected, to treat the full linearized equation as perturbations of equation \eqref{eq:intro11}. To achieve high accuracy in various approximations, we modify our equations in a number of nontrivial ways, for example, through change of variables in \eqref{eq:changeofcoordinate} to bring the perturbed shear flow closer to the Kolmogorov flow; through renormalization of time in section \ref{sec:spacetime} to absorb the decay factor $\mathrm{e}^{-\nu t}$ in the Kolmogorov flow; and through dividing into smaller time intervals. A guiding principle in these approximation schemes is to prioritize extracting the precise oscillation of our solution, since small error in the phase can lead to nontrivial growth over large times.
\end{itemize}

\subsubsection{Organization of the paper}
The rest of our paper is organized as follows. In section \ref{sec:mainprop} we present the main propositions and give the proof of Theorem \ref{thm:main} assuming these propositions. In section \ref{sec:reseuler} - section \ref{sec:spacetime} we prove the required linear estimates. In section \ref{appro} we prove the main properties of the approximate solution used in section \ref{sec:mainprop}. In section \ref{sec:nonlinearenergy} we prove all the nonlinear estimates used in section \ref{sec:mainprop} except for the ones for the reaction term. In section \ref{sec:reactionpart} we prove the bounds on the reaction term and complete the proof of Theorem \ref{thm:main}. Lastly, the three appendices contain some technical results used in the proof.

\section{Main propositions and proof of Theorem \ref{thm:main}}\label{sec:mainprop}

First of all, we need to introduce a multiple timescale analysis based on different behaviors of the solution in different time regimes. Namely, we divide the whole time interval into multiple  regimes: 
short-time regime $[0,T_0]$ with $T_0=\nu^{-\f16}$, medium-time regime $[T_0,T_1]$ with $T_1=\nu^{-\f49}$,
and large-time regime $[T_1, T_2]$ with $T_2=\nu^{-1}$, and long-time regime $[T_2,+\infty)$.

Throughout this paper, we denote 
\beno
t_*=\int_0^t\mathrm{e}^{-\nu s}\mathrm{d}s=(1-\mathrm{e}^{-\nu t})/\nu,\quad b(y)=\cos y.
\eeno

\subsection{Quasilinear approximation} \label{sec:quasi}

Let  $V(y)=P_0U_0^x$. Thanks to our assumption on the initial condition, we have
\begin{align}
   &\|V-b\|_{C^3}\lesssim\|V-b\|_{H^4}\leq \|U_0^x-b\|_{H^4}\leq C\|\Omega_0{-\sin y}\|_{H^3}\leq C\epsilon_1\nu^{\f13},\label{ass:U-b}\\
   &\|V+V''\|_{C^1}+\|V'+V'''\|_{L^\infty}\lesssim \|V-b\|_{C^3}\lesssim \|U_0^x-b\|_{H^4}\lesssim \|\Omega_0{-\sin y}\|_{H^3}. \label{ass:U'-b'}
\end{align}
We then introduce the perturbation
\beno
u=U-(\mathrm{e}^{-\nu t}V(y),0):=(u^x,u^y),\quad \omega=\Omega+\mathrm{e}^{-\nu t}V'(y)=\partial_xu^y-\partial_yu^x. 
\eeno
It holds that {(here we normalize that $\int_{\mathbb{T}_{\mathfrak{p}}\times\mathbb{T}_{2\pi}}U=0 $, then
$\int_{\mathbb{T}_{\mathfrak{p}}\times\mathbb{T}_{2\pi}}u=0 $, $\int_{\mathbb{T}_{2\pi}}V(y)\mathrm{d}y=0 $)}
\begin{align}\label{eq:NS-vorticity-p}
\left\{
\begin{aligned}
   & \partial_t\omega-\nu\Delta \omega+\mathrm{e}^{-\nu t}V(y)\partial_x\omega-
   \mathrm{e}^{-\nu t}V''(y)\partial_x\phi+u\cdot\nabla \omega=-\nu\mathrm{e}^{-\nu t}(V'+V'''),\\
   &u=(-\partial_y,\partial_x)\phi,\quad \phi=\Delta^{-1}\omega,\\
   &\om|_{t=0}=\Om_0(x,y)+V'(y):=\om_0=P_{\neq}\Om_0.
   \end{aligned}\right.
\end{align}

The first key step of the quasilinear approximation method is to construct a good approximate solution to the linearized Navier-Stokes system. Compared with the case of Couette flow, this construction is far more subtle, and proving the associated properties is significantly more challenging.

In order to apply Theorem \ref{thm:LE} to derive the desired estimates for the approximate solution, we introduce a crucial transform for the flow $V(y)$.
Thanks to $\|V-b\|_{H^4}\le \epsilon_1\nu^{\f13}\le \epsilon_0$, by {Morse} type Lemma \ref{lem:genV} , there exist constants $a, d$ and the function $\theta(y):\T_{2\pi}\to \T_{2\pi}$ with $\|\theta(y)-y\|_{H^3}\le C\|V-b\|_{H^4}\ll1$  such that 
\begin{equation}\label{eq:changeofcoordinate}
V(y)=ab(\theta(y))+d.
\end{equation} 

 Let $\widetilde{w}_k$ solve the linearized Euler equation around Kolmogorov flow:
 \begin{align}\label{def:tildewk}
 \partial_t\widetilde{w}_k+\mathrm{i}k b(y)(\widetilde{w}_k+\widetilde{\psi}_k)=0,\quad \Delta_k\widetilde{\psi}_k=\widetilde{w}_k,\quad 
     \widetilde{w}_k|_{t=0}=\widetilde{w}_{0,k},\ k\neq 0,
 \end{align} 
 where
 \begin{align}\label{def:w0k}
       & w_{0,k}(y)=\dfrac{1}{\mathfrak{p}}\int_{\mathbb{T}_{\mathfrak{p}}}\omega_0(x,y)\mathrm{e}^{-\mathrm{i}kx}\mathrm{d}x,\quad
       w_{0,k}(y)=\widetilde{w}_{0,k}(\theta(y)).
    \end{align}

 Let
  \begin{align}
\gamma_1( t)=\int_{0}^{t}t_{*}^2dt=\dfrac{1}{\nu}\left[ t+\dfrac{1}{2\nu}(1-\mathrm{e}^{-2\nu t})-\dfrac{2}{\nu }(1-\mathrm{e}^{-\nu t})\right].\label{def:gamma1}
\end{align} 
 It is easy to verify that for $0\leq t\leq T_1=\nu^{\f49}$, $\nu$ small enough,
 \beno
  \f {t^3} 4\leq\gamma_1(t)\leq \f {t^3} 3.
 \eeno
 We introduce 
 \begin{align}
&w_{k,1}(t,y)=\widetilde{w}_{k}(at_*,\theta(y))\mathrm{e}^{-\mathrm{i}kdt_*},\quad
     \psi_{k,1}(t,y)=\widetilde{\psi}_{k}(at_*,\theta(y))\mathrm{e}^{-\mathrm{i}kdt_*},\label{def:wk1}\\
     &w_{k,2}=w_{k,1}\mathrm{e}^{-\nu k^2\gamma_1(t)|V'|^2},\quad \psi_{k,2}=\Delta_k^{-1}w_{k,2}.\label{def:wk2}
\end{align}
Now the approximate solution $\om_L$ is defined by 
\begin{align}
\omega_L(t,x,y)=\sum_{k\in \Lambda_*}w_{k,2}(t,y)\mathrm{e}^{\mathrm{i}kx}.\label{def:omL}
\end{align}
where
\begin{align}
&\Lambda_{*}=[-\nu^{-\f13},\nu^{-\f13}]\cap\alpha\mathbb{Z}\setminus\{0\},\quad \alpha=2\pi/\mathfrak{p}.\label{def:Lambda*}
\end{align}
Let $\phi_L=\Delta^{-1}\om_L$ and $u_L=(-\partial_y,\partial_x)\phi_L=\big(u^x_L, u^y_L\big)$ given by 
\begin{align}
u_L^x(t,x,y)=-\sum_{k\in \Lambda_*}\partial_y\psi_{k,2}(t,y)\mathrm{e}^{\mathrm{i}kx},\quad u_L^y(t,x,y)=\sum_{k\in \Lambda_*}\mathrm{i}k\psi_{k,2}(t,y)\mathrm{e}^{\mathrm{i}kx}.\label{def:u1L}
\end{align}
We denote
\begin{align}\label{def:Er-L}
  Er_{L}=\partial_t\omega_L-\nu \Delta\omega_L+\mathrm{e}^{-\nu t}V\partial_x(\omega_L+\phi_L).
   \end{align}

 Let $\eta$ be a smooth cut-off function such that $\eta(s)=1$ for $|s|\le 1$ and $\eta(s)=0$ for $|s|\ge 2$. We denote
\begin{align}
&w_{k*}=(1-\eta(\sqrt{t}V'))w_{k,2}=(1-\eta(\sqrt{t}V'))w_{k,1}\mathrm{e}^{-\nu k^2\gamma_1(t)|V'|^2},\label{def:wk*}\\
 &w_{k,3}=\eta(\sqrt{t}V')w_{k,2}=\eta(\sqrt{t}V')w_{k,1}\mathrm{e}^{-\nu k^2\gamma_1(t)|V'|^2},\label{def:wk3}
\end{align}
and 
\begin{align}
\widetilde{\omega}_L(t,x,y)=\sum_{k\in \Lambda_*}w_{k,3}(t,y)\mathrm{e}^{\mathrm{i}kx},\quad\omega_L^*(t,x,y)=\sum_{k\in \Lambda_*}w_{k*}(t,y)\mathrm{e}^{\mathrm{i}kx}.\label{def:tilomL}
\end{align}
Thus, $\om_L=\widetilde{\om}_L+\om_L^*$.

In Section 7,  based on various sharp inviscid damping and vorticity depletion estimates for the linearized Euler system(namely, Theorem \ref{thm:LE}), we will prove that the approximate solution $\om_L$ enjoys the following important properties, which will play a crucial role in the energy estimates. 

 \begin{Proposition}\label{prop:app}
 It holds that for $ t\leq  T_1=\nu^{-\f49}$, 
    \begin{align}
        &\langle \nu t^3\rangle\big(\|\partial_x\omega_L(t)\|_{L^\infty}+ \|\nabla u_L(t)\|_{L^\infty}\big)
        +\langle \nu t^3\rangle^{1/2}\|(\partial_y+t_*V'\partial_x)\omega_L(t)\|_{L^\infty}\leq C\|\omega_{0}\|_{H^3},\label{est:paromL-inf}\\
        &\|(u^x_L-t_*V'u^y_L)(t)\|_{L^{\infty}} \leq C\big(\langle t\rangle^{-1.2}+| \nu t^3|^{1/2}\langle t\rangle^{-1}\big)\|\omega_{0}\|_{H^3},\label{est:u2L2}\\
        &|u_L^x(t,x,y)|\leq C\big(\langle t\rangle^{-3/2}+\langle t\rangle^{-1}|V'|\big)\|\omega_0\|_{H^3},\label{est:u1Linf}\\
        &\|u^y_L(t)\|_{L^{\infty}} \leq C\langle \nu t^3\rangle^{1/2}\langle t\rangle^{-2}\|\omega_{0}\|_{H^3},\label{est:u2Linf}\\
        &\|u_L\cdot\nabla\omega_L\|_{{L^\infty}}\leq C\big(\langle t\rangle^{-1.2}+| \nu t^3|^{1/2}\langle \nu t^3\rangle^{-1}\langle t\rangle^{-1}\big)\|\omega_{0}\|_{H^3}^2,\label{est:uLnawL}\\
        &\|t_*V'\partial_x\widetilde{\omega}_L(t)\|_{L^\infty}\leq C\|\omega_{0}\|_{H^3},\quad\|t_*V'\partial_x\omega_L(t)\|_{L^\infty}\leq Ct\langle \nu t^3\rangle^{-1}\|\omega_0\|_{H^3}.\label{est:t*Vparw}
    \end{align}
\end{Proposition}

Furthermore,  we have the following estimate for the error $Er_L$.

\begin{Proposition}\label{prop:error}
The error $Er_L$ can be decomposed as $Er_{L}=Er_{L1}+Er_{L2}+Er_{L3}$, where it holds that for $t\leq T_1$, 
\begin{align*}
&\|Er_{L1}(t)\|_{L^2}\leq C\nu\langle t\rangle\langle \nu t^3\rangle^{-1}\|\omega_{0}\|_{H^3},\\
& \|Er_{L2}(t)\|_{L^2}\leq C\nu^{1/3}\langle t\rangle^{-2}\|\omega_{0}\|_{H^3},\\
&\||D_x|^{\f12}\nabla Er_{L3}(t)\|_{L^2}\leq  C\langle t\rangle^{-2}|\nu t^3|^{1/2}\|\omega_{0}\|_{H^3}.
\end{align*}
\end{Proposition}

Next, it is natural to introduce
\begin{align*}
   & \omega_e=\omega-\omega_L,\quad u_e=u-u_{L},\quad \phi_e=\phi-\phi_L.
\end{align*}
Then we have
\begin{equation}\label{eq:ome}
  \left\{\begin{aligned}
  &\partial_t\omega_e-\nu\Delta\omega_e+\mathrm{e}^{-\nu t}V\partial_x\omega_e-\mathrm{e}^{-\nu t}V''\partial_x\phi_e+ u\cdot\nabla \omega_e+u_e\cdot\nabla \omega_L+Er=0,\\
  &u_e=(-\partial_y,\partial_x)\phi_e,\quad \phi_e=\Delta^{-1}\omega_e,
  \quad\omega_e|_{t=0}=\om_0-\om_{L}(0),
  \end{aligned}
  \right.
\end{equation}
where
\begin{align}\label{def:Er}
 Er=Er_{L}-\mathrm{e}^{-\nu t}(V''+V)\partial_x\phi_L+u_L\cdot\nabla\omega_L+\nu\mathrm{e}^{-\nu t}(V'+V''').
\end{align}

In the small-time regime($t\in [0,T_0$)), we have the following energy estimate. 

\begin{Proposition}\label{prop:om-e-short}
  There exists $ \epsilon_1\in(0,1/4)$ such that for all $0<\nu \leq 1$, if the initial vorticity $\omega_0$ satisfies $\|\omega_0\|_{H^3}\leq \epsilon_1\nu^{\f13}$, then it holds that for $t\leq T_0$,    \begin{align*}
      & \|\omega_e(t)\|_{L^2}\leq C\nu^{\f13}\|\omega_{0}\|_{H^3}.
   \end{align*}
 \end{Proposition}
 
 The proof relies on a direct energy estimate under a new inner product; see Section \ref{Energy-short-regime} for details.

 \subsection{Reaction equation and energy estimate in medium-time regime} 
To analyze the dynamics of the solution in the regime $t\in [T_0,T_1]$,  the second key step of the quasilinear approximation method is to introduce a reaction equation for controlling the reaction term$u_e\cdot\na \om_L$. 
Notice that
\begin{align*}
   &u_e\cdot\nabla\omega_L+t_*V'u^y_e\partial_x\omega_L= u^x_e\partial_x\omega_L +u^y_e(\partial_y+t_*V'\partial_x)\omega_L.
\end{align*}
We denote
\begin{align*}
   & \mathrm{P}_{\text{low}}f:=\sum_{0<|k|\leq\nu^{-1/3},k\in \alpha\mathbb{Z}}f_k(y)\mathrm{e}^{\mathrm{i}kx},\quad \quad \mathrm{P}_{\text{high}}f:=\sum_{|k|>\nu^{-1/3},k\in \alpha\mathbb{Z}}f_k(y)\mathrm{e}^{\mathrm{i}kx}.
\end{align*}
where  $\alpha=2\pi/\mathfrak{p}>1$.  
Then we introduce the reaction part $\omega_{re}$ for $T_0\leq t\leq T_1$, which solves 
\begin{align}\label{eq:reaction}
   & \partial_t\omega_{re}-\nu \Delta \omega_{re}+\mathrm{e}^{-\nu t}V\partial_x(\omega_{re}+\Delta^{-1}\omega_{re}) -t_*V'\mathrm{P}_{\text{low}}u^y_e\partial_x\omega_L^* =0, \quad\omega_{re}|_{t=T_0}=0.
\end{align}

Let $\omega_* =\omega_e-\omega_{re}$, which solves
\begin{align}\label{eq:om-star}
   &  \partial_t\omega_{*}-\nu \Delta \omega_{*}+\mathrm{e}^{-\nu t}V\partial_x(\omega_{*}+\Delta^{-1}\omega_{*}) +G=0, \quad\omega_{*}|_{t=T_0}=\omega_{e}|_{t=T_0},
\end{align}
where the source term $G$ is given by 
\begin{align}\label{def:G}
 G=&u\cdot\nabla \omega_e +{u^x_e\partial_x\omega_L}  +u^y_e(\partial_y+t_*V'\partial_x)\omega_L-\mathrm{e}^{-\nu t}(V+V'')\partial_x\phi_e\\ \notag&+Er-t_*V'P_{\text{high}}u^y_e\partial_x\omega_L-t_*V'P_{\text{low}}u^y_e\partial_x \widetilde{\omega}_L.
\end{align}

To control $\om_{*}$, we need to establish the inviscid damping and enhanced dissipation estimates for the following linearized Navier-Stokes system
 for {$t\le \nu^{-1}$}:
\begin{align}\label{eq:partf-g-1}
  \partial_tf-\nu \Delta f+\mathrm{e}^{-\nu t}V\partial_x(f+\Delta^{-1}f)=\partial_y g_1+\partial_xg_2+g_3+g_4.
\end{align}

For an interval $I$, we introduce the space-time norms as
 \begin{align}
     \label{def:XT}  \|f\|_{X_I}=\|f\|_{X_{I,\nu}}=&\|f\|_{L^\infty(I;L^2)}+\nu^{\f12}\|\nabla f\|_{L^2(I;L^2)}+ \||D_x|^{\f12}\nabla \Delta^{-1}f\|_{L^2(I;L^2)}
      \\ \notag&+ \||V'|^{\f12}\partial_x\nabla \Delta^{-1}f\|_{L^2(I;L^2)}.
 \end{align}
 \begin{Proposition}\label{prop:LNS-sp1}
Let $f$ solve  \eqref{eq:partf-g-1} with $P_0g_3=0$. There exists $\epsilon_0>0$ such that if $V=V(y)$ with $\|V-\cos y\|_{H^4}\leq \epsilon_0\nu^{\f13}$, then it holds that for $I= [t_1,t_2]\subset[0,1/\nu]$, 
\begin{align*}
&\|f\|_{X_I}\leq C\big(\|f(t_1)\|_{L^2}+\nu^{-\f12}\|(g_1,g_2)\|_{L^2(I;L^2)}+\||D_x|^{-\f12}\nabla g_3\|_{L^2(I;L^2)}+\|g_4\|_{L^1(I;L^2)}\big).
\end{align*}
\end{Proposition}  

The proof  relies on the resolvent estimates established in Proposition \ref{prop:res-LNS}, see Section 6.1 for the details.

The control of the reaction part $\omega_{re}$ is the most difficult part in this paper.  Prior to applying Proposition \ref{prop:LNS-sp1}, we still need to complete several key preprocessing steps, such as further decomposing the time scales and establishing weighted estimates and inviscid damping estimates for the following transport equation
\beno
\partial_t{\om}+\mathrm{i}k\mathrm{e}^{-\nu t}V {\om}+f=0.
\eeno
Since $V$ is non-monotone, the proof is rather involved. See Section {9} for the details. 
  
   \begin{Proposition}\label{prop:reaction}
 Let $I=[T_0,T]$, $T_0\leq T\leq T_1$ and $\om_{re}$ solve \eqref{eq:reaction}. Then we have
 \begin{align*}
 \|\omega_{re}\|_{X_I} \leq C\nu^{-\f13}\|\omega_e\|_{X_I}\|\omega_0\|_{H^3}.
\end{align*} 
 \end{Proposition}

In Section 8, based on Proposition \ref{prop:LNS-sp1},  we can prove the following energy estimate for $\om_*$. 
 \begin{Proposition}\label{prop:om-star}
   Let $I=[T_0,T]$ and $\om_*$ solve  \eqref{eq:om-star}. Then we have
     \begin{align*}
        \|\omega_{*}\|_{X_I}\leq & C(\nu^{-\f23}\|\omega_e\|_{X_I}^2+\nu^{-\f13}\|\omega_e\|_{X_I}\|\omega_0\|_{H^3}+
 \nu^{\f13}\|\omega_{0}\|_{H^3} +\|\omega_0\|_{H^3}^2)\\
 &+C\nu^{\f12}|\ln(\nu)|\|\omega_0\|_{H^3} +C\|\omega_*(T_0)\|_{L^2}.
     \end{align*}Here the energy functional $X_I$ is defined in \eqref{def:XT}.
 \end{Proposition}

With the hand of  Proposition \ref{prop:om-star} and Proposition \ref{prop:reaction},  it is easy to obtain the following energy estimate for the remainder $\om_e$. 
  \begin{Proposition}\label{prop:om-e}
There exists $\epsilon_1\in(0,1/4)$ such that for all $0<\nu \leq 1$, if the initial vorticity $\omega_0$ satisfies $\|\omega_0\|_{H^3}\leq \epsilon_1\nu^{\f13}$, 
then it holds that for $T_0\leq t\leq T_1$, 
\begin{align*}
   & \|\omega_e(t)\|_{L^2}\leq C\nu^{\f13}\|\omega_0\|_{H^3}.
\end{align*}  
\end{Proposition}

\begin{proof}
The proof is based on a bootstrap argument. Recall that $X_I=X_{[T_0,T]}$.  We assume that $\|\omega_e\|_{X_{[T_0,T]}}\leq \varepsilon_0\nu^{\f23}$ for small $\varepsilon_0$ determined later, which holds for some $T>T_0$ by Proposition \ref{prop:om-e-short}(for $\epsilon_1$ small enough).

As $\omega_*(T_0)=\omega_e(T_0)$, we get by Proposition \ref{prop:om-e-short} that
\begin{align*}
   &\|\omega_*(T_0)\|_{L^2}=\|\omega_e(T_0)\|_{L^2}\leq C\nu^{\f13}\|\omega_0\|_{H^3}.
\end{align*}
Then it follows  from Proposition \ref{prop:om-star} that (for $I=[T_0,T]$)
\begin{align*}
\|\omega_*\|_{X_I}
\leq &C\big(\nu^{-\f23}\|\omega_e\|_{X_I}^2+\nu^{-\f13}\|\omega_e\|_{X_I}\|\omega_0\|_{H^3}+
\nu^{\f13}\|\omega_0\|_{H^3} +\|\omega_0\|_{H^3}^2\big)\\
&+C\nu^{\f12}|\ln(\nu)|\|\omega_0\|_{H^3}+ C\nu^{\f13}\|\omega_0\|_{H^3}\\
\leq &C\big(\nu^{-\f23}\|\omega_e\|_{X_I}^2+\nu^{\f13}\|\omega_0\|_{H^3}\big).
\end{align*}
By Proposition \ref{prop:reaction}, we have
\begin{align*}
   \|\omega_{re}\|_{X_I}\leq& C\nu^{-\f13}\|\omega_e\|_{X_I}\|\omega_0\|_{H^3}\leq C\nu^{\f13}\|\omega_0\|_{H^3}.
\end{align*}
As $\omega_e=\omega_{re}+\omega_{*}$, we infer that
 \begin{align}\nonumber
    \|\omega_e\|_{X_I} \leq &\|\omega_*\|_{X_I}+\|\omega_{re}\|_{X_I}\leq C\big(\nu^{-\f23}\|\omega_e\|_{X_I}^2+\nu^{\f13}\|\omega_0\|_{H^3}\big) .
 \end{align}
which implies    
 \begin{align}\nonumber
   \|\omega_e\|_{X_{[T_0,T]}}\leq C\nu^{\f13}\|\omega_0\|_{H^3}\leq  C\epsilon_1\nu^{\f23}.
 \end{align}
Taking $\epsilon_1=\varepsilon_0/(2C)$, we can conclude our result for $T=T_1$.
\end{proof}

\subsection{Energy estimate in large-time regime}

 For $t\ge T_1$, since $\|\om(T_1)\|_{L^2}$ is of order $\nu^{\f23}$, there is no need to perform the quasilinear approximation.
  
Let $V(t,y)$ solve $\partial_t(\mathrm{e}^{-\nu t}V)= \nu\partial_y^2(\mathrm{e}^{-\nu t}V)$ for $t\geq T_1$ with $V|_{t=T_1}=P_0U_0^x$. Then
\begin{align}
   &\|V(t,\cdot)-b\|_{H^4}\leq \|U_0^x-b\|_{H^4}\leq 2\|\Omega_0{-\sin y}\|_{H^3}\leq 2\epsilon_1\nu^{\f13},\label{ass:U-b1}\\
   &\partial_tV=\nu(\partial_y^2+1)V=\nu(\partial_y^2+1)(V-b), \nonumber\\
   &\|\partial_tV\|_{H^2}\leq\nu\|V-b\|_{H^4}\leq 2\epsilon_1\nu^{\f43}.\label{ass:U-b2}  
\end{align} 
{Here we used that $ \int_{\T_{2\pi}}V(T_1,y)\mathrm{d}y=0$, then $ \int_{\T_{2\pi}}V(t,y)\mathrm{d}y=0$.}
We then introduce the perturbation (for $t\in [T_1,+\infty)$, here $V'=\partial_yV$, $V''=\partial_y^2V$)
\beno
u=U-(\mathrm{e}^{-\nu t}V(t,y),0):=(u^x,u^y),\quad \omega=\Omega+\mathrm{e}^{-\nu t}V'(t,y)=\partial_xu^y-\partial_yu^x. 
\eeno
It holds that for $t\in [T_1,+\infty)$,
\begin{align}\label{eq:>T1}
\left\{
\begin{aligned}
   & \partial_t\omega-\nu\Delta \omega+\mathrm{e}^{-\nu t}V(t,y)\partial_x\omega-
   \mathrm{e}^{-\nu t}V''(t,y)\partial_x\phi+u\cdot\nabla \omega=0,\\
   &u=(-\partial_y,\partial_x)\phi,\quad \phi=\Delta^{-1}\omega.
   \end{aligned}\right.
\end{align}

  \begin{Proposition}\label{prop:omega-e-large}
There exists $\epsilon,\ \epsilon_2\in(0,1/4)$ such that for all $0<\nu \leq 1$, if $\|\omega(T_1)\|_{L^2}\leq \epsilon_2\nu^{\f23}$, then it holds that for $  t\in [T_1,1/\nu]$, 
\begin{align*}
   & \|\om_{\neq}(t)\|_{L^2}\leq C\mathrm{e}^{-\epsilon\nu^{\f12} t}\|\omega(T_1)\|_{L^2},\quad 
   \|\om(t)\|_{L^2}\leq C\|\omega(T_1)\|_{L^2}.
   \end{align*}  
\end{Proposition}

The proof is based on the inviscid damping and enhanced dissipation estimates for the following linearized Navier-Stokes system
\begin{align}\label{eq2}
  \partial_tf-\nu \Delta f+\mathrm{e}^{-\nu t}(V\partial_xf-V''\partial_x\Delta^{-1}f)=\partial_y g_1+\partial_xg_2.
\end{align} 
 \begin{Proposition}\label{prop:LNS-sp2}
There exists $\epsilon$, $\epsilon_0>0$ such that if $V=V(t,y)$, $\|V-\cos y\|_{H^4}\leq \epsilon_0\nu^{\f13}$, 
$\|\partial_tV\|_{H^2}\leq \epsilon_0\nu^{\f43}$ for $t\in[0,1/\nu]$, $0<\nu \leq 1$, and $f$ satisfies \eqref{eq2}, then for $I=[t_1,T]\subseteq [0,1/\nu]$, 
$\tilde{\epsilon}\in [0,\epsilon]$,  it holds that 
\begin{align*}
&\|f\|_{X_I}+\|\mathrm{e}^{\tilde{\epsilon}\nu^{\f12}t}P_{\neq}f\|_{X_I}\leq C\big(\|f(T_1)\|_{L^2}+\nu^{-\f12}\|\mathrm{e}^{\tilde{\epsilon}\nu^{\f12}t}(g_1,g_2)\|_{L^2(I;L^2)}\big).
\end{align*}
\end{Proposition}

Beyond the viscous timescale, namely $t\geq 1/\nu$, we have the following energy  bound.

\begin{Proposition}\label{prop2}
There exists $\epsilon_3,\nu_0\in(0,1/4)$ such that for all $0<\nu \leq \nu_0$, if $\|\omega(T_1)\|_{L^2}\leq \epsilon_3\nu^{\f23}$, then it holds that for $t\in [1/\nu,+\infty)$, 
\begin{align*}
   & \|\om_{\neq}(t)\|_{L^2}\leq C\mathrm{e}^{-\nu t}\|\omega_{\neq}|_{t=1/\nu}\|_{L^2},\quad 
   \|\om(t)\|_{L^2}\leq C\mathrm{e}^{-\nu t/2}\|\omega|_{t=1/\nu}\|_{L^2}.
   \end{align*}  
\end{Proposition}
\def\YI{Y_I}
The proof relies on the following space-time estimate for the linearized Navier-Stokes system \eqref{eq2}.  

\begin{Proposition}\label{Prop1}
Let $f$ solve \eqref{eq2} with $\int_{\mathbb{T}_{\mathfrak{p}}\times\mathbb{T}_{2\pi}}f=0$.There exists $\epsilon_0,\nu_0\in(0,1/4)$ such that if  $0<\nu \leq \nu_0$, $V=V(t,y)$, $\|V-\cos y\|_{H^4}\leq \epsilon_0\nu^{\f13}$, 
$\|\partial_tV\|_{H^2}\leq \epsilon_0\nu^{\f43}$ for $t\in[1/\nu,+\infty)$, then it holds that for $I=[1/\nu,T]$, 
\begin{align*}
&\|P_{\neq}f\|_{\YI}\leq 
C\big(\|P_{\neq}f|_{t=1/\nu}\|_{L^2}+\nu^{-\f12}\|\mathrm{e}^{\nu t}(g_1,g_2)\|_{L^2(I;L^2)}\big),\\
&\|\mathrm{e}^{\nu t/2}f\|_{L^{\infty}(I;L^2)}\leq C\big(\|f|_{t=1/\nu}\|_{L^2}+\nu^{-\f12}\|\mathrm{e}^{\nu t}(g_1,g_2)\|_{L^2(I;L^2)}\big).
\end{align*}
Here we define 
\begin{align*}
&\|f\|_{\YI}=\|\mathrm{e}^{\nu t}f\|_{L^{\infty}(I;L^2)}+\nu^{\f16}\|\mathrm{e}^{\nu t/2}\partial_x \nabla\Delta^{-1}f\|_{L^2(I;L^2)}
+\nu^{\f16}\|\mathrm{e}^{\nu t/2}\nabla\Delta^{-1}f\|_{L^2(I;L^\infty)}.
\end{align*}

\end{Proposition}

    \subsection{Proof of Theorem \ref{thm:main}} 

Let us complete the proof of Theorem \ref{thm:main} by admitting the above propositions.

\begin{proof}[Proof of Theorem \ref{thm:main}]
First, we consider $t\in[0,T_1]$. As $T_1=\nu^{-\f49}\leq \nu^{-\f12}$, we have $\nu t^2\leq 1$, $\nu t^3\leq t $ for $t\in[0,T_1]$.
By \eqref{est:u1Linf} and \eqref{est:u2Linf}, we get
\begin{align*}
    \|u_L(t)\|_{L^\infty}&\leq \|u^x_{L}(t)\|_{L^\infty}+ \|u^y_{L}(t)\|_{L^\infty}\\
   &\leq C\langle t\rangle^{-1}\|\omega_0\|_{H^3}+ C\langle \nu t^3\rangle^{1/2}\langle t\rangle^{-2}\|\omega_{0}\|_{H^3}\\
   &\leq C\langle t\rangle^{-1}\|\omega_0\|_{H^3}.
\end{align*}
 By  \eqref{est:paromL-inf} and $P_0\omega_L=0$, we have 
  \begin{align*}
     &\|\omega_L(t)\|_{L^2}\leq \|\partial_x\omega_L\|_{L^2}\leq C\|\partial_x\omega_L\|_{L^\infty}\leq C\langle \nu t^3\rangle^{-1}\|\omega_{0}\|_{H^3}.
  \end{align*}
Then we infer from Proposition \ref{prop:om-e-short} and Proposition \ref{prop:om-e} that 
\begin{align*}
   & \|\omega_{\neq}(t)\|_{L^2}\leq {\|\omega(t)\|_{L^2}\leq} \|\omega_{e}(t)\|_{L^2}+ \|\omega_{L}(t)\|_{L^2}\leq 
   C(\nu^{\f13}+\langle \nu t^3\rangle^{-1})\|\omega_0\|_{H^3},
\end{align*} 
and 
\begin{align*}
   \|u_{\neq}(t)\|_{L^2}& \leq \|u_{e}(t)\|_{L^2}+ \|u_{L}(t)\|_{L^2}\leq C(\|\omega_{e}(t)\|_{L^2}+ \|u_{L}(t)\|_{L^\infty})\\
   &\leq C(\nu^{\f13}+\langle t\rangle^{-1})\|\omega_0\|_{H^3},\end{align*}
and (recall that $\omega=\Omega+\mathrm{e}^{-\nu t}V'(y) $, $b(y)=\cos y$ and using \eqref{ass:U-b})
\begin{align*}&\|\Omega(t)-\mathrm{e}^{-\nu t}\sin y\|_{L^2}=\|\omega(t)-\mathrm{e}^{-\nu t}V'(y)+\mathrm{e}^{-\nu t}b'( y)\|_{L^2}\\ &\leq \|\omega(t)\|_{L^2}+\mathrm{e}^{-\nu t}\|V'-b'\|_{L^2}\leq C(\nu^{\f13}+\langle \nu t^3\rangle^{-1})\|\omega_0\|_{H^3}+\|V-b\|_{H^4}\\
&\leq C\|\omega_0\|_{H^3}+C\|\Omega_0-\sin y\|_{H^3}\leq C\|\Omega_0-\sin y\|_{H^3}.
\end{align*}
Thus, the result follows from the facts that $\Omega_{\neq}=\omega_{\neq}, U_{\neq}=u_{\neq}, \omega_0=P_{\neq}\Omega_0$ and $\nu^{\f12}t\leq \nu^{\f12}T_1\leq1$, $\nu t\leq 1$ for $t\in[0,T_1]$. 

Next, we consider $t\in [T_1,1/\nu]$. As $T_1=\nu^{-\f49}$, we have $\nu T_1^3=\nu^{-\f13}$ and then
 \begin{align*}
   & {\|\omega(T_1)\|_{L^2}}\leq  
   C(\nu^{\f13}+\langle \nu T_1^3\rangle^{-1})\|\omega_0\|_{H^3}\leq  
   C\nu^{\f13}\|\omega_0\|_{H^3}=C\nu^{\f13}\|P_{\neq}\Omega_0\|_{H^3}\leq C\epsilon_1\nu^{\f23}.
\end{align*}
Taking $\epsilon_1 $ smaller if necessary, we have $\|\omega(T_1)\|_{L^2}\leq \epsilon_2\nu^{\f23} $. Then it follows from  Proposition \ref{prop:omega-e-large} that
\begin{align*}
   &\|U_{\neq}(t)\|_{L^2}\leq \|\Omega_{\neq}(t)\|_{L^2}=\|\omega_{\neq}(t)\|_{L^2}\leq C\mathrm{e}^{-\epsilon\nu^{\f12} t}\|\omega(T_1)\|_{L^2}\leq C\mathrm{e}^{-\epsilon\nu^{\f12} t}\nu^{\f13}\|P_{\neq}\Omega_0\|_{H^3},\\
   &{\|\omega(t)\|_{L^2}}\leq C\|\omega(T_1)\|_{L^2}{\leq C\nu^{\f13}\|P_{\neq}\Omega_0\|_{H^3}}\leq C\epsilon_1\nu^{\f23},
\end{align*}and by \eqref{ass:U-b1},
\begin{align*}&\|\Omega(t)-\mathrm{e}^{-\nu t}\sin y\|_{L^2}=\|\omega(t)-\mathrm{e}^{-\nu t}V'(t,y)+\mathrm{e}^{-\nu t}b'( y)\|_{L^2}\\ 
&\leq \|\omega(t)\|_{L^2}+\mathrm{e}^{-\nu t}\|V'-b'\|_{L^2}\leq C\nu^{\f13}\|P_{\neq}\Omega_0\|_{H^3}+\|V-b\|_{H^4}\leq C\|\Omega_0-\sin y\|_{H^3}.
\end{align*}
Thus, the result follows from the facts that $t=\min(t,\nu^{-1})$, $\nu t\leq 1$ for $t\in[T_1,1/\nu]$. 

Finally, we consider $t\geq1/\nu$. 
Taking $\epsilon_1$ smaller if necessary, we have $\|\omega|_{t=1/\nu}\|_{L^2}\leq C\epsilon_1\nu^{\f23}\leq \epsilon_3\nu^{\f23} $, 
$0< \nu\leq \epsilon_1\leq \nu_0$. Then it follows from  Proposition \ref{prop2} that
\begin{align*}
   &\|U_{\neq}(t)\|_{L^2}\leq \|\Omega_{\neq}(t)\|_{L^2}=\|\omega_{\neq}(t)\|_{L^2}\leq C\mathrm{e}^{-\nu t}\|{\omega_{\neq}}|_{t=1/\nu}\|_{L^2}\leq C\mathrm{e}^{-\epsilon\nu^{\f12} /\nu-\nu t}\nu^{\f13}\|P_{\neq}\Omega_0\|_{H^3},\\
    &\|\omega(t)\|_{L^2}\leq C\mathrm{e}^{-\nu t/2}\|{\omega}|_{t=1/\nu}\|_{L^2}\leq C\mathrm{e}^{-\nu t/2}\nu^{\f13}\|P_{\neq}\Omega_0\|_{H^3},
    \end{align*}
and then
\begin{align*}
&\|\Omega(t)-\mathrm{e}^{-\nu t}\sin y\|_{L^2}=\|\omega(t)-\mathrm{e}^{-\nu t}V'(t,y)+\mathrm{e}^{-\nu t}b'( y)\|_{L^2}\leq \|\omega(t)\|_{L^2}+\mathrm{e}^{-\nu t}\|V'-b'\|_{L^2}\\ &\leq C\mathrm{e}^{-\nu t/2}\nu^{\f13}\|P_{\neq}\Omega_0\|_{H^3}+\mathrm{e}^{-\nu t}\|V-b\|_{H^4}\leq C\mathrm{e}^{-\nu t/2}\|\Omega_0-\sin y\|_{H^3}.
\end{align*}
Thus, the result follows from the fact that $-\epsilon\nu^{\f12} /\nu=-\epsilon\nu^{\f12}\min(t,\nu^{-1})$ for $t\geq1/\nu$.

This completes the proof of Theorem \ref{thm:main}. 
\end{proof}

\section{Resolvent estimate for the linearized Euler}\label{sec:reseuler}

 Throughout this section, we assume that the real valued function V satisfies  $\|V-b\|_{H^4}\leq \epsilon_0$ with $\epsilon_0 > 0$ sufficiently small.
Thus, there exist $y_1, y_2$ with the following properties(see the proof of Lemma \ref{lem:genV}):
 \begin{align*}
&V'(y_1)=V'(y_2)=0,\ \text{where}\ |y_1|\leq 1/10,\ |y_2-\pi|\leq 1/10,\\
& M=\max_{y\in\mathbb{T}} V(y)= V(y_1),\ m=\min_{y\in\mathbb{T}} V(y)=V(y_2).
 \end{align*}
 We denote
 \begin{align*}
 \theta(k,\lambda):=1+|k|\lambda_0^{1/2}+|\lambda|,\quad
 \lambda_0:=\min\{|\lambda-m|,|\lambda-M|\}.      
  \end{align*}

The main result of this section is the following resolvent estimate for the linearized Euler equations, which plays a crucial role in both the inviscid damping estimate and the resolvent estimate for the linearized Navier-Stokes equations. 

\begin{Proposition}\label{lem:rayleigh}
  If $f$ satisfies 
  \begin{align}\label{eq:Ray-eps-0}
   & V(f+\Delta^{-1}_k f)-(\lambda+\mathrm{i}\epsilon) f=F,
\end{align}
here $\epsilon\in\R\setminus\{0\}$, $\lambda\in\mathbb{R}$.
Then we have\begin{align*}
   & \theta(k,\lambda)\|(\partial_y,k)\Delta_k^{-1}f\|_{L^2} \leq C\|(\partial_y,k)F\|_{L^2}.
\end{align*}Here the constant $C$ is independent with $k,\ \epsilon$ and $\lambda$.
\end{Proposition}

We also need the following lemma.

\begin{lemma}\label{lemb}
  If $f$ satisfies \eqref{eq:Ray-eps-0}
with $\epsilon\neq 0$, then we have
\begin{align*}
   & |\epsilon||k|^2\|(\partial_y,k)\Delta_k^{-1}f\|_{L^2} \leq C\|(\partial_y,k)F\|_{L^2}.
\end{align*}
\end{lemma}\begin{proof}
Without lose of generality, we assume $ \epsilon>0$.
Let $w=f+\Delta^{-1}_k f$, $\phi=\Delta^{-1}_k f$. Then we have $w=(\partial_y^2-k^2+1)\phi$, $f=w-\phi$ and 
\begin{align*}
   & (V-\lambda-\mathrm{i}\epsilon)w+(\lambda+\mathrm{i}\epsilon)\phi=F,\quad
   \epsilon\|w\|_{L^2}^2-\epsilon\langle \phi,w\rangle=-\mathbf{Im}\langle F,w\rangle,\\
   &\epsilon\|w\|_{L^2}^2+\epsilon\|(\partial_y,\tilde{k})\phi\|_{L^2}^2=-\mathbf{Im}\langle F,w\rangle\leq 
      \| F\|_{L^2}\|w\|_{L^2},\quad \epsilon\|w\|_{L^2}\leq 
      \| F\|_{L^2},
      \end{align*}
      which implies that
      \begin{align*}
          & |\epsilon||k|^2\|(\partial_y,k)\Delta_k^{-1}f\|_{L^2}\leq |\epsilon||k|\|f\|_{L^2}\leq C|\epsilon||k|\|w\|_{L^2}\leq C|k|\|F\|_{L^2}
      \leq C\|F\|_{H_k^1}.
      \end{align*}
      Here $\tilde{k}=\sqrt{k^2-1}$. 
 \end{proof}

The following lemma gives the case of $\lambda\in\R\setminus[m,M]$ in Proposition \ref{lem:rayleigh}.
\begin{lemma}\label{lema}
  If $f$ satisfies \eqref{eq:Ray-eps-0}
with $\epsilon\neq 0$ and $\lambda\in\R\setminus[m,M]$, then we have 
\begin{align*}
   & \theta(k,\lambda)\|(\partial_y,k)\Delta_k^{-1}f\|_{L^2} \leq C\|(\partial_y,k)F\|_{L^2}.
\end{align*}Here the constant $C$ is independent of $k,\ \epsilon$ and $\lambda$.
\end{lemma}

\begin{proof}
Without lose of generality, we assume $ \lambda>M$, then $\lambda>M>1/2$.
Let $w=f+\Delta^{-1}_k f$, $\phi=\Delta^{-1}_k f$. Then $ w=(\partial_y^2-k^2+1)\phi$, $f=w-\phi$, 
$V-\lambda\leq m-\lambda\leq-\lambda_0$, and 
\begin{align*}
   & (V-\lambda-\mathrm{i}\epsilon)w+(\lambda+\mathrm{i}\epsilon)\phi=F,\\
   &\mathbf{Re}\langle F,w\rangle= \langle(V-\lambda)w,w\rangle+\lambda\langle\phi,w\rangle\leq
   -\lambda_0\|w\|_{L^2}^2+\lambda\langle\phi,w\rangle\leq
   \lambda\langle\phi,w\rangle\leq0,
 \end{align*}
 which gives  
   \begin{align*}
  C^{-1}\|\phi\|_{H_k^1}^2\leq&\|\partial_y\phi\|_{L^2}^2+(k^2-1)\|\phi\|_{L^2}^2=-\langle\phi,w\rangle\leq|\lambda|^{-1}|\langle F,w\rangle|\\ 
   \leq&|\lambda|^{-1}\|F\|_{H_k^1}\|w\|_{H_k^{-1}}\leq C|\lambda|^{-1}\|F\|_{H_k^1}\|\phi\|_{H_k^{1}}\\\leq& C(1+|\lambda|)^{-1}\|F\|_{H_k^1}\|\phi\|_{H_k^{1}}.
   \end{align*}   
  This shows that
  \beno
  \|\phi\|_{H_k^{1}}\leq C(1+|\lambda|)^{-1}\|F\|_{H_k^1}.
  \eeno
 We also have
 \begin{align*}
   &\lambda_0\|w\|_{L^2}^2\leq|\langle F,w\rangle|\leq \|F\|_{L^2}\|w\|_{L^2},\quad 
   \lambda_0\|w\|_{L^2}\leq\|F\|_{L^2}\leq|k|^{-1}\|F\|_{H_k^1},\\
   &\|\phi\|_{H_k^{1}}\leq C\|w\|_{H_k^{-1}}\leq C|k|^{-1}\|w\|_{L^2}\leq C\lambda_0^{-1}|k|^{-2}\|F\|_{H_k^1}.
\end{align*}
Thus, we obtain
\begin{align*}
  \|F\|_{H_k^1}\geq& C^{-1}(1+\lambda_0|k|^{2}+|\lambda|)\|\phi\|_{H_k^{1}}\\
  \geq& C^{-1}(1+\lambda_0^{1/2}|k|+|\lambda|)\|\phi\|_{H_k^{1}}=C^{-1}\theta(k,\lambda)\|\phi\|_{H_k^{1}}.
  \end{align*}
  
   This completes the proof.
\end{proof}
\subsection{Spectral properties}\label{subspec}
We observe first that $B_k: L^2(\T) \to L^2(\T)$ with $k\in\frac{2\pi}{\tp}\Z\backslash\{0\}$ has no eigenvalues $\lambda\in\mathbb{C}\backslash\R$, since $B_k$ is symmetric with respect to the inner product 
\begin{equation}\label{fg*}
    \langle f,g\rangle_\ast := \langle f, (1+\Delta_k^{-1})g\rangle_{L^2(\T)}.
\end{equation}
Therefore, the spectrum $\sigma(B_k)\subseteq\R$. 

Due to the non-monotonicity of $b(y) =\cos y$, we require the following definition of {``}generalized embedded eigenvalues" which was already introduced in \cite{WZZ-APDE}. The formulation below, applicable to general non-monotonic shear flows $(b(y), 0), y\in\T$ with non-degenerate critical points, can be found in \cite{BCJ2024}.

\begin{definition}\label{gemeig}
    $\lambda\in b(\T)$ is called a limiting singular value or generalized embedded eigenvalue of the operator for $k\in\frac{2\pi}{\tp}\Z\backslash\{0\}$, 
    $$L_k= b(y)-b''(y)\Delta_k^{-1}: L^2(\T) \to L^2(\T),$$
    if either (i) $\lambda$ is an embedded eigenvalue, or (ii) $\lambda$ is not a critical value of $b$ and there exists a nontrivial $\psi\in H^1(\T)$ such that on $\T$, 
    \begin{equation}\label{gemeig2}
        (k^2-\partial_y^2)\psi+{\rm P.V.}\,\frac{b''\psi}{b(y)-\lambda}+\mathrm{i}\pi\sum_{z\in\T,\, b(z)=\lambda} \frac{b''(z)}{|b'(z)|}\psi(z)\delta(y-z)=0.
    \end{equation}
\end{definition}

The significance of generalized embedded eigenvalues lies in their being the only possible accumulation points of the (approximate) eigenvalues of  $L_k$.

In this subsection, we prove that the Kolmogorov flow $b(y)=\cos y$ has neither eigenvalues nor generalized embedded eigenvalues. First, we note that for any $\lambda\in b(\T)$,  the second derivative $b''$ has the same sign on the set $\{z\in \T:\, b(z)=\lambda\}$. In addition, the two critical points are $y_{1}=0$, $y_{2}=\pi$,   with corresponding critical values $\lambda\in\{-1,1\}$. 

\begin{proof}
Suppose $\lambda$ is an eigenvalue or a generalized embedded eigenvalue. Using Definition \ref{gemeig}, we analyze several cases.\smallskip

{\it Case 1. $\lambda\in (-1,1)$}. Let $\psi\in H^1(\T)$ be a nontrivial solution to equation \eqref{gemeig2}. We multiply by $\overline{\psi}$ and integrate over $\T$. Then we obtain
\begin{equation}\label{exa1}
    \int_{\T} k^2|\psi|^2+|\partial_y\psi|^2+{\rm P.V.} \frac{b''\psi\overline{\psi}}{b(y)-\lambda}\,dy+\mathrm{i}\pi \sum_{z\in\T, b(z)=\lambda}\frac{b''(z)}{|b'(z)|}|\psi(z)|^2=0.
\end{equation}
Thus, \eqref{exa1} implies that $b''(z)\psi(z)=0$ by taking the imaginary part, since $b''=-b$ does not change sign on $\{z\in\T: \, b(z)=\lambda\}$. Equation \eqref{gemeig2} reduces to
 \begin{equation}\label{exa1.1}
        (k^2-\partial_y^2)\psi+\frac{b''\psi(y)}{b(y)-\lambda}=0,
    \end{equation}
    for $y\in\T$. 
If $\lambda=0$, then \eqref{exa1.1} implies $\psi\equiv 0$ since $|k|>1$, a contradiction. We can therefore assume that $\lambda\neq0$, which implies $\psi(z)=0$ on $\{z\in\T: \, b(z)=\lambda\}$. Assume without loss of generality that $\psi$ is real valued. By simple integration by parts, we have
\begin{equation}\label{exa2}
   \int_{\T}\frac{b''\psi^2(y)}{b(y)-\lambda}\,dy = -\int_{\T} \frac{2b'\psi\partial_y\psi}{b(y)-\lambda}\,dy+\int_{\T}\frac{(b')^2\psi^2}{(b(y)-\lambda)^2}\,dy. 
\end{equation}
It follows from \eqref{exa1.1}-\eqref{exa2} and the fact $\psi(z)=0$ on $\{z\in\T: \, b(z)=\lambda\}$ that
\begin{equation}\label{exa3}
    \int_{\T}k^2|\psi|^2+|\partial_y\psi|^2\,dy-\int_{\T} \frac{2b'\psi\partial_y\psi}{b(y)-\lambda}\,dy+\int_{\T}\frac{(b')^2\psi^2}{(b(y)-\lambda)^2}\,dy=0,
\end{equation}
which implies by Cauchy-Schwarz inequality that $\psi\equiv0$. Hence, $\lambda\in b(\T)\backslash\{-1, 1\}$ is not a generalized embedded eigenvalue. \smallskip

{\it Case 2. $\lambda\in\R\backslash [-1,1]$.} In this case, we can use directly the equation \eqref{exa1.1} to arrive at a contradiction.\smallskip

{\it Case 3. $\lambda\in\{-1,1\}$. } In this case, $\lambda$ is an embedded eigenvalue. For the sake of concreteness, we assume that $\lambda=1$. Then there exists a nontrivial $\psi\in H^2(\T)$ such that 
\begin{equation}\label{exa4}
    (b(y)-1)(\partial_y^2-k^2)\psi-b''\psi=0,
\end{equation}
which implies that $\psi(0)=\psi'(0)=0$ and 
\begin{equation}\label{exa5}
    (\partial_y^2-k^2)\psi-\frac{b''}{b(y)-1}\psi=0,
\end{equation}
for $y\in\T$. Similar integration by parts as in {\it Case 1} implies that $\psi\equiv0$. Hence, $\lambda$ is not an embedded eigenvalue. 

To summarize Cases 1 through 3, the operator $B_k$ has no eigenvalues or generalized embedded eigenvalues.
\end{proof}

\subsection{The non-degenerate case} 
When $|\lambda|\leq 7/8$, we have $\lambda_0\leq C $, $\theta(k,\lambda)\leq C|k| $. The desired bound in Proposition \ref{lem:rayleigh} follows from the following stronger bound  
\begin{equation}\label{porb701}
       \big\|f\big\|_{H_k^{-1}}\leq C |k|^{-1}\big\|F\big\|_{H_k^{1}},
\end{equation}
where $f$ and  $ F$ satisfy \eqref{eq:Ray-eps-0}.
We need the following lemma. 
\begin{lemma}\label{porb8}
For $k\in \frac{2\pi}{\tp}\Z\backslash\{0\}$, $\lambda\in[-7/8,7/8]$, and $\epsilon\in\R$ with $0<|\epsilon|<1/10$ and $|c_0|<1/10$, we have the following bound for any $h\in H^1(\T)$,
\begin{equation}\label{porb8.1}
    \bigg\|(-\Delta_k)^{-1}\Big[\frac{h}{V-\lambda-\mathrm{i}\epsilon}\Big]\bigg\|_{H^1_k}\lesssim|k|^{-1}  \|h\|_{H^1_k}. 
\end{equation}
Moreover, for any sequences $\lambda_j\in[-7/8,7/8]$, $k_j\in \frac{2\pi}{\tp}\Z\backslash\{0\}$, $\epsilon_j\to 0$ with $0<|\epsilon_j|<1$, $V_j\to \cos y$ as $j\to\infty$ in $C^3(\T)$, and functions $h_j\in H^1_{k_j}(\T)$ with $\|h_j\|_{H^1_{k_j}(\T)}\leq1$, the sequence $(-\Delta_{k_j})^{-1}\Big[\frac{h_j}{V_j-\lambda_j-\mathrm{i}\epsilon_j}\Big]$ has a convergent subsequence in $H^1(\T)$. 
\end{lemma}

\begin{proof}
The proof is similar to that of the limiting-absorption principle in \cite{JiaG}; see Lemmas 3.1–3.2, and we shall therefore be brief.

Now we assume $\lambda\in[-7/8, 7/8]$. We may choose a smooth cutoff function $\varphi:\T\to[0,1]$ that is identically 1 in a small neighborhood of the critical points of $V$ and satisfies $|V-\lambda|\gtrsim1$ on the support of $\varphi$. In addition, we require that $|V'(y)|\gtrsim1$ on the support of $1-\varphi$. We decompose 
\begin{equation}\label{porb8.3}
\begin{split}
      \Delta_k^{-1}\bigg[\frac{h}{V-\lambda-\mathrm{i}\epsilon}\bigg]&=\Delta_k^{-1}\bigg[\frac{h\varphi}{V-\lambda-\mathrm{i}\epsilon}\bigg] + \Delta_k^{-1}\bigg[\frac{h(1-\varphi)}{V-\lambda-\mathrm{i}\epsilon}\bigg]\\
      &:= H_1+H_2. 
\end{split}      
\end{equation}
The desired conclusion for $H_1$ follows easily since $|V-\lambda|\gtrsim1$ on its support. We focus on the term $H_2$. We can compute
\begin{equation}\label{porb8.4}
\begin{split}
    H_2&=\Delta_k^{-1}\bigg[\frac{(1-\varphi)h}{V'(y)}\partial_y\log(V(y)-\lambda-\mathrm{i}\epsilon)\bigg]\\
    &=\partial_y\Delta_k^{-1}\bigg[\frac{(1-\varphi)h}{V'(y)}\log(V(y)-\lambda-\mathrm{i}\epsilon)\bigg]-
    \Delta_k^{-1}\bigg[\log(V(y)-\lambda-\mathrm{i}\epsilon)\partial_y\frac{(1-\varphi)h}{V'(y)}\bigg]. 
\end{split}
\end{equation}
   We have the preliminary bounds 
   \begin{equation}\label{porb8.5}
     \begin{split}
         &\bigg\|\Delta_k^{-1}\bigg[\log(V(y)-\lambda-\mathrm{i}\epsilon)\partial_y\frac{(1-\varphi)h}{V'(y)}\bigg]\bigg\|_{H^1_k(\T)}
         \lesssim|k|^{-1/3}\|h\|_{H^1_k(\T)},
     \end{split}  
   \end{equation}
   and 
 \begin{equation}\label{porb8.6}
     \begin{split}
         &\bigg\|\partial_y\Delta_k^{-1}\bigg[\frac{(1-\varphi)h}{V'(y)}\log(V(y)-\lambda-\mathrm{i}\epsilon)\bigg]\bigg\|_{H^1_k(\T)}\\
         &\lesssim\bigg\|\frac{(1-\varphi)h}{V'(y)}\log(V(y)-\lambda-\mathrm{i}\epsilon)\bigg\|_{L^{2}(\T)}\lesssim|k|^{-1/3}\|h\|_{H^1_k(\T)}.
     \end{split}  
   \end{equation}
The bounds \eqref{porb8.5}–\eqref{porb8.6}, although weaker than the desired bounds \eqref{porb8.1}, suffice to establish compactness in the second part of Lemma~\ref{porb8}.  The stronger bound \eqref{porb8.1} follows from an energy inequality, and through a change of coordinate, the inequality for $\mu\in \R$,
\begin{equation}\label{porb8.601}
\begin{split}
 &  \bigg| \int_{\R} \frac{g_1g_2(y)}{y-\mu+\mathrm{i}\epsilon}dy\bigg|\lesssim \int_{\T}|\widehat{\,g_1g_2}(\xi)|\,d\xi\\
 &\lesssim |k|^{-1/2}\|(k,\partial_y)(g_1g_2)\|_{L^2(\R)}\lesssim |k|^{-1}\|(k,\partial_y)g_1\|_{L^2(\R)}\|(k,\partial_y)g_2\|_{L^2(\R)}. 
   \end{split}
\end{equation}

To establish the second part of Lemma \ref{porb8}, by \eqref{porb8.1}, we may assume that $k_j= k_0\in  \frac{2\pi}{\tp}\Z\backslash\{0\}, \lambda_j\to \lambda_0\in[-7/8,7/8]$. Using the decomposition \eqref{porb8.3} and the bounds \eqref{porb8.5}-\eqref{porb8.6}, it suffices to show that a subsequence of 
\begin{equation}\label{porb8.7}
    \partial_y(-\Delta_{k_0})^{-1}\bigg[\frac{(1-\varphi)h_j}{V_j'(y)}\log(V_j(y)-\lambda_j-\mathrm{i}\epsilon_j)\bigg]
\end{equation}
   converges in $H^1(\T)$.  This follows from the identity
\begin{align}\label{porb8.8}
  & \partial^2_y(-\Delta_{k_0})^{-1}\bigg[\frac{(1-\varphi)h_j}{V_j'(y)}\log(V_j(y)-\lambda_j-\mathrm{i}\epsilon_j)\bigg]\\
 \notag  =&-\frac{(1-\varphi)h_j}{V_j'(y)}\log(V_j(y)-\lambda_j-\mathrm{i}\epsilon_j)+
   k_0^2(-\Delta_{k_0})^{-1}\bigg[\frac{(1-\varphi)h_j}{V_j'(y)}\log(V_j(y)-\lambda_j-\mathrm{i}\epsilon_j)\bigg], 
\end{align}    
and the fact that the right hand side has a convergent subsequence in $L^2(\T)$. 
\end{proof}
 If $f$ satisfies \eqref{eq:Ray-eps-0}, let $\phi=(-\Delta_k)^{-1}f,$ then $\phi$ satisfies for $y\in\T$,
\begin{equation}\label{porb6}
    (\partial_y^2-k^2)\phi+\frac{V(y)}{V(y)-\lambda-\mathrm{i}\epsilon}\phi = \frac{F}{V(y)-\lambda-\mathrm{i}\epsilon}.
\end{equation}

We now prove \eqref{porb701} in the case $\lambda\in [-7/8,7/8]$ for sufficiently small $\epsilon_0>0$. Suppose, for contradiction, that \eqref{porb701} fails to hold. Then we can find sequences $c_j\to0+$, $V_j$ satisfying $\|V_j-b\|_{H^4}\leq c_j$, $\lambda_j\in[-7/8,7/8]\to \lambda_*$,  $k_j=k_0\in \frac{2\pi}{\tp}\Z\backslash\{0\}$, 
and functions $\phi_j\in H^1_{k_0}(\T)$ with $\|\phi_j\|_{H^1_{k_0}(\T)}=1$, 
such that $\phi_j$ satisfies equation \eqref{porb6} with some $F_j\in H^1(\T)$ with $F_j\to 0$ in $H^1_k(\T)$. 
By Lemma \ref{lemb}, we have $\epsilon_j\to 0$.
Without loss of generality, assume $\epsilon_j\to 0+$. The case where $\epsilon_j\to0-$ is similar.

By Lemma \ref{porb8} and its proof, passing to a subsequence if necessary, we may assume that $\phi_j\to \phi\in H^1(\T)$ with $\|\phi_j\|_{H^1_{k_0}}=1$, and $\phi$ satisfies the equation on $\T$, 
\begin{equation}\label{porb8.9}
    (\partial_y^2-k_0^2)\phi{+}\lim_{j\to0+}\frac{\cos y}{\cos y-\lambda_0+\mathrm{i}\epsilon_j}\phi=0. 
\end{equation}
This contradicts the spectral properties of  $B_{k_0}$(see Subsection \ref{subspec}). The proof in the nondegenerate case is therefore complete.

\subsection{The degenerate case}
Now we turn to the case $\lambda\in[m,M]\setminus[-7/8,7/8]$.  Without loss of generality, we assume that $7/8\leq\lambda\leq M$.
The desired bound in Proposition \ref{lem:rayleigh} follows from the following bound  
\begin{equation}\label{porb331-0}
        (1+|k|(M-\lambda)^{\f12})\big\|f\big\|_{H_k^{-1}}\leq C \big\|F\big\|_{H_k^{1}},
\end{equation}
where $f$ and  $ F$ satisfy \eqref{eq:Ray-eps-0}. Let $V'(y_c)=0$ with $|y_c|<1/10$, and $V(y_c)=M$.

\begin{lemma}\label{critical-1}
Let $\psi_j,h_j \in H^1(a,d), u_j\in H^3(a,d)$ be sequences satisfying 
\begin{align*}
&\psi_j \rightharpoonup \psi, h_j \rightarrow h \ \ in \  \ H^1(a,d),\ u_j \to u_0\ \ in \ \ H^3(a,d),\\
 &u_j (\partial_y^2 \psi_j-k_j^2\psi_j)-u''_j\psi_j=h_j,
 \end{align*}
and $k_j\to k\in[1,+\infty), \, \mathrm{Im}\ u_j<0$. In addition, $\mathrm{Im}\ u_0=0, y\in [a,d], u_0(y_0)=0, y_0\in [a,d], u_0'(y)u'_0(y_0)>0, y\in [a,d]$. Then it holds that $\psi_j\to \psi$ in $H^1(a,d)$ and for any $\phi\in H^1_0(a,d)$, we have
\begin{align*}
\int_{a}^d (\psi'\phi'+k^2\psi\phi)\mathrm{d}y +\mathrm{P.V.}\int_{a}^d \frac{(u_0''\psi+h)\phi}{u_0}\mathrm{d}y+\mathrm{i}\pi \frac{((u_0''\psi+h)\phi)(y_0)}{|u_0'(y_0)|}=0.
 \end{align*}
\end{lemma}

\begin{proof}
This is a compactness lemma for the monotone case. The proof is classical; for example, see Lemma 6.2 in \cite{WZZ-APDE}.
\end{proof}

\begin{lemma}\label{hardy-type-abp}
If $f\in H^1_0(a,d), u\in C^2([a,d]), u(y)\in \mathbb{R}, |u'(y)|\geq c_0>0, |u''(y)|<C_0$ in $[a,d]$, then for $c\not\in \mathbb{R}$, $|\mathbf{Im}\ c|\leq \epsilon_2$, it holds that
 \begin{equation}\label{hardy-type-abp-0}
\begin{aligned}
\left|\int_a^d \frac{f(y)}{u(y)-c}\mathrm{d}y \right|\leq Cc_0^{-1}  |k|^{-\frac{1}{2}} (\|\partial_y f\|_{L^2}+|k|\|f\|_{L^2}+c_0^{-1}\|f\|_{L^2}) ,
\end{aligned}
\end{equation}
where $C>0$ is independent of $|k|,c,\epsilon_2$. Moreover, if $|k|\geq c_0^{-1}$, then it holds that
\begin{equation}\label{hardy-type-abp-1}
\begin{aligned}
\left|\int_a^d \frac{f(y)}{u(y)-c}\mathrm{d}y \right|\leq C c_0^{-1} |k|^{-\frac{1}{2}}(\|\partial_y f\|_{L^2}+|k|\|f\|_{L^2}),
 \end{aligned}
\end{equation}
where $C>0$ is independent of $|k|,c,\epsilon_2$.
\end{lemma}
\begin{proof} See Lemma C.1 in \cite{CDLZ}.
\end{proof}

To proceed with the analysis, we first derive several important properties of the flow $V$. \smallskip

First of all, we know that there exist $\delta,c_0\in (0,1)$ such that
\begin{align}\label{b'-c0}\left\{
\begin{array}{l}
                            |b'(y)|>2c_0,\qquad y\in \Sigma_{\delta}\cap[-3\pi/4,3\pi/4]\\
                           |b''(y)|>2c_0,\qquad y\in B(0,\delta),
                        \end{array}
  \right.
\end{align} where $b(y)=\cos y$, $\Sigma_{\delta}=\mathbb{T}_{2\pi}\setminus B(0,\delta)$. Now we fix $\delta,c_0$ .

We assume $|y_c|\leq 1/10$ such that $V(y_c)=\max\{V\}$. Thus, 
    \begin{align*}
       &|b'(0)-b'(y_c)|=|\sin(0)-\sin(y_c)|\geq |y_c-0|/\cos(1/10)\geq C^{-1}|y_c|,
    \end{align*}
  which gives by $V'(y_c)=0$ that
    \begin{align}\label{est:yc-1}
       &|y_c|\leq C|b'(0)-b'(y_c)|= C|V'(y_c)-b'(y_c)|\leq C\|V'-b'\|_{L^\infty}\leq C\|V-b\|_{H^4}.
    \end{align}
    Thanks to $|V''(y)|\geq |b''(y)|-C\|V-b\|_{H^4}$ and  $|V'(y)|\geq |b'(y)|-C\|V-b\|_{H^4}$, if $C\|V-b\|_{H^4}\leq \min\{ c_0,\delta/10\}$, then $y_c\in B(0,\delta/10)$ and by  \eqref{b'-c0},
\begin{equation*}\left\{\begin{aligned}
                           & |V'(y)|>c_0,\qquad y\in \Sigma_{\delta}\cap[-3\pi/4,3\pi/4],\\
                           &|V''(y)|>c_0,\qquad y\in B(0,\delta),
                        \end{aligned}
  \right.
\end{equation*}  which gives 
\begin{equation}\label{est:yc-2}\left\{\begin{aligned}
                           & |V'(y)|>c_0,\qquad y\in \Sigma_{\delta}\cap[-3\pi/4,3\pi/4],\\
                           &|V'(y)|>c_0|y-y_c|,\qquad y\in B(0,\delta).
                        \end{aligned}
  \right.
\end{equation} 
Therefore, we conclude 
\begin{align}
& |V'(y)|\geq C^{-1}|y-y_c|,\quad y\in [-3\pi/4,3\pi/4]\label{est:yc-30}\\
   & \inf\limits_{\Sigma_{r,y_c}\cap[-3\pi/4,3\pi/4]} |V'|\geq c_0r,\quad \text{for}\quad 0<r<9\delta/10,\label{est:yc-3}
\end{align} 
here $\Sigma_{r,y_c} =\mathbb{T}_{2\pi}\setminus B(y_c,r)$, and $C>0$ only depends on $\|V-b\|_{H^4}$.

Let $y_\lambda\in[y_c,3\pi/4]$ such that $V(y_{\lambda})=\lambda$. Then we get by $V'(y_c)=0$ that 
\begin{align*}
   & M-\lambda=-\int_{y_c}^{y_{\lambda}}\int_{y_c}^{z}V''(z_1)\mathrm{d}z_1\mathrm{d}z\leq (y_\lambda-y_c)^2\|V''\|_{L^\infty}/2\leq C(y_\lambda-y_c)^2,
\end{align*}
which shows 
\begin{align}\label{est:yc-4}
        & (M-\lambda)^{\f12}\leq C|y_\lambda-y_c|,
     \end{align}
     here $C>0$ only depends on $\|V-b\|_{H^4}$. Let $y_{\lambda}\in[y_c,\pi/2]$ satisfy $V(y_{\lambda})=\lambda\in[7/8,M]$. For $y\in [y_\lambda-\f34(y_\lambda-y_c),y_\lambda+\f34(y_\lambda-y_c)]$, we get by \eqref{est:yc-3} and \eqref{est:yc-4} that 
     \begin{align}\label{est:yc-5}
        & |V'(y)|\geq c_0(y_\lambda-y_c)/4\geq C^{-1}c_0(M-\lambda)^{\f12},
     \end{align}
      here $C>0$ only depends on $\|V-b\|_{H^4}$.
\begin{lemma}\label{monotone-11}
Assume that $C\|V-b\|_{H^4}\leq \min\{ c_0,\delta/10\}$, $r\in (0,9\delta/10),  |k|>\frac{2}{c_0r},\lambda\in [7/8,M]$, $M=\max\{V\}=V(y_c)$. Let $\Phi$ be the solution to
$$(V(y)-\lambda-\mathrm{i}\epsilon)(\partial_y^2-|k|^2)\Phi=g_1.$$
Then it holds that
\begin{equation}\nonumber
\begin{aligned}
&
\|(\partial_y\Phi,k\Phi)\|_{L^2(\Sigma_{r,y_c})}\leq C\left((r|k|)^{-1}+|k|^{-2}\right) \left(\|(\partial_y g_1,k g_1)\|_{L^2}+|k|\|\Phi\|_{L^2}\right),
 \end{aligned}
\end{equation}
where $C>0$ is independent of $|k|,\lambda,\epsilon$.

Moreover, assume that $C\|V-b\|_{H^4}\leq \min\{ c_0,\delta/10\}$, $\lambda\in [7/8,M]$, $M=\max\{V\}$, $|k|(M-\lambda)^{\f12}\geq 1$. Then it holds that
\begin{align}\label{monotone-11-est}
   &|k|(M-\lambda)^{\f12} \|(\partial_y\Phi,k\Phi)\|_{L^2}\leq C\|(\partial_y g_1,k g_1)\|_{L^2}.
\end{align}
where $C>0$ is independent of $|k|,\lambda,\epsilon$.
\end{lemma}
\begin{proof}
Let $\eta (y)\in C^\infty_0(\mathbb{R})$ be a cut-off function with $\eta (y)=1$ for $|y|\leq \frac{1}{2}$, $\eta(y)=0$ for $|y|\geq \frac{3}{4}$, and $\eta \in [0,1]$ for $y\in[-\pi,\pi]$. We then define $\eta_r(y)=1-\eta((y-y_c)/r)$ and take a periodic extension. It is easy to see that  $|\eta_r|\in [0,1], |\eta'_r|\lesssim r^{-1},  |\eta''_r|\lesssim r^{-2}$ for $y\in \mathbb{T}_{2\pi}$, $\eta_r=1$ for $y\in\Sigma_{r,y_c},$ $\eta_r=0$ for $y\in \mathbb{T}_{2\pi}\setminus \Sigma_{\frac{r}{2},y_c}$.

Let $\eta_0(y)=\eta(y/\pi)$. A direct calculation yields 
\begin{equation}\label{mono-1}
\begin{aligned}
\int_{\mathbb{T}_{2\pi}}\eta_r &(|\Phi'|^2+|k|^2|\Phi|^2)\mathrm{d}y=\int_{-\pi}^{\pi} \left(\eta_r''\frac{|\Phi|^2}{2}-\eta_r\mathrm{Re} [\Phi''-|k|^2\Phi ]\overline{\Phi} \right)\mathrm{d}y\\
&=\int_{-\pi}^{\pi} \eta_r''\frac{|\Phi|^2}{2}\mathrm{d}y-\mathrm{Re} \int_{ \Sigma_{\frac{r}{2},y_c}}\eta_r \frac{g_1\overline{\Phi}}{V(y)-\lambda-\mathrm{i}\epsilon}\mathrm{d}y\\
&\leq Cr^{-2} \|\Phi\|_{L^2}^2+\left| \int_{ \Sigma_{\frac{r}{2},y_c}}\eta_r \frac{g_1\overline{\Phi}}{V(y)-\lambda-\mathrm{i}\epsilon}\mathrm{d}y\right|\\
&\leq Cr^{-2} \|\Phi\|_{L^2}^2+\left| \int_{ \Sigma_{\frac{r}{2},y_c}\cap[-3\pi/4,3\pi/4]}\eta_0\eta_r \frac{g_1\overline{\Phi}}{V(y)-\lambda-\mathrm{i}\epsilon}\mathrm{d}y\right|\\
&\qquad\qquad\qquad+\left| \int_{ \Sigma_{\frac{r}{2},y_c}}(1-\eta_0)\eta_r \frac{g_1\overline{\Phi}}{V(y)-\lambda-\mathrm{i}\epsilon}\mathrm{d}y\right|.
 \end{aligned}
\end{equation}

Notice that $\inf\limits_{\Sigma_{\frac{r}{2},y_c}\cap[-3\pi/4,3\pi/4]} |V'|\geq \frac{1}{2}c_0r,\, |k|>\left(\frac{1}{2}c_0r\right)^{-1},\, \eta_0|_{y=\pm \pi}=0, \eta_r\left(\pm \frac{r}{2}\right)=0$. We then get by Lemma \ref{hardy-type-abp} that
\begin{equation}\label{mono-2}
\begin{aligned}
&\left| \int_{-\pi}^\pi \frac{\eta_0\eta_r g_1\overline{\Phi}}{V(y)-\lambda-\mathrm{i}\epsilon}\mathrm{d}y\right|\\
&\leq C\left(\frac{1}{2}c_0r\right)^{-1}|k|^{-\frac{1}{2}}\big(\|\partial_y (\eta_0\eta_rg_1\overline{\Phi})\|_{L^2 }+|k|\| (\eta_0\eta_rg_1\overline{\Phi})\|_{L^2}\big)\\
&\leq C\left(\frac{1}{2}c_0r\right)^{-1}|k|^{-\frac{1}{2}}\big(\|\partial_y (\eta_rg_1\overline{\Phi})\|_{L^2 }+|k|\| (\eta_rg_1\overline{\Phi})\|_{L^2}\big).
 \end{aligned}
\end{equation}

Thanks to $|V(y)-\lambda-\mathrm{i}\epsilon|\geq 7/8-C\|V-b\|_{H^4}\geq 3/4$ on the support of $1-\eta_0$, we have
\begin{equation}\label{mono-21}
\begin{aligned}
&\left| \int_{-\pi}^\pi \frac{(1-\eta_0)\eta_r g_1\overline{\Phi}}{V(y)-\lambda-\mathrm{i}\epsilon}\mathrm{d}y\right|
\leq C\|(1-\eta_0)\eta_r g_1\overline{\Phi}\|_{L^1}\leq C\|g_1\|_{L^2}\|\eta_r\Phi\|_{L^2}.
 \end{aligned}
\end{equation}

We introduce
\begin{align*}
&R_1=\|\eta_r\partial_y\Phi\|_{L^2}+|k|\|\eta_r\Phi\|_{L^2}, \quad R_2=\|\partial_y (\eta_r \Phi)\|_{L^2}+ |k|\| \eta_r \Phi\|_{L^2},\\
& R_3=\|\partial_y g_1\|_{L^2} +|k| \| g_1\|_{L^2}.
\end{align*}
Then $R_2\leq R_1+Cr^{-1}\|\Phi\|_{L^2}.$
By the Sobolev embedding, we get
\begin{align*}
\|g_1\|_{L^\infty}\leq C\|g_1\|_{L^2}^{\frac{1}{2}} \|g_1\|_{H^1}^{\frac{1}{2}} \leq C|k|^{-\frac{1}{2}}(\|\partial_y g_1\|_{L^2}+|k|\|g_1\|_{L^2})\leq C|k|^{-\frac{1}{2}} R_3,
 \end{align*}
and $ \|\eta_r\Phi\|_{L^\infty}\leq C|k|^{-\frac{1}{2}} R_2$. Thus, we obtain
\begin{align}\label{mono-3}
\begin{aligned}
&\|\partial_y (\eta_rg_1\overline{\Phi})\|_{L^2 }+|k|\| (\eta_rg_1\overline{\Phi})\|_{L^2}\\
&\leq \|\partial_y g_1\|_{L^2 }\| (\eta_r\overline{\Phi})\|_{L^\infty}+\|g_1\|_{L^\infty}\|\partial_y (\eta_r\overline{\Phi})\|_{L^2} +|k| \|g_1\|_{L^\infty} \| (\eta_r \overline{\Phi})\|_{L^2}\\
&\leq C (R_3\cdot |k|^{-\frac{1}{2}}R_2+|k|^{-\frac{1}{2}}R_3\cdot R_2+|k|\cdot |k|^{-\frac{1}{2}}R_3\cdot|k|^{-\frac{1}{2}}R_2 )\\
&\leq C |k|^{-\frac{1}{2}}R_2R_3.
 \end{aligned}
\end{align}
It follows from \eqref{mono-1}-\eqref{mono-3} that
\begin{align*}
\frac{R_1^2}{2}\leq &\int_{\mathbb{T}_{2\pi}}\eta_r (|\Phi'|^2+|k|^2|\Phi|^2)
\leq Cr^{-2} \|\Phi\|_{L^2}^2\\
&+\left| \int_{-\pi}^{\pi}\eta_0\eta_r \frac{g_1\overline{\Phi}}{V(y)-\lambda-\mathrm{i}\epsilon}\mathrm{d}y\right|+\left| \int_{-\pi}^{\pi}(1-\eta_0)\eta_r \frac{g_1\overline{\Phi}}{V(y)-\lambda-\mathrm{i}\epsilon}\mathrm{d}y\right|\\
\leq & Cr^{-2} \|\Phi\|_{L^2}^2+C\left(\frac{1}{2}c_0r\right)^{-1}|k|^{-1}R_2R_3+C|k|^{-2}R_1R_3\\
\leq & Cr^{-2} \|\Phi\|_{L^2}^2+C(c_0r)^{-1}(R_1+Cr^{-1}\|\Phi\|_{L^2})R_3+C|k|^{-2}R_1R_3,
 \end{align*}
which yields 
\begin{align*}
R_1
\leq  Cr^{-1} \|\Phi\|_{L^2}+C\left((c_0r)^{-1}+|k|^{-1}\right)|k|^{-1} R_3\leq C\left((r|k|)^{-1}+|k|^{-2}\right)(R_3+|k|\|\Phi\|_{L^2}).
 \end{align*}
Then the first inequality of this lemma follows from $\eta_r=1$ in $\Sigma_{r,y_c}$.

Now we turn to prove \eqref{monotone-11-est}. Let $y_{\lambda}\in[y_c,3\pi/4]$ satisfy $V(y_{\lambda})=\lambda$, $\eta_{\lambda}(y)=\eta(\f{y-y_{\lambda}}{y_{\lambda}-y_c})$. By \eqref{est:yc-5}, we know that $|V'(y)|\geq C^{-1}c_0(M-\lambda)^{\f12}$ on $\text{supp}\ \eta_{\lambda}(y)=  [y_\lambda-\f34(y_\lambda-y_c),y_\lambda+\f34(y_\lambda-y_c)]$. We then  have
\begin{align}\label{mono-20}
&\left| \int_{y_c}^\pi \frac{ \eta_{\lambda} g_1\overline{\Phi}}{V(y)-\lambda-\mathrm{i}\epsilon}\mathrm{d}y\right|\nonumber\\
&\leq C\left((M-\lambda)^{\f12}\right)^{-1}|k|^{-\frac{1}{2}}\big(\|\partial_y (\eta_{\lambda}g_1\overline{\Phi})\|_{L^2 }+|k|\| (\eta_\lambda g_1\overline{\Phi})\|_{L^2}\big)\nonumber\\
&\leq  C(M-\lambda)^{-\f12}|k|^{-1}\|(\partial_y g_1,kg_1)\|_{L^2 }\big(\|\partial_y\Phi\|_{L^2}+|k|\| \Phi\|_{L^2}+(M-\lambda)^{-\f12}\|\Phi\|_{L^2}\big)\nonumber\\
&\leq  C(M-\lambda)^{-\f12}|k|^{-1}\|(\partial_y g_1,kg_1)\|_{L^2 }\|(\partial_y \Phi,k\Phi)\|_{L^2 },
 \end{align}
 where we used  \eqref{est:yc-4} and 
 \begin{align*}
    & \|\partial_y (\eta_{\lambda}g_1\overline{\Phi})\|_{L^2 }+|k|\| (\eta_\lambda g_1\overline{\Phi})\|_{L^2}\\
     &\leq C|k|^{-\f12}\|(\partial_y g_1,kg_1)\|_{L^2 }\big(\|\partial_y(\eta_{\lambda}\Phi)\|_{L^2}+|k|\|\eta_{\lambda} \Phi\|_{L^2}\big)\\
     &\leq C|k|^{-\f12}\|(\partial_y g_1,kg_1)\|_{L^2 }\big(\|\partial_y\Phi\|_{L^2}+|k|\| \Phi\|_{L^2}+\|\partial_y\eta_{\lambda}\|_{L^\infty}\|\Phi\|_{L^2}\big)\\
     &\leq C|k|^{-\f12}\|(\partial_y g_1,kg_1)\|_{L^2 }\big(\|\partial_y\Phi\|_{L^2}+|k|\| \Phi\|_{L^2}+(y_{\lambda}-y_{c})^{-1}|\|\Phi\|_{L^2}\big).
 \end{align*}
 
 By \eqref{est:yc-30} and $V'(y)\leq 0$ for $y\in [y_c,\pi/2]$, we obtain
 \begin{align*}
 &V(y_\lambda)-V(y_{\lambda}-(y_{\lambda}-y_c)/2)\geq |V'(y_{\lambda}-(y_{\lambda}-y_c)/2)|(y_\lambda-y_c)/2\geq C^{-1}(y_\lambda-y_c)^2,\\
    & V(y_{\lambda}+(y_{\lambda}-y_c)/2)-V(y_\lambda)\geq |V'(y_{\lambda})|(y_\lambda-y_c)/2\geq C^{-1}(y_\lambda-y_c)^2.
 \end{align*}
 If $y\in [y_c,\pi]\setminus [y_{\lambda}-\f{(y_{\lambda}-y_c)}{2},y_{\lambda}+\f{(y_{\lambda}-y_c)}{2}]$, then we have (using  $V'(y)\leq 0$ for $y\in [y_c,\pi/2]$)
 \begin{align*}
    &|V(y)-\lambda|=|V(y)-V(y_\lambda)|\\
    &\geq \min\left(V(y_\lambda)-V(y_{\lambda}-(y_{\lambda}-y_c)/2), V(y_{\lambda}+(y_{\lambda}-y_c)/2)-V(y_\lambda)\right)\\
    &\geq C^{-1}(y_\lambda-y_c)^2,
 \end{align*} 
 which gives by noticing $[y_c,\pi] \cap\text{supp} (1-\eta_\lambda)\subseteq [y_c,\pi]\setminus [y_{\lambda}-\f{(y_{\lambda}-y_c)}{2},y_{\lambda}+\f{(y_{\lambda}-y_c)}{2}]$ and \eqref{est:yc-4} that
 \begin{align*}
    & \left\|\frac{ 1-\eta_{\lambda}}{V(y)-\lambda-\mathrm{i}\epsilon}\right\|_{L^\infty(y_c,\pi)}\leq C(y_\lambda-y_c)^{-2}\leq C(M-\lambda)^{-1}.
 \end{align*} 
 Then by $|k|(M-\lambda)^{\f12}\geq 1$, we have
 \begin{align}\label{mono-20-1}
\left| \int_{y_c}^\pi \frac{ (1-\eta_{\lambda}) g_1\overline{\Phi}}{V(y)-\lambda-\mathrm{i}\epsilon}\mathrm{d}y\right|
&\leq \left\|\frac{ 1-\eta_{\lambda}}{V(y)-\lambda-\mathrm{i}\epsilon}\right\|_{L^\infty(y_c,\pi)}\|g_1\overline{\Phi}\|_{L^1}
\leq C(M-\lambda)^{-1}\|g_1\|_{L^2}\|\Phi\|_{L^2}\nonumber\\
&\leq  C(M-\lambda)^{-\f12}|k|^{-1}\|(\partial_y g_1,kg_1)\|_{L^2 }\|(\partial_y \Phi,k\Phi)\|_{L^2 },
 \end{align}
which along with with \eqref{mono-20} and \eqref{mono-20-1} gives 
 \begin{align*}
    &\left| \int_{y_c}^\pi \frac{  g_1\overline{\Phi}}{V(y)-\lambda-\mathrm{i}\epsilon}\mathrm{d}y\right|\leq  C(M-\lambda)^{-\f12}|k|^{-1}\|(\partial_y g_1,kg_1)\|_{L^2 }\|(\partial_y \Phi,k\Phi)\|_{L^2 }.
 \end{align*} 
 Similarly, we have
 \begin{align*}
    &\left| \int_{-\pi}^{y_c} \frac{  g_1\overline{\Phi}}{V(y)-\lambda-\mathrm{i}\epsilon}\mathrm{d}y\right|\leq  C(M-\lambda)^{-\f12}|k|^{-1}\|(\partial_y g_1,kg_1)\|_{L^2 }\|(\partial_y \Phi,k\Phi)\|_{L^2 }.
 \end{align*}  
 
 Thus, we arrive at 
 \begin{align*}
    & \|(\partial_y\Phi,k\Phi)\|_{L^2}^2= -\int_{-\pi}^{\pi} \frac{  g_1\overline{\Phi}}{V(y)-\lambda-\mathrm{i}\epsilon}\mathrm{d}y\leq  C(M-\lambda)^{-\f12}|k|^{-1}\|(\partial_y g_1,kg_1)\|_{L^2 }\|(\partial_y \Phi,k\Phi)\|_{L^2 },
 \end{align*}
 which gives \eqref{monotone-11-est}. 
 \end{proof}

\begin{lemma}\label{lem:porb331}
  Let  $f$ satisfy \eqref{eq:Ray-eps-0} and $\phi=\Delta_k^{-1}f$. If
  \begin{equation}\label{porb331}
        \|\phi\|_{H_k^{1}(\mathbb{T}_{2\pi})}\leq C \big\|F\big\|_{H_k^{1}(\mathbb{T}_{2\pi})},
\end{equation}
and $C\|V-b\|_{H^4}\leq \min\{ c_0,\delta/10\}$, $\lambda\in [7/8,M]$, then \eqref{porb331-0} holds.
\end{lemma}
\begin{proof} If $|k|(M-\lambda)^{\f12}\leq 1$, then \eqref{porb331} implies \eqref{porb331-0} obviously. Now we assume $|k|(M-\lambda)^{\f12}\geq 1$.
  We rewrite  \eqref{eq:Ray-eps-0} as the form
  \begin{align*}
    &(V-\lambda-\mathrm{i}\epsilon)f+V\Delta_k^{-1} f=F ,\\
     &(V-\lambda-\mathrm{i}\epsilon)(\partial_y^2-k^2)\phi=F -V\phi.
  \end{align*}
  By \eqref{monotone-11-est} in Lemma \ref{monotone-11}, we have
  \begin{align*}
    |k|(M-\lambda)^{\f12}& \|(\partial_y\phi,k\phi)\|_{L^2}\leq C\|(\partial_y (F -V\phi),k (F -V\phi))\|_{L^2}\\
     \leq &C\|(\partial_y F,kF)\|_{L^2} +C\|(\partial_y\phi,k\phi)\|_{L^2},
  \end{align*}
  from which and \eqref{porb331}, we deduce {\eqref{porb331-0}}. 
\end{proof}

With Lemma \ref{lem:porb331} at our disposal, we only need to prove \eqref{porb331}. Recall  Lemma 6.3 in \cite{WZZ-APDE}.

\begin{lemma}\label{Lemma6.3APDE-1}
  Assume that $\psi,\ h\in H^1(a,d),$ $u\in H^3(a,d)$ satisfy
  \begin{align*}
     & (u-c)(\psi''-k^2\psi)-u''\psi=h,
  \end{align*}
  where $c\notin\mathbb{R}$, and $u$ is real-valued and $u'(y_0)=0,\ y_0=\f{a+d}{2}\in (a,d)$, $u''(y)u''(y_0)>0$ in $[a,d]$. Then we have
  \begin{align*}
     & |(\psi''-k^2\psi)(y_0)|\leq C\min \big(|u(y_0)-c|^{-\f34},|u(y_0)-c|^{-1}\big) \big(\|\psi\|_{H^1(a,d)}+\|h\|_{H^1(a,d)}\big),
  \end{align*}
  where the constant $C$ {depends only on $u,\ k$ and $\delta=y_0-a$}.
\end{lemma}
Using a similar proof, we can refine this lemma as follows.
 
\begin{lemma}\label{Lemma6.3APDE}
  Assume that $\psi,\ h\in H^1(a,d),$ $u\in H^3(a,d)$ satisfy
  \begin{align*}
     & (u-c)(\psi''-k^2\psi)-u''\psi=h,
  \end{align*}
  where $c\notin\mathbb{R}$, and $u$ is real-valued and $u'(y_0)=0,\ y_0=\f{a+d}{2}\in (a,d)$, $u''(y)u''(y_0)>0$ in $[a,d]$. If $|k||u(y_0)-c|^{\f12}\leq 1$, $h(y_0)=0$, then we have
  \begin{align*}
     & |(\psi''-k^2\psi)(y_0)|\leq C\min \big(|u(y_0)-c|^{-\f34},|k|^{-\f12}|u(y_0)-c|^{-1}\big) \big(\|\psi\|_{H^1_k(a,d)}+\|h\|_{H^1_k(a,d)}\big),
  \end{align*}
  where the constant $C$ depends on $\inf_{y\in[a,d]}|u''|$, $\|u\|_{H^3(a,d)}$ and $\delta=y_0-a$
\end{lemma}
\begin{proof}
  We may assume that $|u''|>c_{01}>0$ in $[a, d]$. Then we have
  \begin{align}
     &|u'(y)|\leq |y-y_0|,\ |u(y)-u(y_0)|\leq C|y-y_0|^2,\nonumber\\
     &|u'(y)-u'(y_1)|\geq c_{01}|y-y_1|\quad \text{for }\ y,y_1\in[a,d].\label{eat:Lemma6.3APDE-0}
  \end{align}
  It is easy to see that
  \begin{align*}
     & |(\psi''-k^2\psi)(y_0)|=\left|\dfrac{u''\psi-h}{u-c}\right|(y_0)\leq C|k|^{-\f12}\dfrac{\|\psi\|_{H^1_k}+\|h\|_{H^1_k}}{|u(y_0)-c|}.
  \end{align*} 
  Thus, it suffices to consider the case $|u(y_0)-c|< \delta^2$. Let $\delta_1=|u(y_0)-c|^{\f12}$. So, $\delta_1<\delta$.

 We normalize $\|\psi\|_{H^1_k(a,d)}+\|h\|_{H^1_k(a,d)}=1$. Then for $|y-y_0|\leq \delta_1$,
\begin{align*}
   & |h(y)|\leq C\delta_1^{\f12},\ |u(y)-c|\leq |u(y)-u(y_0)|+|u(y_0)-c|\leq C\delta_1^2. 
\end{align*}
Let $h_1=\big((u-c)\psi'-u'\psi\big)'=k^2\psi(u-c)+h$. Then for $|y-y_0|<\delta_1$(using $|k|\delta_1=|k||u(y_0)-c|^{\f12}\leq 1$), 
\begin{align*}
    |h_1(y)|\leq& |k|^2\|\psi\|_{L^\infty}|u(y)-c|+|h(y)|\\
   \leq& C|k|^{\f32}\|\psi\|_{H^1_k}|u(y)-c|+|h(y)|\leq C|k|^{\f32}\delta_1^2+C\delta_1^{\f12}\leq C\delta_1^{\f12},
\end{align*}
which leads to
\begin{align}\label{eat:Lemma6.3APDE-1}
   & \left|\big((u-c)\psi'-u'\psi\big)|_{y_1}^{y_2}\right|=\left|\int_{y_1}^{y_2}h_1(y)\mathrm{d}y\right| \leq C\delta^{\f32}_1.
\end{align} 

Choose $\delta_2\in(\delta_1/2,\delta_1),$ $y_1=y_0-\delta_2,$ $y_2=y_0+\delta_2$ such that
\begin{align*}
   & |\psi'(y_1)|^2+|\psi'(y_2)|^2\leq 2\delta_1^{-1}\|\psi'\|^2_{L^2(a,d)}\leq 2\delta_1^{-1},
\end{align*}
which gives
\begin{align*}
   &  \left|(u-c)\psi'|_{y_1}^{y_2}\right|\leq (|\psi'(y_1)|+|\psi'(y_2)|)\|u-c\|_{L^\infty([y_1,y_2])}\leq C\delta_1^{-\f12}\delta_1^2=C\delta_1^{\f32}.
\end{align*}
This along with \eqref{eat:Lemma6.3APDE-1} shows that $\left|u'\psi|_{y_1}^{y_2}\right|\leq C\delta_1^{\f32}$. Notice that
\begin{align*}
   & \psi u'|_{y_1}^{y_2}=\psi(y_0)u'|_{y_1}^{y_2}+u'(y_1)\psi|_{y_1}^{y_0}+u'(y_2)\psi|_{y_0}^{y_2}.
\end{align*}
Thus, we get by \eqref{eat:Lemma6.3APDE-0} that
\begin{align*}
   & c_{01}\delta_1|\psi(y_0)|\leq \left|\psi(y_0)u'|_{y_1}^{y_2}\right| \leq \left|\psi u'|_{y_1}^{y_2}\right| +C\|u'\|_{L^\infty([y_1,y_2])}\delta_2^{\f12}\|\psi'\|_{L^2(a,b)},
\end{align*}
which yields 
\begin{align*}
   & |\psi(y_0)|\leq C\delta_1^{\f12}= C|u(y_0)-c|^{\f12},
\end{align*}
and thus,
\begin{align*}
   & |(\psi''-k^2\psi)(y_0)|=\left|\dfrac{u''\psi(y_0)}{u(y_0)-c}\right|\leq C\delta_1^{-2}|\psi(y_0)|\leq C\delta_1^{-\f32}.
\end{align*}
\end{proof}

The proofs of the following two lemmas represent minor revisions of the corresponding lemmas in Appendix C \cite{CDLZ}. For the reader's convenience, we provide complete proofs here.

\begin{lemma}\label{critical-3}
Let $k_j\in\mathbb{R}, \psi_j,h_j\in H^1(a,d), u_j\in H^3(a,d)$ be  sequences satisfying 
\begin{equation}\nonumber
\begin{aligned}
&\|(\partial_y\psi_j,k_j\psi_j)\|_{L^2(a,d)} \leq 2,\quad \|(\partial_yh_j,k_jh_j)\|_{L^2(a,d)} \to 0,\\ &(u_j-c_j)(\psi_j''-k_j^2\psi_j)-u''_j\psi_j=h_j,\quad u_j\rightarrow  \cos (y-(a+d)/2)\ \in H^3(a,d)
 \end{aligned}
\end{equation}
and $|k_j|\to +\infty, 0<\mathrm{Im}\ c_j\rightarrow 0, k_j^2(c_j-u(y_0))\to 0, u'_j(y_0)=h_j(y_0)=0, y_0=\frac{a+d}{2}, \delta=y_0-a \in [0,1], u''_j(y) u''_j(y_0)>0, y\in [a,d]$. Then it holds that
$$\|\psi_j'\|_{L^2(y_0-\delta/2,y_0+\delta/2)}+|k_j|\|\psi_j\|_{L^2(y_0-\delta/2,y_0+\delta/2)} \to 0.$$
\end{lemma}

\begin{proof}
Without loss of generality, let us suppose that $y_0=0$, and then $[a,d]=[-\delta,\delta]$. Let $c_j-u_j(0)=r_j^2 \mathrm{e}^{2\mathrm{i}\theta_j}, r_j>0,\theta_j\in (0,\pi/2),$ then $|k_j| r_j\to 0,r_j\to 0$. By Lemma \ref{Lemma6.3APDE}, and noticing $|k_j||c_j-u_j(0)|^{\f12}=|k_j|r_j\leq 1$ for $j$ large,  we have 
\begin{equation}\nonumber
\begin{aligned}
&|(\psi''_j-k_j\psi_j)(0)|\leq Cr_j^{-\frac{3}{2}},\\
 & |(u''_j\psi_j+h_j)(0)|= |u_j(0)-c_j||(\psi''_j-k_j\psi_j)(0)|\leq C r_j^{\frac{1}{2}},\\
 & |\psi_j (0)|= |(u''_j\psi_j+h_j)(0)|/|u_j''(0)|\leq C r_j^{\frac{1}{2}}.
 \end{aligned}
\end{equation}
This motivates us to introduce
\begin{equation}\nonumber
\begin{aligned}
\widetilde{\psi}_j=r_j^{-\frac{1}{2}}\psi_j(r_j y),\quad \widetilde{h}_j(y)=r_j^{-\frac{1}{2}} h_j(r_jy),\quad \widetilde{ u}_j(y)=r_j^{-2}(u_j(r_j y)-u_j(0)).
 \end{aligned}
\end{equation}
It holds that
\begin{equation}\nonumber
\begin{aligned}
&(\widetilde{u}_j-\mathrm{e}^{2\mathrm{i}\theta_j})(\widetilde{\psi}_j''-(k_j r_j)^2 \widetilde{\psi}_j) -\widetilde{u}_j''  \widetilde{\psi}_j=\widetilde{h}_j(y),\\
&|\widetilde{\psi}_j(0)|=|r_j^{-\frac{1}{2}}\psi_j( 0)|\leq C, \ \|\widetilde{\psi}_j'\|_{L^2(a/r_j,d/r_j)}=\|{\psi}_j'\|_{L^2(a,d)}\leq C,\\
&|\widetilde{h}_j(y)(0)|=0,\ \|\widetilde{h}_j'\|_{L^2(a/r_j,d/r_j)}=\|{h}_j'\|_{L^2(a,d)}\to 0.
 \end{aligned}
\end{equation}
And $\widetilde{\psi}_j$ is bounded in $H^1_{loc}(\mathbb{R})$ and $\widetilde{h}_j(y)\to 0$ in $H^1_{loc}(\mathbb{R})$. Up to a subsequence, let us assume that $\widetilde{\psi}_j \rightharpoonup \widetilde{\psi}_0$ in $H^1_{loc}(\mathbb{R})$, $\theta_j\to \theta_0\in [0,\pi/2]$ with $\widetilde{\psi}_0' \in L^2(\mathbb{R})$.

Using the facts that ($\widetilde{ u}_j(0)=\widetilde{ u}_j'(0)=0$)
\begin{align*}
   & \widetilde{ u}_j(y)=y^2\int_{0}^{1}\int_{0}^{1}t{ u}''_j(r_jyts)\mathrm{d}s\mathrm{d}t,\quad   \widetilde{ u}_j'(y)=y\int_{0}^{1} { u}''_j(r_jyt)\mathrm{d}t,
\end{align*} 
and $\widetilde{u}_j''(y)=u_j''(r_jy)$, $\widetilde{u}_j'''(y)=r_ju_j'''(r_jy)$, $ u_j\rightarrow \cos y\in H^3(-\delta,\delta)$, we can deduce that $\widetilde{u}_j \to -y^2/2$ in $H^3_{loc}(\mathbb{R})$, $\widetilde{\psi}_0=0$ and $\widetilde{\psi}_j \to 0$ in $H^1_{loc}(\mathbb{R})\cap C^1_{loc}(\mathbb{R}\setminus \{\pm 1\})$(See the proof of Lemma 6.4 in \cite{WZZ-APDE}). Then we have
\begin{align}\label{est:critical-31}
&|k_j| \|\psi_j\|_{L^2(-3r_j,3r_j)}+ \|\psi_j'\|_{L^2(-3r_j,3r_j)}
=|k_j r_j| \|\widetilde{\psi}_j\|_{L^2(-3,3)}+ \|\widetilde{\psi}_j'\|_{L^2(-3,3)}\to 0.
 \end{align}
For $j>1$, choose $d_j\in [\delta/2,\delta]$ such that
\begin{align*}
&2|k_j \psi_j' \overline{\psi_j}(d_j)|\leq |k_j \psi_j' (d_j)|^2+|\overline{\psi_j}(d_j)|^2\\
&\leq 2\delta^{-1}(|k_j|\|\psi_j\|_{L^2(a,d)}^2+ \|\psi_j'\|_{L^2(a,d)}^2) \leq 8\delta^{-1},
 \end{align*}
which gives $| \psi_j' \overline{\psi_j}(d_j)|\leq 8(|k_j|\delta)^{-1}\to 0$. We get by integration by part  that
\begin{align}
&\int_{3r_j}^{d_j} \left(|\psi_j'|^2+k_j^2 |\psi_j|^2+\frac{u''_j|\psi_j|^2}{u_j-c_j}\right)\mathrm{d}y\nonumber\\
&\label{ap-int-1}=-\int_{3r_j}^{d_j}\frac{h_j\overline{\psi_j}}{u_j-c_j}\mathrm{d}y+\psi_j' \overline{\psi_j}|_{3r_j}^{d_j}.
 \end{align}
For $j$ large enough and $y\in [3r_j,d]$, we have $u_j(0)-u_j(y)\geq u_j(0)-u_j(3r_j)\geq 9r_j^2/4\geq 2|c_j-u_j(0)|$, and then
\begin{align*}
\mathrm{Re} \frac{1}{u_j(y)-c_j}\geq \frac{1}{Cy^2},\quad -u''_j(y)\geq C^{-1}, \quad \left|\frac{1}{u_j(y)-c_j}\right|\leq \frac{C}{y^2}.
 \end{align*}
Taking the real part of \eqref{ap-int-1}, we obtain
\begin{align*}
& \int_{3r_j}^{d_j} \left(|\psi_j'|^2+k_j^2 |\psi_j|^2\right)\mathrm{d}y+C^{-1}\left\|\frac{\psi_j}{y}\right\|_{L^2(3r_j,d_j)}^2\\
&\leq C \left\|\frac{\psi_j}{y}\right\|_{L^2(3r_j,d_j)}\left\|\frac{h_j}{y}\right\|_{L^2(3r_j,d_j)}
+\left|\psi_j' \overline{\psi_j}|_{3r_j}^{d_j}\right|,
 \end{align*}
which yields 
\begin{align*}
\|\psi_j'\|_{L^2(3r_j,d_j)}^2+k_j^2 \|\psi_j\|_{L^2(3r_j,d_j)}^2 \leq C\left\|\frac{h_j}{y}\right\|_{L^2(3r_j,d_j)}^2+\left|\psi_j' \overline{\psi_j}|_{3r_j}^{b_j}\right|.
 \end{align*}
Due to $\psi_j' \overline{\psi_j}|_{3r_j}^{d_j}=\psi_j' \overline{\psi_j}(d_j)-\widetilde{\psi}_j' \overline{\widetilde{\psi}_j}(3) \to 0$ and
$$\left\|\frac{h_j}{y}\right\|_{L^2(3r_j,d_j)}\leq C\|h'_j\|_{L^2(a,d)}\to 0,$$
we obtain
\begin{align}\label{est:critical-32}
&\|\psi_j'\|_{L^2(3r_j, \delta/2)}+ |k_j| \|\psi_j\|_{L^2(3r_j,\delta/2)}
\leq \|\psi_j'\|_{L^2(3r_j,d_j)}+|k_j| \|\psi_j\|_{L^2(3r_j,d_j)} \to 0.
 \end{align}

Similar arguments yield that $\|\psi_j'\|_{L^2(-\delta/2,-3r_j)}+ |k_j| \|\psi_j\|_{L^2(-\delta/2,-3r_j)}\to 0$, which along with \eqref{est:critical-31} and  \eqref{est:critical-32} completes the proof of the lemma.
\end{proof}

\begin{lemma}\label{critical-4}
Let $k_j\in \mathbb{R}, \psi_j,h_j \in H^1(a,d), u_j\in H^3(a,d)$ be  sequences satisfying
\begin{equation}\nonumber
\begin{aligned}
&\|(\partial_y\psi_j,k_j\psi_j)\|_{L^2(a,d)} \leq 1,\quad \|(\partial_yh_j,k_jh_j)\|_{L^2(a,d)} \to 0,\\ 
&(u_j-c_j)(\psi_j''-k_j^2\psi_j)-u_j''\psi_j=h_j,\quad u_j\rightarrow \cos (y-(a+d)/2)\in H^3(a,d),
 \end{aligned}
\end{equation}
and $|k_j|\to +\infty, 0<\mathrm{Im}\ c_j\rightarrow 0, u_j'(y_0)=0, y_0=\frac{a+d}{2}, \delta=y_0-a \in [0,1], u''_j(y) u''_j(y_0)>0, y\in [a,d]$. Then it holds that
$\widetilde{\psi}_j\to 0$ in $H^1_{loc}(\mathbb{R})$, where $\widetilde{\psi}_j(y)=|k_j|^{\frac{1}{2}} \psi_j (y_0+y/|k_j|).$
\end{lemma}
\begin{proof}
Without loss of generality, let us assume that $y_0=0$ and then $[a,d]=[-\delta,\delta]$. Let  $c_j=r_j^2 \mathrm{e}^{2\mathrm{i}\theta_j}+u_j(0), r_j>0,\theta_j\in (0,\pi/2)$ and $\widetilde{h}_j=|k_j|^{\frac{1}{2}} h_j (y/|k_j|),\widetilde{ u}_j(y)=|k_j|^2 (u_j(y/|k_j|)-u_j(0))$. Then we have
\begin{equation}\label{tilde-critical-4}
\left\{
\begin{aligned}
&(\widetilde{ u}_j-(k_j r_j)^2 \mathrm{e}^{2\mathrm{i}\theta_j}) (\widetilde{\psi}_j''- \widetilde{\psi}_j) -\widetilde{u}_j''  \widetilde{\psi}_j=\widetilde{h}_j(y),\\
& \|(\widetilde{\psi}_j',\widetilde{\psi}_j)\|_{L^2 (|k_j|a,|k_j|d)}= \|({\psi}_j',k_j{\psi}_j)\|_{L^2 (a,d)}\leq 1,\\
& \|(\widetilde{h}_j',\widetilde{h}_j)\|_{L^2 (|k_j|a,|k_j|d)}= \|({h}_j',k_jh_j)\|_{L^2 (a,d)}\to 0,
 \end{aligned}\right.
\end{equation}
which implies that $\widetilde{\psi}_j$ is bounded in $H^1_{loc}(\mathbb{R})$ and $\widetilde{h}_j \to 0$ in $H^1_{loc}(\mathbb{R})$. Up to a subsequence, let us suppose that $\widetilde{\psi}_j \rightharpoonup \widetilde{\psi}_0$ in $H^1_{loc}(\mathbb{R})$, $\theta_j\to \theta_0\in [0,\pi/2], |k_j| r_j \to r_0\in [0,+\infty] $ with $\widetilde{\psi}_0 \in H^1(\mathbb{R})$. 
Using the facts that ($\widetilde{ u}_j(0)=\widetilde{ u}_j'(0)=0$)
\begin{align*}
   & \widetilde{ u}_j(y)=y^2\int_{0}^{1}\int_{0}^{1}t { u}''_j(yts/|k_j|)\mathrm{d}s\mathrm{d}t,\quad   \widetilde{ u}_j'(y)=y\int_{0}^{1} { u}''_j(yt/|k_j|)\mathrm{d}t,
\end{align*} 
and $\widetilde{u}_j''(y)=u_j''(y/|k_j|)$, $\widetilde{u}_j'''(y)=|k_j|^{-1}u_j'''(y/|k_j|)$, $ u_j\rightarrow \cos y\in H^3(-\delta,\delta)$, we can deduce that $\widetilde{u}_j \to -y^2/2$ in $H^3_{loc}(\mathbb{R})$.\smallskip

{\it Case 1.} $r_0=+\infty.$

In this case, it holds that $(|\widetilde{u}_j''|+1)/(\widetilde{u}_j-(|k_j|r_j)^2 e^{2\mathrm{i}\theta_j})\to 0$ in $L^\infty_{loc}(\mathbb{R})$ and $\widetilde{\psi}_j''- \widetilde{\psi}_j\to 0$ in $L^\infty_{loc}(\mathbb{R})$, thus $\widetilde{\psi}_j \to \widetilde{\psi}_0$ in $H^1_{loc}(\mathbb{R})$ with $\widetilde{\psi}_0''- \widetilde{\psi}_0=0$. Note that $\widetilde{\psi}_0 \in H^1(\mathbb{R})$, therefore $\widetilde{\psi}_0=0$ and $\widetilde{\psi}_j \to 0$ in $H^1_{loc}(\mathbb{R})$.\smallskip

{\it Case 2.}  $r_0\in (0,+\infty).$

Firstly, if $\theta_0\not=\pi/2$, then $\widetilde{\psi}_j''- \widetilde{\psi}_j$ is bounded in $L^\infty_{loc}(\mathbb{R})$ and $\widetilde{\psi}_j \to \widetilde{\psi}_0$ in $C^1_{loc}(\mathbb{R})$ with
\begin{align*}
(-y^2/2-r_0^2  \mathrm{e}^{2\mathrm{i}\theta_0})(\widetilde{\psi}_0''- \widetilde{\psi}_0) +  \widetilde{\psi}_0=0, \end{align*}
which implies 
\begin{align*}
\int_{\mathbb{R}}( |\widetilde{\psi}_0'|^2+|\widetilde{\psi}_0|^2)\mathrm{d}y=-\int_{\mathbb{R}}
\frac{2|\widetilde{\psi}_0|^2}{y^2+2r_0^2  \mathrm{e}^{2\mathrm{i}\theta_0}}\mathrm{d}y.
 \end{align*}
Multiplying both sides by $\mathrm{e}^{\mathrm{i}\theta_0}$ and taking the real part of the resulting expression, we obtain 
\begin{align*}
\cos \theta_0\int_{\mathbb{R}}( |\widetilde{\psi}_0'|^2+|\widetilde{\psi}_0|^2)\mathrm{d}y=-\cos \theta_0 \int_{\mathbb{R}}\frac{2(y^2+2r_0^2) |\widetilde{\psi}_0'|^2}{|y^2+2r_0^2  \mathrm{e}^{2\mathrm{i}\theta_0}|^2}\mathrm{d}y,
 \end{align*}
which implies that $\widetilde{\psi}_0=0$ and $\widetilde{\psi}_j \to 0$ in $H^1_{loc}(\mathbb{R})$.

If $\theta_0=\pi/2$, then $\widetilde{\psi}_j''- \widetilde{\psi}_j$ is bounded in $L^\infty_{loc}(\mathbb{R} \setminus \{\pm\sqrt{2} r_0\})$ and $\widetilde{\psi}_j \to \widetilde{\psi}_0$ in $C^1_{loc}(\mathbb{R} \setminus \{\pm\sqrt{2}  r_0\})$ with
\begin{align*}
&(y^2+2r_0^2 \mathrm{e}^{2\mathrm{i}\theta_0})(\widetilde{\psi}_0''- \widetilde{\psi}_0) -2  \widetilde{\psi}_0=0, \ y\in \mathbb{R}\setminus \{\pm \sqrt{2} r_0\},\\
&(y^2-2r_0^2  )(\widetilde{\psi}_0''- \widetilde{\psi}_0) -2  \widetilde{\psi}_0=0, \quad y\in \mathbb{R}\setminus \{\pm\sqrt{2}  r_0\},
\end{align*}
Recall \eqref{tilde-critical-4}, by Lemma \ref{critical-1}, $ \widetilde{\psi}_j \to  \widetilde{\psi}_0$ in $H^1(\sqrt{2} r_0 (1-\delta),\sqrt{2} r_0(1+\delta)) \cap H^1( -\sqrt{2} r_0(1+\delta),-\sqrt{2} r_0 (1-\delta))$, then $\widetilde{\psi}_j \to \widetilde{\psi}_0$ in $H^1_{loc}(\mathbb{R})\cap C^1_{loc}(\mathbb{R} \setminus \{\pm\sqrt{2}  r_0\})$. Moreover, for any $\phi \in H^1(\mathbb{R})$, it holds that
\begin{align*}
\int_{\mathbb{R}} (\widetilde{\psi}_0' \phi'+\widetilde{\psi}_0\phi)\mathrm{d}y+\mathrm{p.v.}\int_{\mathbb{R}} \frac{2\widetilde{\psi}_0\phi}{y^2-2r_0^2}\mathrm{d}y+\mathrm{i}\pi \sum_{y=\pm \sqrt{2} r_0}\frac{\widetilde{\psi}_0 \phi}{\sqrt{2}r_0}=0.
\end{align*}
Taking $\phi=\overline{\widetilde{\psi}_0}$ and taking the real part, we get $\widetilde{\psi}_0(\pm \sqrt{2}r_0)=0$. Thus, $\widetilde{\psi}_0\in H^2 (\mathbb{R})$ with
\begin{align*}
\int_{\mathbb{R}} \left(\left|\widetilde{\psi}_0'+\frac{-2y \widetilde{\psi}_0}{y^2-2r_0^2}\right|^2+|\widetilde{\psi}_0|^2\right) \mathrm{d}y=\int_{\mathbb{R}} \left( |\widetilde{\psi}_0'|^2+|\widetilde{\psi}_0|^2+\frac{2|\widetilde{\psi}_0|^2}{y^2-2r_0^2}\right)\mathrm{d}y=0,
\end{align*}
which yields that $\widetilde{\psi}_0=0$ and $\widetilde{\psi}_j \to 0$ in $H^1_{loc}(\mathbb{R})$.\smallskip

{\it Case 3.} $r_0=0.$

In this case, $|k_j|^2(c_j-u_j(0))\to 0$. Let us introduce a cut-off function $\eta\in C^\infty_0(\mathbb{R})$ with $\eta(y)=\cosh y$ for $|y|\leq 1$, $\eta=0$ for $|y|\geq 2$ and define $\eta_1=\eta''-\eta$. Obviously, $\eta_1=0$ provided that $|y|\leq 1$ or $|y|\geq 2$. Then we define
\begin{align*}
&\psi_{j^*}(y)=\psi_j(y) +\frac{h_j(0)}{u_j''(0)}\eta(|k_j|y),\\
& h_{j^*}(y)=h_j(y)-\frac{u''_j(y)h_j(0)}{u''_j(0)}\eta (|k_j|y)+|k_j|^2(u_j(y)-c_j)\frac{h_j(0)}{u''_j(0)} \eta_1 (|k_j|y).
\end{align*}
It {holds} that $\psi_{j^*},h_{j^*}\in H^1(a,d), h_{j^*}(0)=0$ with
\begin{align*}
(u_j(y)-c_j)(\psi_{j^*}''-k_j^2 \psi_{j^*})-u_j''\psi_{j^*}=h_{j^*}.
\end{align*}
Notice that $\widetilde{ u}_j(y)=|k_j|^2 \big(u_j(y/|k_j|)-u_j(0)\big), \widetilde{u}_j\to -y^2/2$ in $H^3_{loc}(\mathbb{R})$, then $\widetilde{u}''_j(y)=u''_j(y/|k_j|)$, $k_j^2 (u_j(y)-c_j)=\widetilde{u}_j (|k_j| y)-k_j^2 (c_j-u_j(0))$, $|k_j^2 (c_j-u_j(0))|\to r_0^2=0,$ and for $j$ large enough, it holds that $\|\widetilde{u}_j''\|_{H^1(-2,2)}+\|\widetilde{u}_j-k_j^2 (c_j-u_j(0))\|_{H^1(-2,2)}\leq C$. Then
\begin{align*}
& |k_j| \|\eta(|k_j| y )\|_{L^2(a,d)}+  \|\partial_y \eta(|k_j| y )\|_{L^2(a,d)}+|k_j| \|u''_j(y)\eta(|k_j| y )\|_{L^2(a,d)}\\
&\quad+\|\partial_y (u''_j(y)\eta(|k_j| y ))\|_{L^2(a,d)}+ |k_j| \| k_j^2(u_j(y)-c_j)\eta_1(|k_j| y )\|_{L^2(a,d)}\\&\quad+  \|\partial_y \left( k_j^2(u_j(y)-c_j)\eta_1(|k_j| y )\right)\|_{L^2(a,d)}\\
&\leq |k_j|^{\frac{1}{2}} \Big(\|\eta\|_{L^2 (a|k_j|,d|k_j|)}+\|\eta'\|_{L^2 (a|k_j|,d|k_j|)}+\|\widetilde{u}''_j\eta\|_{L^2 (a|k_j|,d|k_j|)}\\
&\quad+\|(\widetilde{u}''_j\eta)'\|_{L^2 (a|k_j|,d|k_j|)}+\| (\widetilde{u}_j- k_j^2(c_n-u_j(0)))\eta_1\|_{L^2 (a|k_j|,d|k_j|)}\\
&\quad +\|\partial_y((\widetilde{u}_j- k_j^2(c_n-u_j(0)))\eta_1)\|_{L^2(a|k_j|,d|k_j|)}\Big)\\
&\leq C|k_j|^{\frac{1}{2}} \Big((1+\|\widetilde{u}''_j\|_{H^1(-2,2)})\|\eta\|_{H^1(\mathbb{R})}+\|\widetilde{u}_j-k_j^2 (c_j-u_j(0))\|_{H^1(-2,2)}\|\eta_1\|_{H^1(\mathbb{R})}\Big)\\
&\leq C|k_j|^{\frac{1}{2}}.
\end{align*}
By Gagliardo-Nirenberg inequality, we have
$$|h_j(0)|\leq C|k_j|^{-\frac{1}{2}} (\|h_j'\|_{L^2 (a,d)}+|k_j|\|h_j\|_{L^2(a,d)}),$$
and then
\begin{align*}
& \|h_{j^*}'\|_{L^2 (a,d)}+|k_j| \|h_{j^*}\|_{L^2 (a,d)}\leq \|h_j'\|_{L^2 (a,d)}+|k_j|\|h_j\|_{L^2(a,d)}+C|k_j|^{\frac{1}{2}} |h_j(0)|\\
&\quad\leq C  (\|h_j''\|_{L^2 (a,d)}+|k_j|\|h_j\|_{L^2(a,d)})\to 0,\\
& |k_j| \|\psi_{j^*}-\psi_j \|_{L^2(a,d)}+ \|\partial_y (\psi_{j^*}-\psi_j )\|_{L^2(a,d)}\to 0.
\end{align*}
Note that $\|\psi_j'\|_{L^2(a,d)}+|k_j|\|\psi_j\|_{L^2(a,d)} \leq 1$; thus, for $j$ large, we have $$|k_j| \|\psi_{j^*} \|_{L^2(a,d)}+ \|\partial_y \psi_{j^*}\|_{L^2(a,d)}\leq 2.$$
By applying Lemma \ref{critical-3}, we obtain $$ |k_j| \|\psi_{j^*} \|_{L^2(-\delta/2,\delta/2)}+ \|\partial_y \psi_{j^*}\|_{L^2(-\delta/2,\delta/2)}\to 0,$$ which further gives
\begin{align*}
& |k_j| \|\psi_j \|_{L^2(-\delta/2,\delta/2)}+ \|\partial_y \psi_j \|_{L^2(-\delta/2,\delta/2)}\leq  |k_j| \|\psi_{j^*}-\psi_j \|_{L^2(a,d)}\\
&\quad+ \|\partial_y (\psi_{j^*}-\psi_j )\|_{L^2(a,d)}+ |k_j| \|\psi_{j^*} \|_{L^2(-\delta/2,\delta/2)}+ \|\partial_y \psi_{j^*}\|_{L^2(-\delta/2,\delta/2)} \to 0.
\end{align*}

Thanks to $|k_j| \to \infty$ and the definition of $\widetilde{\psi}_j$, we have
\begin{align*}
& \|\widetilde{\psi}_j \|_{L^2(-( |k_j|\delta)/2,( |k_j| \delta)/2)}+ \|\widetilde{ \psi}_j' \|_{L^2(-( |k_j|\delta)/2,( |k_j|\delta)/2)}\\
&= |k_j| \|\psi_j \|_{L^2(-\delta/2,\delta/2)}+ \|\partial_y \psi_j \|_{L^2(-\delta/2,\delta/2)}\to 0,
\end{align*}
and $-( |k_j|\delta)/2\to -\infty, ( |k_j|\delta)/2\to +\infty$, then $\widetilde{\psi}_j \to 0$ in $H^1_{loc}(\mathbb{R})$.
\end{proof}

For the case where $k$ is fixed, the proof is similar to that of Lemma 6.4 in \cite{WZZ-APDE}, so we omit the details.

\begin{lemma}\label{critical-5}
Let $k\in \mathbb{R},\ k>1$, $\psi_j,h_j \in H^1(a,d), u_j\in H^3(a,d)$ be sequences satisfying 
\begin{align*}
&\|\psi_j\|_{H^1(a,d)} \leq 1,\quad \|h_j\|_{H^1(a,d)} \to 0,\quad \psi_j\rightharpoonup 0 \ \text{in}\ H^1(\mathbb{T}_{2\pi}),\\ 
&(u_j-c_j)(\psi_j''-k^2\psi_j)-u_j''\psi_j=h_j,\quad u_j\rightarrow \cos (y-(a+d)/2)\in H^3(a,d),
 \end{align*}
and  $0<\mathrm{Im}\ c_j\rightarrow 0, u_j'(y_0)=0, y_0=\frac{a+d}{2}, \delta=y_0-a \in [0,1], u''_j(y) u''_j(y_0)>0, y\in [a,d]$. Then it holds that
${\psi}_j\to 0$ in $H^1(a,d)$.
\end{lemma}

Now we are in a position to prove \eqref{porb331}. 

\begin{proof}

\textit{ Case 1}. $|k|\geq K_0$ for some large  $K_0>0$ and $|\epsilon|\leq \epsilon_{00}$ for some small $\epsilon_{00}\neq 0$. \smallskip

Suppose that \eqref{porb331} fails to hold in Case 1. Then we can find sequences $|k_j|\rightarrow +\infty$, $\phi_j\in H^1_0(\mathbb{T}_{2\pi})$, $F_j\in H^1(\mathbb{T}_{2\pi})$, $0\neq|\epsilon_j|\rightarrow 0$
(we may assume $\epsilon_j>0$) and $\lambda_j\in\mathbb{R}$, such that 
\begin{align}\label{sec3.3-compact-11}
   &\|(\partial_y\phi_j,k_j\phi_j)\|_{L^2(\mathbb{T}_{2\pi})}=1,\quad\|F_j\|_{H^1_{k_j}(\mathbb{T}_{2\pi})}\rightarrow 0,\\
  & V_j\rightarrow b\in H^4(\mathbb{T}),\quad M_j=\max\{V_j\}\rightarrow 1,\quad \lambda_j\rightarrow \lambda_*\in [7/8,1],\label{sec3.3-compact-12}
\end{align}and 
\begin{align}\label{sec3.3-compact-13}
   & V_j(\phi_j''-(k_j^2-1)\phi_j)-(\lambda_j+\mathrm{i}\epsilon_j)\Delta_{k_j}\phi_j=F_j,
\end{align}
then
\begin{align}\label{sec3.3-compact-130}
   & (V_j-\lambda_j-\mathrm{i}\epsilon_j)(\phi_j''-k_j^2\phi_j)=F_j -V_j\phi_j.
\end{align}
Let $|y_{jc}|\leq 1/10$, such that $V_j(y_{jc})=M_j$. 
Taking $\delta,c_0\in (0,1)$ in \eqref{b'-c0},  for large $j$, it holds that $C\|V_j-b\|_{H^4}\leq \min\{c_0,\delta/10\}$, then \eqref{est:yc-2} and \eqref{est:yc-3} hold for $V_j$. For $0<r<9\delta/10, |k_j|\geq \frac{2}{c_0r}$, we get by Lemma \ref{monotone-11} that 
\begin{equation}\nonumber
\begin{aligned}
&
\|(\partial_y\phi_j,k_j\phi_j)\|_{L^2(\Sigma_{r,y_{jc}})}\\
&\leq C\big((r|k_j|)^{-1}+|k_j|^{-2}\big) \left(\|(\partial_y (F_j-V_j\phi_j),k_j( F_j-V_j\phi_j)) \|_{L^2}+|k_j|\|\phi_j\|_{L^2}\right)\\
&\leq C\big((r|k_j|)^{-1}+|k_j|^{-2}\big) \left(\|(\partial_y\phi_j,k_j\phi_j)\|_{L^2}+\|(\partial_yF_j,k_jF_j)\|_{L^2}\right)\\
&\leq C\big((r|k_j|)^{-1}+|k_j|^{-2}\big),
 \end{aligned}
\end{equation}
where $C>0$ is independent of $j,r$.

Set $r_{j}=R/|k_j|$. For any fixed $R>\frac{2}{c_0}$ and  $|k_j|>R/\delta$(for large  $j$),  it holds that $0<{r_{j}}<9\delta/10,\ k_j =R/r_j >\frac{2}{c_0r_j}$ and
\begin{align*}
\|(\partial_y\phi_j,k_j\phi_j)\|_{L^2(\Sigma_{R/|k_j|,y_{jc}})}
\leq C\big((r_j|k_j|)^{-1}+|k_j|^{-2}\big) =C(R^{-1}+|k_j|^{-2}),
 \end{align*}
which gives
\begin{align}\label{lim-sup}
&\limsup_{j\to +\infty} \|(\partial_y\phi_j,k_j\phi_j)\|_{L^2(\Sigma_{R/|k_j|,y_{jc}})}
\leq CR^{-1}.
 \end{align}

It remains to deal with the interval near the critical point $y_{jc}$. For any fixed $R>\frac{2}{c_0}$, let $g_j=F_j-V_j\phi_j$, $\widetilde{\phi}_j(y)=|k_j|^{\frac{1}{2}} \phi_j (y_{jc}+y/|k_j|)$.
To apply Lemma \ref{critical-4}, we denote 
\begin{align}\label{a1}\begin{aligned}
       & a=-\pi/2,\ d=\pi/2,\ y_0=0,\ c_j=\lambda_j+\mathrm{i}\epsilon_j,\\
       &u_j(y)=V_j(y+y_{jc}),\ \psi_j(y)=\phi_j(y+y_{jc}),\\
       &h_j(y)={F_j}(y+y_{jc})-(V_j(y+y_{jc})+V''_j(y+y_{jc}))\phi_{j}(y+y_{jc}) .
    \end{aligned}   \end{align}
    It hods that
    \begin{align*}
       & V_j(y+y_{jc})\rightarrow b(y)\ \text{in}\  H^3(-\pi/2,\pi/2),\ V''_j(y+y_{jc})\rightarrow -b(y)\ \text{in}\  H^1(-\pi/2,\pi/2),\\
       &(V_j(y+y_{jc})+V''_j(y+y_{jc}))\phi_{j}(y+y_{jc})\rightarrow 0\ \text{in}\  H^1(-\pi/2,\pi/2),\\&(u_j-c_j)(\psi_j''-k_j^2\psi_j)-u_j''\psi_j=h_j,\\
       &h_j(y)\rightarrow0\ \text{in}\  H^1(-\pi/2,\pi/2),\quad \|(\partial_y\psi_j,k_j\psi_j)\|_{L^2(-\pi/2,\pi/2)} \leq 1.
    \end{align*}
    Then Lemma \ref{critical-4}
 implies that $\widetilde{\psi}_j\to 0$ in $H^1_{loc}(\mathbb{R})$, then  $\widetilde{\phi}_j\to 0$ in $H^1_{loc}(\mathbb{R})$. Notice that (for large $j$ such that $|k_j|\geq R\pi$)
\begin{align*}
\|\partial_y\phi_j\|_{L^2(-R/|k_j|+y_{jc},R/|k_j|+y_{jc})}
+|k_j|\|\phi_j\|_{L^2(-R/|k_j|+y_{jc},R/|k_j|+y_{jc})}\\
=\|\partial_y\widetilde{\phi}_j\|_{L^2(-R,R)}+\|\widetilde{\phi}_j\|_{L^2(-R,R)},
 \end{align*}
which along with $\widetilde{\psi}_j\to 0$ in $H^1_{loc}(\mathbb{R})$ implies
\begin{align}\label{lim-2}
\lim_{j\to +\infty}& \big(\|\partial_y\phi_j\|_{L^2(-R/|k_j|+y_{jc},R/|k_j|+y_{jc})}
+|k_j|\|\phi_j\|_{L^2(-R/|k_j|+y_{jc},R/|k_j|+y_{jc})}\big)
=0.
 \end{align}
Therefore, we obtain
\begin{align*}
&1= \|(\partial_y\phi_j,k_j\phi_j)\|_{L^2}
\leq   \|(\partial_y\phi_j,k_j\phi_j)\|_{L^2(\Sigma_{R/|k_j|,y_{jc}})}
+ \|(\partial_y\phi_j,k_j\phi_j)\|_{L^2(-R/|k_j|+y_{jc},R/|k_j|+y_{jc})}.
 \end{align*}
This along with \eqref{lim-sup} and \eqref{lim-2} shows that $1\leq CR^{-1}$ for any $R>\frac{2}{c_0}$ with $C>0$ independent of $R$; however, this leads to a contradiction when $R$ is sufficiently large.\smallskip

\textit{ Case 2}. $|k|\leq K_0$ and $|\epsilon|\leq \epsilon_{01}$ for some small $\epsilon_{01}\neq 0$, where $K_0>0$ is decided in \textit{Case 1}. 
In this case, \eqref{porb331} is equivalent to 
\begin{equation}\label{porb332}
        \|\phi\|_{H^{1}(\mathbb{T}_{2\pi})}\leq C \|F\|_{H^{1}(\mathbb{T}_{2\pi})},\quad \phi=\Delta_k^{-1}f.
\end{equation}
Suppose that \eqref{porb332} fails to hold in Case 2. Then we can find sequences $|k_j|=k>1$, $\phi_j\in H^1_0(\mathbb{T}_{2\pi})$, $F_j\in H^1(\mathbb{T}_{2\pi})$, $0\neq|\epsilon_j|\rightarrow 0$
(we may assume $\epsilon_j>0$) and $\lambda_j\in\mathbb{R}$, such that 
\begin{align*}
   &\|\phi_j\|_{H^1(\mathbb{T}_{2\pi})}=1,\quad\|F_j\|_{H^1(\mathbb{T}_{2\pi})}\rightarrow 0,\\
  & V_j\rightarrow b\in H^4(\mathbb{T}),\quad M_j=\max\{V_j\}\rightarrow 1,\quad \lambda_j\rightarrow \lambda_*\in [7/8,1],\\
 &(V_j-\lambda_j-\mathrm{i}\epsilon_j)(\phi_j''-k^2\phi_j)=F_j -V_j\phi_j.
\end{align*}
We can assume 
$\lambda_*=1$; otherwise the proof follows from the argument in the non-degenerate case. Since $L_k$ has no embedded eigenvalue, we have $\phi_j\to 0 $ in $H_{loc}^1(\mathbb{T}_{2\pi}\setminus\{0\})$ and $\phi_j\rightharpoonup 0 $ in $H^1(\mathbb{T}_{2\pi})$. To apply Lemma \ref{critical-5}, we still use the notations from \eqref{a1}; then we have
    \begin{align*}
       & V_j(y+y_{jc})\rightarrow b(y)\ \text{in}\  H^3(-\pi/2,\pi/2),\ V''_j(y+y_{jc})\rightarrow -b(y)\ \text{in}\  H^1(-\pi/2,\pi/2),\\
       &(V_j(y+y_{jc})+V''_j(y+y_{jc}))\phi_{j}(y+y_{jc})\rightarrow 0\ \text{in}\  H^1(-\pi/2,\pi/2),\\&(u_j-c_j)(\psi_j''-k^2\psi_j)-u_j''\psi_j=h_j,\\
       &h_j(y)\rightarrow0,\ \phi_j\rightharpoonup 0\ \text{in}\  H^1(-\pi/2,\pi/2),\quad \|\psi_j\|_{H^1(-\pi/2,\pi/2)} \leq 1.
    \end{align*}
 Lemma \ref{critical-5} implies that ${\psi}_j\to 0$ in $H^1(-\pi/2,\pi/2)$, then  ${\phi}_j\to 0$ in $H^1(-\pi/3,\pi/3)$. Recall that $\phi_j\to 0 $ in $H_{loc}^1(\mathbb{T}_{2\pi}\setminus\{0\})$, then ${\phi}_j\to 0$ in $H^1(\mathbb{T}_{2\pi})$, which is a contradiction.\smallskip

\textit{ Case 3}. $|\epsilon|\geq \min\{\epsilon_{00},\epsilon_{01}\}$, where $\epsilon_{00}$ and  $\epsilon_{01}$ are decided in \textit{Case 1} and  \textit{Case 2}. By Lemma \ref{lemb} and $(1+|k|(M-\lambda)^{\f12})\leq C|k|\leq C|k|^2$, we deduce \eqref{porb331} for \textit{Case 3}.
\end{proof}

\section{Resolvent estimate for the linearized NS}\label{sec:resnavier}

 Throughout this section, we assume that the real valued function V satisfies  $\|V-b\|_{H^4}\leq \epsilon_0$ with $\epsilon_0 > 0$ sufficiently small. Recall the notations from Section 3:
 \begin{align*}
&V'(y_1)=V'(y_2)=0,\ \text{where}\ |y_1|\leq 1/10,\ |y_2-\pi|\leq 1/10,\\
& M=\max_{y\in\mathbb{T}} V(y)= V(y_1),\ m=\min_{y\in\mathbb{T}} V(y)=V(y_2),\\
& \theta(k,\lambda):=1+|k|\lambda_0^{1/2}+|\lambda|,\
 \lambda_0:=\min\{|\lambda-m|,|\lambda-M|\}.      
 \end{align*}
 We denote
  \begin{align*}
& \lambda_{0+}:=\max\{\min\{\lambda-m,M-\lambda\},0\},\ 0\leq \lambda_{0+}\leq\lambda_0, \\
 & \lambda_{0+}>0\Leftrightarrow\lambda\in(m,M)\Rightarrow\lambda_{0+}=\lambda_{0}.
       \end{align*}
       
In this section, we prove the following resolvent estimates, which will be used to establish the space-time estimates for the linearized Navier-Stokes system.
\begin{Proposition}\label{prop:res-LNS}
Let $\nu\in(0,1)$ and $\lambda \in\R $. If $f$ with $P_0f=0$ solves the inhomogeneous eigenvalue problem
\begin{align*}
   & -\nu\Delta f+V\partial_x(f+\Delta^{-1} f)-\mathrm{i}\lambda f=F,
\end{align*}
then we have 
\begin{align*}
   & \||D_x|^{\f12}\nabla\Delta^{-1}f\|_{L^2}+ \||V'|^{\f12}\partial_x\nabla \Delta^{-1}f\|_{L^2}\leq C\||D_x|^{-\f12}\nabla F\|_{L^2},\\
   &\||D_x|^{\f12}\nabla\Delta^{-1}f\|_{L^2}+ \||V'|^{\f12}\partial_x\nabla \Delta^{-1}f\|_{L^2}\leq C\nu^{-\f12}\| F\|_{H^{-1}}.
\end{align*}
\end{Proposition}

We consider the following equation ($k\in2\pi\mathbb{Z}/\mathfrak{p}\setminus\{0\}$, $\lambda$ is replaced by $k\lambda$)
\begin{align*}
   & -\nu\Delta_k f+\mathrm{i}kV(1+\Delta^{-1}_k )f-\mathrm{i}k\lambda f=F.
\end{align*}
Thus, it is enough to show that
\begin{align*}
   &|k|\|(\partial_y,k)\Delta_k^{-1}f\|_{L^2}+|k|^{\f32}\||V'|^{\f12}(\partial_y,k)\Delta_k^{-1}f\|_{L^2} \leq C\|(\partial_y,k)F\|_{L^2},\\
   &|k|\|(\partial_y,k)\Delta_k^{-1}f\|_{L^2}+|k|^{\f32}\||V'|^{\f12}(\partial_y,k)\Delta_k^{-1}f\|_{L^2} \leq C\nu^{-\f12}|k|^{\f12}\|(\partial_y,k)\Delta_k^{-1}F\|_{L^2},
\end{align*}
which follow from the following Proposition \ref{pro:res1-os-V'} and $\theta(k,\lambda)\geq 1$.

\begin{Proposition}\label{pro:res1-os-V'}
If $\nu\in(0,1)$, $\lambda \in\R $, and (f, F) solves 
\begin{align*}
   & -\nu\Delta_k f+\mathrm{i}kV(1+\Delta_k^{-1})f-\mathrm{i}k\lambda f=F,
\end{align*}
then we have
\begin{align}\label{est:res1-os-1-V'}
   & |k|\theta(k,\lambda)\|(\partial_y,k)\Delta_k^{-1}f\|_{L^2}+ |k|^{\f32}\||V'|^{\f12}(\partial_y,k)\Delta_k^{-1}f\|_{L^2}\leq C\|(\partial_y,k) F\|_{L^2},\\
   &|\nu k\theta(k,\lambda)|^{\f12}\|(\partial_y,k)\Delta_k^{-1}f\|_{L^2}+ \nu^{\f12}|k|\||V'|^{\f12}(\partial_y,k)\Delta_k^{-1}f\|_{L^2}\leq C\|(\partial_y,k)\Delta_k^{-1} F\|_{L^2}.\label{est:res1-os-2-V'}
\end{align}\end{Proposition}

Let $w=f+\Delta_k^{-1}f=(\Delta_k+1)\Delta_k^{-1}f$. Then we have
\begin{align*}
&-\nu\Delta_k w+\mathrm{i}k(V-\lambda)w+\mathrm{i}k\lambda (\Delta_k+1)^{-1}w=F-\nu f,
\end{align*}
We denote $\mathcal{I}_k=(\Delta_k+1)^{-1}=\Delta_{\tilde{k}}^{-1}$, where $\tilde{k}=\sqrt{k^2-1}$.

\subsection{Some properties of the flow \texorpdfstring{$V(y)$}{V(y)}}
\begin{lemma}\label{lem:V'}
  It holds that
  \begin{align*}
&-V'(y)\sim\min\{|y-y_1|,|y-y_2|\}\quad \text{for}\,\, y\in[y_1,y_2],\\
& V'(y)\sim\min\{|y-y_1|,|y-y_2+2\pi|\}\quad \text{for}\,\,y\in[y_2-2\pi,y_1],\\
&C^{-1}|y-y_1|^2\leq M-V(y)\leq |y-y_1|^2\quad \text{for}\,\, y\in[y_2-2\pi,y_2],\\
&C^{-1}|y-y_2|^2\leq V(y)-m\leq |y-y_2|^2\quad \text{for}\,\, y\in[y_1,y_1+2\pi],\\
&C^{-1}|y-y_2+2\pi|^2\leq V(y)-m\leq |y-y_2+2\pi|^2\quad \text{for}\,\,y\in[y_1-2\pi,y_1],\\
    &|V'|^2\leq C\min\{V-m,M-V\}.
\end{align*} 
\end{lemma}
\begin{proof}
  It is obvious that  $V : [y_1, y_2] \rightarrow \mathbb{R}$ is strictly decreasing. 
   If $z \in [y_1,\pi/3]$, we have
  \begin{align*}
     & -V''(z)\geq -b''+(b''-V'')\geq \cos z-C\epsilon_0\geq \cos (\pi/3)-C\epsilon_0\geq 1/3,
  \end{align*}
  then for $y\in[y_1,\pi/3]$, by noting $V'(y_1)=0$, we get
  \beno
     -V'(y)=-\int_{y_1}^{y}V''(z)\mathrm{d}z\geq (y-y_1)/3;
   \eeno
  meanwhile, for $y\in[\pi/3,(y_1+y_2)/2]\subseteq [\pi/3,2\pi/3]$, we have
  \begin{align*}
     -V'(y)=& -b'+(b'-V')\geq {\sin y}-C\epsilon_0\geq \sqrt{3}/2-C\epsilon_0\geq 1/2\\
     =&(y-y_1)/(y_2-y_1)\geq(y-y_1)/2.
  \end{align*} 
  Here we used $y_2-y_1\leq \pi/2+2/10\leq 2$. 
  On the other hand, for $y\in[y_1,(y_1+y_2)/2]$, it holds that
  \begin{align*}
     &  -V'(y)=-\int_{y_1}^{y}V''(z)\mathrm{d}z\leq  \|V''(z)\|_{L^\infty}(y-y_1)\leq 3(y-y_1)/2.
  \end{align*}
 Thus, we conclude for $y\in[y_1,(y_1+y_2)/2]$, $ (y-y_1)/{3}\leq -V'(y)\leq 3(y-y_1)/2.$ Similarly, it holds that for $y\in[(y_1+y_2)/2,y_2]$, $ (y_2-y)/{3}\leq -V'(y)\leq 3(y_2-y)/2$. This shows that 
  \begin{align*}
&-V'(y)\sim\min\{|y-y_1|,|y-y_2|\},\quad \forall\ y\in[y_1,y_2].
\end{align*} 

Similarly, we have 
\begin{align*}
& V'(y)\sim\min\{|y-y_1|,|y-y_2+2\pi|\},\quad \forall\ y\in[y_2-2\pi,y_1].
\end{align*} 
  
  As $V : [y_1, y_2] \rightarrow \mathbb{R}$ is strictly decreasing, for $y\in[y_1,y_2]$, we have
  \begin{align*}
     M-V(y)&=V(y_1)-V(y)=\int_{y_1}^{y}-V'(z)\mathrm{d}z= - \int_{y_1}^{y}\int_{y_1}^{ z}V''(\tau)\mathrm{d}\tau\mathrm{d}z\\
     &\leq \|V''\|_{L^\infty} \int_{y_1}^{y}\int_{y_1}^{ z}1\mathrm{d}\tau\mathrm{d}z\leq \|V''\|_{L^\infty}(y_1-y)^2/2\leq (y_1-y)^2.
  \end{align*}
  For $y\in[y_1,y_2]$, we have
  \begin{align*}
     &M-V(y)=\int_{y_1}^{y}-V'(z)\mathrm{d}z\geq C^{-1} \int_{y_1}^{y}\min\{|z-y_1|,|z-y_2|\}\mathrm{d}z\\
     &\geq C^{-1} \int_{y_1}^{\min\{y,\f{y_1+y_2}{2}\}}|z-y_1|\mathrm{d}z \geq C^{-1} \min\{|y-y_1|^2/2,|y_1-y_2|^2/8\}\geq C^{-1} |y-y_1|^2.
  \end{align*}
  This shows that for $y\in[y_1,y_2]$,
  \beno  
  C^{-1}|y -y_1|^2\leq M-V(y)\leq |y -y_1|^2.
  \eeno
  Similarly, we have $C^{-1}|y -y_1|^2\leq M-V(y)\leq |y -y_1|^2$, for $y\in[y_2-2\pi,y_1]$. In a similar way, we can prove the bounds for $V(y)-m$.
    
  Therefore, for $y\in [y_1,y_2]$, we have
  \begin{align*}
     |V'(y)|^2\leq& C\min\{|y-y_1|^2,|y-y_2|^2\}\leq C\min \{C(V(y)-m),C(M-V(y))\}\\
     \leq& C\min \{V(y)-m,M-V(y)\}.
  \end{align*}
  Similarly, for $y\in [y_2-2\pi,y_1]$, $|V'(y)|^2\leq C\min \{V(y)-m,M-V(y)\}.$
  \end{proof}

\begin{lemma}\label{lem1}
It holds that for  $\lambda\in\mathbb{R}$,
\begin{align*}
    \|V'f\|_{L^2}\leq C\left|\int_{\T}(V-\lambda)|f|^2\mathrm{d}y\right|
   +C\lambda_{0+}\|f\|_{L^2}^2.
\end{align*} \end{lemma}   

\begin{proof}
 As $|V'|^2\leq C\min\{V-m,M-V\}$ by Lemma \ref{lem:V'}, we infer that
\begin{align*}
   &\|V'f\|_{L^2}^2\leq C\int_{\T}\min\{V-m,M-V\}|f|^2
   \leq C\min\left\{\int_{\T}(V-m)|f|^2,\int_{\T}(M-V)|f|^2\right\},\\
   &\int_{\T}(V-m)|f|^2=\int_{\T}(V-\lambda)|f|^2+(\lambda-m)\|f\|_{L^2}^2
   \leq\left|\int_{\T}(V-\lambda)|f|^2\right|+(\lambda-m)\|f\|_{L^2}^2,\\
   &\int_{\T}(M-V)|f|^2=-\int_{\T}(V-\lambda)|f|^2+(M-\lambda)\|f\|_{L^2}^2
   \leq\left|\int_{\T}(V-\lambda)|f|^2\right|+(M-\lambda)\|f\|_{L^2}^2,\end{align*}
from which, we  conclude 
\begin{align*}
   &\|V'f\|_{L^2}^2\leq C\left|\int_{\T}(V-\lambda)|f|^2\right|+
   \min\{\lambda-m,M-\lambda\}\|f\|_{L^2}^2.
   \end{align*}
Then the result follows from $ \lambda_{0+}\geq\min\{\lambda-m,M-\lambda\}$.
\end{proof}

\begin{lemma}\label{lem:V3}
If $\nu\in(0,1)$, $\lambda \in\R $, $ \delta>0$, $ \chi_1=\mathrm{Re}[(V-\lambda+\mathrm{i}\delta)^{-1}]$, $\lambda_{0+}\geq |\nu/k|^{\f12}$, 
then\begin{align*} 
&\|\chi_1'\|_{L^{\infty}}\leq C(\lambda_0+\delta)^{\f12}\delta^{-2},\quad 
  \||V-\lambda+\mathrm{i}\delta|^{-1}\|_{L^{2}}^2\leq C\delta^{-1}\lambda_0^{-1/2}.\end{align*}
 Moreover, if $\lambda_{0}=M-\lambda$, then 
 \beno
 \|\chi_1|y-y_1|^{\f12}\|_{L^2(y_2-2\pi+1/2,y_2+1/2)}\leq C\delta^{-\f12}.
 \eeno
\end{lemma}

\begin{proof}
As $\lambda_{0+}\geq |\nu/k|^{\f12}>0$, then $\lambda_{0}=\lambda_{0+}$, $\lambda\in(m,M) $. Thus,
 $\lambda_0=M-\lambda$ or $\lambda_0=\lambda-m$. Without lose of generality, assume $\lambda_0=M-\lambda$. Then 
 $\lambda-m\geq M-\lambda=\lambda_0 $, $\lambda-m\geq (M-m)/2>1/2 $, $\lambda_0\leq (M-m)/2<2$.

 By Lemma \ref{lem:V'}, we have 
\begin{align*} 
\chi_1'&=-\mathrm{Re}[V'(V-\lambda+\mathrm{i}\delta)^{-2}],\quad |V'|^2\leq C(M-V),
\end{align*}
which gives
\begin{align*} 
|\chi_1'|&\leq|V'||V-\lambda+\mathrm{i}\delta|^{-2}\leq C|M-V|^{\f12}|V-\lambda+\mathrm{i}\delta|^{-2}\\&= C|\lambda_0+\lambda-V|^{\f12}|V-\lambda+\mathrm{i}\delta|^{-2}
\leq C\lambda_0^{1/2}|V-\lambda+\mathrm{i}\delta|^{-2}+C|V-\lambda+\mathrm{i}\delta|^{-\f32}\\&\leq C\lambda_0^{1/2}\delta^{-2}+C\delta^{-\f32}\leq C(\lambda_0+\delta)^{\f12}\delta^{-2}.
\end{align*}
This shows the first inequality of the lemma.

Notice that
\begin{align*} 
\int_{\T}|V'||V-\lambda+\mathrm{i}\delta|^{-2}\mathrm{d}y&=
\int_{y_2-2\pi}^{y_1}|V'||V-\lambda+\mathrm{i}\delta|^{-2}\mathrm{d}y+
\int_{y_1}^{y_2}|V'||V-\lambda+\mathrm{i}\delta|^{-2}\mathrm{d}y\\
&=2\int_m^M|z-\lambda+\mathrm{i}\delta|^{-2}\mathrm{d}z\leq C\delta^{-1}.
\end{align*}
By Lemma \ref{lem:V'}, we deduce that if $y_1+\lambda_0^{1/2}/2\leq y\leq y_2-1/2$, then 
\beno
\lambda_0^{1/2}+|y-y_1|\leq 3|y-y_1|\leq C|y-y_2|\leq C|V'(y)|;
\eeno
if $y_2-2\pi+1/2\leq y \leq y_1-\lambda_0^{1/2}/2$, then 
\beno
\lambda_0^{1/2}+|y-y_1|\leq 3|y-y_1|\leq C|y-y_2+2\pi|\leq C|V'(y)|;
\eeno
if $y_1-\lambda_0^{1/2}/2\leq y\leq y_1+\lambda_0^{1/2}/2$, then $\lambda_0^{1/2}+|y-y_1|\leq 2\lambda_0^{1/2}$, and 
\begin{align*}
V(y)-\lambda\geq M-|y-y_1|^2-\lambda=\lambda_0-|y-y_1|^2\geq \lambda_0-\lambda_0/4\geq \lambda_0/2>0;
\end{align*}
if $y_2-1/2\leq y\leq y_2+1/2$, then $\lambda_0^{1/2}+|y-y_1|\leq C$, and (as $\lambda-m>1/2 $)
\begin{align*}V(y)-\lambda\leq m+|y-y_2|^2-\lambda\leq|y-y_2|^2-1/2\leq 1/4-1/2=-1/4<0.
\end{align*}
Thus, we conclude 
\begin{align*} 
&\int_{y_2-2\pi+1/2}^{y_2+1/2}(\lambda_0^{1/2}+|y-y_1|)|V-\lambda+\mathrm{i}\delta|^{-2}\mathrm{d}z\\
&=
\int_{y_2-2\pi+1/2}^{y_1-\lambda_0^{1/2}/2}(\lambda_0^{1/2}+|y-y_1|)
|V-\lambda+\mathrm{i}\delta|^{-2}+
\int_{y_1+\lambda_0^{1/2}/2}^{y_2-1/2}(\lambda_0^{1/2}+|y-y_1|)|V-\lambda+\mathrm{i}\delta|^{-2}\\&\quad+
\int_{y_1-\lambda_0^{1/2}/2}^{y_1+\lambda_0^{1/2}/2}(\lambda_0^{1/2}+|y-y_1|)|V-\lambda+\mathrm{i}\delta|^{-2}+
\int_{y_2-1/2}^{y_2+1/2}(\lambda_0^{1/2}+|y-y_1|)|V-\lambda+\mathrm{i}\delta|^{-2}\\
&\leq C\int_{\T}|V'||V-\lambda+\mathrm{i}\delta|^{-2}+\int_{y_1-\lambda_0^{1/2}/2}^{y_1+\lambda_0^{1/2}/2}
(2\lambda_0^{1/2})|\lambda_0/2+\mathrm{i}\delta|^{-2}+C\int_{y_2-1/2}^{y_2+1/2}|1/4+\mathrm{i}\delta|^{-2}\\
&\leq C\delta^{-1}+2\lambda_0|\lambda_0/2+\mathrm{i}\delta|^{-2}+C|1/4+\mathrm{i}\delta|^{-2}\leq C\delta^{-1}.
\end{align*}
This implies the second and third inequality of the lemma. 
\end{proof}

\subsection{Resolvent estimate for the linearized Euler}

First, Proposition \ref{lem:rayleigh} can be reformulated as follows. 

\begin{lemma}\label{lem:rayleigh-1}
    If $w$ satisfies 
  \begin{align}\label{eq:Ray-eps}
   & (V-\lambda-\mathrm{i}\epsilon)w+(\lambda+\mathrm{i}\epsilon) \mathcal{I}_kw=F,
\end{align}
here $\epsilon\in\R\setminus\{0\}$, then we have
\begin{align*}
   &\theta(k,\lambda)\|(\partial_y,k)\Delta_k^{-1}w\|_{L^2} \leq C\|(\partial_y,k)F\|_{L^2}.
\end{align*}
Here the constant $C$ is independent with $k,\ \epsilon$ and $\lambda$.
\end{lemma}

\begin{proof}
Let $f_0 = \mathcal{I}_kw$, $f=\Delta_kf_0$.  Then $w=f+f_0=f+\Delta^{-1}_k f$, and $f$ solves 
\begin{align*}
   &  V(f+\Delta^{-1}_k f)-(\lambda+\mathrm{i}\epsilon) f=F.
\end{align*}
Thus, we  infer from Proposition \ref{lem:rayleigh} that 
\begin{align*}
   &\theta(k,\lambda)\|(\partial_y,k)\Delta_k^{-1}f\|_{L^2} \leq C\|(\partial_y,k)F\|_{L^2}.
\end{align*}
 Thanks to $k\in2\pi\mathbb{Z}/\mathfrak{p}\setminus\{0\}$, $|k|>1$, we have
\begin{align*}
   & \|(\partial_y,k)\Delta_k^{-1}f\|_{L^2}\sim \|(\partial_y,k)\Delta_k^{-1}w\|_{L^2},
\end{align*}
which yields 
\begin{align*}
   &\theta(k,\lambda)\|(\partial_y,k)\Delta_k^{-1}w\|_{L^2}\leq C\theta(k,\lambda)\|(\partial_y,k)\Delta_k^{-1}f\|_{L^2} \leq C\|(\partial_y,k)F\|_{L^2}.
\end{align*}
\end{proof}

\begin{lemma}\label{lem:eps-w}
If $w$ satisfies 
  \begin{align}\label{eq:Ray-eps-000}
   & (V-\lambda-\mathrm{i}\epsilon)w+(\lambda+\mathrm{i}\epsilon) \mathcal{I}_kw=F,
\end{align}
 then it holds that for $\epsilon\neq 0$, 
  \begin{align*}
   &|\epsilon \theta|^{\f12}\|w\|_{L^2}\leq C\|(\partial_y,k)F\|_{L^2},\\
   &\|V'w\|_{L^2}\leq C\theta^{-\f12}(\lambda_{0+}/|\epsilon|+1)^{\f12}\|(\partial_y,k) F\|_{L^2},\\
   & \|(\partial_y,k)w\|_{L^2}\leq C|\epsilon|^{-3/2}(\lambda_{0+}/\theta+|\epsilon|)^{1/2}\|(\partial_y,k)F\|_{L^2}.
\end{align*}
\end{lemma}
\begin{proof}
  Assume $\epsilon>0$ for convenience.  Taking the inner product of \eqref{eq:Ray-eps-000} with $w$, we obtain
  \begin{align*}
     & \epsilon\|w\|_{L^2}^2-\epsilon\langle \mathcal{I}_kw,w\rangle=-\mathbf{Im}\langle F,w\rangle,\quad 
     \langle (V-\lambda)w,w\rangle+\lambda\langle \mathcal{I}_kw,w\rangle=\mathbf{Re}\langle F,w\rangle,\\
      &\epsilon\|w\|_{L^2}^2+\epsilon\|(\partial_y,\tilde{k})\Delta_{\tilde{k}}^{-1}w\|_{L^2}^2=-\mathbf{Im}\langle F,w\rangle\leq \|(\partial_y,k) F\|_{L^2}\|(\partial_y,k)\Delta_{k}^{-1}w\|_{L^2},
  \end{align*}
  By Lemma \ref{lem:rayleigh-1}, we have 
  \begin{align*}
   &|\epsilon \theta|^{\f12}\|w\|_{L^2}\leq C\|(\partial_y,k)F\|_{L^2}.
\end{align*}
This proves the first inequality of the lemma.  By Lemma \ref{lem1}, we get
\begin{align*}
   &\|V'w\|_{L^2}^2\leq C|\langle (V-\lambda)w,w\rangle|+C\lambda_{0+}\|w\|_{L^2}^2\leq C|\langle F,w\rangle|+C|\lambda\langle \mathcal{I}_kw,w\rangle|+C\lambda_{0+}\|w\|_{L^2}^2\\
   &\leq C\|(\partial_y,k) F\|_{L^2}\|(\partial_y,k)\Delta_{k}^{-1}w\|_{L^2}+
   C|\lambda|\|(\partial_y,\tilde{k})\Delta_{\tilde{k}}^{-1}w\|_{L^2}^2+C\lambda_{0+}\|w\|_{L^2}^2\\
   &\leq C\theta^{-1}\|(\partial_y,k) F\|_{L^2}^2+C|\lambda|\theta^{-2}\|(\partial_y,k) F\|_{L^2}^2+C\lambda_{0+}\|w\|_{L^2}^2\\
   &\leq C\theta^{-1}\|(\partial_y,k) F\|_{L^2}^2+C\lambda_{0+}|\epsilon \theta|^{-1}\|(\partial_y,k) F\|_{L^2}^2,
\end{align*}
which implies the second inequality of the lemma, here we used (see Lemma \ref{lem:rayleigh-1})
\begin{align*}
   &C^{-1}\|(\partial_y,\tilde{k})\Delta_{\tilde{k}}^{-1}w\|_{L^2}\leq\|(\partial_y,{k})\Delta_{{k}}^{-1}w\|_{L^2}\leq 
   C\theta^{-1}\|(\partial_y,k) F\|_{L^2}|.
\end{align*}

By taking $\partial_y$ to \eqref{eq:Ray-eps}, we get
\begin{align*}
   & (V-\lambda-\mathrm{i}\epsilon)\partial_yw+(\lambda+\mathrm{i}\epsilon) \mathcal{I}_k\partial_yw=\partial_yF-V'w.
\end{align*}
Taking the inner product with $\partial_yw$ gives
\begin{align*}
   &\epsilon \|\partial_yw\|_{L^2}^2-\epsilon\langle \mathcal{I}_k\partial_yw,\partial_yw\rangle \leq |\langle \partial_yF-V'w,\partial_yw\rangle|,\\
   &\epsilon \|\partial_yw\|_{L^2}^2+ \epsilon\|(\partial_y,\tilde{k})\Delta_{\tilde{k}}^{-1}\partial_yw\|_{L^2}^2 \leq (\|\partial_yF\|_{L^2}+\|V'w\|_{L^2})\|\partial_yw\|_{L^2},\\
   &\epsilon \|\partial_yw\|_{L^2}\leq C(\|\partial_yF\|_{L^2}+\|V'w\|_{L^2})\leq C(1+\theta^{-\f12}(\lambda_{0+}/\epsilon+1)^{-\f12})\|(\partial_y,k)F\|_{L^2}\\ &\leq  C\epsilon^{-\f12}((\lambda_{0+}+\epsilon)/\theta+\epsilon)^{\f12}\|(\partial_y,k)F\|_{L^2} \leq  C\epsilon^{-\f12}(\lambda_{0+}/\theta+\epsilon)^{\f12}\|(\partial_y,k)F\|_{L^2}.
\end{align*}
 This implies the third inequality of the lemma by noticing that 
  \begin{align*}
      &\epsilon\|w\|_{L^2}^2+\epsilon\|(\partial_y,\tilde{k})\Delta_{\tilde{k}}^{-1}w\|_{L^2}^2=-\mathbf{Im}\langle F,w\rangle\leq \| F\|_{L^2}\|w\|_{L^2},\quad\epsilon\|w\|_{L^2}\leq \| F\|_{L^2},\\
      &\|kw\|_{L^2}\leq \epsilon^{-1}\|k F\|_{L^2}\leq\epsilon^{-\f32}(\lambda_{0+}/\theta+\epsilon)^{\f12}\|(\partial_y,k)F\|_{L^2}.
  \end{align*}
  
  This completes the proof of the lemma.
    \end{proof}
    
 \subsection{Resolvent estimate for the simplified linearized operator}

\begin{lemma}\label{lem:SLNS-res}
If $\nu\in(0,1)$, $\lambda \in\R $, $ \delta>0$, $ \chi_1=\mathrm{Re}[(V-\lambda+\mathrm{i}\delta)^{-1}]$, $\lambda_{0+}\geq |\nu/k|^{\f12}$,  and
\begin{align}\label{eq1}
   & -\nu\Delta_k w+\mathrm{i}k(V-\lambda)w=F_1,
\end{align}
then we have
\begin{align*} 
&\|w\|_{L^{2}} \leq C\big(|k|^{-1}\nu(\lambda_0+\delta)^{\f12}\delta^{-2}+\delta\lambda_0^{-1/2}\big)\|w\|_{H_k^1}+C|k|^{-1}\|\chi_1F_1\|_{L^2},
\end{align*}
where
\begin{align*} 
&\|\chi_1F_1\|_{L^2}  \leq C\delta^{-\f12}\theta^{-\f12}
\big(\|F_1\|_{H_k^1}+|k|\|w\|_{H_k^{-1}}\big)+C\nu\lambda_0^{-1}\delta^{-\f12}\|w\|_{H_k^1}.\end{align*}

\end{lemma}

\begin{proof} 
Taking the inner product  of \eqref{eq1} with $\chi_1{w}$ and considering the imaginary part, we obtain
\begin{align*}
   & \mathrm{Im}\langle \chi_1F_1,w\rangle=
   -\nu\mathrm{Im}\langle \chi_1\Delta_k w,w\rangle+k\langle \chi_1(V-\lambda)w,w\rangle,\\
  & |\mathrm{Im}\langle \chi_1\Delta_k w,w\rangle|=|\mathrm{Im}\langle \chi_1' w',w\rangle|\leq \|\chi_1'\|_{L^{\infty}}\|w'\|_{L^{2}}\|w\|_{L^{2}},\quad |\langle \chi_1F_1,w\rangle|\leq\|\chi_1F_1\|_{L^2}\|w\|_{L^{2}},\\
  &|k||\langle \chi_1(V-\lambda)w,w\rangle|\leq \nu\|\chi_1'\|_{L^{\infty}}\|w'\|_{L^{2}}\|w\|_{L^{2}}+\|\chi_1F_1\|_{L^2}\|w\|_{L^{2}}.
  \end{align*}
  Together with the fact that
  \begin{align*}
  &\chi_1(V-\lambda)=\mathrm{Re}[(V-\lambda)/(V-\lambda+\mathrm{i}\delta)]=\mathrm{Re}[1-\mathrm{i}\delta/(V-\lambda+\mathrm{i}\delta)]
  =1-\delta^2|V-\lambda+\mathrm{i}\delta|^{-2},
  \end{align*}
we infer that
\begin{align*}
  \|w\|_{L^{2}}^2=&\langle \chi_1(V-\lambda)w,w\rangle+\|\delta|V-\lambda+\mathrm{i}\delta|^{-1}w\|_{L^{2}}^2\\ \leq&|k|^{-1}\big(\nu\|\chi_1'\|_{L^{\infty}}\|w'\|_{L^{2}}\|w\|_{L^{2}}+\|\chi_1F_1\|_{L^2}\|w\|_{L^{2}}\big)\\&+\|w\|_{L^{\infty}}^2
  \|\delta|V-\lambda+\mathrm{i}\delta|^{-1}\|_{L^{2}}^2.
  \end{align*}
  
  By Lemma \ref{lem:V3} and Gagliardo-Nirenberg inequality, we obtain
  \begin{align*}
  &\|\chi_1'\|_{L^{\infty}}\leq C(\lambda_0+\delta)^{\f12}\delta^{-2},\quad 
  \|\delta|V-\lambda+\mathrm{i}\delta|^{-1}\|_{L^{2}}^2\leq C\delta\lambda_0^{-1/2},\quad \|w\|_{L^{\infty}}^2\leq C\|w\|_{H_k^1}\|w\|_{L^2},\\
  &\|w\|_{L^{2}}^2 \leq C|k|^{-1}(\nu(\lambda_0+\delta)^{\f12}\delta^{-2}\|w'\|_{L^{2}}\|w\|_{L^{2}}+\|\chi_1F_1\|_{L^2}\|w\|_{L^{2}})+
  C\delta\lambda_0^{-1/2}\|w\|_{H_k^1}\|w\|_{L^2},\end{align*}
  which gives 
  \begin{align*}
  &\|w\|_{L^{2}} \leq C|k|^{-1}\big(\nu(\lambda_0+\delta)^{\f12}\delta^{-2}\|w'\|_{L^{2}}+\|\chi_1F_1\|_{L^2}\big)+
  C\delta\lambda_0^{-1/2}\|w\|_{H_k^1}\\ \leq& C\big(|k|^{-1}\nu(\lambda_0+\delta)^{\f12}\delta^{-2}+\delta\lambda_0^{-1/2}\big)\|w\|_{H_k^1}+C|k|^{-1}\|\chi_1F_1\|_{L^2}.
\end{align*}
This proves the first inequality of the lemma. For the second  inequality, we consider the following two cases.\smallskip

{\it Case 1.} $\lambda_0\geq |k|^{-2}$. In this case, we have $\theta=1+|k|\lambda_0^{1/2}+|\lambda|\sim |k|\lambda_0^{1/2}$ and by Lemma \ref{lem:V3}),
 \begin{align*}
    \|\chi_1F_1\|_{L^2}\leq&\|\chi_1\|_{L^2}\|F_1\|_{L^{\infty}}\leq\||V-\lambda+\mathrm{i}\delta|^{-1}\|_{L^2}\|F_1\|_{L^{\infty}}\leq C\delta^{-\f12}\lambda_0^{-1/4}\|F_1\|_{L^{\infty}}\\ \leq& C\delta^{-\f12}\lambda_0^{-1/4}|k|^{-\f12}\|F_1\|_{H_k^1}\leq C\delta^{-\f12}\theta^{-\f12}\|F_1\|_{H_k^1}.
 \end{align*}
 
 {\it Case 2.} $\lambda_0\leq |k|^{-2}$.  As $\lambda_{0+}\geq |\nu/k|^{\f12}$,  $\lambda_{0}=\lambda_{0+}$, $\lambda\in(m,M) $. Thus,
 $\lambda_0=M-\lambda$ or $\lambda_0=\lambda-m$. Without lose of generality, assume $\lambda_0=M-\lambda$. Then 
 \begin{align*}
    & |F_1(y)-F_1(y_1)|\leq |y-y_1|^{\f12}\|F_1'\|_{L^2},\quad y\in(y_2-2\pi+1/2,y_2+1/2),
 \end{align*}
 By Lemma \ref{lem:V3}, we have
 \begin{align}\label{F1}
    \|\chi_1F_1\|_{L^2}\leq&\|\chi_1(F_1-F_1(y_1))\|_{L^2}+\|\chi_1\|_{L^2}|F_1(y_1)|\\ \notag\leq & 
    \|\chi_1|y-y_1|^{\f12}\|_{L^2(y_2-2\pi+1/2,y_2+1/2)}\|F_1'\|_{L^2}+\|\chi_1\|_{L^2}|F_1(y_1)| 
    \\
    \notag\leq& C\delta^{-\f12}\|F_1'\|_{L^2}+C\delta^{-\f12}\lambda_0^{-1/4}|F_1(y_1)|.
 \end{align}
 
 Let $\Psi$ be a fixed nonnegative smooth function, supported in $[-1,1]$, equal to 1 on\\ $[-1/2,1/2]$. 
  Let $\Psi_1(y)=\Psi(\lambda_0^{-1/2}(y-y_1))$ for $y\in(y_2-2\pi,y_2)$. Then 
  \begin{align*}
    & \langle F_1,\Psi_1\rangle=-\nu\langle \Delta_k w,\Psi_1\rangle+\mathrm{i}k\langle (V-\lambda)w,\Psi_1\rangle,\\
    &|\langle \Delta_k w,\Psi_1\rangle|\leq \|w\|_{H_k^1}\|\Psi_1\|_{H_k^1},\quad |\langle (V-\lambda)w,\Psi_1\rangle|\leq \|w\|_{H_k^{-1}}\|(V-\lambda)\Psi_1\|_{H_k^{1}}.
    \end{align*}
  We also have
  \begin{align*}
    &\|\Psi_1\|_{H_k^1}\leq\|\Psi_1'\|_{L^2}+|k|\|\Psi_1\|_{L^2}=\lambda_0^{-1/4}\|\Psi'\|_{L^2(\R)}+|k|\lambda_0^{1/4}\|\Psi\|_{L^2}\\ &\qquad\leq
    C(\lambda_0^{-1/4}+|k|\lambda_0^{1/4})\leq C\lambda_0^{-1/4},\\
    &|V-\lambda|\leq |V-M|+|M-\lambda|\leq C|y-y_1|^2+\lambda_0,\quad |V'|\leq C|y-y_1|,\end{align*}
 here we used Lemma \ref{lem:V'}, and then
  \begin{align*}
    & \|(V-\lambda)\Psi_1\|_{H_k^{1}}\leq\|V'\Psi_1\|_{L^2}+\|(V-\lambda)\Psi_1'\|_{L^2}+|k|\|(V-\lambda)\Psi_1\|_{L^2}\\
    &\leq C\||y-y_1|\Psi_1\|_{L^2}+
    \|(C|y-y_1|^2+\lambda_0)\Psi_1'\|_{L^2}+|k|\|(C|y-y_1|^2+\lambda_0)\Psi_1\|_{L^2}\\
    &\leq C\lambda_0^{3/4}\||y|\Psi\|_{L^2(\R)}+C\lambda_0^{3/4}\||y|^2\Psi'\|_{L^2(\R)}+\lambda_0^{3/4}\|\Psi'\|_{L^2(\R)}\\&\qquad+
    C|k|\lambda_0^{5/4}\||y|^2\Psi\|_{L^2(\R)}+|k|\lambda_0^{5/4}\|\Psi\|_{L^2(\R)}
    \leq C(\lambda_0^{3/4}+|k|\lambda_0^{5/4})\leq C\lambda_0^{3/4}.
    \end{align*}
    
    Summing up, we arrive at
  \begin{align*}
    &|\langle F_1,\Psi_1\rangle|\leq C\nu\lambda_0^{-1/4}\|w\|_{H_k^1}+C|k|\lambda_0^{3/4}\|w\|_{H_k^{-1}}.
 \end{align*}
 
 On the other hand, we have
\begin{align*}
     &|\langle F_1,{\Psi}_1\rangle-F_1(y_1)\langle 1,{\Psi}_1\rangle|=
     |\langle F_1-F_1(y_1),{\Psi}_1\rangle|
     \leq \langle |y-y_1|^{1/2}\|F_1'\|_{L^2},{\Psi}_1\rangle\\
     &=\|F_1'\|_{L^2}\||y-y_1|^{1/2}{\Psi}_1\|_{L^1}=\|F_1'\|_{L^2}\lambda_0^{3/4}\||y|^{1/2}{\Psi}\|_{L^1(\R)}\leq C\|F_1'\|_{L^2}\lambda_0^{3/4},\end{align*}
     which yields 
     \begin{align*}
     &|F_1(y_1)\langle 1,{\Psi}_1\rangle|\leq C\nu\lambda_0^{-1/4}\|w\|_{H_k^1}+C|k|\lambda_0^{3/4}\|w\|_{H_k^{-1}}+C\|F_1'\|_{L^2}\lambda_0^{3/4}. \end{align*}
     Thanks to $|\langle 1,{\Psi}_1\rangle|=\lambda_0^{1/2}\|{\Psi}\|_{L^1(\R)}\geq \lambda_0^{1/2} $, we get
     \begin{align*}
     &|F_1(y_1)|\leq C\nu\lambda_0^{-3/4}\|w\|_{H_k^1}+C|k|\lambda_0^{1/4}\|w\|_{H_k^{-1}}+C\|F_1'\|_{L^2}\lambda_0^{1/4}.
  \end{align*}
  This along with \eqref{F1} shows 
  \begin{align*}
    \|\chi_1F_1\|_{L^2} 
    \leq& C\delta^{-\f12}\|F_1'\|_{L^2}+C\delta^{-\f12}\lambda_0^{-1/4}|F_1(y_1)|\\ \leq& C\delta^{-\f12}\|F_1'\|_{L^2}+C\nu\lambda_0^{-1}\delta^{-\f12}\|w\|_{H_k^1}+C|k|\delta^{-\f12}\|w\|_{H_k^{-1}}\\ \leq& C\delta^{-\f12}\big(\|F_1\|_{H_k^1}+|k|\|w\|_{H_k^{-1}}\big)+C\nu\lambda_0^{-1}\delta^{-\f12}\|w\|_{H_k^1}.
 \end{align*}
 Then the second inequality of the lemma follows from $\theta=1+|k|\lambda_0^{1/2}+|\lambda|\sim 1$ (in {\it Case 2}).
  \end{proof}

 \subsection{Resolvent estimate for the full linearized operator} 
  We introduce the full linearized operator
\begin{align*}
   & \mathcal{L}_{\lambda}= -\nu\Delta_k +\mathrm{i}k(V-\lambda)+\mathrm{i}k\lambda \mathcal{I}_k.
\end{align*}
 The dual operator of ${\mathcal{L}}_{\lambda}$ (in $L^2$) is
$\overline{\mathcal{L}}_{\lambda}= -\nu\Delta_k -\mathrm{i}k(V-\lambda)-\mathrm{i}k\lambda \mathcal{I}_k$, and the dual operator of 
${\mathcal{L}}_{\lambda}^{-1}$ (in $L^2$) is $\overline{\mathcal{L}}_{\lambda}^{-1}$, thus 
$\|\overline{\mathcal{L}}_{\lambda}^{-1}\|_{H^{a}_k\rightarrow H^b_k}= \|\mathcal{L}_{\lambda}^{-1}\|_{H^{-b}_k\rightarrow H^{-a}_k}$ 
(for $\lambda,a,b\in\mathbb{R}$, $b\leq a+2$). Here $\|f\|_{H^a_k}:=\|(-\Delta_k)^{a/2}f\|_{L^2}$. Then 
\beno
\|f\|_{H^0_k}=\|f\|_{L^2},\quad  \|f\|_{H^1_k}=\|(\partial_y,k)f\|_{L^2},\quad \|f\|_{H^{-1}_k}=\|(\partial_y,k)\Delta_k^{-1}f\|_{L^2}.
\eeno
 As $\overline{\mathcal{L}}_{\lambda}\overline{f}=\overline{\mathcal{L}_{\lambda}f}$, 
$\overline{\mathcal{L}}_{\lambda}^{-1}\overline{f}=\overline{\mathcal{L}_{\lambda}^{-1}f}$, we have $\|{\mathcal{L}}_{\lambda}^{-1}\|_{H^{a}_k\rightarrow H^b_k}= \|\overline{\mathcal{L}}_{\lambda}^{-1}\|_{H^{a}_k\rightarrow H^b_k}$. In summary,
\begin{align}\label{L1}
     &\|{\mathcal{L}}_{\lambda}^{-1}\|_{H^{a}_k\rightarrow H^b_k}=\|\overline{\mathcal{L}}_{\lambda}^{-1}\|_{H^{a}_k\rightarrow H^b_k}= 
     \|\mathcal{L}_{\lambda}^{-1}\|_{H^{-b}_k\rightarrow H^{-a}_k}.
  \end{align}
  
Now we consider $\mathcal{L}_\lambda w=F$, namely, 
\begin{align}\label{eq:w-os}
   & -\nu\Delta_k w+\mathrm{i}k(V-\lambda)w+\mathrm{i}k\lambda \mathcal{I}_kw=F.
\end{align}
We denote
\begin{align}\label{def:A(nu)}
   & A=A(\nu,k,\lambda):=\|\mathcal{L}_{\lambda}^{-1}\|_{H^{1}_{k}\rightarrow H^{-1}_k},
\end{align}
for $0<\nu\leq 1.$  Our goal is to show that $A=A(\nu,k,\lambda)\leq C|k\theta(k,\lambda)|^{-1}.$
\begin{lemma}\label{lem:A}
  It holds that  $A\leq \nu^{-1}|k|^{-4}<+\infty$ for $\nu\in(0,1)$.
\end{lemma}
\begin{proof}
  Let $F\in H^1_k$ and $w\in H^{-1}_k$ solve \eqref{eq:w-os}. Taking the inner product of \eqref{eq:w-os} with 
$w$ and considering the real part, we obtain
  \begin{align*}
     &\nu\|(\partial_y,k)w\|_{L^2}^2\leq |\langle F,w\rangle|\leq |k|^{-2}\|(\partial_y,k)F\|_{L^2}\|(\partial_y,k)w\|_{L^2}.
  \end{align*}
  Then $\|(\partial_y,k)w\|_{L^2}\leq \nu^{-1}|k|^{-2}\|(\partial_y,k)F\|_{L^2}$, which gives
  \begin{align*}
     & \|(\partial_y,k)\Delta_k^{-1} w\|_{L^2}\leq |k|^{-2}\|(\partial_y,k)w\|_{L^2}\leq \nu^{-1}|k|^{-4}\|(\partial_y,k)F\|_{L^2} .
  \end{align*} 
 Thus, $A\leq \nu^{-1}|k|^{-4}<+\infty$.
\end{proof}

\begin{lemma}\label{lem:res1-os-H-1-A}
If $\nu\in(0,1)$, $\lambda \in\R $, then we have
\begin{align*} 
&\|{\mathcal{L}}_{\lambda}^{-1}\|_{H^{-1}_k\rightarrow H^{-1}_k}=\|{\mathcal{L}}_{\lambda}^{-1}\|_{H^{1}_k\rightarrow H^1_k}
\leq |A/\nu|^{\f12}.\end{align*}
Moreover, if $\lambda_{0+}\geq |\nu/k|^{\f12}$, then
\begin{align*}
&\|{\mathcal{L}}_{\lambda}^{-1}\|_{L^2\rightarrow H^{-1}_k}=\|{\mathcal{L}}_{\lambda}^{-1}\|_{H^{1}_k\rightarrow L^2}\leq 
C|\nu/k|^{\f13}\lambda_0^{-1/6}|A/\nu|^{\f12}+C|k|^{-1}|\nu\lambda_0/k|^{-\f16}\theta^{-\f12}(1+|kA|).
\end{align*}
\end{lemma}

\begin{proof}We consider $\mathcal{L}_\lambda w=F$, that is
\begin{align}\label{eq:w-os-overline}
   & -\nu\Delta_k w+\mathrm{i}k(V-\lambda)w+\mathrm{i}k\lambda \mathcal{I}_kw=F.
\end{align}
As $A=\|{\mathcal{L}}_{\lambda}^{-1}\|_{H^{1}_k\rightarrow H^{-1}_k}$, we have(for convenience we write $A(\nu,k,\lambda)$ as $A$, write $\theta(k,\lambda)$ as $\theta$)
\begin{align*}
   & \|{w}\|_{H^{-1}_k}=\|(\partial_y,k)\Delta_k^{-1}{w}\|_{L^2}\leq A\|(\partial_y,k){F}\|_{L^2}=A\|{F}\|_{H^{1}_k}.
\end{align*}
 Taking the inner product of \eqref{eq:w-os-overline}  with ${w}$ and considering the real part, we obtain
  \begin{align*}
     & \nu\|(\partial_y,k){w}\|_{L^2}^2\leq |\langle{F},{w}\rangle|
\leq \|(\partial_y,k){F}\|_{L^2}\|(\partial_y,k)\Delta_k^{-1}{w}\|_{L^2}\leq A\|(\partial_y,k){F}\|_{L^2}^2,\\
     &\|(\partial_y,k){w}\|_{L^2}\leq |A/\nu|^{\f12}\|(\partial_y,k){F}\|_{L^2},
  \end{align*}
 which gives $\|{\mathcal{L}}_{\lambda}^{-1}\|_{H^{1}_k\rightarrow H^1_k}\leq |A/\nu|^{\f12}$.
 By \eqref{L1}, we have
 \begin{align*}
    & \|\mathcal{L}_{\lambda}^{-1}\|_{H^{-1}_k\rightarrow H^{-1}_k}=\|{\mathcal{L}}_{\lambda}^{-1}\|_{H^{1}_k\rightarrow H^1_k}\leq |A/\nu|^{\f12}.
  \end{align*}
  
Now we assume $\lambda_{0+}\geq |\nu/k|^{\f12}$, then $\lambda_{0}=\lambda_{0+}$, $\lambda\in(m,M) $. Let $ \chi_1=\mathrm{Re}[(V-\lambda+\mathrm{i}\delta)^{-1}]$ for some $\delta>0$, $F_1=F-\mathrm{i}k\lambda \mathcal{I}_kw$, then
$-\nu\Delta_k w+\mathrm{i}k(V-\lambda)w=F_1.$
 By Lemma \ref{lem:SLNS-res}, we have
 \begin{align*}
 \|w\|_{L^{2}} \leq& C\big(|k|^{-1}\nu(\lambda_0+\delta)^{\f12}\delta^{-2}+\delta\lambda_0^{-1/2}\big)\|w\|_{H_k^1}+C|k|^{-1}\|\chi_1F_1\|_{L^2}\\
\leq& C\big(|k|^{-1}\nu(\lambda_0+\delta)^{\f12}\delta^{-2}+\delta\lambda_0^{-1/2}\big)\|w\|_{H_k^1}\\
&+C|k|^{-1}\big(\delta^{-\f12}\theta^{-\f12}(\|F_1\|_{H_k^1}+|k|\|w\|_{H_k^{-1}})+\nu\lambda_0^{-1}\delta^{-\f12}\|w\|_{H_k^1}\big)\\
\leq& C\big(|k|^{-1}\nu(\lambda_0+\delta)^{\f12}\delta^{-2}+\delta\lambda_0^{-1/2}+|k|^{-1}\nu\lambda_0^{-1}\delta^{-\f12}\big)\|w\|_{H_k^1}
\\&+C|k|^{-1}\delta^{-\f12}\theta^{-\f12}\big(\|F_1\|_{H_k^1}+|k|\|w\|_{H_k^{-1}}\big).
\end{align*}
Now we take $\delta=|\nu\lambda_0/k|^{\f13}$. Then we have $\lambda_{0}=\lambda_{0+}\geq |\nu/k|^{\f12}$, 
$|\lambda|\leq C $ and\begin{align*}
 &\lambda_0/\delta=\lambda_0^{2/3}|\nu/k|^{-\f13}\geq (|\nu/k|^{\f12})^{2/3}|\nu/k|^{-\f13}=1,\quad \lambda_0+\delta\leq2\lambda_0,\quad \lambda_0^{-1}\delta^{-\f12}\leq \lambda_0^{1/2}\delta^{-2},\\
 &\|F_1\|_{H_k^1}\leq\|F\|_{H_k^1}+|k\lambda|\| \mathcal{I}_kw\|_{H_k^1}=\|F\|_{H_k^1}+|k\lambda|\| \Delta_{\tilde{k}}^{-1}w\|_{H_k^1}\leq\|F\|_{H_k^1}+C|k|\| w\|_{H_k^{-1}}.
\end{align*}
Thus, we arrive at
\begin{align*} &\|w\|_{L^{2}} \leq C|\nu/k|^{\f13}\lambda_0^{-1/6}\|w\|_{H_k^1}+C|k|^{-1}|\nu\lambda_0/k|^{-\f16}\theta^{-\f12}(\|F_1\|_{H_k^1}+|k|\|w\|_{H_k^{-1}})\\
 &\leq C|\nu/k|^{\f13}\lambda_0^{-1/6}|A/\nu|^{\f12}\|F\|_{H_k^1}+C|k|^{-1}|\nu\lambda_0/k|^{-\f16}\theta^{-\f12}(\|F\|_{H_k^1}+|kA|\|F\|_{H_k^1}).
\end{align*}
Hence, $\|{\mathcal{L}}_{\lambda}^{-1}\|_{H^{1}_k\rightarrow L^2}\leq C|\nu/k|^{\f13}\lambda_0^{-1/6}|A/\nu|^{\f12}+C|k|^{-1}|\nu\lambda_0/k|^{-\f16}\theta^{-\f12}(1+|kA|)$.
 This along with \eqref{L1} shows
 \begin{align*}
    & \|{\mathcal{L}}_{\lambda}^{-1}\|_{L^2\rightarrow H^{-1}_k}=\|{\mathcal{L}}_{\lambda}^{-1}\|_{H^{1}_k\rightarrow L^2}\leq 
C|\nu/k|^{\f13}\lambda_0^{-1/6}|A/\nu|^{\f12}+C|k|^{-1}|\nu\lambda_0/k|^{-\f16}\theta^{-\f12}(1+|kA|).
 \end{align*}
 
 This completes the proof of the lemma.
\end{proof}

\begin{lemma}\label{lem:res1-os-H-1-A2}
If $\nu\in(0,1)$, $\lambda \in\R $, then $A=\|\mathcal{L}_{\lambda}^{-1}\|_{H^{1}_{k}\rightarrow H^{-1}_k}\leq C|k\theta(k,\lambda)|^{-1}$.
\end{lemma}

\begin{proof}
{\it Case 1.} $|\nu k^3|\geq \theta$. By Lemma \ref{lem:A}, we have 
$A\leq \nu^{-1}|k|^{-4}=|\nu k^3|^{-1}|k|^{-1}\leq\theta^{-1}|k|^{-1}=|k\theta|^{-1}$.

We consider $\mathcal{L}_\lambda w=F$, namely, 
\begin{align*}
   & -\nu\Delta_k w+\mathrm{i}k(V-\lambda)w+\mathrm{i}k\lambda \mathcal{I}_kw=F.
\end{align*}
  We decompose $w=w_1+w_2$, where $w_1$ and $w_2$ solve
  \begin{equation}\label{eq:w=w123}
    \left\{\begin{aligned}
   & (V-\lambda-\mathrm{i}\delta)w_1+(\lambda+\mathrm{i}\delta)\mathcal{I}_kw_1=F/(\mathrm{i}k),\\
   &-\nu\Delta_kw_2+\mathrm{i}k(V-\lambda)w_2+\mathrm{i}k\lambda\mathcal{I}_kw_2=\nu\Delta_kw_1+ k\delta (w_1-\mathcal{I}_kw_1),
\end{aligned} 
    \right.
  \end{equation}
with $\delta> 0$ to be determined.
    
  By Lemma \ref{lem:rayleigh-1} and Lemma \ref{lem:eps-w}, we get
  \begin{equation}\label{est:w1-00}
    \begin{aligned}
   &\|w_1\|_{L^2}\leq C|\delta \theta|^{-\f12}|k|^{-1}\|(\partial_y,k)F\|_{L^2},\quad \|w_1\|_{H_k^{-1}}\leq C|k\theta|^{-1}\|(\partial_y,k)F\|_{L^2}\\
   & \|(\partial_y,k)w_1\|_{L^2}\leq C\delta^{-3/2}(\lambda_{0+}/\theta+\delta)^{1/2}|k|^{-1}\|(\partial_y,k)F\|_{L^2}.
\end{aligned}
  \end{equation}

{\it Case 2.} $|\nu k^3|\leq \theta$, $\lambda_{0+}\leq |\nu/k|^{\f12}$.
By Lemma \ref{lem:res1-os-H-1-A} and \eqref{est:w1-00}, we get
\begin{align*}
     &\|(\partial_y,k)\Delta_k^{-1}w_2\|_{L^2}\leq |A/\nu|^{\f12}\big(\nu\|(\partial_y,k)w_1\|_{L^2}+|k|\delta\|w_1-\mathcal{I}_kw_1\|_{H_k^{-1}}\big)\\
     &\leq  |A/\nu|^{\f12}\big(\nu\|(\partial_y,k)w_1\|_{L^2}+C|k|\delta\|w_1\|_{H_k^{-1}}\big)\\
     &\leq  C|A/\nu|^{\f12}\big(\nu\delta^{-3/2}(\lambda_{0+}/\theta+\delta)^{1/2}|k|^{-1}+\delta/\theta\big)\|(\partial_y,k)F\|_{L^2}.
  \end{align*}
 Taking $\delta=|\nu\theta/k|^{\f12}$, we obtain (as $\theta\geq 1 $)
 \begin{align*}
     &\lambda_{0+}/\theta\leq \lambda_{0+}\leq |\nu/k|^{\f12}\leq\delta,\quad \delta^{-3/2}(\lambda_{0+}/\theta+\delta)^{1/2}\leq 2\delta^{-1},\\
     &\|(\partial_y,k)\Delta_k^{-1}w_2\|_{L^2}\le C|A/\nu|^{\f12}|\nu/(k\theta)|^{\f12}\|(\partial_y,k)F\|_{L^2}= C|A/(k\theta)|^{\f12}\|(\partial_y,k)F\|_{L^2},
  \end{align*} 
  which gives $A=\|{\mathcal{L}}_{\lambda}^{-1}\|_{H^{1}_k\rightarrow H^{-1}_k}\leq C|A/(k\theta)|^{\f12},$ $A\leq C|k\theta|^{-1}$.\smallskip
  
  {\it Case 3.} $|\nu k^3|\leq \theta$, $\lambda_{0+}\geq |\nu/k|^{\f12}$. Then $\lambda_{0+}=\lambda_{0} $, $\lambda\in(m,M) $.
  By Lemma \ref{lem:res1-os-H-1-A} and \eqref{est:w1-00}, we have
  \begin{align*}
     &\|(\partial_y,k)\Delta_k^{-1}w_2\|_{L^2}\leq |A/\nu|^{\f12}\nu\|(\partial_y,k)w_1\|_{L^2}\\&\quad+ C\big[|\nu/k|^{\f13}\lambda_0^{-1/6}|A/\nu|^{\f12}+|k|^{-1}|\nu\lambda_0/k|^{-\f16}\theta^{-\f12}(1+|kA|)\big]|k|\delta\|w_1-\mathcal{I}_kw_1\|_{L^2}\\
     &\leq  |A\nu|^{\f12}\|(\partial_y,k)w_1\|_{L^2}+C\big[|\nu/k|^{\f13}\lambda_0^{-1/6}|A/\nu|^{\f12}+|k|^{-1}|\nu\lambda_0/k|^{-\f16}\theta^{-\f12}(1+|kA|)\big]|k|\delta\|w_1\|_{L^2}\\
     &\leq  C|A\nu|^{\f12}\delta^{-3/2}(\lambda_0/\theta+\delta)^{1/2}|k|^{-1}\|(\partial_y,k)F\|_{L^2} \\&\quad+C\big[|\nu/k|^{\f13}\lambda_0^{-1/6}|A/\nu|^{\f12}+|k|^{-1}|\nu\lambda_0/k|^{-\f16}\theta^{-\f12}(1+|kA|)\big]|k|\delta|\delta \theta|^{-\f12}|k|^{-1}\|(\partial_y,k)F\|_{L^2}\\
     &=C|A\nu|^{\f12}|k|^{-1}\big[\delta^{-3/2}(\lambda_0/\theta+\delta)^{1/2}+|\nu/k|^{-\f23}\lambda_0^{-1/6}|\delta/\theta|^{\f12}\big]\|(\partial_y,k)F\|_{L^2}\\
     &\quad+C\delta^{\f12}|\nu\lambda_0/k|^{-\f16}\theta^{-1}(|k|^{-1}+|A|)\|(\partial_y,k)F\|_{L^2}.
  \end{align*}
  
  \def\Cd{C_1}
  Now we take $\delta=c_0|\nu\lambda_0/k|^{\f13}$, here $c_0\in(0,1)$ is a small constant. Then we conclude 
  \begin{align*}
  &\theta=1+|k|\lambda_0^{1/2}+|\lambda|\leq C+|k|\lambda_0^{1/2},\quad \theta^{\f43}\leq C(1+|k|^{\f43}\lambda_0^{2/3}),\quad
  \nu^{\f13}|k|=|\nu k^3|^{\f13}\leq \theta^{\f13},\\
  &\theta^{\f43}\delta/\lambda_0=\theta^{\f43}c_0|\nu/k|^{\f13}\lambda_0^{-2/3}\leq C(1+|k|^{\f43}\lambda_0^{2/3})|\nu/k|^{\f13}\lambda_0^{-2/3}=
  C(\lambda_0^{-2/3}+|k|^{\f43})|\nu/k|^{\f13}\\ &\qquad\leq C((|\nu/k|^{\f12})^{-2/3}+|k|^{\f43})|\nu/k|^{\f13}=C(1+\nu^{\f13}|k|)\leq C\theta^{\f13},\quad
  \delta\leq C\lambda_0/\theta,
  \end{align*}
  Thus, we obtain
  \begin{align*}
  &\|(\partial_y,k)\Delta_k^{-1}w_2\|_{L^2}
     =C\big(c_0^{-3/2}|A/(k\theta)|^{1/2}+c_0^{1/2}\theta^{-1}(|k|^{-1}+|A|)\big)\|(\partial_y,k)F\|_{L^2},
  \end{align*}  
  which gives 
 \beno
 A=\|{\mathcal{L}}_{\lambda}^{-1}\|_{H^{1}_k\rightarrow H^{-1}_k}\leq \Cd\big(c_0^{-3/2}|A/(k\theta)|^{1/2}+c_0^{1/2}\theta^{-1}(|k|^{-1}+|A|)\big).
 \eeno 
  Taking  $c_0=(2\Cd)^{-2}$ implies(as $\theta\geq 1 $)
  \begin{align*}
  &A\leq 2\Cd\big(c_0^{-3/2}|A/(k\theta)|^{1/2}+c_0^{1/2}\theta^{-1}|k|^{-1})\leq C(|A/(k\theta)|^{1/2}+|k\theta|^{-1}\big),
  \end{align*}
  which leads to $ A\leq C|k\theta|^{-1}$.
  \end{proof}

\begin{lemma}\label{pro:res1-os}
If $\nu\in(0,1)$, $\lambda \in\R $ and
\begin{align*}
   & -\nu\Delta_k w+\mathrm{i}k(V-\lambda)w+\mathrm{i}k\lambda \mathcal{I}_kw=F,
\end{align*}
then we have
\begin{align}\label{est:res1-os-1}
   & |\nu k\theta(k,\lambda)|^{\f12}\|(\partial_y,k)w\|_{L^2}+|k|\theta(k,\lambda)\|(\partial_y,k)\Delta_k^{-1}w\|_{L^2}\leq C\|(\partial_y,k) F\|_{L^2},\\
   & \nu \|(\partial_y,k)w\|_{L^2}+|\nu k\theta(k,\lambda)|^{\f12}\|(\partial_y,k)\Delta_k^{-1}w\|_{L^2}\leq C\|(\partial_y,k)\Delta_k^{-1} F\|_{L^2}.\label{est:res1-os-2}
\end{align}
\end{lemma}
\begin{proof}
The inequality \eqref{est:res1-os-1} follows from Lemma \ref{lem:res1-os-H-1-A}, Lemma \ref{lem:res1-os-H-1-A2} and the definition of $A(\nu,k,\lambda)$ in  \eqref{def:A(nu)}.
Notice that
\begin{align}
   & \nu\|(\partial_y,k)w\|_{L^2}^2\leq |\langle F,w\rangle|\leq \|(\partial_y,k)\Delta_k^{-1}F\|_{L^2}\|(\partial_y,k)w\|_{L^2},\nonumber\\
   &\nu\|(\partial_y,k)w\|_{L^2}\leq \|(\partial_y,k)F\|_{L^2}.\nonumber
\end{align}
 which along with Lemma \ref{lem:res1-os-H-1-A} and Lemma \ref{lem:res1-os-H-1-A2}  gives \eqref{est:res1-os-2}.
\end{proof}

\begin{lemma}\label{lem:kthea-0}
  If $\nu\in(0,1)$, $\lambda \in\R $,
\begin{align*}
   & -\nu\Delta_k f+\mathrm{i}kV(1+\Delta_k^{-1})f-\mathrm{i}k\lambda f=F,
\end{align*}
then we have
\begin{align*}
&|\nu k\theta(k,\lambda)|^{\f12}\|(\partial_y,k)f\|_{L^2}+|k|\theta(k,\lambda)
\|(\partial_y,k)\Delta_k^{-1}f\|_{L^2}\leq C\|(\partial_y,k) F\|_{L^2},\\
   &\nu \|(\partial_y,k)f\|_{L^2}+ |\nu k\theta(k,\lambda)|^{\f12}\|(\partial_y,k)\Delta_k^{-1}f\|_{L^2}\leq C\|(\partial_y,k)\Delta_k^{-1} F\|_{L^2}.
\end{align*}
\end{lemma}

\begin{proof}
First, we have
\begin{align*}
&\mathrm{Re}\langle F,(1+\Delta_k^{-1})f\rangle=\langle -\nu\Delta_k f,(1+\Delta_k^{-1})f\rangle=
\nu\big(\|\partial_yf\|_{L^2}^2+(k^2-1)\|f\|_{L^2}^2\big),\\
&\nu\big(\|\partial_yf\|_{L^2}^2+(k^2-1)\|f\|_{L^2}^2\big)\leq \|F\|_{H_k^{-1}}\|(1+\Delta_k^{-1})f\|_{H_k^{1}}\leq \|F\|_{H_k^{-1}}\|f\|_{H_k^{1}},\\
&\|f\|_{H_k^{1}}\leq C\big(\|\partial_yf\|_{L^2}^2+(k^2-1)\|f\|_{L^2}^2\big)^{1/2}\leq C\nu^{-1}\|F\|_{H_k^{-1}}\leq C\nu^{-1}|k|^{-2}\|F\|_{H_k^{1}}.
\end{align*}
  Let $w=f+\Delta_k^{-1}f$, then $w$ solves $-\nu\Delta_k w+\mathrm{i}k(V-\lambda)w+\mathrm{i}k\lambda \mathcal{I}_kw=F-\nu f$. We then  infer from  
 Proposition  \ref{pro:res1-os} that 
  \begin{align*}
     |\nu k\theta(k,\lambda)|^{\f12}\|(\partial_y,k)w\|_{L^2}&+
     |k|\theta\|(\partial_y,k)\Delta_k^{-1}w\|_{L^2}\leq C\|(\partial_y,k) F\|_{L^2} +C\nu\|(\partial_y,k)f\|_{L^2}\\
     \leq &C\|(\partial_y,k) F\|_{L^2} +C|  k|^{-2}\|(\partial_y,k) F\|_{L^2}
      \leq C\|(\partial_y,k) F\|_{L^2}. 
  \end{align*} 
  Meanwhile, we have
  \begin{align*}
    &\nu \|(\partial_y,k)w\|_{L^2} +|\nu k\theta|^{\f12}\|(\partial_y,k)\Delta_k^{-1}w\|_{L^2} \leq  C\|F\|_{H_k^{-1}} +C\nu\|f\|_{H_k^{-1}}\\ &\leq  C\|F\|_{H_k^{-1}} +C\nu |k|^{-2}\|f\|_{H_k^{1}}\leq 
     C\|F\|_{H_k^{-1}} +C|  k|^{-2}\|F\|_{H_k^{-1}}
      \leq C\|F\|_{H_k^{-1}}.
  \end{align*}

 Thanks to the  facts that $\|(\partial_y,k)\Delta_k^{-1}f\|_{L^2}\sim \|(\partial_y,k)\Delta_k^{-1}w\|_{L^2}$, $\|(\partial_y,k)f\|_{L^2}\sim \|(\partial_y,k)w\|_{L^2}$, 
  we complete the proof of the lemma.
\end{proof}

Now we are in a position to prove Proposition \ref{pro:res1-os-V'}.

  \begin{proof}
  \textit{Step 1}. We first prove \eqref{est:res1-os-1-V'}. Let $\phi=\Delta_k^{-1}f$, then
\begin{align*}
    \mathrm{Im}\langle F,\phi\rangle=&k\mathrm{Re}\langle V(1+\Delta_k^{-1})f-\lambda f,\phi\rangle
    =k\mathrm{Re}\langle (V-\lambda)\Delta_k\phi+V\phi,\phi\rangle\\
    =&k\int_{\T}\big[-(V-\lambda)(|\phi'|^2+k^2|\phi|^2)+(V+V''/2)|\phi|^2\big],\\
    \left|\int_{\T}(V-\lambda)(|\phi'|^2+k^2|\phi|^2)\right|\leq&|k|^{-1}|\langle F,\phi\rangle|+\int_{\T}|V+V''/2||\phi|^2
    \leq|k|^{-1}|\langle F,\phi\rangle|+C\|\phi\|_{L^2}^2.
\end{align*} 
 We then get by Lemma \ref{lem1}  that 
\begin{align}
\notag   \||V'|(\partial_y,k)\phi\|_{L^2}^2\leq&C\left|\int_{\T}(V-\lambda)(|\phi'|^2+k^2|\phi|^2)\right|
   +C\lambda_{0+}\|(\partial_y,k)\phi\|_{L^2}^2\\
\label{V1}   \leq&C|k|^{-1}|\langle F,\phi\rangle|+C\|\phi\|_{L^2}^2+C\lambda_0\|(\partial_y,k)\phi\|_{L^2}^2\\
 \notag  \leq&C|k|^{-1}\|F\|_{L^2}\|\phi\|_{L^2}+C\|\phi\|_{L^2}^2+C\lambda_0\|(\partial_y,k)\phi\|_{L^2}^2\\ \notag  \leq&C|k|^{-2}\|F\|_{L^2}^2+C(\lambda_0+|k|^{-2})\|(\partial_y,k)\phi\|_{L^2}^2,
\end{align}
which along with Lemma \ref{lem:kthea-0} gives(as $\theta(k,\lambda)\geq1+|k|\lambda_0^{1/2}$)
\begin{align*}
\||V'|(\partial_y,k)\phi\|_{L^2}\leq & C(\lambda_0^{1/2}+|k|^{-1})\|(\partial_y,k)\phi\|_{L^2}+C|k|^{-1}\|F\|_{L^2}
 \\
 \leq&
 C|k|^{-1}\theta(k,\lambda)\|(\partial_y,k)\phi\|_{L^2}+C|k|^{-1}\|F\|_{L^2}\\
 \leq& C|k|^{-2}\|(\partial_y,k)F\|_{L^2}+C|k|^{-1}\|F\|_{L^2}\leq C|k|^{-2}\|(\partial_y,k)F\|_{L^2},\end{align*}
 which leads to 
\begin{align*}
&\||V'|^{\f12}(\partial_y,k)\phi\|_{L^2}\leq\||V'|(\partial_y,k)\phi\|_{L^2}^{1/2}\|(\partial_y,k)\phi\|_{L^2}^{1/2}\leq C|k|^{-1}|k\theta(k,\lambda)|^{-1/2}\|(\partial_y,k)F\|_{L^2}. 
\end{align*}  
This along with Lemma \ref{lem:kthea-0} and $\theta(k,\lambda)\geq1$ proves \eqref{est:res1-os-1-V'}.\smallskip

\textit{Step 2}. We next prove \eqref{est:res1-os-2-V'}. By \eqref{V1} and Lemma \ref{lem:kthea-0}, we have (as $\theta(k,\lambda)\geq1+|k|\lambda_0^{1/2}$)
\begin{align*}
&\|\phi\|_{H_k^{1}}\leq|k|^{-2}\|\phi\|_{H_k^{3}}=|k|^{-2}\|f\|_{H_k^{1}}\leq C|k|^{-2}\nu^{-1}\| F\|_{H_k^{-1}},\\
&|k|^{-1}|\langle F,\phi\rangle|\leq|k|^{-1}\| F\|_{H_k^{-1}}\|\phi\|_{H_k^{1}}\leq C|k|^{-3}\nu^{-1}\| F\|_{H_k^{-1}}^2,\\
&\||V'|(\partial_y,k)\phi\|_{L^2}^2\leq C|k|^{-1}|\langle F,\phi\rangle|+C\|\phi\|_{L^2}^2+C\lambda_0\|(\partial_y,k)\phi\|_{L^2}^2\\
&\quad\leq C|k|^{-3}\nu^{-1}\| F\|_{H_k^{-1}}^2+C(|k|^{-2}+\lambda_0)\|(\partial_y,k)\phi\|_{L^2}^2,
\end{align*}
which yields 
\begin{align*}
\||V'|(\partial_y,k)\phi\|_{L^2}\leq&  C|k|^{-1}|\nu k|^{-\frac12}\| F\|_{H_k^{-1}}+C(|k|^{-1}+\lambda_0^{1/2})\|(\partial_y,k)\phi\|_{L^2}\\
\leq&  C|k|^{-1}|\nu k|^{-\frac12}\| F\|_{H_k^{-1}}+C|k|^{-1}\theta(k,\lambda)\|(\partial_y,k)\phi\|_{L^2}.
\end{align*}
Then by H\"older's inequality, $\theta(k,\lambda)\geq1$ and Lemma \ref{lem:kthea-0}, we have
\begin{align*}
   &|k|^{\frac12}\||V'|^{\frac12}(\partial_y,k)\phi\|_{L^2}\leq|k|^{\frac12}
   \||V'|(\partial_y,k)\phi\|_{L^2}^{1/2}\|(\partial_y,k)\phi\|_{L^2}^{1/2}\\
  &\leq  C\big(|\nu k|^{-\frac12}\| F\|_{H_k^{-1}}+\theta(k,\lambda)\|(\partial_y,k)\phi\|_{L^2}\big)^{1/2}\|(\partial_y,k)\phi\|_{L^2}^{1/2}\\
  &\leq C\big(|\nu k|^{-\frac12}\| F\|_{H_k^{-1}}+|\theta(k,\lambda)|^{\f12}\|(\partial_y,k)\phi\|_{L^2}\big) \leq C|\nu k|^{-\frac12}\| F\|_{H_k^{-1}},
\end{align*}
which along with Lemma \ref{lem:kthea-0} gives \eqref{est:res1-os-2-V'}.
 \end{proof}

\section{Inviscid damping and vorticity depletion for the linearized Euler}\label{sec:invisciddamping}

 Let $b(y)=\cos y$, $\Delta_k=\partial_y^2-k^2$ and $\omega_k$ solve ($ k\in\alpha \mathbb{Z}\setminus\{0\},\ \alpha:=2\pi/\mathfrak{p}>1$)
\begin{equation}\label{eq:wk0}
    \partial_t\omega_k+\mathrm{i}k b(\omega_k+\psi_k)=0,\quad \Delta_k\psi_k=\omega_k,\quad 
     \omega_k|_{t=0}=\omega_{0,k}.
\end{equation}
For fixed $k$, we denote
\begin{align}\label{def:AB}
  & A:=\sin y(1+\Delta_k^{-1}),\qquad B:= \cos y(1+\Delta_k^{-1}),\\
  &\Delta_{k,s}=\mathrm{e}^{-\mathrm{i}sy}\Delta_k\mathrm{e}^{\mathrm{i}sy}= (\partial_y+\mathrm{i}s)^2-k^2.\label{def:Del-ks}
\end{align}
Then $ \Delta_{k,s}^{-1}=\mathrm{e}^{-\mathrm{i}sy}\Delta_k^{-1}\mathrm{e}^{\mathrm{i}sy}$, $B$ is symmetric with respect to the inner product 
  $\langle ,\rangle_{*}$ defined in \eqref{fg*}, i.e., $\langle f,g\rangle_\ast := \langle f, (1+\Delta_k^{-1})g\rangle$.

The goal of this section is to prove Theorem \ref{thm:LE}.  
Thanks to $\omega_k=\mathrm{e}^{-\mathrm{i}tkB}\omega_{0,k},\ \psi_k=\Delta_k^{-1}\omega_k$, Theorem \ref{thm:LE} is equivalent to the following proposition by replacing $kt$ with $t$.

 \begin{Proposition}\label{lem:eibtf}
  Let $\omega=\mathrm{e}^{-\mathrm{i}Bt}f$, $\psi=\Delta_k^{-1}\omega$, it holds that
   \begin{align}\label{eq12}
      &|\psi(t,y)|\leq  C(|t|+k^2)^{-2}|k|^{-1/2}\|f\|_{H_k^3},\\
      &|\partial_y\psi(t,y)|\leq  C((|t|+k^2)^{-3/2}|k|^{-1/2} +(|t|+k^2)^{-1}|b'||k|^{-1/2})\|f\|_{H_k^3},\label{est:vor-dep-0}\\
      &{|\omega(t,y)|\leq  C\min((|t|+k^2)^{-1}|k|^{-1/2} +|b'|^2|k|^{-1/2},|k|^{-5/2})\|f\|_{H_k^3}},\label{est:vor-dep-1}\\
      &|\partial_y(\mathrm{e}^{\mathrm{i}tb}\omega)(t,y)|\leq  C\min((|t|+k^2)^{-1/2}|k|^{-1/2}+|b'||k|^{-1/2},|k|^{-3/2})\|f\|_{H_k^3},\label{est:vor-dep-2}\\
      &|\partial_y(\mathrm{e}^{\mathrm{i}tb}\psi)(t,y)|\leq C(|t|+k^2)^{-1.2}|k|^{-1}\|f\|_{H_k^3},\label{est:pareitbpsi}\\
      &\|(\partial_y,k)^2(\mathrm{e}^{\mathrm{i}tb}\omega)\|_{L^2}\leq C|k|^{-1}\|f\|_{H_k^3}.\label{eq17}
   \end{align}
 \end{Proposition}
 
 \subsection{Construction of vector fields and coercive estimate}
 
 We introduce the operators
 \beno
 &&\Lambda_1=[\Delta_k,B],\quad  \Lambda_2=[\Lambda_1,B],\quad  \Delta_{k,s}=\mathrm{e}^{-\mathrm{i}sy}\Delta_k\mathrm{e}^{\mathrm{i}sy}= (\partial_y+\mathrm{i}s)^2-k^2,\\
 &&\Lambda_3=(1-\Delta_k^{-1})\Delta_{k,-1}^{-1}\Delta_{k,1}^{-1},\quad \Delta_{(t)}=\Delta_k+\mathrm{i}t\Lambda_1 -t^2(1-B^2).  
 \eeno

 \begin{lemma}\label{lem1-dec}
It holds that 
\beno 
\Lambda_2=2(1-B^2)-4k^2\Lambda_3.
\eeno
 Let $\omega_1=\Delta_{(t)}\omega$ and $\omega=\mathrm{e}^{-\mathrm{i}tB}f$. Then we have
 \beno
  (\partial_t+\mathrm{i}B)\omega_1=-4tk^2\Lambda_3\omega.
  \eeno
   \end{lemma}
   \begin{proof}
   First, we have $[\partial_y,B]=-A,$ $[\partial_y,A]=B$, $\Delta_k=\partial_y^2-k^2 $ 
   and 
   \begin{align}\label{lam1}
      &\Lambda_1=[\Delta_k,B]=[\partial_y,B]\partial_y+\partial_y[\partial_y,B]=-A\partial_y- \partial_y A=-2A\partial_y-B=-2\partial_yA+B,
   \end{align}
   and (see \cite{WZ-SCM})
   \begin{align}\label{def:[AB]}
       [A,B]&=\big(\sin y\Delta_k^{-1}\cos y- \cos y\Delta_k^{-1}\sin y\big)(1+\Delta_k^{-1})\nonumber\\
      &=\dfrac{1}{4\mathrm{i}}\left( (\mathrm{e}^{\mathrm{i}y}-\mathrm{e}^{-\mathrm{i}y}) \Delta_k^{-1} (\mathrm{e}^{\mathrm{i}y}+\mathrm{e}^{-\mathrm{i}y})-(\mathrm{e}^{\mathrm{i}y}+ \mathrm{e}^{-\mathrm{i}y}) \Delta_k^{-1} (\mathrm{e}^{\mathrm{i}y}-\mathrm{e}^{-\mathrm{i}y})\right)(1+\Delta_k^{-1})\nonumber\\
      &=\dfrac{1}{2\mathrm{i}}\left( \mathrm{e}^{\mathrm{i}y} \Delta_k^{-1} \mathrm{e}^{-\mathrm{i}y}-\mathrm{e}^{-\mathrm{i}y} \Delta_k^{-1} \mathrm{e}^{\mathrm{i}y}\right)(1+\Delta_k^{-1})=\dfrac{1}{2\mathrm{i}}( \Delta_{k,-1}^{-1}-\Delta_{k,1}^{-1})(1+\Delta_k^{-1})\nonumber\\
      &=\dfrac{1}{2\mathrm{i}}(\Delta_{k,1}-\Delta_{k,-1})\Delta_{k,1}^{-1}\Delta_{k,-1}^{-1}(1+\Delta_k^{-1}) =2\partial_y\Delta_{k,1}^{-1}\Delta_{k,-1}^{-1}(1+\Delta_k^{-1}),
   \end{align}
   which yields 
   \begin{align}
     &\Lambda_2=[\Lambda_1,B]=-2[A,B]\partial_y+2A^2=-4\partial_y^2\Delta_{k,1}^{-1}\Delta_{k,-1}^{-1} (1+\Delta_k^{-1})+2A^2.\label{def:Lam2} 
   \end{align}
   
   Now we calculate $\Lambda_2\omega$.
   Since $\Delta_{k,s}=\mathrm{e}^{-\mathrm{i}sy}\Delta_k\mathrm{e}^{\mathrm{i}sy}= (\partial_y+\mathrm{i}s)^2-k^2$, we have
   \begin{align}
      & B^2+A^2=(1+\cos y\Delta_k^{-1}\cos y+\sin y\Delta_k^{-1}\sin y)(1+\Delta_k^{-1}),\nonumber\\
      &4(\cos y\Delta_k^{-1}\cos y+\sin y\Delta_k^{-1}\sin y)\nonumber\\
      &=(\mathrm{e}^{\mathrm{i}y}+\mathrm{e}^{-\mathrm{i}y})\Delta_k^{-1}(\mathrm{e}^{\mathrm{i}y}+ \mathrm{e}^{-\mathrm{i}y})- (\mathrm{e}^{\mathrm{i}y}-\mathrm{e}^{-\mathrm{i}y})\Delta_k^{-1}(\mathrm{e}^{\mathrm{i}y} -\mathrm{e}^{-\mathrm{i}y})\nonumber\\
      &=2\mathrm{e}^{iy}\Delta_k^{-1}\mathrm{e}^{-\mathrm{i}y}+2\mathrm{e}^{-\mathrm{i}y}\Delta_k^{-1}\mathrm{e}^{\mathrm{i}y} 
   =2\Delta_{k,-1}^{-1}+2\Delta_{k,1}^{-1},\nonumber\\
      \Rightarrow \quad& B^2+A^2=(2+\Delta_{k,-1}^{-1}+\Delta_{k,1}^{-1})(1+\Delta_k^{-1})/2,\label{est:A2+B2}
   \end{align}
   which along with \eqref{def:Lam2} gives
   \begin{align*}
       \Lambda_2+2B^2=&-4\partial_y^2\Delta_{k,1}^{-1}\Delta_{k,-1}^{-1}(1+\Delta_k^{-1})+2(A^2+B^2)\\
      =&-4\partial_y^2\Delta_{k,1}^{-1}\Delta_{k,-1}^{-1}(1+\Delta_k^{-1}) +(2+\Delta_{k,-1}^{-1}+\Delta_{k,1}^{-1})(1+\Delta_k^{-1})\\
      =&(2+(\Delta_{k,-1}+\Delta_{k,1}-4\partial_y^2)\Delta_{k,-1}^{-1}\Delta_{k,1}^{-1})(1+\Delta_k^{-1})\\
      =&2(1-(\Delta_{k}+1+2k^2)\Delta_{k,-1}^{-1} \Delta_{k,1}^{-1})(1+\Delta_k^{-1})\\
      =&2\big(1+\Delta_k^{-1}-\Delta_k^{-1}(\Delta_k+1)(\Delta_{k}+1+2k^2)\Delta_{k,-1}^{-1} \Delta_{k,1}^{-1}\big)\\
      =&2\big(1-2k^2\Delta_k^{-1}(\Delta_k-1)\Delta_{k,-1}^{-1} \Delta_{k,1}^{-1}\big)=2-4k^2\Lambda_3,
   \end{align*}
   This shows $\Lambda_2=2(1-B^2)-4k^2\Lambda_3$. Here we have used the facts that 
  \begin{align}\label{L1-1}
      & \Delta_{k,-1}+\Delta_{k,1}-4\partial_y^2=2(\Delta_{k}-1)-4(\Delta_{k}+k^2)=-2(\Delta_{k}+1+2k^2),
   \end{align}
   and 
   \begin{align}
      \label{L2}\Delta_{k,1}\Delta_{k,-1}=&(\Delta_{k}-1+2\mathrm{i}\partial_y)(\Delta_{k}-1-2\mathrm{i}\partial_y)=(\Delta_{k}-1)^2+4\partial_y^2\nonumber\\
    \notag  =&(\Delta_{k}-1)^2+4(\Delta_{k}+k^2)=(\Delta_{k}+1)^2+4k^2\\=&(\Delta_k+1)(\Delta_{k}+1+2k^2)-2k^2(\Delta_{k}-1).
   \end{align}   
   
 As $\omega=\mathrm{e}^{-\mathrm{i}tB}f$, we have
   \begin{align}
      &\partial_t\omega=- \mathrm{i}B\omega,\quad\partial_t\Delta_k\omega=-\Delta_k\mathrm{i}B\omega=-\mathrm{i}B\Delta_k\omega -\mathrm{i}[\Delta_k,B]\omega,\nonumber\\
      &\partial_t\Lambda_1\omega=-\Lambda_1\mathrm{i}B\omega=-\mathrm{i}B\Lambda_1\omega-\mathrm{i}[\Lambda_1,B]\omega, \label{def:Lam1}
   \end{align} 
   This shows 
   \begin{align}
      &(\partial_t+\mathrm{i}B)\Delta_k\omega=-\mathrm{i}\Lambda_1\omega,\quad (\partial_t+\mathrm{i}B)\Lambda_1\omega=-\mathrm{i}\Lambda_2\omega,\nonumber
   \end{align}
   and then 
\begin{align}
&(\partial_t+\mathrm{i}B)(\Delta_k\omega+\mathrm{i}t\Lambda_1\omega)=t\Lambda_2\omega =2t(1-B^2)\omega-4tk^2\Lambda_3\omega,\nonumber\\
   & (\partial_t+\mathrm{i}B)(\Delta_k\omega+\mathrm{i}t\Lambda_1\omega -t^2(1-B^2)\omega )=-4tk^2\Lambda_3\omega, \nonumber
\end{align}
which gives $(\partial_t+\mathrm{i}B)\omega_1=-4tk^2\Lambda_3\omega$.
\end{proof}

The following striking coercive estimate constitutes a crucial ingredient in our proof. Given that the proof is rather technical, it will be presented in a separate subsection.

\begin{Proposition}\label{prop:coer}
   For $s\in[0,0.4]$, we have
   \begin{align*}
      &\Lambda_1(-\Delta_k)^{s}\Lambda_1+2(1-B^2)(-\Delta_k)^{1+s}+2(-\Delta_k)^{1+s}(1-B^2)\geq 4k^2A(-\Delta_k)^{s}A+\delta(s)(-\Delta_k)^{s}.
   \end{align*}
   Here $\delta(s):=4-\max((8s^2+7s+2)s/2,8s^2+8s-1)\in[0.52,4]$.
 \end{Proposition}
 
 Based on Proposition \ref{prop:coer},  we can prove the following key lemma.

 \begin{lemma}\label{lem:t2Lamom-om1}
   Let $\psi=\Delta_k^{-1}\omega, \widetilde{\psi}=(-\Delta_k)^{-s}\psi$, and $\omega_1=\Delta_{(t)}\omega$. Then it holds that  for $s\in[0,0.4]$,
   \begin{align*}
     |\langle\omega_1,\widetilde{\psi}\rangle_*|\geq& \|(-\Delta_k)^{s/2}(\Delta_k\widetilde{\psi}+\mathrm{i}t\Lambda_1\widetilde{\psi}/2)\|_{*}^2 \\& +t^2k^2\|(-\Delta_k)^{s/2}A\widetilde{\psi}\|_{*}^2+\delta(s)t^2\|(-\Delta_k)^{s/2}\widetilde{\psi}\|_{*}^2/4,
   \end{align*}
   and
      \begin{align*}
    &t^2\|\psi\|_{H_k^{-s}}\leq C\|\omega_1\|_{H_k^{-s}}\,\, \forall\ t\geq 0,\quad t\|(\partial_y-\mathrm{i}tA)\psi\|_{H_k^{1-s}}\leq C\|\omega_1\|_{H_k^{-s}}\,\, \forall\ t\geq k^2.
   \end{align*}   \end{lemma}

 \begin{proof}
 As $\psi=\Delta_k^{-1}\omega$ and $\widetilde{\psi}=(-\Delta_k)^{-s}\psi$, we have $\omega=-(-\Delta_k)^{1+s}\widetilde{\psi}$ and 
  \begin{align*}
     & \mathbf{Re}\langle \omega_1,\tilde{\psi}\rangle_*=-\mathbf{Re}\big\langle \big(\Delta_k+\mathrm{i}t\Lambda_1-t^2(1-B^2)\big)(-\Delta_k)^{1+s}\widetilde{\psi},\tilde{\psi}\big\rangle_*\\
     &=\|(-\Delta_k)^{1+s/2}\widetilde{\psi}\|_{*}^2-\mathbf{Re}\big\langle (-\Delta_k)^{1+s}\widetilde{\psi},\mathrm{i}t\Lambda_1\widetilde{\psi}\big\rangle_{*} +t^2\mathbf{Re}\langle(1-B^2)(-\Delta_k)^{1+s}\widetilde{\psi},\widetilde{\psi}\rangle_*\\
     &=\|(-\Delta_k)^{s/2}(\Delta_k\widetilde{\psi}+\mathrm{i}t\Lambda_1\widetilde{\psi}/2)\|_{*}^2-
     t^2\|(-\Delta_k)^{s/2}\mathrm{i}\Lambda_1\widetilde{\psi}/2\|_{*}^2 +t^2\mathbf{Re}\big\langle(1-B^2)(-\Delta_k)^{1+s}\widetilde{\psi},\widetilde{\psi}\big\rangle_*\\
     &=\|(-\Delta_k)^{s/2}(\Delta_k\widetilde{\psi}+\mathrm{i}t\Lambda_1\widetilde{\psi}/2)\|_{*}^2 +t^2\mathbf{Re}\big\langle \Lambda_1(-\Delta_k)^{s}\Lambda_1\tilde{\psi}/4+(1-B^2)(-\Delta_k)^{1+s}\widetilde{\psi},\widetilde{\psi}\big\rangle_*.
  \end{align*}
Here we used that $\Lambda_1$ is antisymmetric. We get by Proposition \ref{prop:coer} that 
\begin{align*}
     & \mathbf{Re}\big\langle \Lambda_1(-\Delta_k)^{s}\Lambda_1\widetilde{\psi}/4+(1-B^2)(-\Delta_k)^{1+s}\widetilde{\psi},\widetilde{\psi}\big\rangle_*\\
     &=\mathbf{Re}\big\langle \Lambda_1(-\Delta_k)^{s}\Lambda_1\widetilde{\psi}+2(1-B^2)(-\Delta_k)^{1+s}\widetilde{\psi}+2(-\Delta_k)^{1+s}(1-B^2)\widetilde{\psi},\widetilde{\psi}\big\rangle_*/4
     \\ &\geq \mathbf{Re}\big\langle 4k^2A(-\Delta_k)^{s}A\widetilde{\psi}+\delta(s)(-\Delta_k)^{s}\widetilde{\psi},\widetilde{\psi}\big\rangle_*/4 
     \\ &= k^2\|(-\Delta_k)^{s/2}A\widetilde{\psi}\|_{*}^2+\delta(s)\|(-\Delta_k)^{s/2}\widetilde{\psi}\|_{*}^2/4.
  \end{align*}
  This shows the first inequality of the lemma, which further gives
  \begin{align*}
&\delta(s)t^2\|(-\Delta_k)^{s/2}\widetilde{\psi}\|_{*}^2/4\leq |\langle\omega_1,\widetilde{\psi}\rangle_*|\leq \|(-\Delta_k)^{-s/2}\omega_1\|_{*}\|(-\Delta_k)^{s/2}\widetilde{\psi}\|_{*},\\
&\delta(s)t^2\|(-\Delta_k)^{s/2}\widetilde{\psi}\|_{*}/4\leq \|(-\Delta_k)^{-s/2}\omega_1\|_{*}.
   \end{align*}
As $\delta(s)\in[0.52,4]$, $\|\cdot\|_{*}\sim\|\cdot\|_{L^2}$ and $\widetilde{\psi}=(-\Delta_k)^{-s}\psi$, we have
\begin{align*}
&t^2\|\psi\|_{H_k^{-s}}=t^2\|\widetilde{\psi}\|_{H_k^{s}}=t^2\|(-\Delta_k)^{s/2}\widetilde{\psi}\|_{L^2}\leq C\|(-\Delta_k)^{-s/2}\omega_1\|_{L^2}= C\|\omega_1\|_{H_k^{-s}}.
   \end{align*}
Using the first inequality of this lemma,  we also have
   \begin{align*}
      &\|(\Delta_k\widetilde{\psi}+\mathrm{i}t\Lambda_1\widetilde{\psi}/2)\|_{H_k^{s}}^2 +t^2k^2\|A\widetilde{\psi}\|_{H_k^{s}}^2\\ &\leq C\big(\|(-\Delta_k)^{s/2}(\Delta_k\tilde{\psi}+\mathrm{i}t\Lambda_1\widetilde{\psi}/2)\|_{*}^2 +t^2k^2\|(-\Delta_k)^{s/2}A\widetilde{\psi}\|_{*}^2\big)
     \leq C|\langle\omega_1,\widetilde{\psi}\rangle_*|\\ &\leq \|(-\Delta_k)^{-s/2}\omega_1\|_{*}\|(-\Delta_k)^{s/2}\widetilde{\psi}\|_{*}\leq Ct^{-2}\|(-\Delta_k)^{-s/2}\omega_1\|_{*}^2\leq Ct^{-2}\|\omega_1\|_{H_k^{-s}}^2.
   \end{align*}
Now we assume $t\geq k^2$. By \eqref{lam1}, we get
\begin{align*}
\Delta_k\widetilde{\psi}+\mathrm{i}t\Lambda_1\widetilde{\psi}/2=(\partial_y^2-k^2)\widetilde{\psi}-\mathrm{i}t\partial_yA\widetilde{\psi}+\mathrm{i}tB\widetilde{\psi}/2
=\partial_y(\partial_y-\mathrm{i}tA)\widetilde{\psi}-k^2\widetilde{\psi}+\mathrm{i}tB\widetilde{\psi}/2.
\end{align*}
Then we have (as $t\geq k^2$, $t^2\|\widetilde{\psi}\|_{H_k^{s}}\le {C}\|\omega_1\|_{H_k^{-s}}$)
\begin{align*}
\|\partial_y(\partial_y-\mathrm{i}tA)\widetilde{\psi}\|_{H_k^{s}}\leq&\|\Delta_k\tilde{\psi}+\mathrm{i}t\Lambda_1\widetilde{\psi}/2\|_{L^2}+
k^2\|\widetilde{\psi}\|_{H_k^{s}}+t\|B\widetilde{\psi}\|_{H_k^{s}}\\
\leq&Ct^{-1}\|\omega_1\|_{H_k^{-s}}+k^2\|{\widetilde{\psi}}\|_{H_k^{s}}+Ct\|{\widetilde{\psi}}\|_{H_k^{s}}\leq Ct^{-1}\|\omega_1\|_{H_k^{-s}},
\end{align*}
and
\begin{align*}
\|(\partial_y-\mathrm{i}tA)\widetilde{\psi}\|_{H_k^{s+1}}=&
\|(\partial_y^2-k^2)(\partial_y-\mathrm{i}tA)\widetilde{\psi}\|_{H_k^{s-1}}\\ \leq&
\|\partial_y^2(\partial_y-\mathrm{i}tA)\widetilde{\psi}\|_{H_k^{s-1}}+k^2\|\partial_y\widetilde{\psi}\|_{H_k^{s-1}}+
k^2t\|A\widetilde{\psi}\|_{H_k^{s-1}}\\ \leq&
\|\partial_y(\partial_y-\mathrm{i}tA)\widetilde{\psi}\|_{H_k^{s}}+t\|\widetilde{\psi}\|_{H_k^{s}}+
|k|t\|A\widetilde{\psi}\|_{H_k^{s}}\leq Ct^{-1}\|\omega_1\|_{H_k^{-s}}.
\end{align*}
Thanks to $ A=\sin y(1+\Delta_k^{-1})=(\mathrm{e}^{\mathrm{i}y}-\mathrm{e}^{-\mathrm{i}y})(1+\Delta_k^{-1})/(2\mathrm{i})$, we have
\begin{align*}
[A,(-\Delta_k)^s]=&\big(\mathrm{e}^{\mathrm{i}y}(-\Delta_k)^s-(-\Delta_k)^s\mathrm{e}^{\mathrm{i}y}
-\mathrm{e}^{-\mathrm{i}y}(-\Delta_k)^s+(-\Delta_k)^s\mathrm{e}^{-\mathrm{i}y}\big)(1+\Delta_k^{-1})/(2\mathrm{i})\\
=&\big(\mathrm{e}^{\mathrm{i}y}((-\Delta_k)^s-(-\Delta_{k,1})^s)+
((-\Delta_k)^s-(-\Delta_{k,1})^s)\mathrm{e}^{-\mathrm{i}y}\big)(1+\Delta_k^{-1})/(2\mathrm{i}).\end{align*}
Then we have 
\beno
\|[A,(-\Delta_k)^s]\widetilde{\psi}\|_{H_k^{1-s}}\leq C\|\widetilde{\psi}\|_{H_k^{s}},
\eeno
which follows from the facts that  $\Delta_{k,1}=(\partial_y+\mathrm{i})^2-k^2 $ and $|(n^2+k^2)^s-((n+1)^2+k^2)^s|\leq C(n^2+k^2)^{s-1/2}$ (by taking Fourier transform).
Thus,
\begin{align*}
\|(\partial_y-\mathrm{i}tA){\psi}\|_{H_k^{1-s}}=&\|(\partial_y-\mathrm{i}tA)(-\Delta_k)^s\widetilde{\psi}\|_{H_k^{1-s}}\\ \leq &
\|(-\Delta_k)^s(\partial_y-\mathrm{i}tA)\widetilde{\psi}\|_{H_k^{1+s}}+t\|[A,(-\Delta_k)^s]\widetilde{\psi}\|_{H_k^{1-s}}
\\ \leq &\|(\partial_y-\mathrm{i}tA)\widetilde{\psi}\|_{H_k^{1-s}}+Ct\|\widetilde{\psi}\|_{H_k^{s}}\leq Ct^{-1}\|\omega_1\|_{H_k^{-s}}.
\end{align*}

This completes the proof of the lemma.
\end{proof}

\subsection{Inviscid damping estimates}

In this subsection, we prove Proposition \ref{lem:eibtf}.

\subsubsection{Bounds on  the semigroup $\mathrm{e}^{-\mathrm{i}tB}$}

\begin{lemma}\label{prop:timespace-B-0}
  If $\omega$ satisfies $\partial_t\omega+\mathrm{i}B\omega=F$, $\omega|_{t=0}=\omega_0$, then for $t\ge 0$, we have 
  \begin{align*}
     &\|\omega(t)\|_{L^2}^2\leq C\Big(\|\omega_0\|_{L^2}^2+\int_{0}^{t}\|F(s)\|_{H_k^1}^2\mathrm{d}s\Big).
  \end{align*}
\end{lemma}
\begin{proof}
  We decompose $\omega(t_1,y)=\omega_I(t_1,y)+\mathrm{e}^{-\mathrm{i}Bt_1}\omega_0(y)$, where $\omega_I$ satisfies(for fixed $t_1\geq0$)
  \begin{align*}
     & \partial_t\omega_I+\mathrm{i}B\omega_I=F\mathbf{1}_{t<t_1},\quad\omega_I|_{t=0}=0.
  \end{align*}
  Taking the inner product of the above equation with $\omega_I$, we obtain(for $t\in[0,t_1]$) 
  \begin{align*}
     &\dfrac{1}{2}\dfrac{\mathrm{d}}{\mathrm{d}t}\|\omega_I(t)\|_{*}^2\leq \left|\langle F,\omega_I\rangle_*\right|=
     \left|\langle F,(1+\Delta_k^{-1})\omega_I\rangle\right|\leq \|F\|_{H^1_{{k}}}\|(1+\Delta_k^{-1})\omega_I\|_{H^{-1}_{{k}}}, 
      \end{align*}
  which gives
   \begin{align*}
     &\|\omega_I(t)\|_{*}^2\leq \int_{0}^{t}\|F(s)\|_{H_k^1}^2\mathrm{d}s+\int_{0}^{t}
     \|(1+\Delta_k^{-1})\omega_I(s)\|_{H_k^{-1}}^2\mathrm{d}s,\\
     &\|\omega_I(t)\|_{L^2}^2\leq C\left(\int_{0}^{t}\|F(s)\|_{H_k^1}^2\mathrm{d}s+\int_{0}^{t}
     \|\omega_I(s)\|_{H_k^{-1}}^2\mathrm{d}s\right).
  \end{align*}  
  
  We denote $\hat{\omega}(\lambda,y)=\int_{0}^{+\infty}\mathrm{e}^{-\mathrm{i}\lambda t}\omega_I(t,y)\mathrm{d}t$ and $\hat{F}(\lambda,y)=\int_{0}^{+\infty}\mathrm{e}^{-\mathrm{i}\lambda t}F(t,y)\mathrm{d}t$. Then we have (for $\mathrm{Im}\,\lambda<0$)
  \begin{align*}
     &\mathrm{i}\lambda\hat{\omega}+\mathrm{i}B\hat{\omega}=\hat{F}.
  \end{align*} 
  It follows from Proposition \ref{lem:rayleigh} and $\theta(k,\lambda)\geq1$ that $\|\hat{\omega}(\lambda)\|_{H^{-1}_k}\leq C\|\hat{F}(\lambda)\|_{H^1_k}$, which gives
   \begin{align*}
   &\int_{\mathbb{R}}\|\hat{\omega}(\lambda-\mathrm{i}\delta,y)\|_{H_k^{-1}}^2\mathrm{d}\lambda\leq C\int_{\mathbb{R}}\|\hat{F}(\lambda-\mathrm{i}\delta,y)\|_{H_k^1}^2\mathrm{d}\lambda,\quad\ \forall\ \delta>0.
  \end{align*}
  This along with Plancherel's formula shows 
  \begin{align*}
     \int_{0}^{+\infty}\mathrm{e}^{-2s\delta}\|\omega_I(s)\|_{H_k^{-1}}^2\mathrm{d}s=&C \int_{\mathbb{R}}\|\hat{\omega}(\lambda-\mathrm{i}\delta,y)\|_{H_k^{-1}}^2\mathrm{d}\lambda\\
     \leq& C\int_{\mathbb{R}}\|\hat{F}(\lambda-\mathrm{i}\delta,y)\|_{H_k^1}^2\mathrm{d}\lambda 
    = C\int_{0}^{t_1}\mathrm{e}^{-2s\delta}\|F(s)\|_{H_k^1}^2\mathrm{d}s.
  \end{align*}
  Letting $\delta\to0+ $, we get
  \begin{align*}
     & \int_{0}^{+\infty}\|\omega_I(s)\|_{H_k^{-1}}^2\mathrm{d}s\leq C\int_{0}^{t_1}\|F(s)\|_{H_k^1}^2\mathrm{d}s,
  \end{align*}
  Hence, $\|\omega_I(t_1)\|_{L^2}^2\leq C\int_{0}^{t_1}\|F(s)\|_{H_k^1}^2\mathrm{d}s$, and then using $\|\mathrm{e}^{-\mathrm{i}Bt_1}\omega_0\|_{*}= \|\omega_0\|_{*}$, $\|f\|_{*}\sim\|f\|_{L^2}$,
  we arrive at 
  \begin{align*}
     & \|\omega(t_1)\|_{L^2}^2\leq C\|\mathrm{e}^{-\mathrm{i}Bt_1}\omega_0\|_{L^2}^2+C\|\omega_I(t_1)\|_{L^2}^2\leq C\Big(\|\omega_0\|_{L^2}^2+\int_{0}^{t_1}\|F(s)\|_{H_k^1}^2\mathrm{d}s\Big).
  \end{align*}
 {Replacing $t_1$ by $t$ gives the result.}
  \end{proof}

The following lemma provides bounds for $\mathrm{e}^{-\mathrm{i}tB}f$ with $f\in H^s_k$ via the vector field method.

 \begin{lemma}\label{lem2b}
 For $s\in[0,2]$, $t\in\R$, it holds that
   \begin{align}
   &(k^2+|t|)^{s}\|\mathrm{e}^{-\mathrm{i}tB}f\|_{H_k^{-s}}\leq C\|f\|_{H_k^{s}},\label{est:eibtf-1}\\
      & \|\mathrm{e}^{\mathrm{i}tb}\mathrm{e}^{-\mathrm{i}tB}f\|_{H_k^{s}}\leq C\|f\|_{H_k^{s}}, \label{est:eibtf-2}\\
   \label{est:eibtf-3}   &{|\mathrm{e}^{-\mathrm{i}tB}f(t,y)|\leq  C((k^2+|t|)^{-1/4}+|b'|^{1/2})\|f\|_{H_k^{1}}}.
   \end{align}
   \end{lemma} \begin{proof}
 Without lose of generality, we assume $t\geq 0$.  The case $t\leq 0$ can be proved by taking the conjugation. Let $\omega=\mathrm{e}^{-\mathrm{i}tB}f$, $\omega_1:=\Delta_{(t)}\omega$, $\psi=\Delta_k^{-1}\omega$. Then by Lemma  \ref{lem1-dec}, we have $(\partial_t+\mathrm{i}B)\omega_1=-4tk^2\Lambda_3\omega$. Thus, by {Lemma} \ref{prop:timespace-B-0} and  Lemma \ref{lem:t2Lamom-om1} (for $s=0$), we obtain
\begin{align}\label{om8-dec}
C^{-1}\|t^2\Delta_k^{-1}&\omega(t)\|_{L^2}^2\leq \|\omega_1(t)\|_{L^2}^2\leq C\Big(\|\omega_1(0)\|_{L^2}^2+\int_{0}^{t}t'^2k^4\|\Lambda_3\omega(t')\|_{H_k^1}^2\mathrm{d}t'\Big)\nonumber\\
   \leq& C\Big(\|\Delta_kf\|_{L^2}^2+\int_{0}^{t}t'^2k^4\|\omega(t')\|_{H_k^{-3}}^2\mathrm{d}t'\Big)
    \leq C\Big(\|\Delta_kf\|_{L^2}^2+\int_{0}^{t}t'^2k^2\|\Delta_k^{-1}\omega(t')\|_{L^2}^2\mathrm{d}t'\Big)\nonumber\\
    \leq& C\Big(\|\Delta_kf\|_{L^2}^2+
    \int_{0}^{t}k^{2}(t'^2+k^4)^{-1}\|(t'^2+k^4)\Delta_k^{-1}\omega(t')\|_{L^2}^2\mathrm{d}t'\Big).
\end{align}
Here we used $\omega_1(0)=\Delta_k\omega(0)=\Delta_kf$ and
\begin{align}\label{om8}
   \|\Lambda_3\omega\|_{H_k^1}\leq C\|\omega\|_{H_k^{-3}}.
\end{align}
In fact, as $\Lambda_3=(1-\Delta_k^{-1})\Delta_{k,-1}^{-1}\Delta_{k,1}^{-1}$, \eqref{om8} follows from (taking Fourier transform)
\begin{align*}
   (n^2+k^2)^{1/2}(1+(n^2+k^2)^{-1})((n-1)^2+k^2)^{-1}((n+1)^2+k^2)^{-1}\leq C(n^2+k^2)^{-3/2}.
\end{align*}
As $\omega=\mathrm{e}^{-\mathrm{i}tB}f$ and $B$ is antisymmetric under the inner product $\langle,\rangle_*$, we have $\|\omega(t)\|_{*}=\|f\|_{*}$. Thus, $\|\omega(t)\|_{L^2}\le C\|f\|_{L^2}$ and
   \begin{align*}
      &|k|^4\|\Delta_k^{-1}\omega\|_{L^2}\leq k^2\|\omega\|_{L^{2}}\lesssim k^2\|f\|_{L^2}\leq \|\Delta_kf\|_{L^2},
   \end{align*}
   which along with \eqref{om8-dec} shows 
   \begin{align*}
\|(t^2+k^4)\Delta_k^{-1}\omega(t)\|_{L^2}^2\leq C\Big(\|\Delta_kf\|_{L^2}^2+\int_{0}^{t}k^{2}(t'^2+k^4)^{-1}\|(t'^2+k^4)\Delta_k^{-1}\omega(t')\|_{L^2}^2\mathrm{d}t'\Big).
\end{align*}
By Gr\"{o}nwall's inequality, we have
  \begin{align*}
     & \|(t^2+k^4)\Delta_k^{-1}\omega(t)\|_{L^2}\leq C\|\Delta_kf\|_{L^2}^2\mathrm{e}^{C\int_{0}^{t}k^{2}(t'^2+k^4)^{-1}\mathrm{d}t'}\leq C\|\Delta_kf\|_{L^2}^2= C\|f\|_{H_k^2}^2.
  \end{align*}
  This proves \eqref{est:eibtf-1} for $s=2$; moreover $\|\omega(t)\|_{L^2}\le C\|f\|_{L^2}$, which gives \eqref{est:eibtf-1} for $s=0$.  Thus,  \eqref{est:eibtf-1} holds for $s\in[0,2]$ by the interpolation.

  By \eqref{om8-dec} again, we have
  \begin{align}\label{est:om1-Deltaf}
 \|\omega_1(t)\|_{L^2}^2 \leq& C\Big(\|\Delta_kf\|_{L^2}^2+
   k^{2} \int_{0}^{t}(t'^2+k^4)^{-1}\|(t'^2+k^4)\Delta_k^{-1}\omega(t')\|_{L^2}^2\mathrm{d}t'\Big)\nonumber\\
    \leq& C\Big(\|\Delta_kf\|_{L^2}^2+
   k^{2}\|f\|_{H^2_k}^2 \int_{0}^{t}(t'^2+k^4)^{-1}\mathrm{d}t'\Big)\leq  C\|f\|_{H^2_k}^2 .
\end{align}
  
   Recall that (using \eqref{lam1}), 
   \begin{align}\notag
      \omega_1=&\Delta_{(t)}\omega=(\Delta_k+\mathrm{i}t\Lambda_1 -t^2(1-B^2))\omega\\
       =&(\partial_y^2-k^2)\omega-\mathrm{i}t(A\partial_y+\partial_yA)\omega-t^2 (1-B^2)\omega\nonumber\\
    \label{om11}  =&(\partial_y-\mathrm{i}tA)^2\omega-k^2\omega-t^2(1-B^2-A^2)\omega.
   \end{align}
 By \eqref{est:A2+B2} and Lemma \ref{lem:t2Lamom-om1} for $s=0$, we obtain 
   \begin{align*}
   &1-B^2-A^2=-\Delta_k^{-1}-(\Delta_{k,-1}^{-1}+\Delta_{k,1}^{-1})(1+\Delta_k^{-1})/2,\\
      &\|t^2(1-B^2-A^2)\omega\|_{L^2}\leq Ct^2\|\Delta_k^{-1}\omega\|_{L^2}\leq C\|\omega_1\|_{L^2}.
   \end{align*}Then we conclude
   \begin{align*}
      &\|(\partial_y-\mathrm{i}tA)^2\omega-k^2\omega\|_{L^2}\leq C\|\omega_1\|_{L^2}\leq C\|\Delta_kf\|_{L^2}.
   \end{align*}
   Recall that $A=\sin y(1+\Delta_k^{-1})=-b'(1+\Delta_k^{-1})$, then 
   \begin{align*}
      (\partial_y-\mathrm{i}tA)^2-(\partial_y+\mathrm{i}tb')^2 &=(\partial_y-\mathrm{i}tA)^2-(\partial_y+\mathrm{i}tb')(\partial_y-\mathrm{i}tA)\\
      &\quad+ (\partial_y+\mathrm{i}tb')(\partial_y-\mathrm{i}tA) -(\partial_y+\mathrm{i}tb')^2\\
      &= \mathrm{i}tb'\Delta_k^{-1}(\partial_y-\mathrm{i}tA)+  (\partial_y+\mathrm{i}tb')\mathrm{i}tb'\Delta_k^{-1},
   \end{align*}
  from which and the fact $A$ is bounded on $H_k^{-2}$, we infer that 
   \begin{align*}
      & \|(\partial_y-\mathrm{i}tA)^2\omega-(\partial_y+\mathrm{i}tb')^2\omega\|_{L^2}\leq t\|b'\Delta_k^{-1}(\partial_y-\mathrm{i}tA)\omega\|_{{L^2}}+ t\|(\partial_y+\mathrm{i}tb')(b'\Delta_k^{-1}\omega)\|_{{L^2}}\\
      &\leq Ct\|(-\Delta_k)^{-1/2}\omega\|_{{L^2}}+C(t+t^2)\|\Delta_k^{-1}\omega\|_{{L^2}}\leq C\|\Delta_kf\|_{{L^2}}.
   \end{align*}
   Here we used \eqref{est:eibtf-1} for $s=1$ and $s=2$. Thus,
  \begin{align*}
     & \|(\partial_y+\mathrm{i}tb')^2\omega-k^2\omega\|_{{L^2}}\leq C\|\Delta_kf\|_{{L^2}}.
  \end{align*}
  Notice that
  \begin{align*}
     & \mathrm{e}^{\mathrm{i}tb}[(\partial_y+\mathrm{i}tb')^2\omega-k^2\omega] =(\partial_y^2 -k^2)(\mathrm{e}^{\mathrm{i}tb}\omega)=\Delta_k(\mathrm{e}^{\mathrm{i}tb}\omega),
  \end{align*}
  we have
  \begin{align*}
     & \|\mathrm{e}^{\mathrm{i}tb}\omega\|_{H_k^2}=\|\Delta_k(\mathrm{e}^{\mathrm{i}tb}\omega)\|_{L^2}= \|(\partial_y+\mathrm{i}tb')^2\omega-k^2\omega\|_{L^2}
     \leq C\|f\|_{H_k^2}.
  \end{align*}
  This gives \eqref{est:eibtf-2} for $s=2$; moreover, $\|\mathrm{e}^{\mathrm{i}tb}\omega\|_{L^2}=\|\omega(t)\|_{L^2}\le C\|f\|_{L^2}$, which gives  \eqref{est:eibtf-2} for $s=0$.  Then \eqref{est:eibtf-2} holds for $s\in[0,2]$ by the interpolation. 
  
Let $\widetilde{\omega}=\mathrm{e}^{\mathrm{i}tb}\omega=\mathrm{e}^{\mathrm{i}tb}\mathrm{e}^{-\mathrm{i}tB}f$, then $|\widetilde{\omega}|=|{\omega}|$. By \eqref{est:eibtf-2} for $s=1$, we have 
$\|\widetilde{\omega}\|_{H_k^1}\leq C\|f\|_{H_k^1}$ and 
  \begin{align*}
     &|\widetilde{\omega}(t,y)-\widetilde{\omega}(t,0)|\le |y|^{1/2}\|\widetilde{\omega}\|_{H_k^1}\leq C|y|^{1/2}\|f\|_{H_k^1},
  \end{align*}
  
  Let $\Psi$ be a nonnegative function, which is supported in $[-1,1]$, and equals to 1 on $[-1/2,1/2]$. 
  Let $\Psi_t(y)=\Psi((k^2+|t|)^{1/2}y)$ and $\widetilde{\Psi}_t=\mathrm{e}^{-\mathrm{i}tb}{\Psi}_t$ for $|y|\leq \pi$. Then we get by \eqref{est:eibtf-1} for $s=1$ that 
  \begin{align*}
     &|\langle\widetilde{\omega},{\Psi}_t\rangle|=|\langle{\omega},\widetilde{\Psi}_t\rangle|\leq \|{\omega}\|_{H_k^{-1}}\|\widetilde{\Psi}_t\|_{H_k^{1}}\leq C(k^2+|t|)^{-1}\|f\|_{H_k^1}\|\widetilde{\Psi}_t\|_{H_k^{1}},
  \end{align*}
  where
  \begin{align*}
    \|\widetilde{\Psi}_t\|_{H_k^{1}}\leq& |k|\|\widetilde{\Psi}_t\|_{L^2}+\|\partial_y\widetilde{\Psi}_t\|_{L^2}=
     |k|\|{\Psi}_t\|_{L^2}+\|\partial_y{\Psi}_t-\mathrm{i}tb'{\Psi}_t\|_{L^2}\\
     \leq&|k|\|{\Psi}_t\|_{L^2}+\|\partial_y{\Psi}_t\|_{L^2}+|t|\|b'{\Psi}_t\|_{L^2}\leq
     |k|\|{\Psi}_t\|_{L^2}+\|\partial_y{\Psi}_t\|_{L^2}+|t|\|y{\Psi}_t\|_{L^2}\\=&
     |k|(k^2+|t|)^{-1/4}\|{\Psi}\|_{L^2(\R)}+(k^2+|t|)^{1/4}\|{\Psi}'\|_{L^2(\R)}+|t|(k^2+|t|)^{-3/4}\|y{\Psi}\|_{L^2(\R)}\\
     \leq&C(|k|(k^2+|t|)^{-1/4}+(k^2+|t|)^{1/4}+|t|(k^2+|t|)^{-3/4})\leq C(k^2+|t|)^{1/4}.
  \end{align*}  
  This shows 
  \beno
  |\langle\widetilde{\omega},{\Psi}_t\rangle|\leq C(k^2+|t|)^{-3/4}\|f\|_{H_k^1}.
  \eeno
  
   On the other hand, we {have}
  \begin{align*}
     |\langle\widetilde{\omega},{\Psi}_t\rangle-\widetilde{\omega}(t,0)\langle 1,{\Psi}_t\rangle|=&
     |\langle\widetilde{\omega}(t,y)-\widetilde{\omega}(t,0),{\Psi}_t\rangle|\\
     \leq& \langle C|y|^{1/2}\|f\|_{H_k^1},{\Psi}_t\rangle=C\|f\|_{H_k^1}\||y|^{1/2}{\Psi}_t\|_{L^1}\\
     =&C\|f\|_{H_k^1}(k^2+|t|)^{-3/4}\||y|^{1/2}{\Psi}\|_{L^1(\R)}\leq C(k^2+|t|)^{-3/4}\|f\|_{H_k^1}.
  \end{align*}
  Thus, we obtain 
  \beno
   |\widetilde{\omega}(t,0)\langle 1,{\Psi}_t\rangle|\leq C(k^2+|t|)^{-3/4}\|f\|_{H_k^1}.
   \eeno
  Notice that $\|{\Psi}\|_{L^1(\R)}\geq1$ and
  \begin{align*}
     &|\langle 1,{\Psi}_t\rangle|=(k^2+|t|)^{-1/2}\|{\Psi}\|_{L^1(\R)}\geq (k^2+|t|)^{-1/2}.
  \end{align*}
  Thus, $ |\widetilde{\omega}(t,0)|\leq C(k^2+|t|)^{-1/4}\|f\|_{H_k^1}$ and 
  \begin{align*}
    |{\omega}(t,y)|=|\widetilde{\omega}(t,y)|\leq&|\widetilde{\omega}(t,0)|+|\widetilde{\omega}(t,y)-\widetilde{\omega}(t,0)|\\
    \le& C\big((k^2+|t|)^{-1/4}+|y|^{1/2}\big)\|f\|_{H_k^1},
  \end{align*}
  which implies \eqref{est:eibtf-3} for $|y|\leq \pi/2$ due to $|b'(y)|=|\sin y|\sim|y|$. Similarly, we have $ |\widetilde{\omega}(t,\pi)|\leq C(k^2+|t|)^{-1/4}\|f\|_{H_k^1}$, and if $|y-\pi|\leq \pi/2$, then 
  \begin{align*}
     |{\omega}(t,y)|=|\widetilde{\omega}(t,y)|\leq& C((k^2+|t|)^{-1/4}+|y-\pi|^{1/2})\|f\|_{H_k^1}\\
     \leq& C((k^2+|t|)^{-1/4}+|b'|^{1/2})\|f\|_{H_k^1}.
  \end{align*}
  
  This completes the proof of \eqref{est:eibtf-3} by noticing $\omega(t,y)=\mathrm{e}^{-\mathrm{i}tB}f(y)$.
 \end{proof}

   Next, we establish additional bounds on $\mathrm{e}^{-itB}f$ by iterating the results of Lemma \ref{lem2b}.

   \begin{lemma}\label{lem3-dec}
Let $\omega=\mathrm{e}^{-\mathrm{i}tB}f$, $\omega_1=\Delta_{(t)}\omega$ and $\psi=\Delta_k^{-1}\omega$. Then we have
   \begin{align}\label{om9}
   &\|\omega_1(t)\|_{L^2}\leq C\|f\|_{H_k^{2}},\quad \|\omega_1(t)\|_{H_k^{-1}}\leq C(k^2+|t|)^{-1}\ln(|t|/k^2+2)\|f\|_{H_k^{3}},\\
   &\|(\partial_y,k)(\partial_y+\mathrm{i}tb')\psi\|_{L^2}+\|(\partial_y+\mathrm{i}tb')\partial_y\psi\|_{L^2}+|k|\|\partial_y\psi\|_{L^2}\leq 
      C(k^2+|t|)^{-1}\|f\|_{H_k^2},\label{psi1}\\
      & \|(\partial_y+\mathrm{i}tb')\psi\|_{L^{\infty}}+|t|\|b'\psi\|_{L^{\infty}}+\|\partial_y\psi\|_{L^{\infty}}\leq 
      C(k^2+|t|)^{-1}|k|^{-1/2}\|f\|_{H_k^2}, \label{psi2}\\
    \label{omH-3} & \|\omega(t)\|_{H_k^{-3}}=\|\psi(t)\|_{H_k^{-1}}\leq C(k^2+|t|)^{-2.3}|k|^{-1.4}\|f\|_{H_k^{3}},\\
   \label{psi3}  &\|(\partial_y+\mathrm{i}tb')\psi\|_{L^{\infty}}\leq C(k^2+|t|)^{-1.2}|k|^{-1}\|f\|_{H_k^{3}},\\
   \label{om1t}  &|\omega_1(t,y)|\leq C((k^2+|t|)^{-1/4}+|b'|^{1/2})\|f\|_{H_k^{3}}.
   \end{align}
   \end{lemma} 
   
   \begin{proof}
 Without lose of generality, we assume $t\geq 0$.  The case $t\leq 0$ can be proved by taking the conjugation. Thanks to $(\partial_t+\mathrm{i}B)\omega_1=-4tk^2\Lambda_3\omega$, we get by Duhamel's formula  that 
 \begin{align}\label{Duhamel}
      &\omega_1(t)=\mathrm{e}^{-\mathrm{i}tB}\Delta_kf-\int_0^t4t'k^2\mathrm{e}^{-\mathrm{i}(t-t')B}\Lambda_3\omega(t')\mathrm{d}t'.\end{align}
  By \eqref{om8} and \eqref{est:eibtf-1} (for $s=2$), we have 
   \begin{align*}
      &\|\Lambda_3\omega(t)\|_{H_k^{1}}\leq C\|\omega(t)\|_{H_k^{-3}},  
       \end{align*}
   and
    \begin{align*}
      \|\omega(t)\|_{H_k^{-3}}&\leq |k|^{-1}\|\omega(t)\|_{H_k^{-2}}=|k|^{-1}\|\mathrm{e}^{-\mathrm{i}tB}f\|_{H_k^{-2}}\\
      &\leq C|k|^{-1}(k^2+|t|)^{-2}\|f\|_{H_k^{2}}\leq C|k|^{-2}(k^2+|t|)^{-2}\|f\|_{H_k^{3}}.
   \end{align*}   
  By \eqref{est:eibtf-1} (for $s=1$), we have
   \begin{align*}
      \|\omega_1(t)\|_{H_k^{-1}}\leq&\|\mathrm{e}^{-\mathrm{i}tB}\Delta_kf\|_{H_k^{-1}}+
      \int_0^t4t'k^2\|\mathrm{e}^{-\mathrm{i}(t-t')B}\Lambda_3\omega(t')\|_{H_k^{-1}}\mathrm{d}t'\\
      \leq&C(k^2+t)^{-1}\|\Delta_kf\|_{H_k^{1}}+C\int_0^tt'k^2(k^2+t-t')^{-1}\|\Lambda_3\omega(t')\|_{H_k^{1}}\mathrm{d}t'\\
      \leq&C(k^2+t)^{-1}\|f\|_{H_k^{3}}+C\int_0^tt'k^2(k^2+t-t')^{-1}\|\omega(t')\|_{H_k^{-3}}\mathrm{d}t'\\
      \leq&C(k^2+t)^{-1}\|f\|_{H_k^{3}}+C\int_0^tt'(k^2+t-t')^{-1}(k^2+t')^{-2}\|f\|_{H_k^{3}}\mathrm{d}t'\\
      \leq&C(k^2+t)^{-1}\|f\|_{H_k^{3}}+C\int_0^t(k^2+t-t')^{-1}(k^2+t')^{-1}\|f\|_{H_k^{3}}\mathrm{d}t'\\
      \leq&C(k^2+t)^{-1}\|f\|_{H_k^{3}}+C(k^2+t)^{-1}\ln(t/k^2+1)\|f\|_{H_k^{3}}\\\leq& C(k^2+t)^{-1}\ln(t/k^2+2)\|f\|_{H_k^{3}}.
   \end{align*}
This along with \eqref{est:om1-Deltaf} gives \eqref{om9}. Here we used the fact
 \begin{align*}
\int_0^t(k^2+t-t')^{-1}(k^2+t')^{-1}\mathrm{d}t'&=(2k^2+t)^{-1}\ln\frac{k^2+t'}{k^2+t-t'}\Big|_{t'=0}^{t'=t}\\
&=2(2k^2+t)^{-1}\ln(t/k^2+1).
\end{align*}

 By Lemma \ref{lem:t2Lamom-om1} (for $s=0$), \eqref{est:om1-Deltaf}, and \eqref{est:eibtf-1} (for $s=2$), we have 
\begin{align}\label{psi9}\|\Delta_k\psi+\mathrm{i}t\Lambda_1\psi/2\|_{*}^2 +t^2k^2\|A\psi\|_{*}^2\leq |\langle\omega_1,\psi\rangle_*|\leq \|\omega_1\|_*\|\psi\|_*\leq C(k^2+t)^{-2}\|f\|_{H_k^2}^2.\end{align} 
By \eqref{lam1}, we have
\begin{align*}
\Delta_k\psi+\mathrm{i}t\Lambda_1\psi/2=(\partial_y^2-k^2)\psi-\mathrm{i}tA\partial_y\psi-\mathrm{i}tB\psi/2
=(\partial_y-\mathrm{i}tA)\partial_y\psi-k^2\psi-\mathrm{i}tB\psi/2.\end{align*}
Then by \eqref{est:eibtf-1} (for $s=2$) and \eqref{psi9}, we obtain
\begin{align*}
\|(\partial_y-\mathrm{i}tA)\partial_y\psi\|_{L^2}\leq&\|\Delta_k\psi+\mathrm{i}t\Lambda_1\psi/2\|_{L^2}+
k^2\|\psi\|_{L^2}+t\|B\psi\|_{L^2}\\
\leq&C(k^2+t)^{-1}\|f\|_{H_k^2}+k^2\|\psi\|_{L^2}+t\|\psi\|_{L^2}\\\leq& C(k^2+t)^{-1}\|f\|_{H_k^2}.
\end{align*}

Recall that $A=\sin y(1+\Delta_k^{-1})=-b'(1+\Delta_k^{-1})$. Thus, 
   \begin{align*}
      & \|(\partial_y-\mathrm{i}tA)\partial_y\psi-(\partial_y+\mathrm{i}tb')\partial_y\psi\|_{L^2}= t\|b'\Delta_k^{-1}\partial_y\psi\|_{{L^2}}\\
      &\quad\leq Ct\|\psi\|_{{L^2}}\leq Ct(k^2+t)^{-2}\|f\|_{H_k^2}\leq (k^2+t)^{-1}\|f\|_{H_k^2},\\
      &\|A\psi+b'\psi\|_{L^2}=\|b'\Delta_k^{-1}\psi\|_{{L^2}}\leq|k|^{-2}\|\psi\|_{{L^2}}\leq C|k|^{-2}(k^2+t)^{-2}\|f\|_{H_k^2}.
   \end{align*}
   Thus, we arrive at
   \beno
   \|(\partial_y+\mathrm{i}tb')\partial_y\psi\|_{L^2}\leq  C(k^2+t)^{-1}\|f\|_{H_k^2},
   \eeno
   and  by \eqref{psi9},
      \begin{align*}
   t|k|\|b'\psi\|_{L^2}\leq& t|k|\|A\psi\|_{L^2}+t|k|\|A\psi+b'\psi\|_{L^2}\\ \leq& C(k^2+t)^{-1}\|f\|_{H_k^2}+Ct|k|^{-1}(k^2+t)^{-2}\|f\|_{H_k^2}\\\leq& C(k^2+t)^{-1}\|f\|_{H_k^2}.
   \end{align*}
   Then, using $\partial_y(\partial_y+\mathrm{i}tb')\psi=(\partial_y+\mathrm{i}tb')\partial_y\psi+\mathrm{i}tb''\psi$ and \eqref{est:eibtf-1}, we conclude   
\begin{align*}
   &\|(\partial_y,k)(\partial_y+\mathrm{i}tb')\psi\|_{L^2}+\|(\partial_y+\mathrm{i}tb')\partial_y\psi\|_{L^2}+|k|\|\partial_y\psi\|_{L^2}\\
   &\leq 2\|(\partial_y+\mathrm{i}tb')\partial_y\psi\|_{L^2}+t|k|\|b'\psi\|_{L^2}+2|k|\|\partial_y\psi\|_{L^2}+t\|b''\psi\|_{L^2}\\
  &\leq C(k^2+t)^{-1}\|f\|_{H_k^2}+2|k|\|\omega\|_{H_k^{-1}}+t\|\psi\|_{L^2}\\
  &\leq C(k^2+t)^{-1}\|f\|_{H_k^2}+C|k|(k^2+t)^{-1}\|f\|_{H_k^{1}}+t(k^2+t)^{-2}\|f\|_{H_k^2}\\&\leq C(k^2+t)^{-1}\|f\|_{H_k^2},
   \end{align*}
   which gives \eqref{psi1}. Thanks to $\mathrm{e}^{\mathrm{i}tb'}(\partial_y+\mathrm{i}tb')\partial_y\psi=\partial_y(\mathrm{e}^{\mathrm{i}tb'}\partial_y\psi)$, we obtain
   \begin{align*}
     &\|(\partial_y+\mathrm{i}tb')\psi\|_{L^{\infty}}\leq C|k|^{-1/2}\|(\partial_y,k)(\partial_y+\mathrm{i}tb')\psi\|_{L^2}\leq 
      C|k|^{-1/2}(k^2+t)^{-1}\|f\|_{H_k^2},\\
      &\|\partial_y\psi\|_{L^{\infty}}=\|\mathrm{e}^{\mathrm{i}tb'}\partial_y\psi\|_{L^{\infty}}\leq C|k|^{-1/2}\|(\partial_y,k)(\mathrm{e}^{\mathrm{i}tb'}\partial_y\psi)\|_{L^2}\\ &\quad\leq 
     C|k|^{-1/2}\big(\|(\partial_y+\mathrm{i}tb')\partial_y\psi\|_{L^2}+|k|\|\partial_y\psi\|_{L^2}\big)\leq C|k|^{-1/2}(k^2+t)^{-1}\|f\|_{H_k^2},
   \end{align*}
   which along with  $|(\partial_y+\mathrm{i}tb')\psi-\partial_y\psi|=|tb'\psi|$ show \eqref{psi2}.

   By \eqref{om9} and the interpolation, we get
   \begin{align}\label{om1s}
      &\|\omega_1(t)\|_{H_k^{-s}}\leq C(k^2+t)^{-s}|k|^{s-1}\ln(t/k^2+2)\|f\|_{H_k^{3}},\quad \forall\ s\in[0,1].
   \end{align}
  By Lemma \ref{lem:t2Lamom-om1} for $s=0.4$, we have 
  \begin{align*}
      \|\omega(t)\|_{H_k^{-3}}=\|\psi(t)\|_{H_k^{-1}}\le& |k|^{s-1}\|\psi(t)\|_{H_k^{-s}}\le C|k|^{s-1}t^{-2}\|\omega_1(t)\|_{H_k^{-s}}\\
      \le& C|k|^{s-1}t^{-2}(k^2+t)^{-s}|k|^{s-1}\ln(t/k^2+2)\|f\|_{H_k^{3}}\\\le& 
      C|k|^{s-1}t^{-2}t^{-s}|k|^{s-1}(t/k^2)^{0.1}\|f\|_{H_k^{3}}
      =Ct^{-2.3}|k|^{-1.4}\|f\|_{H_k^{3}}.
   \end{align*}
   On the other hand,
   \beno
  \|\omega(t)\|_{H_k^{-3}}\leq |k|^{-3}\|\omega(t)\|_{L^2}\leq C|k|^{-3}\|f\|_{L^2}\leq C|k|^{-6}\|f\|_{H_k^{3}}.
  \eeno
  Thus, we arrive at
 \begin{align*}
      &\|\omega(t)\|_{H_k^{-3}}\leq C\min(t^{-2.3}|k|^{-1.4},|k|^{-6})\|f\|_{H_k^{3}}\leq C(k^2+t)^{-2.3}|k|^{-1.4}\|f\|_{H_k^{3}},
  \end{align*}
  which gives \eqref{omH-3}.
  
   For \eqref{psi3}, if $0\leq t\leq k^2$, then we get by \eqref{psi2}  that
   \begin{align*}
     \|(\partial_y+\mathrm{i}tb')\psi\|_{L^{\infty}}\leq&
      C|k|^{-1/2}(k^2+t)^{-1}\|f\|_{H_k^2}\leq C|k|^{-2.5}\|f\|_{H_k^2}\leq C|k|^{-3.5}\|f\|_{H_k^3}\\
      \leq& C|k|^{-3.4}\|f\|_{H_k^3}=C(k^2)^{-1.2}|k|^{-1}\|f\|_{H_k^3}\\\leq& C(k^2+t)^{-1.2}|k|^{-1}\|f\|_{H_k^3}.
   \end{align*}
 Now we assume $t\geq k^2$. By Lemma \ref{lem:t2Lamom-om1} for $s=0.4$ and \eqref{om1s}, we have
 \begin{align*}
     \|(\partial_y-\mathrm{i}tA)\psi\|_{L^{\infty}}\leq& C|k|^{s-1/2}\|(\partial_y-\mathrm{i}tA)\psi\|_{H_k^{1-s}}\leq C|k|^{s-1/2}t^{-1}\|\omega_1\|_{H_k^{-s}}\\
      \leq& C|k|^{s-1/2}t^{-1}(k^2+t)^{-s}|k|^{s-1}\ln(t/k^2+2)\|f\|_{H_k^{3}}\\\leq& Ct^{-1-s}|k|^{2s-3/2}(t/k^2)^{0.2}\|f\|_{H_k^{3}}\\
      =&Ct^{-1.2}|k|^{-1.1}\|f\|_{H_k^{3}}\leq Ct^{-1.2}|k|^{-1}\|f\|_{H_k^{3}}.
   \end{align*} Recall that $A=\sin y(1+\Delta_k^{-1})=-b'(1+\Delta_k^{-1})$, then we get by \eqref{omH-3} that 
   \begin{align*}
    \|(\partial_y-\mathrm{i}tA)\psi-(\partial_y+\mathrm{i}tb')\psi\|_{L^{\infty}}\leq& t\|b'\Delta_k^{-1}\psi\|_{L^{\infty}}\leq Ct|k|^{-1/2}\|\Delta_k^{-1}\psi\|_{H_k^{1}}\\=&Ct|k|^{-1/2}\|\psi\|_{H_k^{-1}}=Ct|k|^{-1/2}\|\omega(t)\|_{H_k^{-3}}\\\leq&
      Ct|k|^{-1/2}t^{-2.3}|k|^{-1.4}\|f\|_{H_k^{3}}\leq Ct^{-1.2}|k|^{-1}\|f\|_{H_k^{3}}.
   \end{align*}
   Thus, we obtain
   \begin{align*}
      & \|(\partial_y+\mathrm{i}tb')\psi\|_{L^{\infty}}\leq Ct^{-1.2}|k|^{-1}\|f\|_{H_k^{3}}\leq C(k^2+t)^{-1.2}|k|^{-1}\|f\|_{H_k^3}.
   \end{align*}
  By combining the two cases $0\leq t\leq k^2$ and $t\geq k^2$, we arrive at \eqref{psi3}. 
   
   By \eqref{Duhamel}, {\eqref{est:eibtf-3}, \eqref{om8} and \eqref{omH-3}}, we have
   \begin{align*}
      &|\omega_1(t,y)|\leq|\mathrm{e}^{-\mathrm{i}tB}\Delta_kf(t,y)|+
      \int_0^t4t'k^2|\mathrm{e}^{-\mathrm{i}(t-t')B}\Lambda_3\omega(t',y)|\mathrm{d}t'\\
      &\leq C\big((k^2+t)^{-1/4}+|b'|^{1/2})\|\Delta_kf\|_{H_k^{1}}+C\int_0^tt'k^2((k^2+t-t')^{-1/4}+|b'|^{1/2}\big)
      \|\Lambda_3\omega(t')\|_{H_k^{1}}\mathrm{d}t'\\
      &\leq C\big((k^2+t)^{-1/4}+|b'|^{1/2}\big)\|f\|_{H_k^{3}}\\&\quad+C\int_0^tt'k^2((k^2+t-t')^{-1/4}+|b'|^{1/2})
      (k^2+t')^{-2.3}|k|^{-1.4}\|f\|_{H_k^{3}}\mathrm{d}t'\\
      &\leq C\big((k^2+t)^{-1/4}+|b'|^{1/2}\big)\|f\|_{H_k^{3}}\\&\quad+C\int_0^t|k|^{0.6}((k^2+t-t')^{-1/4}+|b'|^{1/2})(k^2+t')^{-1.3}\|f\|_{H_k^{3}}\mathrm{d}t'\\
      &\leq C\big((k^2+t)^{-1/4}+|b'|^{1/2}\big)\|f\|_{H_k^{3}}.
   \end{align*}
   This gives \eqref{om1t}. Here we used the fact that 
   \begin{align*}
      &\int_0^t|k|^{0.6}(k^2+t')^{-1.3}\mathrm{d}t'\leq C|k|^{0.6}(k^2)^{-0.3}=C,
   \end{align*}
   and 
    \begin{align*}
     &\int_0^t|k|^{0.6}(k^2+t-t')^{-1/4}(k^2+t')^{-1.3}\mathrm{d}t'\\
      &\leq \int_0^{t/2}|k|^{0.6}(k^2+t/2)^{-1/4}(k^2+t')^{-1.3}\mathrm{d}t'+\int_{t/2}^t|k|^{0.6}(k^2+t-t')^{-1/4}(k^2+t/2)^{-1.3}\mathrm{d}t'\\
      &\leq C|k|^{0.6}(k^2+t/2)^{-1/4}(k^2)^{-0.3}+C|k|^{0.6}(k^2+t/2)^{3/4}(k^2+t/2)^{-1.3}\\
      &= C(k^2+t/2)^{-1/4}+C|k|^{0.6}(k^2+t/2)^{-0.55}\\ &\leq C(k^2+t/2)^{-1/4}+C|k|^{0.6}(k^2+t/2)^{-1/4}(k^2)^{-0.3}\\&\leq C(k^2+t/2)^{-1/4}\leq C(k^2+t)^{-1/4}.
   \end{align*}

   This completes the proof of the lemma.
   \end{proof} 

\subsubsection{Decay estimates near critical point}
In what follows, we denote
\beno
&&\omega=\mathrm{e}^{-\mathrm{i}tB}f,\quad \psi=\Delta_k^{-1}\omega,\\
&&\omega_1=\Delta_{(t)}\omega,\quad  \omega_2=(\partial_y+\mathrm{i}tb')^2\omega+2t^2\psi,\quad  \omega_3=(\partial_y+\mathrm{i}tb')^2\partial_y^2\psi+2t^2\psi.
\eeno
We introduce the operator $L=(\partial_y+\mathrm{i}tb')\partial_y$. Notice that 
\beno
L=\pa_y(\partial_y+\mathrm{i}tb')-\mathrm{i}tb''=\pa_y(\partial_y+\mathrm{i}tb')+\mathrm{i}tb.
\eeno
Then we have
\begin{align*}
      \omega_3=&(\partial_y+\mathrm{i}tb')^2\partial_y^2 {\psi}+2t^2{\psi}
      =(\partial_y+\mathrm{i}tb')[\partial_y(\partial_y+\mathrm{i}tb')+\mathrm{i}tb]\partial_y{\psi}+2t^2{\psi}\\
      =&(\partial_y+\mathrm{i}tb')\partial_y(\partial_y+\mathrm{i}tb')\partial_y{\psi}+
      \mathrm{i}tb(\partial_y+\mathrm{i}tb')\partial_y{\psi}+\mathrm{i}tb'\partial_y{\psi}+2t^2{\psi}.   
      \end{align*}
  This gives
  \begin{align}
      \omega_3
      =&L^2{\psi}+\mathrm{i}tbL{\psi}+\mathrm{i}tb'\partial_y{\psi}+2t^2{\psi}.\label{om3a}
   \end{align}  
   
The main result of this subsection consists of the following decay estimates for $\mathrm{e}^{-\mathrm{i}tB}f$ near the critical point of $b(y)$. 

    \begin{lemma}\label{lem6}      
It {holds} that  for $t\geq k^2$, $|y|\leq 1/|k|$, 
   \begin{align*}
      &|{\psi}(t,y)|\leq Ct^{-2}|k|^{-1/2}\|f\|_{H_k^{3}},\\
      &|\partial_y{\psi}(t,y)|\leq C(t^{-3/2}|k|^{-1/2}+|y|t^{-1}|k|^{-1/2})\|f\|_{H_k^{3}},\\
       &|\partial_y(\mathrm{e}^{\mathrm{i}tb}\omega)(t,y)|\leq C(t^{-1/2}+|y|)|k|^{-1/2}\|f\|_{H_k^{3}},\\
       &|\omega(t,y)|=|\mathrm{e}^{\mathrm{i}tb}\omega(t,y)|\leq C(t^{-1}+|y|^2)|k|^{-1/2}\|f\|_{H_k^{3}}.
   \end{align*}
   \end{lemma}

The proof is based on the following lemmas.

\begin{lemma}\label{lem4a}
It holds that for $t\geq k^2$, $|y|\leq 1/|k|$,
   \begin{align}\label{om3}
   &|\omega_2(t,y)|\leq C(t^{-1/4}+|b'|^{1/2})\|f\|_{H_k^{3}},\quad |\omega_3(t,y)|\leq C(t^{-1/4}+|b'|^{1/2})\|f\|_{H_k^{3}}.
   \end{align}
\end{lemma}    

 \begin{proof}    
From \eqref{om11}, \eqref{L1-1}, \eqref{L2} and  \eqref{est:A2+B2}, we know that  
   \begin{align*}
      &\omega_1=(\partial_y-\mathrm{i}tA)^2\omega-k^2\omega-t^2(1-B^2-A^2)\omega,\\
      & \Delta_{k,-1}+\Delta_{k,1}=2(\Delta_{k}-1),\quad \Delta_{k,-1}\Delta_{k,1}=(\Delta_{k}+1)^2+4k^2,\\
      &B^2+A^2=(2+\Delta_{k,-1}^{-1}+\Delta_{k,1}^{-1})(1+\Delta_k^{-1})/2\\
      &=(1+(\Delta_{k}-1)((\Delta_{k}+1)^2+4k^2)^{-1})(1+\Delta_k^{-1})\\
      &=1+2\Delta_k^{-1}-2(\Delta_{k}+1+2k^2)((\Delta_{k}+1)^2+4k^2)^{-1}\Delta_k^{-1}.
   \end{align*}
Let 
\beno
h_1=(1+2\Delta_k^{-1}-B^2-A^2)\omega=2(\Delta_{k}+1+2k^2)((\Delta_{k}+1)^2+4k^2)^{-1}\Delta_k^{-1}\omega.
\eeno
Similar to \eqref{om8}, we have $\|h_1\|_{H_k^{1}}\leq C\|\omega\|_{H_k^{-3}}$. Recall that $A=\sin y(1+\Delta_k^{-1})=-b'(1+\Delta_k^{-1})$, then
\begin{align*}
      & \omega_1-\omega_2=\mathrm{i}tb'\Delta_k^{-1}(\partial_y-\mathrm{i}tA)\omega+\mathrm{i}t(\partial_y+\mathrm{i}tb')(b'\Delta_k^{-1}\omega)-k^2\omega-t^2h_1\\
      &\qquad=2\mathrm{i}tb'\partial_y\psi+t^2b'\Delta_k^{-1}(A\omega)+\mathrm{i}tb''\psi-t^2b'^2\psi-k^2\omega-t^2h_1,\\
      &-\Delta_k^{-1}(A\omega)=\Delta_k^{-1}(b'(\omega+\psi))=b'\psi+2\Delta_k^{-1}\partial_y(b\psi),\quad (b''=-b).  
       \end{align*}
   which gives
   \begin{align*}
      |\omega_1-\omega_2|\leq 2t|b'\partial_y\psi|+t|\psi|+2t^2|b'|^2|\psi|+2t^2|b'||\Delta_k^{-1}\partial_y(b\psi)|+k^2|\omega|+t^2|h_1|.
   \end{align*}   
   
As $\omega=\Delta_k\psi=(\partial_y^2-k^2)\psi$, we have\begin{align*}
      & \omega_2-\omega_3={-}k^2(\partial_y+\mathrm{i}tb')^2\psi={-}k^2(\partial_y^2\psi+2\mathrm{i}tb'\partial_y\psi+
      \mathrm{i}tb''\psi-t^2b'^2\psi),\\
      &|\omega_2-\omega_3|\leq k^2(|\omega|+k^2|\psi|+2t|b'\partial_y\psi|+t|\psi|+t^2|b'|^2|\psi|).
   \end{align*}
   Thus, we obtain
   \begin{align*}
      &|\omega_1-\omega_2|+|\omega_2-\omega_3|\\ 
      &\leq Ck^2(|\omega|+(k^2+t)|\psi|+t|b'\partial_y\psi|+t^2|b'|^2|\psi|)+2t^2|b'||\Delta_k^{-1}\partial_y(b\psi)|+t^2|h_1|.
   \end{align*}
   As $\omega=\mathrm{e}^{-\mathrm{i}tB}f$, we get by {\eqref{est:eibtf-3}}  that 
   \begin{align*}
      &k^2|\omega|\leq Ck^2(t^{-1/4}+|b'|^{1/2})\|f\|_{H_k^{1}}\leq C(t^{-1/4}+|b'|^{1/2})\|f\|_{H_k^{3}}.
   \end{align*}
   By \eqref{est:eibtf-1} for $s=5/4$, we have
   \begin{align*}
      &k^2(k^2+t)|\psi|\leq Ck^2(k^2+t)|k|^{-1/4}\|\psi\|_{H_k^{3/4}}=
      C|k|^{7/4}(k^2+t)\|\omega\|_{H_k^{-5/4}}\\ &\leq C|k|^{7/4}(k^2+t)(k^2+t)^{-5/4}\|f\|_{H_k^{5/4}}\leq
      C|k|^{7/4}t^{-1/4}\|f\|_{H_k^{5/4}}\le Ct^{-1/4}\|f\|_{H_k^{3}}.\end{align*}
   By \eqref{psi2}, we have (for $t\geq k^2$, $|y|\leq 1/|k|$, as $|b'(y)|=|\sin y|\leq |y|\leq 1/|k|$)
   \begin{align}
   \notag k^2(t|b'\partial_y\psi|+t^2|b'|^2|\psi|)&\leq k^2t|b'|(\|\partial_y\psi\|_{L^{\infty}}+t\|b'\psi\|_{L^{\infty}})\leq 
      Ck^2t|b'|(k^2+t)^{-1}|k|^{-1/2}\|f\|_{H_k^2}\\
    \label{psia}  &\leq C|k|^{3/2}|b'|\|f\|_{H_k^2}\leq C|k|^{1/2}|b'|\|f\|_{H_k^3}\leq C|b'|^{1/2}\|f\|_{H_k^3}.
   \end{align} 
   By \eqref{est:eibtf-1} for $s=2$, we have
   \begin{align*}
      &t^2|b'||\Delta_k^{-1}\partial_y(b\psi)|\leq Ct^2|b'|\|\Delta_k^{-1}\partial_y(b\psi)\|_{H_k^1}\leq Ct^2|b'|\|b\psi\|_{L^2}\leq Ct^2|b'|\|\psi\|_{L^2}\\ &\leq Ct^2|b'|(k^2+t)^{-2}\|f\|_{H_k^2}\leq C|b'|\|f\|_{H_k^2}\leq C|b'|\|f\|_{H_k^3}\leq C|b'|^{1/2}\|f\|_{H_k^3}.\end{align*}
      As $ \|h_1\|_{H_k^{1}}\leq C\|\omega\|_{H_k^{-3}}$, we get by \eqref{omH-3} that
      \begin{align*}
      t^2|h_1|\leq& Ct^2\|h_1\|_{H_k^{1}}\leq Ct^2\|\omega\|_{H_k^{-3}}\leq Ct^2(k^2+|t|)^{-2.3}|k|^{-1.4}\|f\|_{H_k^{3}}
      \\ \leq& 
    C(k^2+t)^{-0.3}\|f\|_{H_k^{3}}\leq Ct^{-1/4}\|f\|_{H_k^{3}}.\end{align*}
    
Summing up, we conclude that for $t\geq k^2$, $|y|\leq 1/|k|$,
      \begin{align*}
      &|\omega_1-\omega_2|+|\omega_2-\omega_3|\leq C(t^{-1/4}+|b'|^{1/2})\|f\|_{H_k^{3}},
   \end{align*}
   which along with \eqref{om1t} gives \eqref{om3}. 
    \end{proof}

 \begin{lemma}\label{lem5}      
Let $\psi_1=(L-\mathrm{i}t){\psi}$ and $\psi_2=(L+2\mathrm{i}t){\psi}$. Then it holds that for $t\geq k^2$, $|y|\leq 1/|k|$,
\begin{align*}
      &|(L+\mathrm{i}t(b+1))\psi_1(t,y)|+|(L+\mathrm{i}t(b-2))\psi_2(t,y)|\leq C(t^{-1/4}+|b'|^{1/2})\|f\|_{H_k^{3}},\\
      &\|L\psi(t)\|_{L^2}+\|\psi_1(t)\|_{L^2}+\|\psi_2(t)\|_{L^2}\leq  Ct^{-1}|k|^{-1}\|f\|_{H_k^3},\\
      &|\psi_1(t,y)|+|\psi_2(t,y)|\leq Ct^{-1}|k|^{-1/2}\|f\|_{H_k^{3}},\\ 
      &|\partial_y\psi_1(t,0)|+|\partial_y\psi_2(t,0)|\leq Ct^{-1/2}|k|^{-1/2}\|f\|_{H_k^{3}}.
      \end{align*}
   \end{lemma}

The proof relies on the following lemma on the $L^\infty$ estimate for the second order ODE. Given that the proof is rather technical, it will be presented in Appendix C.
      
\begin{lemma}\label{lem5a}      
Let $I=[-1/|k|,1/|k|]$, $t\geq k^2$. If $\phi$ solves $(L-\mathrm{i}tb)\phi=F_1$, then we have
\begin{align*}
      &t\|\phi\|_{L^{\infty}(I)}+t^{1/2}|\phi'(0)|\leq C\|F_1/(t^{-1/2}+|y|)\|_{L^1(I)}+Ct|k|^{1/2}\|\phi\|_{L^2(I)}.
      \end{align*}
Let $f_2=b-1+1/(2\mathrm{i}t)$ and $h=Lf_2/f_2$. If $\phi$ solves $(L-h)\phi=F_2$, then we have 
\begin{align*}
      &t\|\phi\|_{L^{\infty}(I)}+t^{1/2}|\phi'(0)|\leq C\|F_2/(t^{-1/2}+|y|)\|_{L^1(I)}+Ct|k|^{1/2}\|\phi\|_{L^2(I)}.
      \end{align*}
   \end{lemma}
   
   Now we prove Lemma \ref{lem5}.
   
   \begin{proof}
   By {\eqref{om3a}}, we have
   \begin{align*}
      &\omega_3-\mathrm{i}tb'\partial_y{\psi}=L^2{\psi}+\mathrm{i}tbL{\psi}+2t^2{\psi}=(L+2\mathrm{i}t)(L-\mathrm{i}t){\psi} +\mathrm{i}t(b-1)L{\psi}.
   \end{align*}
   As $\psi_1=(L-\mathrm{i}t){\psi}$, $\psi_2=(L+2\mathrm{i}t){\psi}$, we get
   \begin{align*}
      \omega_3-\mathrm{i}tb'\partial_y{\psi}=&(L+2\mathrm{i}t)\psi_1+\mathrm{i}t(b-1)L{\psi}
      =(L+2\mathrm{i}t)\psi_1+\mathrm{i}t(b-1)(\psi_1+\mathrm{i}t{\psi})\\
      =&(L+\mathrm{i}t(b+1))\psi_1-t^2(b-1){\psi},\\
      \omega_3-\mathrm{i}tb'\partial_y{\psi}=&(L-\mathrm{i}t)\psi_2+\mathrm{i}t(b-1)L{\psi}
      =(L-\mathrm{i}t)\psi_2+\mathrm{i}t(b-1)(\psi_2-2\mathrm{i}t{\psi})\\
      =&(L+\mathrm{i}t(b-2))\psi_2+2t^2(b-1){\psi}.
   \end{align*}
    {Thus,}
   \begin{align*}
      &|(L+\mathrm{i}t(b+1))\psi_1|\leq |\omega_3-\mathrm{i}tb'\partial_y{\psi}|+t^2|(b-1){\psi}|,\\
      &|(L+\mathrm{i}t(b-2))\psi_2|\leq |\omega_3-\mathrm{i}tb'\partial_y{\psi}|+2t^2|(b-1){\psi}|.
   \end{align*}
   {By using \eqref{om3}, \eqref{psia} and $0\le 1-b\leq1-b^2=b'^2$ as $b(y)=\cos y>0$ for $|y|\leq 1/|k|$,} we infer that 
    \begin{align*}
     |(L+\mathrm{i}t(b+1))\psi_1|+|(L+\mathrm{i}t(b-2))\psi_2|\leq& 2|\omega_3|+2t|b'\partial_y{\psi}|+3t^2|(b-1){\psi}|\\
      \leq& 2|\omega_3|+2t|b'\partial_y{\psi}|+3t^2|b'^2{\psi}|\\
      \leq& C(t^{-1/4}+|b'|^{1/2})\|f\|_{H_k^{3}}+C|k|^{-2}|b'|^{1/2}\|f\|_{H_k^{3}}.
   \end{align*}
    This proves the first inequality of the lemma. 
    
    As $\psi_1=(L-\mathrm{i}t){\psi}$, $\psi_2=(L+2\mathrm{i}t){\psi}$, $L=(\partial_y+\mathrm{i}tb')\partial_y$, we get by \eqref{psi1} and \eqref{est:eibtf-1} (for $s=2$)  that 
    \begin{align*}
      &\|L\psi(t)\|_{L^2}+\|\psi_1(t)\|_{L^2}+\|\psi_2(t)\|_{L^2}\leq  3\|L\psi(t)\|_{L^2}+3t\|\psi(t)\|_{L^2}\\
      &\leq Ct^{-1}\|f\|_{H_k^{2}} +Ctt^{-2}\|f\|_{H_k^{2}}\leq Ct^{-1}\|f\|_{H_k^{2}}\leq Ct^{-1}|k|^{-1}\|f\|_{H_k^{3}}.
      \end{align*}
      This shows the second inequality of the lemma.

     Notice that $(L+\mathrm{i}t(b+1))f_2=(1-b)/2$, $f_2=b-1+1/(2\mathrm{i}t)$, we have $|f_2|\geq|b-1|=2|(L+\mathrm{i}t(b+1))f_2|$ and
     \begin{align*}
      &|h+\mathrm{i}t(b+1)|=|(L+\mathrm{i}t(b+1))f_2|/|f_2|\leq1/2 .\end{align*} 
      By the first inequality of the lemma, we have
      \begin{align*}
      &|(L-\mathrm{i}tb)\psi_2|=|(L+\mathrm{i}t(b-2))\psi_2|+2t|(b-1)\psi_2|\leq C(t^{-1/4}+|b'|^{1/2})\|f\|_{H_k^{3}}+t|y\psi_2|,\\
      &|(L-h)\psi_1|=|(L+\mathrm{i}t(b+1))\psi_1|+|(h+\mathrm{i}t(b+1))\psi_1|\leq C(t^{-1/4}+|b'|^{1/2})\|f\|_{H_k^{3}}+|\psi_1|.
      \end{align*}
      Here we used $|b-1|=1-\cos y\leq y^2/2\leq |y|/2$ for $|y|\leq 1/|k|\leq 1$. Then by Lemma \ref{lem5a}, and 
      $|b'(y)|=|\sin y|\leq |y|$, we obtain(for $t\geq k^2$, $y\in I=[-1/|k|, 1/|k|]$)
      \begin{align*}
      &t|\psi_2(t,y)|+t^{1/2}|\partial_y\psi_2(t,0)|\leq Ct|k|^{1/2}\|\psi_2(t)\|_{L^2}+C\|(L-\mathrm{i}tb)\psi_2/(t^{-1/2}+|y|)\|_{L^1(I)},\\
      &t|\psi_1(t,y)|+t^{1/2}|\partial_y\psi_1(t,0)|\leq Ct|k|^{1/2}\|\psi_1(t)\|_{L^2}+C\|(L-h)\psi_1/(t^{-1/2}+|y|)\|_{L^1(I)},
      \end{align*}
      which yield by the second inequality of the lemma that
       \begin{align*}
            &t|\psi_1(t,y)|+t^{1/2}|\partial_y\psi_1(t,0)|+t|\psi_2(t,y)|+t^{1/2}|\partial_y\psi_2(t,0)|\\
      &\leq Ct|k|^{1/2}(\|\psi_1(t)\|_{L^2}+\|\psi_2(t)\|_{L^2})+C\|(t^{-1/4}+|b'|^{1/2})/(t^{-1/2}+|y|)\|_{L^1(I)}\|f\|_{H_k^{3}}\\&\qquad+
      Ct\|y\psi_2/(t^{-1/2}+|y|)\|_{L^1(I)}+C\|\psi_1/(t^{-1/2}+|y|)\|_{L^1(I)}\\
      &\leq Ct|k|^{1/2}(\|\psi_1(t)\|_{L^2}+\|\psi_2(t)\|_{L^2})+C\|(t^{-1/2}+|y|)^{-1/2}\|_{L^1(I)}\|f\|_{H_k^{3}}\\&\qquad+
      Ct\|\psi_2\|_{L^1(I)}+Ct^{1/2}\|\psi_1\|_{L^1(I)}\\
      &\leq Ct|k|^{1/2}(\|\psi_1(t)\|_{L^2}+\|\psi_2(t)\|_{L^2})+C|k|^{-1/2}\|f\|_{H_k^{3}}\leq C|k|^{-1/2}\|f\|_{H_k^{3}}.
      \end{align*}
     This proves the third inequality of the lemma.
   \end{proof}
   
   Now we prove Lemma \ref{lem6}.
   
     \begin{proof} 
Notice that {(here $\psi_1=(L-\mathrm{i}t){\psi}$ and $\psi_2=(L+2\mathrm{i}t){\psi}$ as in Lemma \ref{lem5})}
\beno
L{\psi}=(2\psi_1+\psi_2)/3,\quad \mathrm{i}t{\psi}=(\psi_2-\psi_1)/3.
\eeno
 We infer from Lemma \ref{lem5} that for $t\geq k^2$, $|y|\leq 1/|k|$,
\begin{align}\label{psi6}
      &t|{\psi}(t,y)|+|L{\psi}(t,y)|\leq Ct^{-1}|k|^{-1/2}\|f\|_{H_k^{3}},\\
     \label{psi7} &t|\partial_y{\psi}(t,0)|+|\partial_yL{\psi}(t,0)|\leq Ct^{-1/2}|k|^{-1/2}\|f\|_{H_k^{3}},
   \end{align}
   which give the first inequality of the lemma.
    
   As $\mathrm{e}^{\mathrm{i}tb}L{\psi}=\mathrm{e}^{\mathrm{i}tb}(\partial_y+\mathrm{i}tb')\partial_y\psi=
      \partial_y(\mathrm{e}^{\mathrm{i}tb}\partial_y{\psi}) $, we have \begin{align*}
      &|\partial_y{\psi}(t,y)|=|\mathrm{e}^{\mathrm{i}tb}\partial_y{\psi}(t,y)|\leq
      |\mathrm{e}^{\mathrm{i}tb}\partial_y{\psi}(t,0)|+\mathrm{sgn}(y)\int_0^y|\partial_y(\mathrm{e}^{\mathrm{i}tb}\partial_y{\psi})(t,y')|\mathrm{d}y'\\
      &=|\partial_y{\psi}(t,0)|+\mathrm{sgn}(y)\int_0^y|\mathrm{e}^{\mathrm{i}tb}L{\psi}(t,y')|\mathrm{d}y'
      \leq |\partial_y{\psi}(t,0)|+|y|\|L{\psi}(t,\cdot)\|_{L^{\infty}(|y|\leq 1/|k|)}\\
      &\leq Ct^{-3/2}|k|^{-1/2}\|f\|_{H_k^{3}}+C|y|t^{-1}|k|^{-1/2}\|f\|_{H_k^{3}},
   \end{align*}
   which yields the second inequality of the lemma.
     
   Recall  $\omega_2=(\partial_y+\mathrm{i}tb')^2\omega+2t^2\psi$. It follows from \eqref{om3} and \eqref{psi6}  that for $t\geq k^2$, $|y|\leq 1/|k|$, 
   \begin{align}\label{om4}
      |\partial_y^2(\mathrm{e}^{\mathrm{i}tb}\omega)|&=|(\partial_y+\mathrm{i}tb')^2\omega|\leq2t^2|\psi|+|\omega_2|
       \leq C\big(|k|^{-1/2}+t^{-1/4}+|b'|^{1/2}\big)\|f\|_{H_k^{3}}\\& \notag  \leq C\big(|k|^{-1/2}+|k^2|^{-1/4}+|y|^{1/2}\big)\|f\|_{H_k^{3}}
       \leq C|k|^{-1/2}\|f\|_{H_k^{3}}.
   \end{align}
   As $L{\psi}=(\partial_y+\mathrm{i}tb')\partial_y\psi$, $b'(0)=0$, $b''(0)=-1$ and $\omega=(\partial_y^2-k^2){\psi}$, we find
   \begin{align*}
      &L{\psi}(t,0)=\partial_y^2{\psi}(t,0)=\omega(t,0)+k^2{\psi}(t,0),\\
       &\partial_yL{\psi}(t,0)=\partial_y^3{\psi}(t,0)-\mathrm{i}t\partial_y{\psi}(t,0)=\partial_y\omega(t,0)+(k^2-\mathrm{i}t)\partial_y{\psi}(t,0).
   \end{align*}
   Then we get by \eqref{psi6}, \eqref{psi7} and $b'(0)=0$  that for $t\geq k^2$, $|y|\leq 1/|k|$,
   \begin{align}
    \label{om5}  &|\omega(t,0)|\leq |L{\psi}(t,0)|+k^2|{\psi}(t,0)|\leq |L{\psi}(t,0)|+t|{\psi}(t,0)|\leq Ct^{-1}|k|^{-1/2}\|f\|_{H_k^{3}},\\
    \label{om6}   &|\partial_y(\mathrm{e}^{\mathrm{i}tb}\omega)(t,0)|=|\partial_y\omega(t,0)|\leq 
       |\partial_yL{\psi}(t,0)|+(k^2+t)|\partial_y{\psi}(t,0)|\\
   \notag    &\qquad\leq |\partial_yL{\psi}(t,0)|+2t|\partial_y{\psi}(t,0)|\leq Ct^{-1/2}|k|^{-1/2}\|f\|_{H_k^{3}}.
   \end{align}
   By \eqref{om4} and \eqref{om6}, we have (for $t\geq k^2$, $|y|\leq 1/|k|$)
   \begin{align}\label{om7}
      |\partial_y(\mathrm{e}^{\mathrm{i}tb}\omega)(t,y)|&\leq
      |\partial_y(\mathrm{e}^{\mathrm{i}tb}\omega)(t,0)|+\mathrm{sgn}(y)\int_0^y|\partial_y^2(\mathrm{e}^{\mathrm{i}tb}\omega)(t,y')|\mathrm{d}y'\\
     \notag  &\leq Ct^{-1/2}|k|^{-1/2}\|f\|_{H_k^{3}}+\mathrm{sgn}(y)\int_0^yC|k|^{-1/2}\|f\|_{H_k^{3}}\mathrm{d}y'\\
      &=C(t^{-1/2}+|y|)|k|^{-1/2}\|f\|_{H_k^{3}},\nonumber
   \end{align}
   which shows the third inequality of the lemma. 
   
   By \eqref{om5} and \eqref{om7}, we have (for $t\geq k^2$, $|y|\leq 1/|k|$)\begin{align*}
      &|\omega(t,y)|=|\mathrm{e}^{\mathrm{i}tb}\omega(t,y)|\leq
      |\mathrm{e}^{\mathrm{i}tb}\omega(t,0)|+\mathrm{sgn}(y)\int_0^y|\partial_y(\mathrm{e}^{\mathrm{i}tb}\omega)(t,y')|\mathrm{d}y'\\
      &\leq |\omega(t,0)|+\mathrm{sgn}(y)\int_0^yC(t^{-1/2}+|y'|)|k|^{-1/2}\|f\|_{H_k^{3}}\mathrm{d}y'\\
      &\leq Ct^{-1}|k|^{-1/2}\|f\|_{H_k^{3}}+C(t^{-1/2}|y|+|y|^2)|k|^{-1/2}\|f\|_{H_k^{3}}\leq C(t^{-1}+|y|^2)|k|^{-1/2}\|f\|_{H_k^{3}},
   \end{align*}
   which gives the fourth inequality of the lemma.
 \end{proof}

\subsubsection{Proof of Proposition \ref{lem:eibtf}}
\begin{proof}
As $|\partial_y(\mathrm{e}^{\mathrm{i}tb}\psi)|=|(\partial_y+\mathrm{i}tb')\psi|$, \eqref{est:pareitbpsi} follows from \eqref{psi3}. By \eqref{est:eibtf-2} for $s=2$, we can obtain \eqref{eq17} as follows
\begin{align*}
      &\|(\partial_y,k)^2(\mathrm{e}^{\mathrm{i}tb}\omega)\|_{L^2}=\|\mathrm{e}^{\mathrm{i}tb}\omega\|_{H_k^2}\leq C\|f\|_{H_k^2}\leq C|k|^{-1}\|f\|_{H_k^3}.
   \end{align*}
  By Gagliardo-Nirenberg inequality and $\psi=\Delta_k^{-1}\omega$, we have
  \begin{align}\label{om1}
       &\|\omega\|_{L^{\infty}}=\|\mathrm{e}^{\mathrm{i}tb}\omega\|_{L^{\infty}}\leq 
       C|k|^{-3/2}\|(\partial_y,k)^2(\mathrm{e}^{\mathrm{i}tb}\omega)\|_{L^2}\leq   C|k|^{-5/2}\|f\|_{H_k^3},\\\label{om2}
      &\|\partial_y(\mathrm{e}^{\mathrm{i}tb}\omega)\|_{L^{\infty}}\leq C|k|^{-1/2}\|(\partial_y,k)^2(\mathrm{e}^{\mathrm{i}tb}\omega)\|_{L^2}
       \leq   C|k|^{-3/2}\|f\|_{H_k^3},\\ \label{psi4}
       &\|\psi\|_{L^{\infty}}\leq C|k|^{-2}\|\omega\|_{L^{\infty}}\leq   C|k|^{-9/2}\|f\|_{H_k^3},\\ \label{psi5}
       &\|\partial_y\psi\|_{L^{\infty}}\leq C|k|^{-1}\|\omega\|_{L^{\infty}}\leq   C|k|^{-7/2}\|f\|_{H_k^3}.
   \end{align}
   
   For \eqref{eq12}--\eqref{est:vor-dep-2}, we consider the following three cases.\smallskip
   
   {\it Case 1.} $|t|\leq k^2$. In this case, \eqref{eq12}--\eqref{est:vor-dep-2} follow directly from \eqref{om1}--\eqref{psi5}.\smallskip
      
    {\it Case 2.} $|t|\geq k^2$ and $|b'(y)|\geq \sin|1/k|$. In this case, $|b'(y)|\geq 1/|2k|$, we get by \eqref{psi2}  that
    \begin{align*}
      |\psi(t,y)|&\leq  C|k|\|b'\psi(t)\|_{L^{\infty}}\leq C|k||t|^{-1}(|t|+k^2)^{-1}|k|^{-1/2}\|f\|_{H_k^2}\\ 
      &\leq C(|t|+k^2)^{-2}|k|^{1/2}\|f\|_{H_k^2}\leq C(|t|+k^2)^{-2}|k|^{-1/2}\|f\|_{H_k^3},\\
      |\partial_y\psi(t,y)|&\leq  C(|t|+k^2)^{-1}|k|^{-1/2}\|f\|_{H_k^2}\leq  C(|t|+k^2)^{-1}|b'||k|^{1/2}\|f\|_{H_k^2}\\&\leq  C(|t|+k^2)^{-1}|b'||k|^{-1/2}\|f\|_{H_k^3}.
   \end{align*}
   This proves  \eqref{eq12} and \eqref{est:vor-dep-0}. Moreover, \eqref{est:vor-dep-1} and \eqref{est:vor-dep-2} follow from 
   \eqref{om1} and \eqref{om2} by noticing that 
   \beno
   |k|^{-5/2}=|1/k|^2|k|^{-1/2}\leq C|b'|^2|k|^{-1/2},\quad 
   |k|^{-3/2}=|1/k||k|^{-1/2}\leq C|b'||k|^{-1/2}.
   \eeno
   
 {\it Case 3.} $|t|\geq k^2$ and $|b'(y)|\leq \sin|1/k|$. In this case,  $y\in[-1/|k|,1/|k|]\cup[\pi-1/|k|,\pi+1/|k|]$, as $|b'(y)|=|\sin y|$. 
 Without lose of generality, we assume $t\geq k^2$, $y\in[-1/|k|,1/|k|]$. Then \eqref{eq12}--\eqref{est:vor-dep-2} follows from Lemma \ref{lem6} 
 by noting that  $|t|+k^2\sim t$, $|b'(y)|=|\sin y|\sim |y|$.  The case $y\in[\pi-1/|k|,\pi+1/|k|]$ can be proved by the translation: 
 $\om\to\om(t,y+\pi)$, and the case $t\leq -k^2$ can be proved by taking the conjugation. 
  \end{proof}

\subsection{Proof of coercive estimate}\label{sec:coercive}

The proof relies on the following key decomposition.

\begin{lemma}\label{lem8}
  It holds that
   \begin{align*}
      &\Lambda_1(-\Delta_k)^{s}\Lambda_1+2(1-B^2)(-\Delta_k)^{1+s}+2(-\Delta_k)^{1+s}(1-B^2)= 4AH_0A+H_*,  
       \end{align*}
   where
    \begin{align*}
      &H_0=((-\Delta_{k,-1})^{1+s}+(-\Delta_{k,1})^{1+s})/2+(\partial_y^2+1/4)(-\Delta_k)^{s},\\
      &H_*=(\partial_y-\mathrm{i}/2)^2(1+\Delta_{k,-1}^{-1})(-\Delta_{k,-1})^{s}(1+\Delta_k^{-1})+
     (\partial_y+\mathrm{i}/2)^2(1+\Delta_{k,1}^{-1})(-\Delta_{k,1})^{s}(1+\Delta_k^{-1})\\
     &\qquad+(4-(2+\Delta_{k,-1}^{-1}+\Delta_{k,1}^{-1})(1+\Delta_k^{-1}))(-\Delta_{k})^{1+s}\\&\qquad-(1+\Delta_{k,-1}^{-1})H_{-1}(1+\Delta_k^{-1})
     -(1+\Delta_{k,1}^{-1})H_1(1+\Delta_k^{-1}).
   \end{align*}   
   with $H_j=\mathrm{e}^{-\mathrm{i}jy}H_0\mathrm{e}^{\mathrm{i}jy},\ j\in\{\pm1\}$.

 \end{lemma}

\begin{proof}
Recall that
  \begin{align*}
     A&=\sin y(1+\Delta_k^{-1}),\quad   \Lambda_1=-A\partial_y- \partial_y A,\\
\mathrm{i}\Lambda_1&=\mathrm{i}\big(-2\sin y\partial_y-\cos y\big)(1+\Delta_k^{-1}), \\
&=-\mathrm{e}^{\mathrm{i}y}(\partial_y+\mathrm{i}/2)(1+\Delta_k^{-1})+
     \mathrm{e}^{-\mathrm{i}y}(\partial_y-\mathrm{i}/2)(1+\Delta_k^{-1}).
     \end{align*}
 Using the fact that 
 \beno
   \mathrm{e}^{\mathrm{i}y}(\partial_y+\mathrm{i}/2) =(\partial_y-\mathrm{i}/2)\mathrm{e}^{\mathrm{i}y},\quad \mathrm{e}^{-\mathrm{i}y}(\partial_y-\mathrm{i}/2) =(\partial_y+\mathrm{i}/2)\mathrm{e}^{-\mathrm{i}y},
  \eeno
  we deduce 
     \begin{align*} \Lambda_1(-\Delta_k)^{s}\Lambda_1=&-\mathrm{e}^{\mathrm{i}y}(\partial_y+\mathrm{i}/2)(1+\Delta_k^{-1})(-\Delta_k)^{s}
     \mathrm{e}^{\mathrm{i}y}(\partial_y+\mathrm{i}/2)(1+\Delta_k^{-1})\\
     &-\mathrm{e}^{-\mathrm{i}y}(\partial_y-\mathrm{i}/2)(1+\Delta_k^{-1})(-\Delta_k)^{s}
     \mathrm{e}^{-\mathrm{i}y}(\partial_y-\mathrm{i}/2)(1+\Delta_k^{-1})\\
     &+\mathrm{e}^{\mathrm{i}y}(\partial_y+\mathrm{i}/2)(1+\Delta_k^{-1})(-\Delta_k)^{s}
     \mathrm{e}^{-\mathrm{i}y}(\partial_y-\mathrm{i}/2)(1+\Delta_k^{-1})\\
     &+\mathrm{e}^{-\mathrm{i}y}(\partial_y-\mathrm{i}/2)(1+\Delta_k^{-1})(-\Delta_k)^{s}
     \mathrm{e}^{\mathrm{i}y}(\partial_y+\mathrm{i}/2)(1+\Delta_k^{-1})\\
     =&-\mathrm{e}^{\mathrm{i}y}(\partial_y+\mathrm{i}/2)(1+\Delta_k^{-1})(-\Delta_k)^{s}
     (\partial_y-\mathrm{i}/2)\mathrm{e}^{\mathrm{i}y}(1+\Delta_k^{-1})\\
     &-\mathrm{e}^{-\mathrm{i}y}(\partial_y-\mathrm{i}/2)(1+\Delta_k^{-1})(-\Delta_k)^{s}
    (\partial_y+\mathrm{i}/2)\mathrm{e}^{-\mathrm{i}y}(1+\Delta_k^{-1})\\
     &+(\partial_y-\mathrm{i}/2)(1+\Delta_{k,-1}^{-1})(-\Delta_{k,-1})^{s}
     (\partial_y-\mathrm{i}/2)(1+\Delta_k^{-1})\\
     &+(\partial_y+\mathrm{i}/2)(1+\Delta_{k,1}^{-1})(-\Delta_{k,1})^{s}
     (\partial_y+\mathrm{i}/2)(1+\Delta_k^{-1})\\
     =&-\mathrm{e}^{\mathrm{i}y}(\partial_y^2+1/4)(1+\Delta_k^{-1})(-\Delta_k)^{s}\mathrm{e}^{\mathrm{i}y}(1+\Delta_k^{-1})\\
     &-\mathrm{e}^{-\mathrm{i}y}(\partial_y^2+1/4)(1+\Delta_k^{-1})(-\Delta_k)^{s}\mathrm{e}^{-\mathrm{i}y}(1+\Delta_k^{-1})\\
     &+(\partial_y-\mathrm{i}/2)^2(1+\Delta_{k,-1}^{-1})(-\Delta_{k,-1})^{s}(1+\Delta_k^{-1})\\
     &+(\partial_y+\mathrm{i}/2)^2(1+\Delta_{k,1}^{-1})(-\Delta_{k,1})^{s}(1+\Delta_k^{-1}).
  \end{align*}
  
  As $B=\cos y(1+\Delta_k^{-1})=(\mathrm{e}^{\mathrm{i}y}+\mathrm{e}^{-\mathrm{i}y})(1+\Delta_k^{-1})/2,$ we have
  \begin{align*}
     &4B^2(-\Delta_{k})^{1+s}\\=&\mathrm{e}^{\mathrm{i}y}(1+\Delta_k^{-1})\mathrm{e}^{\mathrm{i}y}(1+\Delta_k^{-1})(-\Delta_{k})^{1+s}
     +\mathrm{e}^{-\mathrm{i}y}(1+\Delta_k^{-1})\mathrm{e}^{-\mathrm{i}y}(1+\Delta_k^{-1})(-\Delta_{k})^{1+s}\\
     &+\mathrm{e}^{\mathrm{i}y}(1+\Delta_k^{-1})\mathrm{e}^{-\mathrm{i}y}(1+\Delta_k^{-1})(-\Delta_{k})^{1+s}+
     \mathrm{e}^{-\mathrm{i}y}(1+\Delta_k^{-1})\mathrm{e}^{\mathrm{i}y}(1+\Delta_k^{-1})(-\Delta_{k})^{1+s}\\
     =&\mathrm{e}^{\mathrm{i}y}(1+\Delta_k^{-1})(-\Delta_{k,-1})^{1+s}\mathrm{e}^{\mathrm{i}y}(1+\Delta_k^{-1})+
     \mathrm{e}^{-\mathrm{i}y}(1+\Delta_k^{-1})(-\Delta_{k,1})^{1+s}\mathrm{e}^{-\mathrm{i}y}(1+\Delta_k^{-1})\\
     &+(1+\Delta_{k,-1}^{-1})(1+\Delta_k^{-1})(-\Delta_{k})^{1+s}+(1+\Delta_{k,1}^{-1})(1+\Delta_k^{-1})(-\Delta_{k})^{1+s},
  \end{align*}
  and then (as $(-\Delta_{k})^{1+s}B^2$ is the dual operator of $B^2(-\Delta_{k})^{1+s}$ under the inner product $\langle, \rangle_*$)
   \begin{align*}
     &2B^2(-\Delta_{k})^{1+s}+2(-\Delta_{k})^{1+s}B^2\\=&\mathrm{e}^{\mathrm{i}y}(1+\Delta_k^{-1})
     \big((-\Delta_{k,-1})^{1+s}+(-\Delta_{k,1})^{1+s}\big)\mathrm{e}^{\mathrm{i}y}(1+\Delta_k^{-1})/2\\
     &+\mathrm{e}^{-\mathrm{i}y}(1+\Delta_k^{-1})\big((-\Delta_{k,-1})^{1+s}+(-\Delta_{k,1})^{1+s}\big)\mathrm{e}^{-\mathrm{i}y}(1+\Delta_k^{-1})/2\\
     &+\big(2+\Delta_{k,-1}^{-1}+\Delta_{k,1}^{-1}\big)(1+\Delta_k^{-1})(-\Delta_{k})^{1+s}.
  \end{align*}
  
  Summing up and using the definition of $H_0$, we conclude 
  \begin{align*}
      &\Lambda_1(-\Delta_k)^{s}\Lambda_1+2(1-B^2)(-\Delta_k)^{1+s}+2(-\Delta_k)^{1+s}(1-B^2)\\
      =&\Lambda_1(-\Delta_k)^{s}\Lambda_1+4(-\Delta_k)^{1+s}-(2B^2(-\Delta_{k})^{1+s}+2(-\Delta_{k})^{1+s}B^2)\\
      =&-\mathrm{e}^{\mathrm{i}y}(1+\Delta_k^{-1})H_0\mathrm{e}^{\mathrm{i}y}(1+\Delta_k^{-1})-
      \mathrm{e}^{-\mathrm{i}y}(1+\Delta_k^{-1})H_0\mathrm{e}^{-\mathrm{i}y}(1+\Delta_k^{-1})\\
     &+(\partial_y-\mathrm{i}/2)^2(1+\Delta_{k,-1}^{-1})(-\Delta_{k,-1})^{s}(1+\Delta_k^{-1})+
     (\partial_y+\mathrm{i}/2)^2(1+\Delta_{k,1}^{-1})(-\Delta_{k,1})^{s}(1+\Delta_k^{-1})\\
     &+\big(4-(2+\Delta_{k,-1}^{-1}+\Delta_{k,1}^{-1})(1+\Delta_k^{-1})\big)(-\Delta_{k})^{1+s}.
   \end{align*}
  As $ A=\sin y(1+\Delta_k^{-1})=(\mathrm{e}^{\mathrm{i}y}-\mathrm{e}^{-\mathrm{i}y})(1+\Delta_k^{-1})/(2\mathrm{i}),$ we have
  \begin{align*}
     4AH_0A=&-\mathrm{e}^{\mathrm{i}y}(1+\Delta_k^{-1})H_0\mathrm{e}^{\mathrm{i}y}(1+\Delta_k^{-1})
     -\mathrm{e}^{-\mathrm{i}y}(1+\Delta_k^{-1})H_0\mathrm{e}^{-\mathrm{i}y}(1+\Delta_k^{-1})\\
     &+\mathrm{e}^{\mathrm{i}y}(1+\Delta_k^{-1})H_0\mathrm{e}^{-\mathrm{i}y}(1+\Delta_k^{-1})
     +\mathrm{e}^{-\mathrm{i}y}(1+\Delta_k^{-1})H_0\mathrm{e}^{\mathrm{i}y}(1+\Delta_k^{-1})\\
     =&-\mathrm{e}^{\mathrm{i}y}(1+\Delta_k^{-1})H_0\mathrm{e}^{\mathrm{i}y}(1+\Delta_k^{-1})
     -\mathrm{e}^{-\mathrm{i}y}(1+\Delta_k^{-1})H_0\mathrm{e}^{-\mathrm{i}y}(1+\Delta_k^{-1})\\
     &+(1+\Delta_{k,-1}^{-1})H_{-1}(1+\Delta_k^{-1})
     +(1+\Delta_{k,1}^{-1})H_1(1+\Delta_k^{-1}),
     \end{align*}
     Thus, we find 
     \begin{align*}
      &\Lambda_1(-\Delta_k)^{s}\Lambda_1+2(1-B^2)(-\Delta_k)^{1+s}+2(-\Delta_k)^{1+s}(1-B^2)-4AH_0A\\
      &=(\partial_y-\mathrm{i}/2)^2(1+\Delta_{k,-1}^{-1})(-\Delta_{k,-1})^{s}(1+\Delta_k^{-1})+
     (\partial_y+\mathrm{i}/2)^2(1+\Delta_{k,1}^{-1})(-\Delta_{k,1})^{s}(1+\Delta_k^{-1})\\
     &\quad+\big(4-(2+\Delta_{k,-1}^{-1}+\Delta_{k,1}^{-1})(1+\Delta_k^{-1})\big)(-\Delta_{k})^{1+s}\\&\quad-(1+\Delta_{k,-1}^{-1})H_{-1}(1+\Delta_k^{-1})
     -(1+\Delta_{k,1}^{-1})H_1(1+\Delta_k^{-1})=H_*.
   \end{align*}
   
   This completes the proof of the lemma.
   \end{proof}
   
   Thanks to Lemma \ref{lem8},  Proposition \ref{prop:coer} follows from the following bounds
   \begin{align}\label{H0}
      &H_0\geq k^2(-\Delta_{k})^{s},\quad H_*\geq \delta(s)(-\Delta_{k})^{s}.
   \end{align}
   This motivates us to introduce the following notations:
   \beno
   a_n=n^2+k^2,\quad  b_n=(a_{n-1}^{1+s}+a_{n+1}^{1+s})/2+(-n^2+1/4)a_n^{s},
    \eeno
   and 
    \begin{align*}
      c_n=&-(n-1/2)^2(1-a_{n-1}^{-1})a_{n-1}^{s}(1-a_n^{-1})-
     (n+1/2)^2(1-a_{n+1}^{-1})a_{n+1}^{s}(1-a_n^{-1})\\
     &+\big(4-(2-a_{n-1}^{-1}-a_{n+1}^{-1})(1-a_n^{-1})\big)a_n^{1+s}\\&-(1-a_{n-1}^{-1})b_{n-1}(1-a_n^{-1})
     -(1-a_{n+1}^{-1})b_{n+1}(1-a_n^{-1}).
   \end{align*}   
   By taking the Fourier transform, \eqref{H0} is equivalent to the following result.
   \begin{lemma}\label{lem9}
  For $s\in[0,0.4]$, let $\delta(s)=4-\max((8s^2+7s+2)s/2,8s^2+8s-1)$. It holds that
  \beno
  b_n\geq {k^2}a_n^s,\quad c_n\ge \delta(s)a_n^s.
    \eeno
   \end{lemma}
   
   We need the following preliminary lemma. 
 
 \begin{lemma}\label{lem9s}
 Let $\lambda=a_n+1$, $s\in[0,0.4]$. It holds that 
 \beno
&& 
k^2a_n^s\leq b_n\leq \big(k^2+1/4+(s+1)(2s+1)\big)a_n^s, \\
&&\big(4-(2-a_{n-1}^{-1}-a_{n+1}^{-1})(1-a_n^{-1})\big)a_n\geq4+(\lambda-2)(2+2/\lambda),
 \eeno
 and 
  \begin{align*}
      &(n+1/2)^2(1-a_{n+1}^{-1})a_{n+1}^{s}+(1-a_{n+1}^{-1})b_{n+1}+(n-1/2)^2(1-a_{n-1}^{-1})a_{n-1}^{s}+(1-a_{n-1}^{-1})b_{n-1}\\ 
      &\leq (2\lambda+8s^2+6s+1)(\lambda-1)\lambda^{s-1}.
      \end{align*}
 \end{lemma}
 \begin{proof} It is easy to see that
 \begin{align*}
    (a_{n-1}^{1+s}+a_{n+1}^{1+s})/2&\geq \left(\dfrac{a_{n-1}+a_{n+1}}{2}\right)^{1+s} = (n^2+k^2+1)^{1+s}\geq(n^2+k^2+1)a_n^s\\
    &\geq (n^2-1/4)a_n^s+(k^2+5/4)a_n^s,
 \end{align*}
 which gives 
 \beno
 b_n\geq (k^2+5/4)a_n^s\geq k^2a_n^s.
 \eeno
 
 As $a_n=n^2+k^2$, $\lambda=a_n+1$, we have $ a_{n+1}=a_n+2n+1=\lambda+2n$, $a_{n-1}=a_n-2n+1=\lambda-2n$, $\lambda=a_n+1>n^2+1\geq|2n|$, and
\begin{align}\label{ans}
      a_{n-1}^{1+s}+a_{n+1}^{1+s}=&(\lambda-2n)^{1+s}+(\lambda+2n)^{1+s}=\lambda^{1+s}\big[(1-2n/\lambda)^{1+s}+(1+2n/\lambda)^{1+s}\big]\\
    \notag  =&2\lambda^{1+s}\sum_{j=0}^{\infty}{1+s\choose 2j}(2n/\lambda)^{2j}.
\end{align}
   Here  ${a\choose m}:=\prod_{j=0}^{m-1}\f{a-j}{j+1}$. Notice that ${1+s\choose 0}=1$ and ${1+s\choose 2j}=\frac{s(s+1)}{2}\prod_{l=1}^{2j-2}\frac{l-s}{l+2}\geq0$ for $j\in \Z_+$ and $n^2<a_n$, then 
    \begin{align*}
   &a_{n-1}^{1+s}+a_{n+1}^{1+s}\leq 2\lambda^{1+s}\sum_{j=0}^{\infty}{1+s\choose 2j}(2/\lambda)^{2j}a_n^j=(\lambda-2\sqrt{a_{n}})^{1+s}+(\lambda+2\sqrt{a_{n}})^{1+s}\\
      &=(a_n+1-2\sqrt{a_{n}})^{1+s}+(a_n+1+2\sqrt{a_{n}})^{1+s}=(\sqrt{a_{n}}-1)^{2+2s}+(\sqrt{a_{n}}+1)^{2+2s}\\
      &=a_n^{1+s}\big[(1-1/\sqrt{a_{n}})^{2+2s}+(1+1/\sqrt{a_{n}})^{2+2s}\big]=2a_n^{1+s}\sum_{j=0}^{\infty}{2+2s\choose 2j}a_n^{-j}.
   \end{align*}
 As ${2+2s\choose 0}=1$, ${2+2s\choose 2}=(s+1)(2s+1)$, ${2+2s\choose 2j}=-(s+1)(2s+1)\frac{2s}{3}\prod_{l=1}^{2j-3}\frac{l-2s}{l+3}\leq0$ (recall that $s\in[0,0.4]$) for $j\in \Z$, $j\geq 2$, we obtain
   \begin{align*}
       a_{n-1}^{1+s}+a_{n+1}^{1+s}\leq 2a_n^{1+s}\big(1+(s+1)(2s+1)a_n^{-1}\big),
   \end{align*}
   which gives
   \begin{align*}
      b_n&=(a_{n-1}^{1+s}+a_{n+1}^{1+s})/2+(-n^2+1/4)a_n^{s}\leq a_n^{1+s}\big(1+(s+1)(2s+1)a_n^{-1}\big)+(-n^2+1/4)a_n^{s}\\
      &=(a_n+(s+1)(2s+1)-n^2+1/4)a_n^{s}=(k^2+1/4+(s+1)(2s+1))a_n^{s}.
   \end{align*}
   This shows  
   \beno
   b_n\le (k^2+1/4+(s+1)(2s+1))a_n^{s},
      \eeno
 from which(with $n$ replaced by $n+1$), we infer that
   \begin{align*}
     &(n+1/2)^2(1-a_{n+1}^{-1})a_{n+1}^{s}+(1-a_{n+1}^{-1})b_{n+1}
     \\ \leq &(n+1/2)^2(1-a_{n+1}^{-1})a_{n+1}^{s}+(1-a_{n+1}^{-1})\big(k^2+1/4+(s+1)(2s+1)\big)a_{n+1}^{s}\\ =&(1-a_{n+1}^{-1})\big(n^2+n+k^2+1/2+(s+1)(2s+1)\big)a_{n+1}^{s}\\ =&(1-a_{n+1}^{-1})\big((a_n+a_{n+1})/2+(s+1)(2s+1)\big)a_{n+1}^{s}\\ =&(1-a_{n+1}^{-1})\big(\lambda-1+a_{n+1}+4s^2+6s+2\big)a_{n+1}^{s}/2\\ 
      =&\big(a_{n+1}^{s+1}+(\lambda+4s^2+6s)a_{n+1}^{s}-(\lambda+4s^2+6s+1)a_{n+1}^{s-1}\big)/2.
   \end{align*}
   Here we used $\lambda=a_n+1$. Similarly, we have
   \begin{align*}
      &(n-1/2)^2(1-a_{n-1}^{-1})a_{n-1}^{s}+(1-a_{n-1}^{-1})b_{n-1}\\ 
      \leq& \big(a_{n-1}^{s+1}+(\lambda+4s^2+6s)a_{n-1}^{s}-(\lambda+4s^2+6s+1)a_{n-1}^{s-1}\big)/2.
   \end{align*}
   Thus, we arrive at
    \begin{align*}
     &(n+1/2)^2(1-a_{n+1}^{-1})a_{n+1}^{s}+(1-a_{n+1}^{-1})b_{n+1}+(n-1/2)^2(1-a_{n-1}^{-1})a_{n-1}^{s}+(1-a_{n-1}^{-1})b_{n-1}\\ 
      &\leq \big(a_{n+1}^{s+1}+a_{n-1}^{s+1}+(\lambda+4s^2+6s)(a_{n+1}^{s}+a_{n-1}^{s})-(\lambda+4s^2+6s+1)(a_{n+1}^{s-1}+a_{n-1}^{s-1})\big)/2\\
      &=\lambda^{s-1}\sum_{j=0}^{\infty}p_j\left|\frac{2n}{\lambda}\right|^{2j},
   \end{align*}
   where
   \beno
    p_j=\lambda^{2}{s+1\choose 2j}+\lambda(\lambda+4s^2+6s){s\choose 2j}-(\lambda+4s^2+6s+1){s-1\choose 2j}.
    \eeno
Here we used \eqref{ans} and the following similar formula
   \begin{align*}
      &a_{n-1}^{s}+a_{n+1}^{s}=2\lambda^{s}\sum_{j=0}^{\infty}{s\choose 2j}(2n/\lambda)^{2j},\quad a_{n-1}^{s-1}+a_{n+1}^{s-1}=2\lambda^{s-1}\sum_{j=0}^{\infty}{s-1\choose 2j}(2n/\lambda)^{2j}.
   \end{align*}
 Notice that  \begin{align*}
     p_0&=\lambda^{2}+\lambda(\lambda+4s^2+6s)-(\lambda+4s^2+6s+1)=(2\lambda+4s^2+6s+1)(\lambda-1),\\
     p_1&=\lambda^{2}s(s+1)/2+\lambda(\lambda+4s^2+6s)s(s-1)/2-(\lambda+4s^2+6s+1)(s-1)(s-2)/2\\
     &=\lambda^{2}s^2-\lambda(4s^2+6s)s(1-s)/2-(\lambda+4s^2+6s+1)(1-s)(2-s)/2\leq\lambda^{2}s^2.
   \end{align*} 
   For $j\in\Z$, $j\geq 2$, we have ${s\choose 2j}=-s\prod_{l=1}^{2j-1}\frac{l-s}{l+1}\leq0$, 
   ${s-1\choose 2j}=\prod_{l=1}^{2j}\frac{l-s}{l}>0$, and
   \beno
   {s+1\choose 2j}\!+\!{s\choose 2j}\!=\!\frac{(s+1)s}{2}\!\prod_{l=1}^{2j-2}\!\frac{l-s}{l+2}+\frac{s(s-2j+1)}{2}\!\prod_{l=1}^{2j-2}\!\frac{l-s}{l+2}\!=\!
   -s(j-1-s)\!\prod_{l=1}^{2j-2}\!\frac{l-s}{l+2}\\ \leq0,
   \eeno
    which leads to
   \begin{align*}
     p_j&=\lambda^{2}\left\{{s+1\choose 2j}+{s\choose 2j}\right\}+\lambda(4s^2+6s){s\choose 2j}-(\lambda+4s^2+6s+1){s-1\choose 2j}<0.
   \end{align*}
   Thus (using $n^2<a_n=\lambda-1$),
   \begin{align*}
     &(n+1/2)^2(1-a_{n+1}^{-1})a_{n+1}^{s}+(1-a_{n+1}^{-1})b_{n+1}+(n-1/2)^2(1-a_{n-1}^{-1})a_{n-1}^{s}+(1-a_{n-1}^{-1})b_{n-1}\\ 
      &\leq \lambda^{s-1}\sum_{j=0}^{\infty}p_j\left|\frac{2n}{\lambda}\right|^{2j}\leq \lambda^{s-1}(p_0+p_1(2n/\lambda)^{2})\leq \lambda^{s-1}(p_0+\lambda^{2}s^2(2n/\lambda)^{2})= \lambda^{s-1}(p_0+4n^2s^2)\\ &\leq \lambda^{s-1}\big(p_0+4(\lambda-1)s^2\big)=\lambda^{s-1}\big((2\lambda+4s^2+6s+1)(\lambda-1)+4(\lambda-1)s^2\big)\\ &=
      \lambda^{s-1}(2\lambda+8s^2+6s+1)(\lambda-1).
   \end{align*} 
   
   Finally, as $\lambda=a_n+1$, $ a_{n+1}=\lambda+2n$, $a_{n-1}=\lambda-2n$, we have
   \begin{align*}
      &\big(4-(2-a_{n-1}^{-1}-a_{n+1}^{-1})(1-a_n^{-1})\big)a_n=4a_n-(2-a_{n-1}^{-1}-a_{n+1}^{-1})(a_n{-1})\\
      &=2(a_n+1)+(a_{n-1}^{-1}+a_{n+1}^{-1})(a_n{-1})=2\lambda+\big((\lambda-2n)^{-1}+(\lambda+2n)^{-1}\big)(\lambda-2)\\
      &=2\lambda+2\lambda(\lambda-2)/(\lambda^2-4n^2)\leq2\lambda+2\lambda(\lambda-2)/\lambda^2=4+2(\lambda-2)+2(\lambda-2)/\lambda\\
      &=4+(\lambda-2)(2+2/\lambda).
   \end{align*}
   
   This completes the proof of the lemma.
 \end{proof}
 
 Now we prove Lemma \ref{lem9}.
 
 \begin{proof}
 By Lemma \ref{lem9s}, we have $b_n\geq k^2a_n^s$, and (here $\lambda=a_n+1$, $a_n=\lambda-1$)
 \begin{align*}
      c_n/a_n^s=&\big(4-(2-a_{n-1}^{-1}-a_{n+1}^{-1})(1-a_n^{-1})\big)a_n-
      \big[(n+1/2)^2(1-a_{n+1}^{-1})a_{n+1}^{s}+(1-a_{n+1}^{-1})b_{n+1}\\
      &+(n-1/2)^2(1-a_{n-1}^{-1})a_{n-1}^{s}+(1-a_{n-1}^{-1})b_{n-1}\big](1-a_n^{-1})/a_n^s\\
      \geq&4+(\lambda-2)(2+2/\lambda)-(2\lambda+8s^2+6s+1)(\lambda-1)\lambda^{s-1}(1-a_n^{-1})/a_n^s\\
      =&4+(\lambda-2)(2+2/\lambda)-(2\lambda+8s^2+6s+1)(\lambda-1)\lambda^{s-1}(1-(\lambda-1)^{-1})/(\lambda-1)^s\\
      =&4+(\lambda-2)(2+2/\lambda)-(2\lambda+8s^2+6s+1)(\lambda-2)\lambda^{s-1}/(\lambda-1)^s\\
      =&4+(\lambda-2)\big(2+2/\lambda-(2\lambda+8s^2+6s+1)(1-1/\lambda)^{-s}/\lambda\big).
   \end{align*}
   As $\lambda=a_n+1\geq k^2+1>2$, we have (using $(-1)^{j}{-s\choose j}= {s+j-1\choose j}, j\geq1$)
   \begin{align*}
     &(1-1/\lambda)^{-s}=\sum_{j=0}^{\infty}{-s\choose j}(-1/\lambda)^{j}=1+\sum_{j=1}^{\infty}{s+j-1\choose j}/\lambda^{j},
   \end{align*} 
   and then
   \begin{align*}
      &2+2/\lambda-(2\lambda+8s^2+6s+1)(1-1/\lambda)^{-s}/\lambda \\ =&2+2/\lambda-2(1-1/\lambda)^{-s} -(8s^2+6s+1)(1-1/\lambda)^{-s}/\lambda \\
      =&2+2/\lambda- 2\Big(1+\sum_{j=1}^{\infty}{s+j-1\choose j}/\lambda^{j}\Big)- (8s^2+6s+1)\Big(1/\lambda+\sum_{j=1}^{\infty}{s+j-1\choose j}/\lambda^{j+1}\Big)\\
      =&-\sum_{j=0}^{\infty}q_j/\lambda^{j+1},
   \end{align*}
 where
 \beno
  &&q_0=-2+ 2{s\choose 1} +(8s^2+6s+1)=8s^2+8s-1,\\
  &&q_j=(8s^2+6s+1){s+j-1\choose j}+2{s+j\choose j+1},\quad \forall\  j\in\Z_+.  
   \eeno

 As $0\leq{s+j\choose j+1}={s+j-1\choose j}\frac{s+j}{j+1}\leq{s+j-1\choose j}$, for $j\in\Z_+$, we have $q_{j+1}\leq q_j$ and 
 $q_{j}\leq q_1\leq 2^jq_1/2$ for all $j\in\Z_+$. Let $q_*:=\max(q_1/2,q_0)$, then $q_{j}\leq  2^jq_*$ for all $j\in\Z$, $j\geq0$, thus,
 \begin{align*}
     &2+2/\lambda-(2\lambda+8s^2+6s+1)(1-1/\lambda)^{-s}/\lambda\geq-\sum_{j=0}^{\infty}q_j/\lambda^{j+1}\geq-\sum_{j=0}^{\infty}2^jq_*/\lambda^{j+1}
     =\frac{-q_*}{\lambda-2}.
   \end{align*}
   Here we used $ 1/(\lambda-2)=\sum_{j=0}^{\infty}2^j/\lambda^{j+1}$.
   Notice that $q_1=(8s^2+6s+1)s+(s+1)s=(8s^2+7s+2)s$, then
   \begin{align*}
      c_n/a_n^s\geq&
4+(\lambda-2)(2+2/\lambda-(2\lambda+8s^2+6s+1)(1-1/\lambda)^{-s}/\lambda)\\
\geq&4+(\lambda-2)(-q_*/(\lambda-2))=4-q_*=4-\max(q_1/2,q_0)\\=&4-\max((8s^2+7s+2)s/2,8s^2+8s-1)=\delta(s).
   \end{align*} 
   This shows $c_n\geq\delta(s)a_n^s $. 
 \end{proof}

\section{Space-time estimates for the linearized NS}\label{sec:spacetime}

Recall the space-time norm $X_I$ defined by 
\begin{align*}
 \|f\|_{X_I}=\|f\|_{X_{I,\nu}}=&\|f\|_{L^\infty(I;L^2)}+\nu^{\f12}\|\nabla f\|_{L^2(I;L^2)}+ \||D_x|^{\f12}\nabla \Delta^{-1}f\|_{L^2(I;L^2)}
      \\ \notag&+ \||V'|^{\f12}\partial_x\nabla \Delta^{-1}f\|_{L^2(I;L^2)}.
 \end{align*}
This norm enables us to establish the following control. 

\begin{lemma}\label{lem:norm-X}
It holds that
   \begin{align*}
      & \nu^{\f16}\|\partial_x\nabla\Delta^{-1}f\|_{L^2(I;L^2)}+
      \nu^{\f16}\|\nabla\Delta^{-1}f_{\neq}\|_{L^2(I;L^{\infty})}+\nu^{\f14}\|f_{\neq}\|_{L^2(I;L^2)}\leq C\|f\|_{X_I}.
   \end{align*}
 \end{lemma}
 
\begin{proof}First, we have
  \begin{align*} 
  \|\partial_x\nabla\Delta^{-1}f\|_{L^2}\leq C\||D_x|^{\f12}\nabla\Delta^{-1}f\|^{\f23}_{L^2}\|\partial_x^{2} \nabla\Delta^{-1}f\|_{L^2}^{\f13}\leq C\||D_x|^{\f12}\nabla\Delta^{-1}f\|^{\f23}_{L^2}\|\nabla f\|_{L^2}^{\f13},
  \end{align*}
  which gives
  \begin{align*} 
  \|\partial_x\nabla\Delta^{-1}f\|_{L^2(I;L^2)}\leq C\||D_x|^{\f12}\nabla\Delta^{-1}f\|^{\f23}_{L^2(I;L^2)}\|\nabla f\|_{L^2(I;L^2)}^{\f13}\leq C \nu^{-\f16}\|f\|_{X_I}.
  \end{align*}

  By Lemma \ref{lem:GN-1} (taking $\delta=\nu^{\f13}$), we get
\begin{align*}   \|\nabla\Delta^{-1}f_{\neq}\|_{L^\infty}\leq& C\nu^{-\f{1}{24}}\||D_x|^{\f12}\nabla\Delta^{-1}f_{\neq}\|_{L^2}^{\f34}\|\nabla\Delta^{-1}f_{\neq}\|_{H^2}^{\f14} +C\nu^{\f13}\|\nabla\Delta^{-1}f_{\neq}\|_{H^2}\\
     \leq& C \nu^{-\f{1}{24}}\||D_x|^{\f12}\nabla \Delta^{-1}f\|_{L^2}^{\f34}\|\nabla f\|_{L^2}^{\f14} +C\nu^{\f13}\|\nabla f\|_{L^2},
     \end{align*}
  which gives
  \begin{align*}    
 \|\nabla\Delta^{-1}f_{\neq}\|_{L^2(I;L^{\infty})}
     \leq& C \nu^{-\f{1}{24}}\||D_x|^{\f12}\nabla \Delta^{-1}f\|_{L^2(I;L^2)}^{\f34}\|\nabla f\|_{L^2(I;L^2)}^{\f14} +C\nu^{\f13}\|\nabla f\|_{L^2(I;L^2)}\\
     \leq& C \nu^{-\f{1}{6}}\|f\|_{X_I}.
  \end{align*}   
  By the interpolation, we get
  \begin{align*} 
  &\|f_{\neq}\|_{L^2(I;L^2)}\leq C\||D_x|^{\f12}\nabla\Delta^{-1}f\|^{\f12}_{L^2(I;L^2)}\|\nabla f\|_{L^2(I;L^2)}^{\f12}\leq C \nu^{-\f14}\|f\|_{X_I}.
  \end{align*}

  This completes the proof.
  \end{proof}

\subsection{Proof of Proposition \ref{prop:LNS-sp1}}
This part is devoted to the space-time estimate for the following linearized Navier-Stokes system
\begin{align}\label{eq:LNS-T1}
  \partial_tf-\nu \Delta f+\mathrm{e}^{-\nu t}V\partial_x(f+\Delta^{-1}f)=\partial_y g_1+\partial_xg_2+g_3+g_4.
\end{align}

We first prove the following energy estimate. {We will use the inner product  $\langle \ ,\ \rangle_{\star}$  defined by (see also \eqref{def:norm-*})}
  \begin{align}\label{8.1}
     & \langle f,g\rangle_\star=\langle P_{\neq} f,P_{\neq}g+\Delta^{-1}P_{\neq}g\rangle +\langle P_0f,P_0g\rangle= \langle f,g+\Delta^{-1}P_{\neq}g\rangle,
  \end{align}

 \begin{lemma}\label{lem2}
Let $f$ solve \eqref{eq:partf-g-1} with $P_0g_3=0$. Then for $I= [t_1,t_2]\subset[0,1/\nu]$, it holds that 
\begin{align*}
&\|f\|_{L^{\infty}({I};L^2)}+\nu^{\f12}\|\nabla f\|_{L^2({I};L^2)}\lesssim\|f(t_1)\|_{L^2}+
\nu^{-\f12}\|(g_{1},g_{2})\|_{L^2({I};L^2)}\\&\qquad+
\||D_x|^{-\f12}\nabla{g}_{3}\|_{L^2({I};L^2)}^{1/2}\||D_x|^{\f12}\nabla\Delta^{-1}f\|_{L^2({I};L^2)}^{1/2}+
\|g_{4}\|_{L^1({I};L^2)}.
\end{align*}
 \end{lemma}
 \begin{proof}
 Taking the inner product $\langle \ ,\ \rangle_{\star}$ with $f$, we obtain(here $\langle \ ,\ \rangle_{\star}$ is defined in
 \eqref{8.1})
 \begin{align*}
&\f12\partial_t\|f\|_{\star}^2+{\nu}\|\nabla f\|_{\star}^2\leq |\langle g,f\rangle_\star|\\&=
|-\langle g_{1},\partial_yf\rangle_\star-\langle g_{2},\partial_xf\rangle_\star-
\langle|D_x|^{-\f12}\nabla{g}_{3},|D_x|^{\f12}\nabla\Delta^{-1}f\rangle_\star+
\langle{g}_{4},f\rangle_\star|\\ &\leq \|(g_{1},g_{2})\|_{\star}\|\nabla f\|_{\star}+
\||D_x|^{-\f12}\nabla{g}_{3}\|_{\star}\||D_x|^{\f12}\nabla\Delta^{-1}f\|_{\star}+\|g_{4}\|_{\star}\|f\|_{\star}.
\end{align*}
Integrating in $t$ and using the fact that $\|h\|_{\star}\sim \|h\|_{L^2} $, we have
\begin{align*}
&\|f\|_{L^{\infty}({I};L^2)}^2+{\nu}\|\nabla f\|_{L^2({I};L^2)}^2\lesssim\|f(t_1)\|_{L^2}^2+
\|(g_{1},g_{2})\|_{L^2({I};L^2)}\|\nabla f\|_{L^2({I};L^2)}\\&\quad+
\||D_x|^{-\f12}\nabla{g}_{3}\|_{L^2({I};L^2)}\||D_x|^{\f12}\nabla\Delta^{-1}f\|_{L^2({I};L^2)}+
\|g_{4}\|_{L^1({I};L^2)}\|f\|_{L^{\infty}({I};L^2)}.
\end{align*}
Then the result follows from Young's inequality.
\end{proof}

It remains to estimate $\||D_x|^{\f12}\nabla\Delta^{-1}f\|_{L^2({I};L^2)}+ \||V'|^{\f12}\partial_x\nabla \Delta^{-1}f\|_{L^2(I;L^2)}$. 
We first consider the case when $\mathrm{e}^{-\nu t}$ is replaced by $1$, i.e.,
\begin{align}\label{7.2}
   & \partial_tf-\nu \Delta f+V\partial_x(f+\Delta^{-1}f)=g=\partial_y g_1+\partial_xg_2+g_3+g_4.
\end{align}

\begin{Proposition}\label{prop:space-time-3}
Let $f$ solve \eqref{7.2} with $P_0g_3=0$. 
There exists $\epsilon_0>0$ such that if $V=V(y)$, $\|V-\cos y\|_{H^4}\leq \epsilon_0$, $\nu\in(0,1)$, then for $I= [t_1,t_2]\subset[0,1/\nu]$, it holds that 
\begin{align*}
&\||D_x|^{\f12}\nabla\Delta^{-1}f\|_{L^2({I};L^2)}+ \||V'|^{\f12}\partial_x\nabla \Delta^{-1}f\|_{L^2(I;L^2)}\\ &\leq C\big(\|f(t_1)\|_{L^2}+\nu^{-\f12}\|(g_1,g_2)\|_{L^2(I;L^2)}+\||D_x|^{-\f12}\nabla g_3\|_{L^2(I;L^2)}+\|g_4\|_{L^1(I;L^2)}\big).
\end{align*}
\end{Proposition}

Now we prove Proposition \ref{prop:LNS-sp1} by using Proposition \ref{prop:space-time-3} and Lemma \ref{lem2}.

\begin{proof}[Proof of Proposition \ref{prop:LNS-sp1}]
Let $t_*=\int_{0}^{t}\mathrm{e}^{-\nu s}\mathrm{d}s= (1-\mathrm{e}^{-\nu t})/\nu,\quad t\geq0$, and
\begin{align*}
      &f(t,y)=f_{*}(t_*,y), \quad g(t,y)=\mathrm{e}^{-\nu t}g_{*}(t_*,y), \quad g_j(t,y)=\mathrm{e}^{-\nu t}g_{j*}(t_*,y),\quad j=1,2,3,4.
   \end{align*}Then we have (as $\mathrm{e}^{\nu t}=1/(1-\nu t_*)$)
   \begin{align*}
      & \partial_tf_{*}-\f{\nu}{1-\nu t} \Delta f_{*}+V\partial_x(1+\Delta^{-1})f_{*}=g_*=\partial_y g_{1*}+\partial_xg_{2*}+g_{3*}+g_{4*},\\
     &\partial_tf_{*}-{\nu}\Delta f_{*}+V\partial_x(1+\Delta^{-1})f_{*}=\partial_y\tilde{g}_{1}+\partial_x\tilde{g}_{2}+g_{3*}+g_{4*}.
   \end{align*}
   Here $\tilde{g}_{1}={g}_{1*}+\f{\nu^2t}{1-\nu t}\partial_yf_{*}$, $\tilde{g}_{2}={g}_{2*}+\f{\nu^2t}{1-\nu t}\partial_xf_{*}$. Let 
   \beno
   t_{1*}=(1-\mathrm{e}^{-\nu t_1})/\nu,\quad t_{2*}=(1-\mathrm{e}^{-\nu t_2})/\nu,\quad \tilde{I}=[t_{1*},t_{2*}].
   \eeno
  As $I=[t_1,t_2]\subset[0,1/\nu]$, we have $\tilde{I}\subset[0,(1-\mathrm{e}^{-1})/\nu]\subset[0,1/\nu]$. 
   As $\mathrm{e}^{-\nu t}\sim1 $ for $t\in I$, and $t\in I\Leftrightarrow t_*\in \tilde{I}$, we find 
   \begin{align*}
&\|f(t_1)\|_{L^2}=\|f_*(t_{1*})\|_{L^2},\quad \||D_x|^{\f12}\nabla\Delta^{-1}f\|_{L^2({I};L^2)}\sim \||D_x|^{\f12}\nabla\Delta^{-1}f_*\|_{L^2(\tilde{I};L^2)},\\
& \||V'|^{\f12}\partial_x\nabla \Delta^{-1}f\|_{L^2(I;L^2)}\sim\||V'|^{\f12}\partial_x\nabla \Delta^{-1}f_*\|_{L^2(\tilde{I};L^2)},\\
& \|\nabla f\|_{L^2({I};L^2)}\sim \|\nabla f_*\|_{L^2(\tilde{I};L^2)},\quad \|(g_1,g_2)\|_{L^2(I;L^2)}\sim \|(g_{1*},g_{2*})\|_{L^2(\tilde{I};L^2)},\\
&\||D_x|^{-\f12}\nabla g_3\|_{L^2(I;L^2)}\sim \||D_x|^{-\f12}\nabla {g}_{3*}\|_{L^2(\tilde{I};L^2)},\quad 
\|g_4\|_{L^1(I;L^2)}\sim \|g_{4*}\|_{L^1(\tilde{I};L^2)}.
\end{align*}
By Proposition \ref{prop:space-time-3}, Lemma \ref{lem2} and $ \f{\nu^2t}{1-\nu t}\leq C\nu$ for $ t\in \tilde I$, we obtain
\begin{align*}
&\||D_x|^{\f12}\nabla\Delta^{-1}f\|_{L^2({I};L^2)}+\||V'|^{\f12}\partial_x\nabla \Delta^{-1}f\|_{L^2(I;L^2)}\\
&\leq C\big(\||D_x|^{\f12}\nabla\Delta^{-1}f_*\|_{L^2(\tilde{I};L^2)}+\||V'|^{\f12}\partial_x\nabla \Delta^{-1}f_*\|_{L^2(\tilde{I};L^2)}\big)\\
&\leq C\big(\|f_*(t_{1*})\|_{L^2}+\nu^{-\f12}\|(\tilde{g}_{1},\tilde{g}_{2})\|_{L^2(\tilde{I};L^2)}+ \||D_x|^{-\f12}\nabla{g}_{3*}\|_{L^2(\tilde{I};L^2)}+\|g_{4*}\|_{L^1(\tilde{I};L^2)}\big)\\
&\leq C\big(\|f_*(t_{1*})\|_{L^2}+\nu^{-\f12}\|(g_{1*},g_{2*})\|_{L^2(\tilde{I};L^2)}+ \||D_x|^{-\f12}\nabla{g}_{3*}\|_{L^2(\tilde{I};L^2)}+\|g_{4*}\|_{L^1(\tilde{I};L^2)}\\&\qquad+
\nu^{\f12}\|\nabla f_*\|_{L^2(\tilde{I};L^2)}\big)\\
&\leq C\big(\|f(t_{1})\|_{L^2}+\nu^{-\f12}\|(g_{1},g_{2})\|_{L^2({I};L^2)}+ \||D_x|^{-\f12}\nabla{g}_{3}\|_{L^2({I};L^2)}+\|g_{4}\|_{L^1({I};L^2)}\\&\qquad+
\nu^{\f12}\|\nabla f\|_{L^2({I};L^2)}\big)\\
&\leq C\big(\|f(t_{1})\|_{L^2}+\nu^{-\f12}\|(g_{1},g_{2})\|_{L^2({I};L^2)}+ \||D_x|^{-\f12}\nabla{g}_{3}\|_{L^2({I};L^2)}+\|g_{4}\|_{L^1({I};L^2)}\\&\qquad+
\||D_x|^{-\f12}\nabla{g}_{3}\|_{L^2({I};L^2)}^{1/2}\||D_x|^{\f12}\nabla\Delta^{-1}f\|_{L^2({I};L^2)}^{1/2}\big).
\end{align*}
By Lemma \ref{lem2} again, we have 
\begin{align*}
&\|f\|_{L^{\infty}({I};L^2)}+\nu^{\f12}\|\nabla f\|_{L^2({I};L^2)}+\||D_x|^{\f12}\nabla\Delta^{-1}f\|_{L^2({I};L^2)}+\||V'|^{\f12}\partial_x\nabla \Delta^{-1}f\|_{L^2(I;L^2)}\\
&\leq C\big(\|f(t_{1})\|_{L^2}+\nu^{-\f12}\|(g_{1},g_{2})\|_{L^2({I};L^2)}+ \||D_x|^{-\f12}\nabla{g}_{3}\|_{L^2({I};L^2)}+\|g_{4}\|_{L^1({I};L^2)}\\&\qquad+
\||D_x|^{-\f12}\nabla{g}_{3}\|_{L^2({I};L^2)}^{1/2}\||D_x|^{\f12}\nabla\Delta^{-1}f\|_{L^2({I};L^2)}^{1/2}\big),
\end{align*}
which implies our result by Young's inequality.
\end{proof}

To prove Proposition \ref{prop:space-time-3}, we first consider the case where $g_4=0$ and $f(t_1)=0$.

\begin{lemma}\label{prop:space-time-4}
Let $f$ solve \eqref{7.2} with $P_0g_3=0$, $g_4=0$ and $f(t_1)=0$.  Then for $I= [t_1,+\infty)$, it holds that 
\begin{align*}
&\||D_x|^{\f12}\nabla\Delta^{-1}f\|_{L^2({I};L^2)}+ \||V'|^{\f12}\partial_x\nabla \Delta^{-1}f\|_{L^2(I;L^2)}\\ 
&\leq C\big(\nu^{-\f12}\|(g_1,g_2)\|_{L^2(I;L^2)}+\||D_x|^{-\f12}\nabla g_3\|_{L^2(I;L^2)}\big).
\end{align*}
\end{lemma}

Then we consider the case where $g=0$.

\begin{lemma}\label{prop:space-time-g0}
Let $f$ solve \eqref{7.2}  with $g=0$. Then for $I= [0,+\infty)$, it holds that
\begin{align*}
\||D_x|^{\f12}\nabla\Delta^{-1}f\|_{L^2({I};L^2)}+ \||V'|^{\f12}\partial_x\nabla \Delta^{-1}f\|_{L^2(I;L^2)}\leq C\|f(0)\|_{L^2}.
\end{align*}
\end{lemma}

Now we prove Proposition \ref{prop:space-time-3} by using Proposition \ref{prop:space-time-4} and Proposition \ref{prop:space-time-g0}.

\begin{proof}[Proof of Proposition \ref{prop:space-time-3}]
  We decompose $f=f_1+f_2$ (for $t\in I=[t_1,t_2]$), where $f_1,\ f_{2}$ solve
  \begin{equation*}\left\{
    \begin{aligned}
     &\partial_tf_1-\nu \Delta f_1+V\partial_x(f_1+\Delta^{-1}f_1)=g_4\mathbf{1}_{t\in [t_1,t_2]},\quad f_{1}|_{t=t_1}=f(t_1),\\
     &\partial_tf_2-\nu \Delta f_2+V\partial_x(f_2+\Delta^{-1}f_2)=(\partial_y g_1+\partial_xg_2+g_3)\mathbf{1}_{t\in [t_1,t_2]},\quad f_{2}|_{t=t_1}=0. 
  \end{aligned}\right.
  \end{equation*}
  Then we get by Lemma \ref{prop:space-time-4} that
  \begin{align*}
     &\||D_x|^{\f12}\nabla\Delta^{-1}f_2\|_{L^2(I;L^2)}+ \||V'|^{\f12}\partial_x\nabla \Delta^{-1}f_2\|_{L^2(I;L^2)}\\ &\leq \||D_x|^{\f12}\nabla\Delta^{-1}f_2\|_{L^2({[t_1,+\infty)};L^2)}+ \||V'|^{\f12}\partial_x\nabla \Delta^{-1}f_2\|_{L^2([t_1,+\infty);L^2)}\\ &\leq C\big(\nu^{-\f12}\|(g_1,g_2)\|_{L^2(I;L^2)}+\||D_x|^{-\f12}\nabla g_3\|_{L^2(I;L^2)}\big).
  \end{align*} 
  Let $ \mathcal{L}=\nu \Delta -V\partial_x(1+\Delta^{-1})$. By Duhamel's formula, we get
\begin{align*}
     &f_1(t)=\mathrm{e}^{(t-t_1)\mathcal{L}}f(t_1)+\int_{t_1}^t\mathrm{e}^{(t-s)\mathcal{L}}g_4(s)\mathrm{d}s,\quad t\in I=[t_1,t_2].
  \end{align*} By Proposition \ref{prop:space-time-g0}, we have
\begin{align*}
\||D_x|^{\f12}\nabla\Delta^{-1}\mathrm{e}^{t\mathcal{L}}h\|_{L^2({[0,+\infty)};L^2)}+ \||V'|^{\f12}\partial_x\nabla \Delta^{-1}\mathrm{e}^{t\mathcal{L}}h\|_{L^2([0,+\infty);L^2)}\leq C\|h\|_{L^2}.
\end{align*}
which along with Minkowski's inequality gives
\begin{align*}
     &\||D_x|^{\f12}\nabla\Delta^{-1}f_1\|_{L^2(I;L^2)}+ \||V'|^{\f12}\partial_x\nabla \Delta^{-1}f_1\|_{L^2(I;L^2)}\\ &\leq \||D_x|^{\f12}\nabla\Delta^{-1}\mathrm{e}^{(t-t_1)\mathcal{L}}f(t_1)\|_{L^2({[t_1,t_2]};L^2)}+ \||V'|^{\f12}\partial_x\nabla \Delta^{-1}\mathrm{e}^{(t-t_1)\mathcal{L}}f(t_1)\|_{L^2([t_1,t_2];L^2)}\\
     &\quad+\int_{t_1}^{t_2}\big(\||D_x|^{\f12}\nabla\Delta^{-1}\mathrm{e}^{(t-s)\mathcal{L}}g_4(s)\|_{L^2([s,t_2];L^2)}+ \||V'|^{\f12}\partial_x\nabla \Delta^{-1}\mathrm{e}^{(t-s)\mathcal{L}}g_4(s)\|_{L^2([s,t_2];L^2)}\big)\mathrm{d}s\\ 
     &\leq C\|f(t_1)\|_{L^2}+C\int_{t_1}^{t_2}\|g_4(s)\|_{L^2}\mathrm{d}s=C\big(\|f(t_1)\|_{L^2}+\|g_4\|_{L^1(I,L^2)}\big).
     \end{align*}
     
      Summing up the estimates of $f_1$ and $f_2$, we conclude Proposition  \ref{prop:space-time-3}.
\end{proof}

Now we prove Lemma \ref{prop:space-time-4} and Lemma \ref{prop:space-time-g0} using Proposition \ref{prop:res-LNS}.
\begin{proof}[Proof of Lemma \ref{prop:space-time-4}]
Let 
\begin{align*}
   & w(\lambda,x,y)=\int_{t_1}^{+\infty}f_{\neq}(t,x,y)\mathrm{e}^{\mathrm{i}\lambda t}\mathrm{d}t,\quad F(\lambda,x,y)=\int_{t_1}^{+\infty}g_{\neq}(t,x,y)\mathrm{e}^{\mathrm{i}\lambda t}\mathrm{d}t,\\
   &F_j(\lambda,x,y)=\int_{t_1}^{+\infty}P_{\neq}g_{j}(t,x,y)\mathrm{e}^{\mathrm{i}\lambda t}\mathrm{d}t,\quad j=1,2,3.
\end{align*}
Then we have
\begin{align*}
   &-\nu\Delta w+V\partial_x(w+\Delta^{-1} w)-\mathrm{i}\lambda w=F=\partial_y F_1+\partial_xF_2+F_3.
\end{align*}
Using Plancherel's formula, we know that (here $I=[t_1,+\infty)$)
\begin{align*}
   & \||D_x|^{\f12}\nabla\Delta^{-1}f\|_{L^2(I;L^2)}^2+ \||V'|^{\f12}\partial_x\nabla \Delta^{-1}f\|_{L^2(I;L^2)}^2\\ &\sim
   \int_{\R}(\||D_x|^{\f12}\nabla\Delta^{-1}w(\lambda)\|_{L^2}^2+\||V'|^{\f12}\partial_x\nabla\Delta^{-1}w(\lambda)\|_{L^2}^2)\mathrm{d}\lambda,\\
   &\|(g_{1},g_2)\|_{L^2(I;L^2)}^2\sim  \int_{\R}\|(F_1,F_2)(\lambda)\|_{L^2}^2\mathrm{d}\lambda,\\
   &\||D_x|^{-\f12}\nabla g_3\|_{L^2(I;L^2)}^2\sim  \int_{\R}\||D_x|^{-\f12}\nabla F_3(\lambda)\|_{L^2}^2\mathrm{d}\lambda.
\end{align*}
It follows from Proposition \ref{prop:res-LNS} that 
\begin{align*}
   & \||D_x|^{\f12}\nabla\Delta^{-1}w(\lambda)\|_{L^2}+\||V'|^{\f12}\partial_x\nabla\Delta^{-1}w(\lambda)\|_{L^2}\\ 
   &\leq    C\nu^{-\f12}\|(\partial_y F_1+\partial_xF_2)(\lambda)\|_{H^{-1}}+C\||D_x|^{-\f12}\nabla F_3(\lambda)\|_{L^2}\\ 
   &\leq    C\nu^{-\f12}\|(F_1,F_2)(\lambda)\|_{L^2}+C\||D_x|^{-\f12}\nabla F_3(\lambda)\|_{L^2}.
\end{align*}
This shows that 
\begin{align*}
   & \||D_x|^{\f12}\nabla\Delta^{-1}f\|_{L^2(I;L^2)}^2+ \||V'|^{\f12}\partial_x\nabla \Delta^{-1}f\|_{L^2(I;L^2)}^2\\ &\sim
   \int_{\R}\big(\||D_x|^{\f12}\nabla\Delta^{-1}w(\lambda)\|_{L^2}^2+\||V'|^{\f12}\partial_x\nabla\Delta^{-1}w(\lambda)\|_{L^2}^2\big)\mathrm{d}\lambda\\
   &\leq C\int_{\R}\big(\nu^{-1}\|(F_1,F_2)(\lambda)\|_{L^2}^2+\||D_x|^{-\f12}\nabla F_3(\lambda)\|_{L^2}^2\big)\mathrm{d}\lambda\\
   &\sim \nu^{-1}\|(g_{1},g_2)\|_{L^2(I;L^2)}^2+\||D_x|^{-\f12}\nabla g_3\|_{L^2(I;L^2)}^2.
\end{align*}

This completes the proof.
\end{proof}

\begin{proof}[Proof of Lemma \ref{prop:space-time-g0}]
Let $f_1=P_{\neq}f$ for $t\geq 0$ and let $f_1$ solve
\begin{align*}
      & \partial_tf_1+\nu \Delta f_1+V\partial_x(f_1+\Delta^{-1}f_1)=0, \ f_1|_{t=0}=P_{\neq}f(0),\quad t\leq0.
   \end{align*}Then we have $P_0f_1=0$ and\begin{align*}
      & \partial_tf_1-\nu\mathrm{sgn}(t) \Delta f_1+V\partial_x(f_1+\Delta^{-1}f_1)=0, \ f_1|_{t=0}=P_{\neq}f(0),\quad\forall\ t\in\R.
   \end{align*} 
  Taking the inner product $\langle \ ,\ \rangle_{\star}$ with $f_1$, we have (here $\langle \ ,\ \rangle_{\star}$ is defined in
 \eqref{def:norm-*})
 \begin{align*}
&\f12\partial_t\|f_1\|_{\star}^2+{\nu}\mathrm{sgn}(t)\|\nabla f_1\|_{\star}^2=0.
\end{align*}
Integrating in $t$ and using that $\|h\|_{\star}\sim \|h\|_{L^2} $, we obtain
\begin{align*}
&{\nu}\|\nabla f\|_{L^2(\R;L^2)}^2\leq C\|f_1(0)\|_{L^2}\leq C\|f(0)\|_{L^2}.
\end{align*}Let 
\begin{align*}
   & w(\lambda,x,y)=\int_{\R}f_{1}(t,x,y)\mathrm{e}^{\mathrm{i}\lambda t}\mathrm{d}t,\quad 
   F(\lambda,x,y)=\int_{-\infty}^{0}f_{1}(t,x,y)\mathrm{e}^{\mathrm{i}\lambda t}\mathrm{d}t.
\end{align*}
Then we get
\begin{align*}
   &-\nu\Delta w+V\partial_x(w+\Delta^{-1} w)-\mathrm{i}\lambda w=-2\nu\Delta F.
\end{align*}
Using Plancherel's formula, we know that 
\begin{align*}
   & \||D_x|^{\f12}\nabla\Delta^{-1}f_1\|_{L^2(\R;L^2)}^2+ \||V'|^{\f12}\partial_x\nabla \Delta^{-1}f_1\|_{L^2(\R;L^2)}^2\\ &\sim
   \int_{\R}(\||D_x|^{\f12}\nabla\Delta^{-1}w(\lambda)\|_{L^2}^2+\||V'|^{\f12}\partial_x\nabla\Delta^{-1}w(\lambda)\|_{L^2}^2)\mathrm{d}\lambda,\\
   &\|\nabla f_1\|_{L^2((-\infty,0];L^2)}^2\sim  \int_{\R}\|\nabla F(\lambda)\|_{L^2}^2\mathrm{d}\lambda.
\end{align*}
By Proposition \ref{prop:res-LNS}, we have
\begin{align*}
   & \||D_x|^{\f12}\nabla\Delta^{-1}w(\lambda)\|_{L^2}+\||V'|^{\f12}\partial_x\nabla\Delta^{-1}w(\lambda)\|_{L^2}\\ 
   &\leq    C\nu^{-\f12}\|2\nu\Delta F(\lambda)\|_{H^{-1}}   \leq C\nu^{\f12}\|\nabla F(\lambda)\|_{L^2}.
\end{align*}
Thus(using $|D_x|^{\f12}\nabla\Delta^{-1}f=|D_x|^{\f12}\nabla\Delta^{-1}f_1 $, 
$\partial_x\nabla\Delta^{-1}f=\partial_x\nabla\Delta^{-1}f_1 $ for $ t\in I=[0,+\infty)$),
\begin{align*}
 & \||D_x|^{\f12}\nabla\Delta^{-1}f\|_{L^2(I;L^2)}^2+ \||V'|^{\f12}\partial_x\nabla \Delta^{-1}f\|_{L^2(I;L^2)}^2\\ 
  &\leq  \||D_x|^{\f12}\nabla\Delta^{-1}f_1\|_{L^2(\R;L^2)}^2+ \||V'|^{\f12}\partial_x\nabla \Delta^{-1}f_1\|_{L^2(\R;L^2)}^2\\ &\sim
   \int_{\R}\big(\||D_x|^{\f12}\nabla\Delta^{-1}w(\lambda)\|_{L^2}^2+\||V'|^{\f12}\partial_x\nabla\Delta^{-1}w(\lambda)\|_{L^2}^2\big)\mathrm{d}\lambda\\
   &\leq C\int_{\R}\nu\|\nabla F(\lambda)\|_{L^2}^2\mathrm{d}\lambda
   \sim \nu\|\nabla f_1\|_{L^2((-\infty,0];L^2)}^2\leq C\|f(0)\|_{L^2}^2.
\end{align*}

This completes the proof.
\end{proof}

\subsection{Proof of Proposition \ref{prop:LNS-sp2}}

This part is devoted to the space-time estimates for the following linearized Navier-Stokes system(i.e., \eqref{eq2})
\begin{align}\label{eq:LNS-T2}
  \partial_tf-\nu \Delta f+\mathrm{e}^{-\nu t}(V\partial_xf-V''\partial_x\Delta^{-1}f)=\partial_y g_1+\partial_xg_2.
\end{align} 
\def\Ce{C_2}
 \begin{lemma}\label{lem4}
Let $f$ solve \eqref{eq:LNS-T1} with $P_0f=0$ and $g=0$. Then for $I= [t_1,t_2]\subset[0,1/\nu]$, it holds that 
\begin{align*}
&\|f(t_2)\|_{L^2}\leq \Ce\nu^{-\f14}(t_2-t_1)^{-\f12}\|f(t_1)\|_{L^2}.
\end{align*}
 \end{lemma}
 
 \begin{proof}
 Taking the inner product $\langle \ ,\ \rangle_{\star}$ with $f$, we have(here $\langle \ ,\ \rangle_{\star}$ is defined in
 \eqref{8.1})
 \begin{align*}
&\f12\partial_t\|f\|_{\star}^2+{\nu}\|\nabla f\|_{\star}^2=0,
\end{align*}
which implies by $\|h\|_{\star}\sim \|h\|_{L^2} $ that 
\begin{align*}
&\|f(t_2)\|_{\star}^2\leq \|f(t)\|_{\star}^2,\quad \|f(t_2)\|_{L^2}^2\leq C\|f(t)\|_{L^2}^2,\quad \forall \ t\in[t_1,t_2].
\end{align*}
By Lemma \ref{lem:norm-X} and Proposition \ref{prop:LNS-sp1}, we have (as $f=f_{\neq}$)
\begin{align*}
      &\nu^{\f14}\|f\|_{L^2(I;L^2)}\leq C\|f\|_{X_I}\leq C\|f(t_1)\|_{L^2}.
   \end{align*}
Thus, we obtain
\begin{align*}
      (t_2-t_1)\|f(t_2)\|_{L^2}^2\leq C\int_{t_1}^{t_2}\|f(t)\|_{L^2}^2\mathrm{d}t=C\|f\|_{L^2(I;L^2)}^2\leq C{\nu}^{-\f12}\|f(t_1)\|_{L^2}^2.
   \end{align*}
\end{proof}

 \begin{lemma}\label{lem3}
 Let $f$ solve  \eqref{eq:LNS-T2} with $P_0f=0$, and $t_0=9\Ce^2\nu^{-\f12}$ with $\Ce$ given by Lemma \ref{lem4}. There exists $\epsilon_0>0$ such that if $V=V(t,y)$ satisfies 
 \beno
 \|V-\cos y\|_{H^4}\leq \epsilon_0\nu^{\f13},\quad  \|\partial_tV\|_{H^2}\leq \epsilon_0\nu^{\f43}\quad \text{for}\,\, t\in[0,1/\nu],
 \eeno
 then for $I= [t_1,t_2]\subset[0,1/\nu]$, $t_2-t_1\leq t_0$, it holds that 
\begin{align*}
&\|f\|_{X_I}\leq C\big(\|f(t_1)\|_{L^2}+\nu^{-\f12}\|(g_1,g_2)\|_{L^2(I;L^2)}\big).
\end{align*}
Moreover, if $t_2-t_1=t_0$, then we have
\begin{align*}
&\|f(t_2)\|_{L^2}\leq \|f(t_1)\|_{L^2}/2+C\nu^{-\f12}\|(g_1,g_2)\|_{L^2(I;L^2)}.
\end{align*}
 \end{lemma}

 \begin{proof}
 Let
 \beno
 V_*(y)=V(t_1,y),\quad  \widetilde V(t,y)=V(t,y)-V_*(y),\quad V_1(t,y)=\partial_y^2V(t,y)+V_*(y).
 \eeno
 Thus,  \eqref{eq2}  is reduced to 
 \begin{align*}
  \partial_tf-\nu \Delta f+\mathrm{e}^{-\nu t}V_*\partial_x(f+\Delta^{-1}f)=\partial_y g_1+\partial_xg_2+g_4,
  \end{align*}
where $g_4=\mathrm{e}^{-\nu t}(-\widetilde V\partial_xf+V_1\partial_x\Delta^{-1}f)$.

For $t\in I=[t_1,t_2]$, we have $ 0\leq t-t_1\leq t_2-t_1\leq t_0\leq C\nu^{-\f12}$. Thanks to $\widetilde V(t,y)=\int_{t_1}^t\partial_tV(s,y)\mathrm{d}s$, we get
\begin{align*}
  \|\widetilde V(t,\cdot)\|_{H^2}\leq \int_{t_1}^t\|\partial_tV\|_{H^2}\mathrm{d}s\leq  (t-t_1)\epsilon_0\nu^{\f43}\leq C\epsilon_0\nu^{\f56}.
\end{align*}
Thanks to $V_1=\partial_y^2(V-\cos y)+(V_*-\cos y)$, we have
\begin{align*}
 \|V_1\|_{H^2}\leq \|V-\cos y\|_{H^4}+\|V_*-\cos y\|_{H^4}\leq 2\epsilon_0\nu^{\f13}.
\end{align*}
Then we can apply Proposition \ref{prop:LNS-sp1} (with $V$ replaced by $V_*$)  to obtain
\begin{align*}
      &\|f\|_{X_I}\leq C\big(\|f(t_1)\|_{L^2}+\nu^{-\f12}\|(g_1,g_2)\|_{L^2(I;L^2)}+\|g_4\|_{L^1(I;L^2)}\big).
   \end{align*}
For $g_4$, we have 
\begin{align*}
\|g_4\|_{L^1(I;L^2)}\leq& \|\widetilde V\|_{L^2(I;L^{\infty})}\|\partial_xf\|_{L^2(I;L^2)}+
\|V_1\|_{L^2(I;L^{\infty})}\|\partial_x\Delta^{-1}f\|_{L^2(I;L^2)}
\\ \leq& C\big(\|\widetilde V\|_{L^2(I;H^2)}\|\nabla f\|_{L^2(I;L^2)}+\|V_1\|_{L^2(I;H^2)}\||D_x|^{\f12}\nabla\Delta^{-1}f\|_{L^2(I;L^2)}\big)\\ \leq& C\epsilon_0(t_2-t_1)^{\f12}\big(\nu^{\f56}\|\nabla f\|_{L^2(I;L^2)}+\nu^{\f13}\||D_x|^{\f12}\nabla\Delta^{-1}f\|_{L^2(I;L^2)}\big)\\ \leq& C\epsilon_0(t_2-t_1)^{\f12}\nu^{\f13}\|f\|_{X_I}\leq C\epsilon_0\nu^{-\f14}\nu^{\f13}\|f\|_{X_I}\leq C\epsilon_0\|f\|_{X_I}.
\end{align*}
Thus, we obtain 
\begin{align*}
      &\|f\|_{X_I}\leq C\big(\|f(t_1)\|_{L^2}+\nu^{-\f12}\|(g_1,g_2)\|_{L^2(I;L^2)}+\epsilon_0\|f\|_{X_I}\big),
   \end{align*}
   which gives the first inequality of the lemma by taking $\epsilon_0$ small enough. 
   
   Now we consider $t_2-t_1=t_0$. We decompose $f=f_1+f_2$, where $f_1,\ f_{2}$ solve
  \begin{equation*}\left\{
    \begin{aligned}
     &\partial_tf_1-\nu \Delta f_1+V_*\partial_x(f_1+\Delta^{-1}f_1)=0,\quad f_{1}|_{t=t_1}=f(t_1),\\
     &\partial_tf_2-\nu \Delta f_2+V_*\partial_x(f_2+\Delta^{-1}f_2)=\partial_y g_1+\partial_xg_2+g_4,\quad f_{2}|_{t=t_1}=0. 
  \end{aligned}\right.
  \end{equation*}
  By Proposition \ref{prop:LNS-sp1}, $\|g_4\|_{L^1(I;L^2)}\leq {C}\epsilon_0\|f\|_{X_I}$ and the first inequality of the lemma, we get
  \begin{align*}
      \|f_2\|_{X_I}&\leq C\big(\nu^{-\f12}\|(g_1,g_2)\|_{L^2(I;L^2)}+\|g_4\|_{L^1(I;L^2)}\big)\leq 
      C\big(\nu^{-\f12}\|(g_1,g_2)\|_{L^2(I;L^2)}+\epsilon_0\|f\|_{X_I}\big)\\
      &\leq C\big(\epsilon_0\|f(t_1)\|_{L^2}+\nu^{-\f12}\|(g_1,g_2)\|_{L^2(I;L^2)}\big).
   \end{align*}
As $P_0f_1(t_1)=P_0f(t_1)=0$, we have $P_0f_1=0$. Then we get  by Lemma \ref{lem4}  that
\begin{align*}
&\|f_1(t_2)\|_{L^2}\leq \Ce\nu^{-\f14}(t_2-t_1)^{-\f12}\|f_1(t_1)\|_{L^2}=\Ce\nu^{-\f14}t_0^{-\f12}\|f(t_1)\|_{L^2}=\|f(t_1)\|_{L^2}/3.
\end{align*}Here we used $t_0=9\Ce^2\nu^{-\f12}$. As $f=f_1+f_2$, we have\begin{align*}
\|f(t_2)\|_{L^2}\leq& \|f_1(t_2)\|_{L^2}+\|f_2(t_2)\|_{L^2}\leq \|f_1(t_2)\|_{L^2}+\|f_2\|_{L^{\infty}(I;L^2)}\leq \|f_1(t_2)\|_{L^2}+\|f_2\|_{X_I}\\ \leq& \|f(t_1)\|_{L^2}/3+
C\big(\epsilon_0\|f(t_1)\|_{L^2}+\nu^{-\f12}\|(g_1,g_2)\|_{L^2(I;L^2)}\big),
\end{align*}
which implies the second inequality of the lemma by taking $\epsilon_0$ small enough such that $C\epsilon_0\leq 1/6$.
 \end{proof}
 
 Now we are in a position to prove  Proposition \ref{prop:LNS-sp2}.
  
 \begin{proof}[Proof of Proposition \ref{prop:LNS-sp2}] 
 It is easy to see that
 \begin{align}\label{f1}
\|f\|_{X_I}+\|\mathrm{e}^{\tilde{\epsilon}\nu^{\f12}t}P_{\neq}f\|_{X_I}&\leq \|P_0f\|_{X_I}+\|P_{\neq}f\|_{X_I}+\|\mathrm{e}^{\tilde{\epsilon}\nu^{\f12}t}P_{\neq}f\|_{X_I}\\ \notag&\leq \|P_0f\|_{X_I}+2\|\mathrm{e}^{\tilde{\epsilon}\nu^{\f12}t}P_{\neq}f\|_{X_I}.
\end{align}

For the zero mode, we have
\begin{align}
\notag  &\partial_tP_0f-\nu \Delta P_0f=\partial_y P_0g_1.
\end{align}
The energy estimate gives
\begin{align}
&\f12\partial_t\|P_0f\|_{L^2}^2+{\nu}\|\nabla P_0f\|_{L^2}^2\leq|\langle\partial_y P_0g_1,P_0f\rangle|=|\langle P_0g_1,\partial_yP_0f\rangle|\leq \|g_1\|_{L^2}\|\nabla P_0f\|_{L^2},
 \notag 
 \end{align}
which implies 
\begin{align}
\label{f2}  &\|P_0f\|_{X_I}=\|P_0f\|_{L^\infty(I;L^2)}+\nu^{\f12}\|\nabla P_0f\|_{L^2(I;L^2)}\leq C\big(\|P_0f(T_1)\|_{L^2}+\nu^{-\f12}\|g_1\|_{L^2(I;L^2)}\big).
\end{align}

\def\A{a}
Now we consider the nonzero mode. Let $t_0=9\Ce^2\nu^{-\f12}$, ${\epsilon}=1/(18\Ce^2) $ with $\Ce$ given by Lemma \ref{lem4}. Let 
$N=\inf(\Z\cap[(T-t_1)/t_0,+\infty))$ and 
\beno
t_{N+1}=T,\quad  t_j=t_1+(j-1)t_0,\quad  I_j=[t_j,t_{j+1}]\quad \text{for}\,\, j\in\Z\cap(0,N].
\eeno
 Then 
$I=[t_1,T]=\cup_{j=1}^NI_j$, $\mathrm{e}^{\tilde{\epsilon}\nu^{\f12}t}\sim\mathrm{e}^{\tilde{\epsilon}\nu^{\f12}jt_0}=\A^{j} $ for $ t\in I_j$, here 
$\A:=\mathrm{e}^{\tilde{\epsilon}\nu^{\f12}t_0}=\mathrm{e}^{\tilde{\epsilon}/(2\epsilon)}\in[1,\mathrm{e}^{1/2}]$ (recall that $\tilde{\epsilon}\in[0,\epsilon]$) and \begin{align*}
  &\|\mathrm{e}^{\tilde{\epsilon}\nu^{\f12}t}P_{\neq}f\|_{X_I}^2\leq C\sum_{j=1}^N\A^{2j}\|P_{\neq}f\|_{X_{I_j}}^2,\quad 
  \|\mathrm{e}^{\tilde{\epsilon}\nu^{\f12}t}(g_1,g_2)\|_{L^2(I;L^2)}^2\sim\sum_{j=1}^N\A^{2j}\|(g_1,g_2)\|_{L^2(I_j;L^2)}^2.
\end{align*}
Since $I_j=[t_j,t_{j+1}]$ and $t_{j+1}-t_j\leq t_0 $ for $j\in\Z\cap(0,N]$, and
\begin{align*}
  \partial_tP_{\neq}f-\nu \Delta P_{\neq}f+\mathrm{e}^{-\nu t}(V\partial_xP_{\neq}f-V''\partial_x\Delta^{-1}P_{\neq}f)=\partial_y P_{\neq}g_1+\partial_xP_{\neq}g_2.
\end{align*}
we get by Lemma \ref{lem3}  that 
\begin{align*}
  &\|P_{\neq}f\|_{X_{I_j}}\leq C(\|P_{\neq}f(t_j)\|_{L^2}+\nu^{-\f12}\|(P_{\neq}g_1,P_{\neq}g_2)\|_{L^2(I_j;L^2)})\leq C(a_j+\nu^{-\f12}b_j),\end{align*}
where
\beno
a_j:=\|P_{\neq}f(t_j)\|_{L^2},\quad b_j:=\|(g_1,g_2)\|_{L^2(I_j;L^2)},\quad\forall\ j\in \Z\cap(0,N].
\eeno

As $I_j=[t_j,t_{j+1}]$ and $t_{j+1}-t_j= t_0 $ for $j\in\Z\cap(0,N)$, by Lemma \ref{lem3}, we have
\begin{align*}
  &\|P_{\neq}f(t_{j+1})\|_{L^2}\leq \|P_{\neq}f(t_j)\|_{L^2}/2+C\nu^{-\f12}\|(P_{\neq}g_1,P_{\neq}g_2)\|_{L^2(I_j;L^2)},
\end{align*}
which yields 
\beno
a_{j+1}\leq a_j/2+C\nu^{-\f12}b_j.
\eeno
Then by the induction, we have (recall that  $\A\in[1,\mathrm{e}^{1/2}]$, $\A/2\in[1/2,\mathrm{e}^{1/2}/2]\subset(0,1)$)
\begin{align*}
  &a_j\leq 2^{1-j}a_1+C\sum_{k=1}^{j-1}2^{k+1-j}\nu^{-\f12}b_k,\quad \A^{j}a_j\leq 2(\A/2)^{j}a_1+C\sum_{k=1}^{j-1}2(\A/2)^{j-k}\nu^{-\f12}\A^{k}b_k,\\
  &\A^{2j}a_j^2\leq C(\A/2)^{2j}a_1^2+C\sum_{k=1}^{j-1}(\A/2)^{j-k}\nu^{-1}\A^{2k}b_k^2,\quad \forall\ j\in\Z\cap(0,N].
\end{align*}
Thus,
\begin{align} \notag&\|\mathrm{e}^{\tilde{\epsilon}\nu^{\f12}t}(g_1,g_2)\|_{L^2(I;L^2)}^2\sim
\sum_{j=1}^N\A^{2j}\|(g_1,g_2)\|_{L^2(I_j;L^2)}^2=\sum_{j=1}^N\A^{2j}b_j^2,
\end{align}
and 
\begin{align} 
  \|\mathrm{e}^{\tilde{\epsilon}\nu^{\f12}t}P_{\neq}f\|_{X_I}^2\leq& C\sum_{j=1}^N\A^{2j}\|P_{\neq}f\|_{X_{I_j}}^2
 \label{f3} \leq C\sum_{j=1}^N\A^{2j}(a_j+\nu^{-\f12}b_j)^2\leq C\sum_{j=1}^N\A^{2j}(a_j^2+\nu^{-1}b_j^2)\\
\notag \leq& C\sum_{j=1}^N((\A/2)^{2j}a_1^2+\nu^{-1}\A^{2j}b_j^2)+C\sum_{j=1}^N\sum_{k=1}^{j-1}(\A/2)^{j-k}\nu^{-1}\A^{2k}b_k^2\\
\notag \leq& Ca_1^2+C\nu^{-1}\sum_{j=1}^N\A^{2j}b_j^2+C\nu^{-1}\sum_{k=1}^{N-1}\A^{2k}b_k^2\\
\notag \leq& C\|P_{\neq}f(t_1)\|_{L^2}^2+C\nu^{-1}\|\mathrm{e}^{\tilde{\epsilon}\nu^{\f12}t}(g_1,g_2)\|_{L^2(I;L^2)}^2.
\end{align}

Now the result follows from \eqref{f1}, \eqref{f2}, \eqref{f3}.
 \end{proof}\subsection{Proof of Proposition \ref{Prop1}}
\begin{lemma}\label{lem3a}
 Let $f$ solve  \eqref{eq:LNS-T2} with $P_0f=0$. There exist $\epsilon_0,\nu_0\in(0,1/8)$ such that if $0<\nu\leq\nu_0$, $V=V(t,y)$ satisfies 
 \beno
 \|V-\cos y\|_{H^4}\leq \epsilon_0\nu^{\f13},\quad  \|\partial_tV\|_{H^2}\leq \epsilon_0\nu^{\f43}\quad \text{for}\,\, t\in[1/\nu,+\infty),
 \eeno
 then for $I= \big[t_1,t_2]\subset[1/\nu,\frac{2}{3\nu}\ln(1/\nu)\big]$, $t_2-t_1\leq 1/\nu$, it holds that 
\begin{align*}
&\|f\|_{L^{\infty}(I;L^2)}+\nu^{\f16}\mathrm{e}^{-\nu t_1/3}\big(\|\partial_x \nabla\Delta^{-1}f\|_{L^2(I;L^2)}
+\|\nabla\Delta^{-1}f\|_{L^2(I;L^\infty)}\big)\\&\leq C\big(\|f(t_1)\|_{L^2}+\nu^{-\f12}\|(g_1,g_2)\|_{L^2(I;L^2)}\big).
\end{align*}
Moreover, if $t_2-t_1=1/\nu$, then we have
\begin{align*}
&\|f(t_2)\|_{L^2}\leq \|f(t_1)\|_{L^2}/4+C\nu^{-\f12}\|(g_1,g_2)\|_{L^2(I;L^2)}.
\end{align*}
 \end{lemma}
   \begin{proof}Let $\nu_*=\nu \mathrm{e}^{\nu t_1}$, $\tilde{t}_2=\mathrm{e}^{-\nu t_1}(t_{2}-t_1)$, $\tilde{I}=[0,\tilde{t}_2] $ and (for $j=1,2$)
\begin{align*}
      &\widetilde f(t,y)=f(t_1+\mathrm{e}^{\nu t_1}t,y), \quad \widetilde{g}_j(t,y)=\mathrm{e}^{-\nu t_1}g_{j}(t_1+\mathrm{e}^{\nu t_1}t,y), \quad \widetilde V(t,y)=V(t_1+\mathrm{e}^{\nu t_1}t,y).
   \end{align*}
   Then we have 
   \begin{align*}
      & \partial_t\widetilde f-\nu_* \Delta\widetilde f+\mathrm{e}^{-\nu_* t}(\widetilde V\partial_x\widetilde f-\widetilde V''\partial_x\Delta^{-1}\widetilde f)=\partial_y \widetilde{g}_{1}+\partial_x\widetilde{g}_{2},   
      \end{align*}
   and
    \begin{align*}
     &\|f(t_1)\|_{L^2}=\|\tilde f(0)\|_{L^2},\quad \|f(t_2)\|_{L^2}=\|\widetilde f(\tilde{t}_2)\|_{L^2},\quad \|f\|_{L^{\infty}(I;L^2)}=\|\widetilde f\|_{L^{\infty}(\tilde I;L^2)},\\
      &\|\partial_x \nabla\Delta^{-1}f\|_{L^2(I;L^2)}
+\|\nabla\Delta^{-1}f\|_{L^2(I;L^\infty)}=\mathrm{e}^{\nu t_1/2}\big(\|\partial_x \nabla\Delta^{-1}\widetilde f\|_{L^2(\tilde I;L^2)}
+\|\nabla\Delta^{-1}\widetilde f\|_{L^2(\tilde I;L^\infty)}\big),\\
&\|(g_1,g_2)\|_{L^2(I;L^2)}=\mathrm{e}^{-\nu t_1/2}\|(\widetilde{g}_1,\widetilde{g}_2)\|_{L^2(\tilde{I};L^2)}.
   \end{align*}   
 Due to $t_2-t_1\leq 1/\nu$, we have $\tilde{t}_2=\mathrm{e}^{-\nu t_1}(t_{2}-t_1)\leq \mathrm{e}^{-\nu t_1}/\nu=1/\nu_*$, 
  $\tilde{I}=[0,\tilde{t}_2]\subseteq[0,1/\nu_*]$. For $I= [t_1,t_2]\subset[1/\nu,\frac{2}{3\nu}\ln(1/\nu)]$, we have 
  $t_1<t_2\leq\frac{2}{3\nu}\ln(1/\nu)$, $\mathrm{e}^{\nu t_1}\leq \nu^{-\frac{2}{3}} $, 
  $\nu_*=\nu \mathrm{e}^{\nu t_1}\leq \nu^{\frac{1}{3}}\leq1$. For $t\geq0$, we have 
  \beno
 &&\|\widetilde V(t,\cdot)-\cos y\|_{H^4}=\| V(t_1+\mathrm{e}^{\nu t_1}t,\cdot)-\cos y\|_{H^4}\leq \epsilon_0\nu^{\f13}\leq \epsilon_0\nu_*^{\f13},\\
   &&\|\partial_t\widetilde V(t,\cdot)\|_{H^2}=\mathrm{e}^{\nu t_1}\|\partial_t V(t_1+\mathrm{e}^{\nu t_1}t,\cdot)\|_{H^2}\leq \epsilon_0\nu^{\f43}\mathrm{e}^{\nu t_1}\leq \epsilon_0\nu^{\f43}\mathrm{e}^{\f43\nu t_1}=\epsilon_0\nu_*^{\f43}.
 \eeno
Thus, we can apply Proposition \ref{prop:LNS-sp2} with $\nu$ replaced by $\nu_*$ and $\tilde\epsilon=0 $ to obtain
\begin{align*}
&\|\widetilde f\|_{X_{\tilde I,\nu_*}}\leq C\big(\|\widetilde f(0)\|_{L^2}+\nu_*^{-\f12}\|(\widetilde{g}_1,\widetilde{g}_2)\|_{L^2(\tilde I;L^2)}\big),
\end{align*}
where
\begin{align*}
\nu_*^{-\f12}\|(\widetilde{g}_1,\widetilde{g}_2)\|_{L^2(\tilde I;L^2)}=\nu^{-\f12}\mathrm{e}^{-\nu t_1/2}\|(\widetilde{g}_1,\widetilde{g}_2)\|_{L^2(\tilde I;L^2)}
=\nu^{-\f12}\|({g}_1,{g}_2)\|_{L^2(I;L^2)}.
\end{align*}

Due to $P_0f=0$, we have $f=f_{\neq}$, $\widetilde f=\widetilde{f}_{\neq}$. By Lemma \ref{lem:norm-X} with $\nu$ replaced by $\nu_*$, we get
   \begin{align*}
      & \nu_*^{\f16}\|\partial_x\nabla\Delta^{-1}\widetilde f\|_{L^2(\tilde I;L^2)}+
      \nu_*^{\f16}\|\nabla\Delta^{-1}\widetilde f\|_{L^2(\tilde I;L^{\infty})}\leq C\|\widetilde f\|_{X_{\tilde I,\nu_*}}.
   \end{align*}
 Notice that $\nu_*=\nu \mathrm{e}^{\nu t_1} $ and $\nu^{\f16}\mathrm{e}^{-\nu t_1/3}\mathrm{e}^{\nu t_1/2}=\nu^{\f16}\mathrm{e}^{\nu t_1/6}=\nu_*^{\f16}$. Thus, 
 \begin{align*}
&\|f\|_{L^{\infty}(I;L^2)}+\nu^{\f16}\mathrm{e}^{-\nu t_1/3}\big(\|\partial_x \nabla\Delta^{-1}f\|_{L^2(I;L^2)}
+\|\nabla\Delta^{-1}f\|_{L^2(I;L^\infty)}\big)\\&=\|\widetilde f\|_{L^{\infty}(\tilde I;L^2)}+\nu^{\f16}\mathrm{e}^{-\nu t_1/3}\mathrm{e}^{\nu t_1/2}\big(\|\partial_x \nabla\Delta^{-1}\tilde f\|_{L^2(\tilde I;L^2)}
+\|\nabla\Delta^{-1}\widetilde f\|_{L^2(\tilde I;L^\infty)}\big)\\&=\|\widetilde f\|_{L^{\infty}(\tilde I;L^2)}+\nu_*^{\f16}\big(\|\partial_x \nabla\Delta^{-1}\tilde f\|_{L^2(\tilde I;L^2)}
+\|\nabla\Delta^{-1}\widetilde f\|_{L^2(\tilde I;L^\infty)}\big)\\ &\leq C\|\widetilde f\|_{X_{\tilde I,\nu_*}}\leq C\big(\|\widetilde f(0)\|_{L^2}+\nu_*^{-\f12}\|(\widetilde{g}_1,\widetilde{g}_2)\|_{L^2(\tilde I;L^2)}\big)\\&=C\big(\|f(t_1)\|_{L^2}+\nu^{-\f12}\|({g}_1,{g}_2)\|_{L^2(I;L^2)}\big).
\end{align*} 
This proves the first inequality of the lemma. 
   
   Now we consider $t_2-t_1=1/\nu$. Then we have $\tilde{t}_2=\mathrm{e}^{-\nu t_1}(t_{2}-t_1)= \mathrm{e}^{-\nu t_1}/\nu=1/\nu_*$, 
  $\tilde{I}=[0,\tilde{t}_2]=[0,1/\nu_*]$. We decompose $\widetilde f=f_1+f_2$, where $f_1, f_{2}$ solve
  \begin{equation*}\left\{
    \begin{aligned}
     &\partial_t f_1-\nu_* \Delta f_1+\mathrm{e}^{-\nu_* t}\big(\widetilde V\partial_x f_1-\widetilde V''\partial_x\Delta^{-1} f_1\big)=0,\quad f_{1}|_{t=0}=\widetilde f(0)=f(t_1),\\
     &\partial_t f_2-\nu_* \Delta f_2+\mathrm{e}^{-\nu_* t}\big(\widetilde V\partial_x f_2-\widetilde V''\partial_x\Delta^{-1} f_2\big)=\partial_y \widetilde{g}_{1}+\partial_x\widetilde{g}_{2},\quad f_{2}|_{t=0}=0. 
  \end{aligned}\right.
  \end{equation*}
 Due to $P_0f=0$, we have $P_0f_1|_{t=0}=0$, $P_0f_1=0$, $f_1=P_{\neq}f_1$. By Proposition \ref{prop:LNS-sp2} with $\nu$ replaced by $\nu_*$ and $\tilde\epsilon=\epsilon $, we obtain 
\begin{align*}
&\| \mathrm{e}^{{\epsilon}\nu_*^{\f12}t}f_1\|_{X_{\tilde I,\nu_*}}\leq C\|f(t_1)\|_{L^2}.
\end{align*}
Due to $\nu_*\leq \nu^{\frac{1}{3}}\leq \nu_0^{\frac{1}{3}}$, we have 
\begin{align*}
&\| f_1|_{t=1/\nu_*}\|_{L^2}\leq \mathrm{e}^{-{\epsilon}\nu_*^{-1/2}}
\| \mathrm{e}^{{\epsilon}\nu_*^{\f12}t}f_1\|_{L^{\infty}(\tilde I;L^2)}\leq\mathrm{e}^{-{\epsilon}\nu^{-1/6}}\| \mathrm{e}^{{\epsilon}\nu_*^{\f12}t}f_1\|_{X_{\tilde I,\nu_*}}\leq C\mathrm{e}^{-{\epsilon}\nu_0^{-1/6}}\|f(t_1)\|_{L^2}.
\end{align*}
By Proposition \ref{prop:LNS-sp2} with $\nu$ replaced by $\nu_*$ and $\tilde\epsilon=0 $, we get
\begin{align*}
&\| f_2|_{t=1/\nu_*}\|_{L^2}\leq\|f_2\|_{X_{\tilde I,\nu_*}}\leq C\nu_*^{-\f12}\|(\widetilde{g}_1,\widetilde{g}_2)\|_{L^2(\tilde I;L^2)}=C\nu^{-\f12}\|({g}_1,{g}_2)\|_{L^2(I;L^2)}.
\end{align*}
Thus, 
\begin{align*}
\|f(t_2)\|_{L^2}&=\|\widetilde f(\tilde{t}_2)\|_{L^2}=\|\widetilde f|_{t=1/\nu_*}\|_{L^2}\leq \| f_1|_{t=1/\nu_*}\|_{L^2}+\| f_2|_{t=1/\nu_*}\|_{L^2}\\ &\leq C\mathrm{e}^{-{\epsilon}\nu_0^{-1/6}}\|f(t_1)\|_{L^2}+C\nu^{-\f12}\|(g_1,g_2)\|_{L^2(I;L^2)},
\end{align*}
which shows the second inequality of the lemma by taking $ \nu_0$ small   such that $C\mathrm{e}^{-{\epsilon}\nu_0^{-1/6}}\leq 1/4$.
 \end{proof}
 
 Recall that 
\begin{align*}
&\|f\|_{\YI}=\|\mathrm{e}^{\nu t}f\|_{L^{\infty}(I;L^2)}+\nu^{\f16}\|\mathrm{e}^{\nu t/2}\partial_x \nabla\Delta^{-1}f\|_{L^2(I;L^2)}
+\nu^{\f16}\|\mathrm{e}^{\nu t/2}\nabla\Delta^{-1}f\|_{L^2(I;L^\infty)}.
\end{align*}
Let $T_2=1/\nu$ and $T_3=\frac{2}{3\nu}\ln(1/\nu)$. We estimate the nonzero mode for $t\in [T_2,T_3]=[1/\nu,\frac{2}{3\nu}\ln(1/\nu)]$ and 
$t\in [T_3,+\infty)$ separately.

\begin{lemma}\label{lem11}
Let $f$ solve \eqref{eq:LNS-T2} with $P_0f=0$. There exists $\epsilon_0,\nu_0\in(0,1/4)$ such that if  $0<\nu \leq \nu_0$, $V=V(t,y)$, $\|V-\cos y\|_{H^4}\leq \epsilon_0\nu^{\f13}$, 
$\|\partial_tV\|_{H^2}\leq \epsilon_0\nu^{\f43}$ for $t\in[1/\nu,+\infty)$ , then
 for $I=[T_2,T]\subseteq[T_2,T_3]$, 
 it holds that 
\begin{align*}
&\|f\|_{\YI}\leq 
C\big(\|f(T_2)\|_{L^2}+\nu^{-\f12}\|\mathrm{e}^{\nu t}(g_1,g_2)\|_{L^2(I;L^2)}\big);
\end{align*}
and for $I=[T_3,T]\subset[T_3,+\infty)$, it holds that 
\begin{align*}
&\|f\|_{\YI}\leq 
C\big(\mathrm{e}^{\nu T_3}\|f(T_3)\|_{L^2}+\nu^{-\f12}\|\mathrm{e}^{\nu t}(g_1,g_2)\|_{L^2(I;L^2)}\big).
\end{align*}
\end{lemma}

\begin{proof}
We first consider the case where $I=[T_2,T]\subseteq[T_2,T_3]$. Let 
$N=\inf(\Z\cap[\nu T,+\infty))$ and 
\beno
t_{N}=T,\quad  t_j=j/\nu,\quad  I_j=[t_j,t_{j+1}]\quad \text{for}\,\, j\in\Z\cap(0,N).
\eeno
 Then $T_2=1/\nu=t_1$,
$I=[T_2,T]=\cup_{j=1}^{N-1}I_j$, $\mathrm{e}^{\nu t}\sim\mathrm{e}^{j}$ for $ t\in I_j$,  and 
\begin{align*}
  &\|f\|_{\YI}^2\leq C\sum_{j=1}^{N-1}\big(\mathrm{e}^{2j}\|f\|_{L^{\infty}(I_j;L^2)}^2+\nu^{\f13}\mathrm{e}^{j}\|\partial_x \nabla\Delta^{-1}f\|_{L^2(I_j;L^2)}^2
+\nu^{\f13}\mathrm{e}^{j}\|\nabla\Delta^{-1}f\|_{L^2(I_j;L^\infty)}^2\big),\\
  &\|\mathrm{e}^{\nu t}(g_1,g_2)\|_{L^2(I;L^2)}^2\sim\sum_{j=1}^{N-1}\mathrm{e}^{2j}\|(g_1,g_2)\|_{L^2(I_j;L^2)}^2.
\end{align*}

Since $I_j=[t_j,t_{j+1}]\subseteq[T_2,T_3]=[1/\nu,\frac{2}{3\nu}\ln(1/\nu)]$ and $t_{j+1}-t_j\leq 1/\nu $ for $j\in\Z\cap(0,N)$, 
we get by Lemma \ref{lem3a}  and $ \mathrm{e}^{-\nu t_j/3}=\mathrm{e}^{-j/3}$ that
\begin{align*}
  &\|f\|_{L^{\infty}(I_j;L^2)}+\nu^{\f16}\mathrm{e}^{-j/3}\big(\|\partial_x \nabla\Delta^{-1}f\|_{L^2(I_j;L^2)}
+\|\nabla\Delta^{-1}f\|_{L^2(I_j;L^\infty)}\big)\\&\leq C\big(\|f(t_j)\|_{L^2}+\nu^{-\f12}\|(g_1,g_2)\|_{L^2(I_j;L^2)})= C(a_j+\nu^{-\f12}b_j\big),
\end{align*}
where
\beno
a_j:=\|f(t_j)\|_{L^2},\quad b_j:=\|(g_1,g_2)\|_{L^2(I_j;L^2)},\quad\forall\ j\in \Z\cap(0,N).
\eeno As $I_j=[t_j,t_{j+1}]$ and $t_{j+1}-t_j= 1/\nu $ for $j\in\Z\cap(0,N-2]$, we get by Lemma \ref{lem3a} that
\begin{align*}
  &\|f(t_{j+1})\|_{L^2}\leq \|f(t_j)\|_{L^2}/4+C\nu^{-\f12}\|(g_1,g_2)\|_{L^2(I_j;L^2)},
\end{align*}
which gives
\beno
a_{j+1}\leq a_j/4+C\nu^{-\f12}b_j.
\eeno
Then by the induction, we obtain 
\begin{align*}
  &a_j\leq 4^{1-j}a_1+C\sum_{k=1}^{j-1}4^{k+1-j}\nu^{-\f12}b_k,\quad \mathrm{e}^{j}a_j\leq 4(\mathrm{e}/4)^{j}a_1+C\sum_{k=1}^{j-1}4(\mathrm{e}/4)^{j-k}\nu^{-\f12}\mathrm{e}^{k}b_k,\\
  &\mathrm{e}^{2j}a_j^2\leq C(\mathrm{e}/4)^{2j}a_1^2+C\sum_{k=1}^{j-1}(\mathrm{e}/4)^{j-k}\nu^{-1}\mathrm{e}^{2k}b_k^2,\quad \forall\ j\in\Z\cap(0,N).
\end{align*}
Thus, we conclude 
\begin{align} \notag&\|\mathrm{e}^{\nu t}(g_1,g_2)\|_{L^2(I;L^2)}^2\sim\sum_{j=1}^{N-1}\mathrm{e}^{2j}\|(g_1,g_2)\|_{L^2(I_j;L^2)}^2=\sum_{j=1}^{N-1}\mathrm{e}^{2j}b_j^2,
\end{align}
and
\begin{align*}
  \|f\|_{\YI}^2&\leq C\sum_{j=1}^{N-1}\big(\mathrm{e}^{2j}\|f\|_{L^{\infty}(I_j;L^2)}^2+\nu^{\f13}\mathrm{e}^{j}\|\partial_x \nabla\Delta^{-1}f\|_{L^2(I_j;L^2)}^2
+\nu^{\f13}\mathrm{e}^{j}\|\nabla\Delta^{-1}f\|_{L^2(I_j;L^\infty)}^2\big)\\ 
&\leq C\sum_{j=1}^{N-1}\mathrm{e}^{2j}\left(\|f\|_{L^{\infty}(I_j;L^2)}+\nu^{\f16}\mathrm{e}^{-j/3}(\|\partial_x \nabla\Delta^{-1}f\|_{L^2(I_j;L^2)}
+\|\nabla\Delta^{-1}f\|_{L^2(I_j;L^\infty)})\right)^2\\ 
&\leq C\sum_{j=1}^{N-1}\mathrm{e}^{2j}(a_j+\nu^{-\f12}b_j)^2\leq C\sum_{j=1}^{N-1}\mathrm{e}^{2j}(a_j^2+\nu^{-1}b_j^2)\\
 &\leq C\sum_{j=1}^{N-1}((\mathrm{e}/4)^{2j}a_1^2+\nu^{-1}\mathrm{e}^{2j}b_j^2)+
 C\sum_{j=1}^{N-1}\sum_{k=1}^{j-1}(\mathrm{e}/4)^{j-k}\nu^{-1}\mathrm{e}^{2k}b_k^2\\
 &\leq Ca_1^2+C\nu^{-1}\sum_{j=1}^{N-1}\mathrm{e}^{2j}b_j^2+C\nu^{-1}\sum_{k=1}^{N-2}\mathrm{e}^{2k}b_k^2\\
 &\leq C\|f(t_1)\|_{L^2}^2+C\nu^{-1}\|\mathrm{e}^{\nu t}(g_1,g_2)\|_{L^2(I;L^2)}^2,
\end{align*}
which leads to the first inequality of the lemma due to $T_2=1/\nu=t_1$.

Next, we assume that $I=[T_3,T]\subset[T_3,+\infty)$ and rewrite \eqref{eq:LNS-T2} as
\begin{align*}
  \partial_tf-\nu \Delta f+\mathrm{e}^{-\nu t}V\partial_x(f+\Delta^{-1}f)=g:=\partial_y g_1+\partial_xg_2+g_4,
  \end{align*}
where $g_4=\mathrm{e}^{-\nu t}(V+V'')\partial_x\Delta^{-1}f$. Taking the inner product $\langle \ ,\ \rangle_{\star}$ with $f$, we obtain
 \begin{align*}
\f12\partial_t\|f\|_{\star}^2+{\nu}\|\nabla f\|_{\star}^2\leq& |\langle g,f\rangle_\star|\\ =&
|-\langle g_{1},\partial_yf\rangle_\star-\langle g_{2},\partial_xf\rangle_\star+
\langle{g}_{4},f\rangle_\star|\\ \leq& \|(g_{1},g_{2})\|_{\star}\|\nabla f\|_{\star}+\|g_{4}\|_{\star}\|f\|_{\star}.
\end{align*}
As $P_0f=0$, we have $\|\nabla f\|_{\star}\geq\alpha\| f\|_{\star} $, here $\alpha=2\pi/\mathfrak{p}>1$, and\begin{align*}
  &\|g_{4}\|_{\star}\| f\|_{\star}\leq\|g_{4}\|_{L^2}\| f\|_{\star}\leq \mathrm{e}^{-\nu t}\|V+V''\|_{L^{\infty}}\|\partial_x\Delta^{-1}f\|_{L^2}\| f\|_{\star},\\
  &\|V+V''\|_{L^{\infty}}=\|(\partial_y^2+1)(V-\cos y)\|_{L^{\infty}}\le C\|V-\cos y\|_{H^4}\le C\epsilon_0\nu^{\f13},\\
  &\|\partial_x\Delta^{-1}f\|_{L^2}\leq \|f\|_{L^2}\leq C\|f\|_{*}.
  \end{align*}
  Thus, for $t\in I=[T_3,T]$, we have(as $T_3=\frac{2}{3\nu}\ln(1/\nu)$, $\mathrm{e}^{-\nu T_3}=\nu^{\f23} $)
  \begin{align*}
  &\|g_{4}\|_{\star}\| f\|_{\star}\leq C\mathrm{e}^{-\nu t}\epsilon_0\nu^{\f13}\| f\|_{\star}^2\leq C\mathrm{e}^{-\nu T_3}\epsilon_0\nu^{\f13}\| \nabla f\|_{\star}^2=C\epsilon_0\nu\| \nabla f\|_{\star}^2,  
  \end{align*}
  which gives
  \begin{align*}
  &\f12\partial_t(\mathrm{e}^{2\nu t}\|f\|_{\star}^2)+{\nu}\mathrm{e}^{2\nu t}\|\nabla f\|_{\star}^2-\nu\mathrm{e}^{2\nu t}\|f\|_{\star}^2 \leq \mathrm{e}^{2\nu t}|\langle g,f\rangle_\star|\\ &\leq \mathrm{e}^{2\nu t}(\|(g_{1},g_{2})\|_{\star}\|\nabla f\|_{\star}+\|g_{4}\|_{\star}\|f\|_{\star})\leq \mathrm{e}^{2\nu t}(\|(g_{1},g_{2})\|_{\star}\|\nabla f\|_{\star}+C\epsilon_0\nu\| \nabla f\|_{\star}^2),\\
  &\f12\partial_t(\mathrm{e}^{2\nu t}\|f\|_{\star}^2)+{\nu}\mathrm{e}^{2\nu t}(1-C\epsilon_0-\alpha^{-2})\|\nabla f\|_{\star}^2\leq \mathrm{e}^{2\nu t}\|(g_{1},g_{2})\|_{\star}\|\nabla f\|_{\star}.
  \end{align*}
 Here we used $\|\nabla f\|_{\star}\geq\alpha\| f\|_{\star} $. Taking $\epsilon_0 $ small enough yields 
  \begin{align*}
  &\f12\partial_t(\mathrm{e}^{2\nu t}\|f\|_{\star}^2)+\f12{\nu}\mathrm{e}^{2\nu t}(1-\alpha^{-2})\|\nabla f\|_{\star}^2\leq \mathrm{e}^{2\nu t}\|(g_{1},g_{2})\|_{\star}\|\nabla f\|_{\star}.
  \end{align*}Integrating in $t$ and using the fact $\|h\|_{\star}\sim \|h\|_{L^2} $, we obtain
\begin{align*}
&\|\mathrm{e}^{\nu t}f\|_{L^{\infty}({I};L^2)}^2+{\nu}\|\mathrm{e}^{\nu t}\nabla f\|_{L^2({I};L^2)}^2\lesssim\mathrm{e}^{2\nu T_3}\|f(T_3)\|_{L^2}^2+
\|\mathrm{e}^{\nu t}(g_{1},g_{2})\|_{L^2({I};L^2)}\|\mathrm{e}^{\nu t}\nabla f\|_{L^2({I};L^2)},
\end{align*}
which along with Young's inequality shows 
\begin{align*}
&\|\mathrm{e}^{\nu t}f\|_{L^{\infty}({I};L^2)}+{\nu}^{\f12}\|\mathrm{e}^{\nu t}\nabla f\|_{L^2({I};L^2)}
\lesssim\mathrm{e}^{\nu T_3}\|f(T_3)\|_{L^2}+{\nu}^{-\f12}\|\mathrm{e}^{\nu t}(g_{1},g_{2})\|_{L^2({I};L^2)}.
\end{align*}
As $P_0f=0$, we get by using $T_3=\frac{2}{3\nu}\ln(1/\nu)$, $\mathrm{e}^{-\nu T_3/2}=\nu^{\f13}$ that
\begin{align*}
&\|\partial_x \nabla\Delta^{-1}f\|_{L^2}+\|\nabla\Delta^{-1}f\|_{L^\infty}\leq C\|\nabla\Delta^{-1}f\|_{H^2}\leq C\|\nabla f\|_{L^2},\\
&\nu^{\f16}\|\mathrm{e}^{\nu t/2}\partial_x \nabla\Delta^{-1}f\|_{L^2(I;L^2)}
+\nu^{\f16}\|\mathrm{e}^{\nu t/2}\nabla\Delta^{-1}f\|_{L^2(I;L^\infty)}\\
&\leq C\nu^{\f16}\|\mathrm{e}^{\nu t/2}\nabla f\|_{L^2(I;L^2)}\leq C\nu^{\f16}\mathrm{e}^{-\nu T_3/2}\|\mathrm{e}^{\nu t}\nabla f\|_{L^2(I;L^2)}= C\nu^{\f12}\|\mathrm{e}^{\nu t}\nabla f\|_{L^2(I;L^2)},
\end{align*}and\begin{align*}
\|f\|_{\YI}=&\|\mathrm{e}^{\nu t}f\|_{L^{\infty}(I;L^2)}+\nu^{\f16}\|\mathrm{e}^{\nu t/2}\partial_x \nabla\Delta^{-1}f\|_{L^2(I;L^2)}
+\nu^{\f16}\|\mathrm{e}^{\nu t/2}\nabla\Delta^{-1}f\|_{L^2(I;L^\infty)}\\
\leq& \|\mathrm{e}^{\nu t}f\|_{L^{\infty}(I;L^2)}+C\nu^{\f12}\|\mathrm{e}^{\nu t}\nabla f\|_{L^2(I;L^2)}\\ \leq& C\big(\mathrm{e}^{\nu T_3}\|f(T_3)\|_{L^2}+{\nu}^{-\f12}\|\mathrm{e}^{\nu t}(g_{1},g_{2})\|_{L^2({I};L^2)}\big).
\end{align*}
This proves the second inequality of the lemma.
 \end{proof}  
 
 Now we are in a position to prove Proposition \ref{Prop1}.
  
 \begin{proof}[Proof of Proposition \ref{Prop1}] Recall that $T_2=1/\nu,$ $T_3=\frac{2}{3\nu}\ln(1/\nu)$. For the nonzero mode, we have 
 \begin{align*}
  \partial_tP_{\neq}f-\nu \Delta P_{\neq}f+\mathrm{e}^{-\nu t}(V\partial_xP_{\neq}f-V''\partial_x\Delta^{-1}P_{\neq}f)=\partial_y P_{\neq}g_1+\partial_xP_{\neq}g_2.
\end{align*}

{\it Case 1.} $T\leq T_3$. We  infer from Lemma \ref{lem11}  that 
\begin{align*}
  \|P_{\neq}f\|_{\YI}&\leq 
C\big(\|P_{\neq}f(T_2)\|_{L^2}+\nu^{-\f12}\|\mathrm{e}^{\nu t}(P_{\neq}g_1,P_{\neq}g_2)\|_{L^2(I;L^2)}\big)\\ &\leq C\big(\|P_{\neq}f|_{t=1/\nu}\|_{L^2}+\nu^{-\f12}\|\mathrm{e}^{\nu t}(g_1,g_2)\|_{L^2(I;L^2)}\big).
\end{align*}

{\it Case 2.} $T> T_3$. It follows from Lemma \ref{lem11}  that \def\Y{Y}
\begin{align*}
  &\|P_{\neq}f\|_{\Y_{[T_2,T_3]}}\leq 
C\big(\|P_{\neq}f(T_2)\|_{L^2}+\nu^{-\f12}\|\mathrm{e}^{\nu t}(P_{\neq}g_1,P_{\neq}g_2)\|_{L^2([T_2,T_3];L^2)}\big),\\ 
&\|P_{\neq}f\|_{\Y_{[T_3,T]}}\leq 
C\big(\mathrm{e}^{\nu T_3}\|P_{\neq}f(T_3)\|_{L^2}+\nu^{-\f12}\|\mathrm{e}^{\nu t}(P_{\neq}g_1,P_{\neq}g_2)\|_{L^2([T_3,T];L^2)}\big),
\end{align*}
from which and the fact that $\mathrm{e}^{\nu T_3}\|P_{\neq}f(T_3)\|_{L^2}\leq \|\mathrm{e}^{\nu t}P_{\neq}f\|_{L^{\infty}({[T_2,T_3]};L^2)}\leq \|P_{\neq}f\|_{\Y_{[T_2,T_3]}}$, we infer that
\begin{align*}
&\|P_{\neq}f\|_{\Y_{[T_3,T]}}\leq 
C\big(\|P_{\neq}f\|_{\Y_{[T_2,T_3]}}+\nu^{-\f12}\|\mathrm{e}^{\nu t}(P_{\neq}g_1,P_{\neq}g_2)\|_{L^2([T_3,T];L^2)}\big),
\end{align*}
and
\begin{align*}
\|P_{\neq}f\|_{\YI}&\leq\|P_{\neq}f\|_{\Y_{[T_2,T_3]}}+ \|P_{\neq}f\|_{\Y_{[T_3,T]}}\\
&\leq 
C\big(\|P_{\neq}f\|_{\Y_{[T_2,T_3]}}+\nu^{-\f12}\|\mathrm{e}^{\nu t}(P_{\neq}g_1,P_{\neq}g_2)\|_{L^2([T_3,T];L^2)}\big)\\ 
&\leq C\big(\|P_{\neq}f(T_2)\|_{L^2}+\nu^{-\f12}\|\mathrm{e}^{\nu t}(P_{\neq}g_1,P_{\neq}g_2)\|_{L^2([T_2,T];L^2)}\big)\\ &\leq C\big(\|P_{\neq}f|_{t=1/\nu}\|_{L^2}+\nu^{-\f12}\|\mathrm{e}^{\nu t}(g_1,g_2)\|_{L^2(I;L^2)}\big).
\end{align*}
This proves the first inequality of the proposition. 

For the zero mode, we have
\begin{align}
\notag  &\partial_tP_0f-\nu \Delta P_0f=\partial_y P_0g_1.
\end{align}
The energy estimate gives
\begin{align}
&\f12\partial_t\|P_0f\|_{L^2}^2+{\nu}\|\nabla P_0f\|_{L^2}^2\leq|\langle\partial_y P_0g_1,P_0f\rangle|=|\langle P_0g_1,\partial_yP_0f\rangle|\leq \|g_1\|_{L^2}\|\nabla P_0f\|_{L^2},
 \notag 
 \end{align}
 As $\int_{\mathbb{T}_{\mathfrak{p}}\times\mathbb{T}_{2\pi}}f=0 $, we have $\int_{\mathbb{T}_{\mathfrak{p}}\times\mathbb{T}_{2\pi}}P_0f=0 $ and 
 $\|\nabla P_0f\|_{L^2}\geq \| P_0f\|_{L^2} $. Then
 \begin{align*}
&\f12\partial_t(\mathrm{e}^{\nu t}\|P_0f\|_{L^2}^2)+{\nu}\mathrm{e}^{\nu t}\|\nabla P_0f\|_{L^2}^2-\f12{\nu}\mathrm{e}^{\nu t}\| P_0f\|_{L^2}^2\leq\mathrm{e}^{\nu t}\|g_1\|_{L^2}\|\nabla P_0f\|_{L^2},\\
&\f12\partial_t(\mathrm{e}^{\nu t}\|P_0f\|_{L^2}^2)+\f12{\nu}\mathrm{e}^{\nu t}\|\nabla P_0f\|_{L^2}^2\leq\mathrm{e}^{\nu t}\|g_1\|_{L^2}\|\nabla P_0f\|_{L^2}.
  \end{align*}Integrating in $t$ and using Young's inequality, we obtain
\begin{align*}
&\|\mathrm{e}^{\nu t/2}P_0f\|_{L^\infty(I;L^2)}+\nu^{\f12}\|\mathrm{e}^{\nu t/2}\nabla P_0f\|_{L^2(I;L^2)}\leq C\big(\|P_0f(T_2)\|_{L^2}+\nu^{-\f12}\|\mathrm{e}^{\nu t/2}g_1\|_{L^2(I;L^2)}\big).
\end{align*}
Thus, we get by using the first inequality of the proposition and $T_2=1/\nu$ that
\begin{align*}
&\|\mathrm{e}^{\nu t/2}f\|_{L^{\infty}(I;L^2)}\leq\|\mathrm{e}^{\nu t/2}P_0f\|_{L^{\infty}(I;L^2)}+\|\mathrm{e}^{\nu t}P_{\neq}f\|_{L^{\infty}(I;L^2)}\\ 
&\leq C\big(\|P_0f(T_2)\|_{L^2}+\nu^{-\f12}\|\mathrm{e}^{\nu t/2}g_1\|_{L^2(I;L^2)}\big)+\|P_{\neq}f\|_{\YI}\\
&\leq C\big(\|f(T_2)\|_{L^2}+\nu^{-\f12}\|\mathrm{e}^{\nu t}g_1\|_{L^2(I;L^2)}\big)+ C\big(\|P_{\neq}f|_{t=1/\nu}\|_{L^2}+\nu^{-\f12}\|\mathrm{e}^{\nu t}(g_1,g_2)\|_{L^2(I;L^2)}\big)\\ &\leq C\big(\|f|_{t=1/\nu}\|_{L^2}+\nu^{-\f12}\|\mathrm{e}^{\nu t}(g_1,g_2)\|_{L^2(I;L^2)}\big),
\end{align*}
which gives the second inequality of the proposition. 
 \end{proof} 
 
 \section{The approximate solution}\label{appro}

This section is devoted to the proof of  Proposition \ref{prop:app}.

\subsection{Preliminary lemmas}

  We first provide the estimates for $w_{k,1},\ \psi_{k,1}$ defined in \eqref{def:wk1}.
  
   \begin{lemma}\label{lem:wk1}
   If $\|V-b\|_{H^4}\le \epsilon_0$ with $\epsilon_0$ given by Lemma \ref{lem:genV} , then it holds that for $t\leq T_1$, 
  \begin{align*}
     &|\psi_{k,1}(t,y)|\lesssim\langle t\rangle^{-2}|k|^{-5/2}M_k,\\
     &|w_{k,1}(t,y)|\lesssim\min(\langle t\rangle^{-1}|k|^{-3/2}+|V'|^2|k|^{-1/2},|k|^{-5/2})M_k,\\
     &|\partial_y\psi_{k,1}(t,y)|\lesssim(\langle t\rangle^{-3/2}|k|^{-2}+\langle t\rangle^{-1}|V'||k|^{-3/2})M_k,\\ 
     &|\partial_y(\mathrm{e}^{\mathrm{i}kt_*V}\psi_{k,1})(t,y)|\lesssim \langle t\rangle^{-1.2}|k|^{-2.2}M_k,\\ &|\partial_y(\mathrm{e}^{\mathrm{i}kt_*V}w_{k,1})(t,y)|\lesssim \min(\langle t\rangle^{-1/2}|k|^{-1}+|V'||k|^{-1/2},|k|^{-3/2})M_k,\\
     &\|(\partial_y,k)^2(\mathrm{e}^{\mathrm{i}kt_*V}w_{k,1})\|_{L^2}\lesssim|k|^{-1}M_k.
  \end{align*}  
  Here we recall $M_k=\|\omega_{0,k}\|_{H_k^3}=\|(\partial_y,k)^3\omega_{0,k}\|_{L^2}$.  
   \end{lemma}
  
   \begin{proof}
   Thanks to $\|V-b\|_{H^4}\le \epsilon_0$, we get by Lemma \ref{lem:genV} that
    \begin{align*}
       & |a-1|+|d|+|\theta(y)-y|+|V(y)-b(y)|+|V'(y)-b'(y)|\leq C\|V-b\|_{H^4}\leq 1/2,
    \end{align*}
    which gives 
   \begin{align}\label{est:athetab}
      &|a|\sim 1,\ \theta'(y)\sim 1, \ (\theta^{-1}(y))'\sim1, \ b'(\theta(y))\sim V'(y).
   \end{align}
   
   By Theorem \ref{thm:LE}, we have (here $\widetilde{M}_k=\|\widetilde{w}_{0,k}\|_{H^3_k}$)
   \begin{align*}
     &|\widetilde{\psi}_{k}(at_*,\theta(y))|\lesssim\langle at_*\rangle^{-2}|k|^{-5/2}\widetilde{M}_k,\\
     &|\partial_y[\widetilde{\psi}_{k}(at_*,\theta(y))]|\lesssim|\theta'(y)|(\langle at_*\rangle^{-3/2}|k|^{-2}+\langle at_*\rangle^{-1}|b'(\theta(y))||k|^{-3/2})\widetilde{M}_k,\\
     &|\widetilde{w}_{k}(at_*,\theta(y))|\lesssim\min(\langle at_*\rangle^{-1}|k|^{-3/2}+|b'(\theta(y))|^2|k|^{-1/2},|k|^{-5/2})\widetilde{M}_k,\\
     &|\partial_y[\mathrm{e}^{\mathrm{i}kat_*b((\theta(y)))}\widetilde{w}_{k}(at_*,\theta(y))]| \\
     &\quad\lesssim|\theta'(y)|\min(\langle at_*\rangle^{-1/2}|k|^{-1}+|b'(\theta(y))||k|^{-1/2},|k|^{-3/2}))\widetilde{M}_k,\\
     &|\partial_y[\mathrm{e}^{\mathrm{i}kat_*b(\theta(y))}\widetilde{\psi}_{k}(at_*,\theta(y))]|\lesssim \langle t\rangle^{-1.2}|k|^{-2.2}\widetilde{M}_k,\\
     &\|(\partial_y,k)^2[\mathrm{e}^{\mathrm{i}kat_*b(\theta(y))}\widetilde{w}_{k}(at_*,\theta(y))]\|_{L^2}\lesssim (|\theta'(y)|^2+|\theta''(y)|+1)|k|^{-1}\widetilde{M}_k.
  \end{align*}
   Recall that $w_{k,1}(t,y)=\widetilde{w}_{k}(at_*,\theta(y))\mathrm{e}^{-\mathrm{i}kdt_*}$ and $\psi_{k,1}(t,y)=\widetilde{\psi}_{k}(at_*,\theta(y))\mathrm{e}^{-\mathrm{i}kdt_*}$. 
   Then the lemma follows from \eqref{est:athetab} and $V(y)=ab(\theta(y))+d.$ and the fact $t_*\sim t$ for $ t\leq T_1$.
      \end{proof}

       Thanks to the definition of $w_{k,1},\ \psi_{k,1}$, we have
     \begin{align*}
     &\partial_tw_{k,1}+\mathrm{i}k\mathrm{e}^{-\nu t} V(w_{k,1}+\psi_{k,1})={\mathrm{i}k\mathrm{e}^{-\nu t}d\psi_{k,1}},\quad \widetilde{\Delta}_k\psi_{k,1}=w_{k,1},\quad 
     w_{k,1}|_{t=0}={w}_{0,k},   
     \end{align*}
      where
       \begin{align*}
    \widetilde{\Delta}_kf=h_1\partial_y(h_1\partial_yf)-k^2f=h_2\partial_y^2f+h_2'\partial_yf/2-k^2f,    
      \end{align*}
    where $ h_1={1/\theta'}, h_2=h_1^2$ and $h_3=h_2-1$ with   
       \begin{align*}
    \|h_3\|_{H^2}=\|h_2-1\|_{H^2}\lesssim\|h_1-1\|_{H^2}\lesssim\|\theta(y)-y\|_{H^3}\lesssim\|V-b\|_{H^4}.
      \end{align*}      
     Let  $\psi_{k,*}=\Delta_k^{-1}w_{k,1}$. We have      
      \begin{equation}\label{eq:wk1-11}
        \left\{\begin{aligned}
         &\partial_tw_{k,1}+\mathrm{i}k\mathrm{e}^{-\nu t} V(w_{k,1}+\psi_{k,*})=\mathrm{i}k\mathrm{e}^{-\nu t}(V(\psi_{k,*}-\psi_{k,1})+d\psi_{k,1}),\\
         & w_{k,1}|_{t=0}={w}_{0,k},\ k\neq 0.
      \end{aligned}
        \right.
      \end{equation}      
      
      The following lemma shows the estimates for $\psi_{k,*}$.
      
        \begin{lemma}\label{lem:psik*}
      Let $\psi_{k,*}=\Delta_k^{-1}w_{k,1}$. If $\|V-b\|_{H^4}\ll \nu^{1/3}$, $|k|\leq \nu^{-\f13}$, then it holds that for $t\le T_1$,
     \begin{align*}
     &|\psi_{k,*}(t,y)-\psi_{k,1}(t,y)|\lesssim \nu^{1/3}\langle t\rangle^{-2}|k|^{-5/2}M_k,\\ &|\psi_{k,*}(t,y)|\lesssim\langle t\rangle^{-2}|k|^{-5/2}M_k,\\
     &|\partial_y\psi_{k,*}(t,y)|\lesssim(\langle t\rangle^{-3/2}|k|^{-2}+\langle t\rangle^{-1}|V'||k|^{-3/2})M_k,\\
     &|\partial_y(\mathrm{e}^{\mathrm{i}kt_*V}\psi_{k,*})(t,y)|\lesssim \langle t\rangle^{-1.2}|k|^{-2.2}M_k.
  \end{align*} 
         \end{lemma}
   
      \begin{proof}

For the operator $\widetilde{\Delta}_k$, we find (as $h_3=h_2-1$)
\begin{align}
      &\widetilde{\Delta}_kf={\Delta}_k(h_2f)-(3/2)\partial_y(h_2'f)+h_2''f/2+k^2(h_2-1)f,\nonumber\\
      &{\Delta}_k^{-1}\widetilde{\Delta}_kf=h_2f-(3/2)\partial_y{\Delta}_k^{-1}(h_3'f)+
      {\Delta}_k^{-1}(h_3''f/2+k^2h_3f).\label{est:De-1Def}
      \end{align}
    Then  we get by Lemma  \ref{lem:comm-1} that
      \begin{align}
        &\|\Delta^{-1}_k\widetilde{\Delta}_kf-f\|_{L^\infty} \nonumber\\
        &\leq \|h_3\|_{L^\infty}\|f\|_{L^\infty} +C\|\partial_y{\Delta}_k^{-1}(h_3'f)\|_{L^\infty}+C \|{\Delta}_k^{-1}(h_3''f/2+k^2h_3f)\|_{L^\infty}\nonumber\\
        &\leq C\|h_3\|_{L^\infty}\|f\|_{L^\infty} +C\|h_3'f\|_{L^1}+C|k|^{-1}\|h_3''f\|_{L^1}+C\|h_3f\|_{L^\infty}\nonumber\\
        &\leq C\|h_3\|_{H^2}\|f\|_{L^\infty}\leq C\nu^{1/3}\|f\|_{L^\infty},\label{est:De-1Def-f}
      \end{align}
   and       
   \begin{align}
         &\|\partial_y{\Delta}_k^{-1}\widetilde{\Delta}_kf-h_2\partial_yf\|_{L^{\infty}}\nonumber\\
          &\leq \|h_3'f\|_{L^
          \infty}+C\|\partial_y^2{\Delta}_k^{-1}(h_3'f)\|_{L^\infty}+C \|\partial_y{\Delta}_k^{-1}(h_3''f/2+k^2h_3f)\|_{L^\infty}\nonumber\\
        &\leq C\|h_3'f\|_{L^\infty}+C|k|^{-1/2}\|h_3''f\|_{L^2}+C|k|\|h_3f\|_{L^\infty}\nonumber\\
        &\leq C|k|\|h_3\|_{H^2}\|f\|_{L^\infty}\leq C\nu^{1/3}|k|\|f\|_{L^\infty}.\label{est:paryDe-1Def-h2f}
      \end{align}
      
        By \eqref{est:De-1Def-f} and \eqref{est:paryDe-1Def-h2f}, we have
      \begin{align*}
      &|\psi_{k,*}(t,y)-\psi_{k,1}(t,y)|\leq \|\Delta_k^{-1}\tilde{\Delta}_k\psi_{k,1}(t,y)-\psi_{k,1}(t,y)\|_{L^\infty}\leq C\nu^{1/3}\|\psi_{k,1}\|_{L^\infty},\\
         & |\psi_{k,*}(t,y)|\leq |\psi_{k,1}(t,y)|+\|\Delta_k^{-1}\tilde{\Delta}_k\psi_{k,1}-\psi_{k,1}\|_{L^\infty}\leq C\|\psi_{k,1}\|_{L^\infty},\\
         & |\partial_y\psi_{k,*}(t,y)|\leq |h_2\partial_y\psi_{k,1}(t,y)|+\|\partial_y\Delta_k^{-1}\tilde{\Delta}_k\psi_{k,1}- h_2\partial_y\psi_{k,1}\|_{L^\infty} \\
         &\quad\qquad\leq C|\partial_y\psi_{k,1}(t,y)|+C\nu^{1/3}|k|\|\psi_{k,1}\|_{L^\infty},\\
         &|\partial_y(\mathrm{e}^{\mathrm{i}kt_*V}\psi_{k,*})(t,y)|\leq |\partial_y(\mathrm{e}^{\mathrm{i}kt_*V}\psi_{k,1})(t,y)|+|k|t_*|V'| \|\Delta_k^{-1}\tilde{\Delta}_k\psi_{k,1}-\psi_{k,1}\|_{L^\infty}\\
         &\quad\qquad+\|\partial_y\Delta_k^{-1}\tilde{\Delta}_k\psi_{k,1}-h_2\partial_y\psi_{k,1} \|_{L^\infty}+\|h_3\partial_y\psi_{k,1} \|_{L^\infty}\\
         &\quad\qquad\leq |\partial_y(\mathrm{e}^{\mathrm{i}kt_*V}\psi_{k,1})(t,y)| +C\nu^{1/3}|k|(t_*|V'|+1)\|\psi_{k,1}\|_{L^\infty} +C\nu^{1/3}\|\partial_y\psi_{k,1}\|_{L^\infty},
      \end{align*}
    which along with Lemma \ref{lem:wk1}  show that  for $t\leq T_1=\nu^{-\f49}$($|k|\leq \nu^{-\f13}$),
        \begin{align*}
     &|\psi_{k,*}(t,y)-\psi_{k,1}(t,y)|\lesssim \nu^{1/3}\langle t\rangle^{-2}|k|^{-5/2}M_k,\\ &|\psi_{k,*}(t,y)|\lesssim\langle t\rangle^{-2}|k|^{-5/2}M_k,\\
     &|\partial_y\psi_{k,*}(t,y)|\lesssim(\langle t\rangle^{-3/2}|k|^{-2}+\langle t\rangle^{-1}|V'||k|^{-3/2}+
     \nu^{1/3}|k|\langle t\rangle^{-2}|k|^{-5/2})M_k,\\
   &|\partial_y(\mathrm{e}^{\mathrm{i}kt_*V}\psi_{k,*})(t,y)|\lesssim (\langle t\rangle^{-1.2}|k|^{-2.2}+
     \nu^{1/3}\langle t\rangle^{-1}|k|^{-3/2})M_k.
      \end{align*} 
  Then the lemma follows by using the facts that 
   \begin{align*}
     &\nu^{1/3}|k|\langle t\rangle^{-2}|k|^{-5/2}\leq \langle t\rangle^{-2}|k|^{-5/2}\leq\langle t\rangle^{-3/2}|k|^{-2},\\
    &\nu^{1/3}\langle t\rangle^{-1}|k|^{-3/2}\leq \nu^{1/3-0.1}\langle t\rangle^{-1.2}|k|^{-3/2}
     \leq |k|^{-3(1/3-0.1)}\langle t\rangle^{-1.2}|k|^{-3/2}=\langle t\rangle^{-1.2}|k|^{-2.2}.
  \end{align*} 
      \end{proof}
      
      Now we derive the estimates for $w_{k,2},\ \psi_{k,2}$ defined in \eqref{def:wk2}.
       \begin{lemma}\label{lem:wk2}
       If $\|V-b\|_{H^4}\ll \nu^{1/3}$, then it holds that for $t\le T_1$,
       \begin{align*}
     &|w_{k,2}(t,y)|\lesssim\langle \nu t^3\rangle^{-1}|k|^{-5/3}M_k,\quad|\partial_y(\mathrm{e}^{\mathrm{i}kt_*V}w_{k,2})(t,y)|\lesssim\langle \nu t^3\rangle^{-1/2}|k|^{-2/3} M_k,\\
      &|\psi_{k,2}(t,y)|\lesssim\langle \nu t^3\rangle^{1/2}\langle t\rangle^{-2}|k|^{-5/2}M_k,\\
     &|\partial_y\psi_{k,2}(t,y)|\lesssim(\langle t\rangle^{-3/2}+\langle t\rangle^{-1}|V'|)|k|^{-3/2}M_k,\\
     &|\partial_y(\mathrm{e}^{\mathrm{i}kt_*V}\psi_{k,2})(t,y)|\lesssim (\langle t\rangle^{-1.2}|k|^{-2.2}+
    |\nu t^3|^{\f12}\langle t\rangle^{-1}|k|^{-3/2})M_k.
      \end{align*}
      \end{lemma}
   
   \begin{proof}
     Recalling that $w_{k,2}=w_{k,1}\mathrm{e}^{-\nu k^2\gamma_1(t)|V'|^2}$ and $\gamma_1(t)\sim t^3$, we get by Lemma \ref{lem:wk1}  that
     \begin{align*}
        |w_{k,2}(t,y)|\leq&\mathrm{e}^{-\nu k^2\gamma_1(t)|V'|^2/2}|w_{k,1}(t,y)|\\ \leq& C\mathrm{e}^{-\nu k^2\gamma_1(t)|V'|^2/2}\min(\langle t\rangle^{-1}|k|^{-\f32}+|V'|^2|k|^{-\f12},|k|^{-\f52})M_k,\\
        \leq& C\min(\langle t\rangle^{-1}|k|^{-\f32}+(\nu k^2t^3)^{-1}|k|^{-\f12},|k|^{-\f52})M_k\\
        \leq &C\langle \nu t^3\rangle^{-1}|k|^{-\f74}M_k\leq C\langle \nu t^3\rangle^{-1}|k|^{-\f53}M_k.
     \end{align*}
     Here we have used the fact that
     \beno
     \min(\langle t\rangle^{-1}|k|^{-\f32},|k|^{-\f52})\leq\langle t\rangle^{-\f34}|k|^{-\f74}\leq C(\nu t^3)^{-1}|k|^{-\f74},\quad \text{for}\, t\leq T_1=\nu^{-\f49}.
     \eeno

  By Lemma \ref{lem:wk1}, we have
     \begin{align*}
       &|\partial_y(\mathrm{e}^{\mathrm{i}kt_*V}w_{k,2})(t,y)| \\
       &\leq C(\nu k^2\gamma_1(t)|V'V''|)|w_{k,2}(t,y)|+C\mathrm{e}^{-\nu k^2\gamma_1(t)|V'|^2}|\partial_y(\mathrm{e}^{\mathrm{i}kt_*V}w_{k,1})(t,y)|\\
       &\leq C\langle\nu k^2t^3\rangle^{\f12}\mathrm{e}^{-\nu k^2\gamma_1(t)|V'|^2/2}|w_{k,1}(t,y)|\\
       &\quad+C\mathrm{e}^{-\nu k^2\gamma_1(t)|V'|^2/2}\min(\langle t\rangle^{-1/2}|k|^{-1}+|V'||k|^{-1/2},|k|^{-3/2})M_k\\
       &\leq C\langle\nu k^2t^3\rangle^{\f12}\langle\nu t^3\rangle^{-1}|k|^{-5/3}M_k+C\min(\langle t\rangle^{-1/2}|k|^{-1}+(\nu k^2t^3)^{-\f12}|k|^{-\f12},|k|^{-\f32})M_k\\
       &\leq C\langle \nu t^3\rangle^{-\f12}|k|^{-\f23}M_k.
     \end{align*}
     
 Let
 \begin{align*}
&g=\mathrm{e}^{-\nu k^2\gamma_1(t)|V'|^2}\quad g_1=\mathrm{e}^{-\mathrm{i}kt_*V}g=\mathrm{e}^{g_2},\quad 
g_2=-\nu k^2\gamma_1(t)|V'|^2-\mathrm{i}kt_*V.
\end{align*}
It is easy to verify that
\begin{align}\label{est:g}
   & \|g'\|_{L^\infty}+\|g''\|_{L^1}\leq C|\nu k^2t^3|^{\f12}, \quad\|g\|_{L^\infty}\leq 1.
\end{align}
Then  by Lemma \ref{lem:psik*}, Lemma \ref{lem:comm-1} and \eqref{est:g}, we obtain
      \begin{align*}
         |\psi_{k,2}(t,y)|\leq& C\|\Delta_k^{-1}(gw_{k,1})\|_{L^\infty}\leq C\|\psi_{k,*}\|_{L^\infty}(\|g\|_{L^\infty}+|k|^{-1}\|g''\|_{L^1})\\
         \leq &C\langle t\rangle^{-2}|k|^{-\f52}M_k(1+|k|^{-1}|\nu k^2t^3|^{\f12})\\
         \leq &C\langle \nu t^3\rangle^{\f12}\langle t\rangle^{-2}|k|^{-\f52}M_k,
      \end{align*}
   and     
    \begin{align*}
        |\partial_y\psi_{k,2}(t,y)|&\leq |\partial_y(g\psi_{k,*})(t,y)| + \|\partial_y(\Delta_k^{-1}(gw_{k,1})-g\Delta_k^{-1}w_{k,1})\|_{L^\infty}\\
        &\leq |\partial_y(\psi_{k,*})(t,y)|\|g\|_{L^\infty} + \|\psi_{k,*}\|_{L^\infty}\|g'\|_{L^\infty}+ C\|\psi_{k,*}\|_{L^\infty}\|g''\|_{L^1}\\
        &\leq C (\langle t\rangle^{-3/2}|k|^{-2}+\langle t\rangle^{-1}|V'||k|^{-3/2})M_k+C|\nu k^2t^3|^{\f12}\langle t\rangle^{-2}|k|^{-\f52}M_k\\
        &\leq C(\langle t\rangle^{-3/2}+\langle t\rangle^{-1}|V'|)|k|^{-3/2}M_k.
     \end{align*}
     Here we have used  $\nu t^2\leq \nu^{\f19}\leq 1$ for $0\leq t\leq T_1=\nu^{-\f49}$.

     By Lemma \ref{lem:psik*}, Lemma \ref{lem:comm-1} and  \eqref{est:g} again,  we have
\begin{align*}
  &|\partial_y(\mathrm{e}^{\mathrm{i}kt_*V}\psi_{k,2})(t,y)|\\
&\leq |\partial_y(\mathrm{e}^{\mathrm{i}kt_*V}g\psi_{k,*})(t,y)|+
|\partial_y(\mathrm{e}^{\mathrm{i}kt_*V}g\Delta_k^{-1}w_{k,1}
-\mathrm{e}^{\mathrm{i}kt_*V}\Delta_k^{-1}(gw_{k,1}))(t,y)|\\
&\leq \|g'\|_{L^\infty}\|\psi_{k,*}\|_{L^\infty}+\|g\|_{L^\infty}
|\partial_y(\mathrm{e}^{\mathrm{i}kt_*V}\psi_{k,*})(t,y)|\\
&\quad+\|\partial_y(g\Delta_k^{-1}w_{k,1}
-\Delta_k^{-1}(gw_{k,1}))\|_{L^\infty}+|kt_*V'|\|g\Delta_k^{-1}w_{k,1}
-\Delta_k^{-1}(gw_{k,1})\|_{L^\infty}\\
&\leq \|g'\|_{L^\infty}\|\psi_{k,*}\|_{L^\infty}+\|g\|_{L^\infty}
|\partial_y(\mathrm{e}^{\mathrm{i}kt_*V}\psi_{k,*})(t,y)|+C\|\psi_{k,*}\|_{L^\infty}\|g''\|_{L^1}\\
&\quad+C|kt||k|^{-1}\|\psi_{k,*}\|_{L^\infty}\|g''\|_{L^1}\\
&\leq C(1+t)|\nu k^2t^3|^{\f12}\langle t\rangle^{-2}|k|^{-\f52}M_k+C \langle t\rangle^{-1.2}|k|^{-2.2}M_k\\
&\leq C (\langle t\rangle^{-1.2}|k|^{-2.2}+
   |\nu t^3|^{\f12}\langle t\rangle^{-1}|k|^{-3/2})M_k.
\end{align*}

This finishes the proof of the lemma.
  \end{proof}
   
    Next  we  show the estimates for $w_{k,3}$ defined in \eqref{def:wk3} .
       \begin{lemma}\label{lem:wk3}
       If $\|V-b\|_{H^4}\ll \nu^{1/3}$, then it holds that for $t\le T_1$,
       \begin{align*}
     &|w_{k,3}(t,y)|\lesssim  \min(\langle t\rangle^{-1}|k|^{-\f12},|k|^{-\f52})M_k,\\
     &|tV'w_{k,3}(t,y)|\lesssim  C\min(\langle t\rangle^{-\f12}|k|^{-\f12},\langle t\rangle^{\f12}|k|^{-\f52})M_k.
      \end{align*}
   \end{lemma}
   
   \begin{proof}
     Thanks to $\eta(\sqrt{t}V')w_{k,2}=\eta(\sqrt{t}V')w_{k,1}\mathrm{e}^{-\nu k^2\gamma_1(t)|V'|^2}$, we get by Lemma \ref{lem:wk1}  that
     \begin{align*}
        |w_{k,3}(t,y)|\leq& C\mathrm{e}^{-\nu k^2\gamma_1(t)|V'|^2}\eta(\sqrt{t}V')\min(\langle t\rangle^{-1}|k|^{-\f32}+|V'|^2|k|^{-\f12},|k|^{-\f52})M_k\\
        \leq& C\min(\langle t\rangle^{-1}|k|^{-\f32}+\langle t\rangle^{-1}|k|^{-\f12},|k|^{-\f52})M_k\\
        \leq &C\min(\langle t\rangle^{-1}|k|^{-\f12},|k|^{-\f52})M_k,
     \end{align*}
  {here we used $ |\eta(x)|\leq C\min(1,|x|^{-2})$} and
     \begin{align*}
        |tV'w_{k,3}(t,y)|\leq& C\mathrm{e}^{-\nu k^2\gamma_1(t)|V'|^2}t|V'|\eta(\sqrt{t}V')\min(\langle t\rangle^{-1}|k|^{-\f32}+|V'|^2|k|^{-\f12},|k|^{-\f52})M_k,\\
        \leq& C\min(\langle t\rangle^{-\f12}|k|^{-\f32}+\langle t\rangle^{-\f12}|k|^{-\f12},\langle t\rangle^{\f12}|k|^{-\f52})M_k\\
        \leq &C\min(\langle t\rangle^{-\f12}|k|^{-\f12},\langle t\rangle^{\f12}|k|^{-\f52})M_k,
     \end{align*}
     {here we used $ |x\eta(x)|\leq C\min(1,|x|^{-2})$}.
    \end{proof}
   
   For $w_{k*}$ defined in \eqref{def:wk*}, we have the following estimates.
   
       \begin{lemma}\label{lem:wk*}
       If $\|V-b\|_{H^4}\ll \nu^{1/3}$, then it holds that   for $ t\leq T_1$,      
        \begin{align*}
     & |w_{k*}(t,y)|\lesssim\min(|V'|^2|k|^{-\f12},|k|^{-\f52})\mathrm{e}^{-\nu k^2\gamma_1(t)|V'|^2}M_k,\\
&|\partial_y(V'\mathrm{e}^{\mathrm{i}kt_*V}w_{k*})(t,y)|\lesssim \min(|V'|^2|k|^{-\f12},|k|^{-\f52})\mathrm{e}^{-\nu k^2\gamma_1(t)|V'|^2/2}M_k,\\
&\|\partial_y(V'\mathrm{e}^{\mathrm{i}kt_*V}w_{k*})\|_{L^2_y}\lesssim |k|^{-\f52}\langle \nu t^3\rangle^{-\f54}
M_k.
     \end{align*}
     \end{lemma}
   
   \begin{proof}
 Recall that  $w_{k*}=(1-\eta(\sqrt{t}V'))w_{k,2}=(1-\eta(\sqrt{t}V'))w_{k,1}\mathrm{e}^{-\nu k^2\gamma_1(t)|V'|^2}.$
            Thanks to $(1-\eta(\sqrt{t}V'))\leq C\min({t}|V'|^2,1)$, we get by Lemma \ref{lem:wk1}  that 
     \begin{align}\label{est:parV'wk*-0}
        &|w_{k*}(t,y)|\\
        &\leq C\min({t}|V'|^2,1)\min(\langle t\rangle^{-1}|k|^{-\f32}+|V'|^2|k|^{-\f12},|k|^{-\f52})M_k\mathrm{e}^{-\nu k^2\gamma_1(t)|V'|^2},\nonumber\\
        &\leq C\min(|V'|^2|k|^{-\f12},|k|^{-\f52})\mathrm{e}^{-\nu k^2\gamma_1(t)|V'|^2}M_k,    \nonumber
         \end{align}
     which gives the first inequality of the lemma.  By  Lemma \ref{lem:wk1} again, we have
       \begin{align*}
        &\mathrm{e}^{\nu k^2\gamma_1(t)|V'|^2}|\partial_y(V'\mathrm{e}^{\mathrm{i}kt_*V}w_{k*})(t,y)|\\
        &\leq C((\nu k^2\gamma_1(t)|V'|^2|V''|+|V''|)(1-\eta(\sqrt{t}V'))+\sqrt{t}|V'V''||\eta'(\sqrt{t}V')|)|w_{k,1}(t,y)|\\
        &\quad+ C|V'|(1-\eta(\sqrt{t}V'))|\partial_y(\mathrm{e}^{\mathrm{i}kt_*V}w_{k,1})(t,y)|\\
        &\leq C|V''|((\nu k^2\gamma_1(t)|V'|^2+1)\min({t}|V'|^2,1)+\sqrt{t}|V'|\min(|\sqrt{t}V'|,|\sqrt{t}V'|^{-1}))|w_{k,1}(t,y)|\\
        &\quad+ C|V'|\min(|\sqrt{t}V'|,1)|\partial_y(\mathrm{e}^{\mathrm{i}kt_*V}w_{k,1})(t,y)|\\
        &\leq C(\nu k^2\gamma_1(t)|V'|^2+1)\min({t}|V'|^2,1)|w_{k,1}(t,y)|\\
        &\quad+C|V'|\min(|\sqrt{t}V'|,1)|\partial_y(\mathrm{e}^{\mathrm{i}kt_*V}w_{k,1})(t,y)|\\
        &\leq C(\nu k^2\gamma_1(t)|V'|^2+1)\min({t}|V'|^2,1)\min(\langle t\rangle^{-1}|k|^{-\f32}+|V'|^2|k|^{-\f12},|k|^{-\f52})M_k\\
        &\quad+ C|V'|\min(|\sqrt{t}V'|,1)\min(\langle t\rangle^{-\f12}|k|^{-1}+|V'||k|^{-\f12},|k|^{-\f32})M_k\\
        &\leq C(\nu k^2\gamma_1(t)|V'|^2+1)\min(|V'|^2|k|^{-\f12},|k|^{-\f52})M_k\\&\quad+ C|V'|\min(|V'||k|^{-\f12},|k|^{-\f32})M_k\\
        &\leq C(\nu k^2\gamma_1(t)|V'|^2+1)\min(|V'|^2|k|^{-\f12},|k|^{-\f52})M_k\\
        &\leq C\mathrm{e}^{\nu k^2\gamma_1(t)|V'|^2/2}\min(|V'|^2|k|^{-\f12},|k|^{-\f52})M_k,
     \end{align*}
     which gives the second inequality of the lemma. Thus, we get by using $\gamma_1(t)\sim t^3 $ that 
     \begin{align*}
        &\|\partial_y(V'\mathrm{e}^{\mathrm{i}kt_*V}w_{k*})\|_{L^2_y}\\
        &\leq C\min(\||V'|^2|k|^{-\f12}\mathrm{e}^{-\nu k^2\gamma_1(t)|V'|^2/2}\|_{L^2_y},|k|^{-\f52})M_k\\
        &\leq C\min(|k|^{-\f12}|\nu k^2 t^3|^{-\f54},|k|^{-\f52})M_k\leq C|k|^{-\f52}\langle \nu t^3\rangle^{-\f54}M_k,
     \end{align*}
     which shows the third inequality of the lemma.
     \end{proof}

     \subsection{Proof of Proposition \ref{prop:app}}

\begin{proof}
  By Lemma \ref{lem:wk2}, $\omega_L(t,x,y)=\sum_{k\in \Lambda_*}w_{k,2}(t,y)\mathrm{e}^{\mathrm{i}kx}$, \eqref{def:u1L} and Lemma \ref{lem:comm-1}, we have
  \begin{align*}
     & \langle \nu t^3\rangle\big(\|\partial_x\omega_L(t)\|_{L^\infty}+ \|\nabla u_L(t)\|_{L^\infty}\big)+\langle \nu t^3\rangle^{1/2}\|(\partial_y+t_*V'\partial_x)\omega_L(t)\|_{L^\infty}\\
     &\leq \langle \nu t^3\rangle\sum_{k\in\Lambda_{*}}\big(|k|\|w_{k,2}(t)\|_{L^\infty}+\|(\partial_y,k)^2\psi_{k,2}(t)\|_{L^\infty}\big)+ \langle \nu t^3\rangle^{1/2}\sum_{k\in\Lambda_{*}}\|\partial_y(\mathrm{e}^{\mathrm{i}kt_*V} w_{k,2})(t)\|_{L^\infty}\\
     &\leq C\sum_{k\in\Lambda_{*}}(|k|^{-2/3}M_k)\leq C\Big(\sum_{k\in\Lambda_{*}}|k|^{-4/3}\Big)^{1/2}\Big(\sum_{k\in\Lambda_{*}}M_k^2\Big)^{1/2}\\
     &\leq C\Big(\sum_{k\in\Lambda_{*}}M_k^2\Big)^{1/2}\leq C\|\omega_0\|_{H^3}.
  \end{align*}
  This shows  \eqref{est:paromL-inf}.
  
  By \eqref{def:u1L} and Lemma \ref{lem:wk2}, we get
  \begin{align*}
     &\|(u^x_L-t_*V'u^y_L)(t)\|_{L^{\infty}}\leq \sum_{k\in\Lambda_{*}}
     \|\partial_y(\mathrm{e}^{\mathrm{i}kt_*V}\psi_{k,2})(t,y)\|_{L^\infty}\\
     &\leq C\sum_{k\in\Lambda_{*}}\big(\big(\langle t\rangle^{-1.2}|k|^{-2.2}+
    |\nu t^3|^{\f12}\langle t\rangle^{-1}|k|^{-3/2}\big)M_k\big)\\
    &\leq C\big(\langle t\rangle^{-1.2}+
    |\nu t^3|^{\f12}\langle t\rangle^{-1}\big)\sum_{k\in\Lambda_{*}}
    \big((|k|^{-2.2}+|k|^{-3/2})M_k\big)\\
    &\leq C\big(\langle t\rangle^{-1.2}+
    |\nu t^3|^{\f12}\langle t\rangle^{-1}\big)\Big(\sum_{k\in\Lambda_{*}}(|k|^{-4.4}+|k|^{-3})\Big)^{1/2}\Big(\sum_{k\in\Lambda_{*}}M_k^2\Big)^{1/2}\\
    &\leq C\big(\langle t\rangle^{-1.2}+
    |\nu t^3|^{\f12}\langle t\rangle^{-1}\big)\|\omega_{0}\|_{H^3},
  \end{align*}
  which gives \eqref{est:u2L2}.

  By \eqref{def:u1L} and Lemma \ref{lem:wk2} again, we have
  \begin{align*}
     & |u_L^x(t,x,y)|\leq \sum_{k\in\Lambda_{*}}|\partial_y\psi_{k,2}(t,y)|\leq C\sum_{k\in\Lambda_{*}}\big(\langle t\rangle^{-3/2}+\langle t\rangle^{-1}|V'|\big)|k|^{-3/2}M_k\\
     &\leq C\big(\langle t\rangle^{-3/2}+\langle t\rangle^{-1}|V'|\big)
     \Big(\sum_{k\in\Lambda_{*}}|k|^{-3}\Big)^{1/2}\Big(\sum_{k\in\Lambda_{*}}M_k^2\Big)^{1/2}\\
     &\leq C\big(\langle t\rangle^{-3/2}+\langle t\rangle^{-1}|V'|\big)\|\omega_0\|_{H^3},
  \end{align*}
  which gives \eqref{est:u1Linf}. Similarly, we have
 \begin{align*}
  &\|u^y_L(t)\|_{L^{\infty}} \leq \sum_{k\in\Lambda_{*}}|k\psi_{k,2}(t,y)|\leq C\sum_{k\in\Lambda_{*}}\langle \nu t^3\rangle^{1/2}\langle t\rangle^{-2}|k|^{-3/2}M_k\\
     &\leq C\langle \nu t^3\rangle^{1/2}\langle t\rangle^{-2}
     \Big(\sum_{k\in\Lambda_{*}}|k|^{-3}\Big)^{1/2}\Big(\sum_{k\in\Lambda_{*}}M_k^2\Big)^{1/2}\leq C\langle \nu t^3\rangle^{1/2}\langle t\rangle^{-2}\|\omega_{0}\|_{H^3}.
  \end{align*}
  This shows \eqref{est:u2Linf}.\smallskip
  
  Notice that
  \begin{align*}
     & u_L\cdot\nabla\omega_L= (u^x_L-t_{*}V'u^y_{L})\partial_x\omega_{L}+u^y_{L}(\partial_y+tV'\partial_x)\omega_L.
  \end{align*}
  Then by \eqref{est:paromL-inf}, \eqref{est:u2L2} and \eqref{est:u2Linf},  we get
  \begin{align*}
     \|u_L\cdot\nabla\omega_L\|_{{L^\infty}}\leq & C\|u^x_L-t_{*}V'u^y_{L}\|_{L^\infty}
     \|\partial_x\omega_{L}\|_{L^\infty}+\|u^y_{L}\|_{L^\infty}
     \|(\partial_y+tV'\partial_x)\omega_L\|_{L^\infty}\\
     \leq &C\big(\langle t\rangle^{-1.2}+| \nu t^3|^{1/2}\langle t\rangle^{-1}\big)\langle\nu t^3\rangle^{-1}\|\omega_0\|_{H^3}^2+C\langle t\rangle^{-2}\|\omega_0\|_{H^3}^2\\
     \leq& C\big(\langle t\rangle^{-1.2}+| \nu t^3|^{1/2}\langle \nu t^3\rangle^{-1}\langle t\rangle^{-1}\big)\|\omega_{0}\|_{H^3}^2.
  \end{align*}
  This proves \eqref{est:uLnawL}.\smallskip
  
  By \eqref{def:tilomL}, Lemma \ref{lem:wk2} and Lemma \ref{lem:wk3}, we have
  \begin{align*}
     &\|t_*V'\partial_x\tilde{\omega}_L(t)\|_{L^\infty}+t^{-1}\langle \nu t^3\rangle\|t_*V'\partial_x{\omega}_L(t)\|_{L^\infty}\\
&\leq \sum_{k\in\Lambda_*}\big(\|t_*kV'w_{k,3}(t)\|_{L^\infty}+t^{-1}\langle \nu t^3\rangle\|t_*kV'w_{k,2}(t)\|_{L^\infty}\big)\\
&\leq C\sum_{k\in\Lambda_*}\big( \min(\langle t\rangle^{-\f12}|k|^{\f12},\langle t\rangle^{\f12}|k|^{-\f32})+|k|^{-\f23}\big)M_k\\
&\leq C\Big(\sum_{k\in\Lambda_{*}}(\min(\langle t\rangle^{-1}|k|,\langle t\rangle|k|^{-3})+|k|^{-\f43})\Big)^{\f12}\Big(\sum_{k\in\Lambda_{*}}M_k^2\Big)^{\f12}\\
&\leq C\Big(\sum_{k\in\Lambda_{*}}M_k^2\Big)^{\f12}\leq C\|\omega_0\|_{H^3},
  \end{align*}
  here we used $\sum_{k\in\Lambda_{*}}\min(\langle t\rangle^{-1}|k|,\langle t\rangle|k|^{-3})\leq C$. This gives \eqref{est:t*Vparw}.
  \end{proof}

\subsection{Proof of Proposition \ref{prop:error}}
Recall $w_{k,2}=w_{k,1}\mathrm{e}^{-\nu k^2\gamma_1(t)|V'|^2}=g w_{k,1}$ with\\
    $g=\mathrm{e}^{-\nu k^2\gamma_1(t)|V'|^2}$. Then we have
      \begin{align*}
         \partial_tw_{k,2}&=g(\partial_tw_{k,1}-\nu k^2t_*^2|V'|^2w_{k,1})=g\partial_tw_{k,1}-\nu k^2t_*^2|V'|^2w_{k,2},
      \end{align*}
     which along with \eqref{eq:wk1-11} gives
      \begin{align}\label{eq:wk2-11}
         &\partial_tw_{k,2}+\nu k^2 t_*^2|V'|^2w_{k,2}+\mathrm{i}k\mathrm{e}^{-\nu t}V(w_{k,2}+g\psi_{k,*})
         =\mathrm{i}k\mathrm{e}^{-\nu t}g (V(\psi_{k,*}-\psi_{k,1})+d\psi_{k,1}).
      \end{align}
      We denote
      \begin{align*}
 &W_{k}=\partial_tw_{k,2}-\nu\Delta_kw_{k,2}+\mathrm{i}k\mathrm{e}^{-\nu t}V(w_{k,2}+\psi_{k,2})=W_{k,1}+W_{k,2}+W_{k,3},
 \end{align*}      
 where     
 \begin{align*}
 &W_{k,1}=-\nu k^2t_*^2|V'|^2w_{k,2}-\nu \Delta_kw_{k,2},\\
&W_{k,2}=\mathrm{i}k\mathrm{e}^{-\nu t}(V(\psi_{k,*}-\psi_{k,1})g+{gd\psi_{k,1}}),\\
&W_{k,3}=\mathrm{i}k\mathrm{e}^{-\nu t}V(\psi_{k,2}-\psi_{k,*}g).
\end{align*}

\begin{lemma}\label{lem:Wkj}
  If $\|V-b\|_{H^4}\ll \nu^{1/3}$, then it holds that for $ t\leq T_1$, 
  \begin{align*}
&\|W_{k,1}\|_{L^2}\leq C\nu(1+t\langle \nu t^3\rangle^{-1})M_k,\\
&\|W_{k,2}\|_{L^2}\leq C\nu^{1/3}\langle t\rangle^{-2}M_k,\\
&\|(\partial_y,k)W_{k,3}\|_{L^{\infty}} \leq C|k|^{-1/2}\langle t\rangle^{-2}|\nu t^3|^{1/2}M_k.
\end{align*}
\end{lemma}

\begin{proof}
We first estimate $W_{k,1}$. Let  $w_{k,4}=\mathrm{e}^{\mathrm{i}kt_*V}w_{k,1}$. Then
\begin{align*}
& w_{k,2}=w_{k,4}g_1,\quad g_1=\mathrm{e}^{-\mathrm{i}kt_*V}g=\mathrm{e}^{g_2},\quad 
g_2=-\nu k^2\gamma_1(t)|V'|^2-\mathrm{i}kt_*V,
\end{align*}
and
\begin{align*}
&\Delta_kw_{k,2}=\big[\Delta_kw_{k,4}+2\partial_yw_{k,4}\partial_yg_2+w_{k,4}\big(\partial_y^2g_2+(\partial_yg_2)^2\big)\big]g_1,\\
&W_{k,1}=-\nu\big(\Delta_kw_{k,4}+2\partial_yw_{k,4}\partial_yg_2+w_{k,4}g_3\big)g_1,\\
&g_3=\partial_y^2g_2+(\partial_yg_2)^2+k^2t_*^2|V'|^2.
\end{align*}
Thanks to the definitions of $g_2$ and $g_3$, it is easy to verify that
 \begin{align*}
&|\partial_yg_2|\lesssim\nu k^2t^3|V'|+|ktV'|,\quad |\partial_y^2g_2|\lesssim \nu k^2t^3+|kt|,\\
&|\partial_yg_2+\mathrm{i}kt_*V'|=2\nu k^2\gamma_1(t)|V'V''|\lesssim \nu k^2t^3|V'|,\\
&|(\partial_yg_2)^2+k^2t_*^2|V'|^2|\lesssim(\nu k^2t^3+|kt|)\nu k^2t^3|V'|^2,\\
&|g_3|\lesssim(\nu k^2t^3+|kt|)(1+\nu k^2t^3|V'|^2).\end{align*}
Then we infer that
\begin{align*}
  &\|W_{k,1}\|_{L^2} \leq \nu\big(\|\Delta_kw_{k,4}\|_{L^2}\|g_1\|_{L^\infty}+2\|\partial_yw_{k,4}\partial_yg_2g_1\|_{L^2}
  +\|w_{k,4}g_3g_1\|_{L^2}\big)\\
\leq& C\nu\big(\|\Delta_kw_{k,4}\|_{L^2}+(\nu k^2t^3+|kt|)( \|V'\partial_yw_{k,4}\cdot g_1\|_{L^2}+\|(1+\nu k^2t^3|V'|^2)w_{k,4}g_1\|_{L^2})\big).
\end{align*}

By Lemma \ref{lem:wk1} and $|g_1|=|\mathrm{e}^{g_2}|\leq \mathrm{e}^{-\nu k^2\gamma_1(t)|V'|^2}$, we have \begin{align*}
&\|(\partial_y,k)^2w_{k,4}\|_{L^2}\leq C|k|^{-1}M_k,
\end{align*}
and (as $|w_{k,4}|=|w_{k,1}|$, $|g|=|g_{1}|$)
\begin{align*}
&\|(1+\nu k^2t^3|V'|^2)w_{k,4}g_1\|_{L^2}+\|V'\partial_yw_{k,4}\cdot g_1\|_{L^2}\\
&\leq C\|\mathrm{e}^{-\nu k^2\gamma_1(t)|V'|^2/2}w_{k,1}\|_{L^2}+ C\|(V'g)\partial_yw_{k,4}\|_{L^2}\\
&\leq C\|(\mathrm{e}^{-\nu k^2\gamma_1(t)|V'|^2/2}\langle t\rangle^{-1}|k|^{-\f32}+(\nu t^3)^{-1}|k|^{-\f52})\|_{L^2}M_k\\
&\quad+C\|(V'g\langle t\rangle^{-1/2}|k|^{-1}+(\nu t^3)^{-1}|k|^{-\f52})\|_{L^2}M_k\\
&\leq C\big((\nu k^2t^3)^{-\f14}\langle t\rangle^{-1}|k|^{-\f32}+(\nu t^3)^{-1}|k|^{-\f52}\big)M_k\\
&\quad+C\big(( \nu k^2t^3)^{-\f12}\langle t\rangle^{-1/2}|k|^{-1}+(\nu t^3)^{-1}|k|^{-\f52}\big)M_k\\
&\leq C|k|^{-2}( \nu t^3)^{-1}M_k.
\end{align*}
Here we used the facts that
\beno
 \|\mathrm{e}^{-\nu k^2\gamma_1(t)|V'|^2/2}\|_{L^2}\leq C(\nu k^2t^3)^{-\f14},\quad \|V'g\|_{L^2}\leq C( \nu k^2t^3)^{-\f12},\quad \nu t^{9/4}\leq 1.
 \eeno
On the other hand, we have
\begin{align*}
&\|(1+\nu k^2t^3|V'|^2)w_{k,4}g_1\|_{L^2}+\|V'\partial_yw_{k,4}\cdot g_1\|_{L^2}\\
&\leq C(\|w_{k,4}\|_{L^2}+ \|\partial_yw_{k,4}\|_{L^2})\leq C|k|^{-1}\||\Delta_kw_{k,4}\|_{L^2}
\leq C|k|^{-2}M_k.
\end{align*}
Thus, we obtain
 \begin{align*}
&\|(1+\nu k^2t^3|V'|^2)w_{k,4}g_1\|_{L^2}+\|V'\partial_yw_{k,4}\cdot g_1\|_{L^2}
\leq C|k|^{-2}\langle \nu t^3\rangle^{-1}M_k.
\end{align*}
Summing up,  we conclude 
\begin{align*}
   \|W_{k,1}\|_{L^2}\leq C\nu\big(|k|^{-1}M_k+(\nu k^2t^3+|kt|)|k|^{-2}\langle \nu t^3\rangle^{-1}M_k\big)
\leq C\nu(1+t\langle \nu t^3\rangle^{-1})M_k.
\end{align*}
  
 Thanks to the definition of $W_{k,2}$, we get by Lemma \ref{lem:wk1}, Lemma \ref{lem:psik*}, and $|d|\leq C\|V-b\|_{H^4}\leq C\nu^{1/3}$ that
\begin{align*}
    \|W_{k,2}\|_{L^2}\leq &C|k|\|\psi_{k,*}-\psi_{k,1}\|_{L^{\infty}}+C|k||d|\|\psi_{k,1}\|_{L^2}\leq C\nu^{1/3}\langle t\rangle^{-2}M_k.
\end{align*}

Thanks to  the definition of $W_{k,3}$, we get by Lemma \ref{lem:psik*} and Lemma \ref{lem:comm-1} that
\begin{align*}
\|(\partial_y,k)W_{k,3}\|_{L^{\infty}}&\leq C|k|\|(\partial_y,k)(\psi_{k,2}-\psi_{k,*}g)\|_{L^{\infty}}
\leq C|k|\|\psi_{k,*}\|_{L^{\infty}}\|g''\|_{L^{1}}\\
&\leq C|k|\langle t\rangle^{-2}|k|^{-5/2}M_k|\nu t^3|^{1/2}|k|=C|k|^{-1/2}\langle t\rangle^{-2}|\nu t^3|^{1/2}M_k.
\end{align*}

This finishes the proof of  the lemma.
\end{proof}

Now we are in a position to prove Proposition \ref{prop:error}.   
   
\begin{proof}
We decompose $Er_L$ as
\begin{align}\label{def:ErL=ErL123}
    &Er_{L}=Er_{L1}+Er_{L2}+Er_{L3},\quad Er_{Lj}=\sum_{k\in \Lambda_*}W_{k,j}(t,y)\mathrm{e}^{\mathrm{i}kx}.
\end{align}
Then the first two inequalities of the lemma follows from Lemma \ref{lem:Wkj} and Plancherel's formula. 
   
   For the third inequality of the lemma, we get by Lemma \ref{lem:Wkj} that
   \begin{align*}
       \||D_x|^{\f12}\nabla Er_{L3}(t)\|_{L^2}^2 =&C\sum_{k\in \Lambda_*}|k|\|(\partial_y,k)W_{k,3}(t)\|_{L^2}^2\\
      \leq&  C\sum_{k\in \Lambda_*}|k|\|(\partial_y,k)W_{k,3}(t)\|_{L^\infty}^2\\
      \leq &C\langle t\rangle^{-4}|\nu t^3|\sum_{k\in\Lambda_*}M_k^2 \leq C\langle t\rangle^{-4}|\nu t^3|\|\omega_0\|_{H^3}^2.
   \end{align*}
   \end{proof}

\section{Nonlinear energy estimates}\label{sec:nonlinearenergy}

This section is devoted to the proof of energy estimate in different timescales.

\subsection{Energy estimate in short-time regime}\label{Energy-short-regime}

In this subsection, we prove Proposition \ref{prop:om-e-short}. The proof is based on a direct energy estimate under new inner product  $\langle \ ,\ \rangle_{\star}$  defined by
  \begin{align}\label{def:norm-*}
     & \langle f,g\rangle_\star=\langle P_{\neq} f,P_{\neq}g+\Delta^{-1}P_{\neq}g\rangle +\langle P_0f,P_0g\rangle= \langle f,g+\Delta^{-1}P_{\neq}g\rangle,
  \end{align}
  where
  \begin{align*}
     & (P_0f)(y):=\dfrac{1}{\mathfrak{p}}\int_{\mathbb{T}_{\mathfrak{p}}}f(x,y)\mathrm{d}x,\quad (P_{\neq}f)(x,y):= f(x,y)-(P_0f)(y).
  \end{align*}
 It is easy to see that 
 \begin{align*}
    & (1-(\mathfrak{p}/(2\pi))^2)\|P_{\neq}f\|_{L^2}^2\leq \|P_{\neq}f\|_{\star}^2\leq \|P_{\neq}f\|_{L^2}^2,\\
    &\|f\|_{\star}^2=\|P_{\neq }f\|_{\star}^2+\|P_0f\|_{L^2}^2.
 \end{align*}  which gives $\|P_{\neq }f\|_{L^2}\sim \|P_{\neq}f\|_{\star}$ and $\|f\|_{L^2}\sim \|f\|_{\star}$ for $\mathfrak{p}\in( 0,2\pi)$.

 \begin{lemma}\label{lem:short time Er}
 Let $Er$ be defined in  \eqref{def:Er}. For $0\leq t\leq T_0$, we have
   \begin{align*}
      & \int_{0}^{t}\|Er(s)\|_{L^2}\mathrm{d}s\leq C\nu^{1/3}\|\omega_{0}\|_{H^3}.
   \end{align*}
 \end{lemma}
 \begin{proof}
   We first estimate $Er_L$. By Proposition \ref{prop:error}, we have  
 \begin{align*}
    \int_{0}^{t}\|Er_L(s)\|_{L^2}\mathrm{d}s \leq & \int_{0}^{t}\big(\|Er_{L1}(s)\|_{L^2}+ \|Er_{L2}(s)\|_{L^2}+ \||D_x|^{\f12}\nabla Er_{L3}(s)\|_{L^2}\big)\mathrm{d}s\\
    \leq &C\|\omega_0\|_{H^3}\int_{0}^{t}\big(\nu \langle s\rangle \langle \nu s^3\rangle^{-1}+\nu^{1/3}\langle s\rangle^{-2}+\langle s\rangle^{-2}|\nu s^3|^{1/2}\big)\mathrm{d}s\\
    \leq &C\|\omega_0\|_{H^3}\big(\nu \langle t\rangle^2+\nu^{1/3} +\nu^{1/2}\langle t\rangle^{1/2}\big)
    \leq C\nu^{1/3}\|\omega_0\|_{H^3}.
 \end{align*}
 
 By \eqref{est:u2Linf} in Proposition \ref{prop:app} and $u_L^y=-\partial_x\phi_L$, we have
  \begin{align*}
     & \int_{0}^{t}\big(\|\mathrm{e}^{-\nu s}(V''+V)\partial_x\phi_L(s)\|_{L^2}+\nu\|\mathrm{e}^{-\nu s}(V'+V''')\|_{L^2}\big)\mathrm{d}s\\
     &\leq  \int_{0}^{t}\big(\|\mathrm{e}^{-\nu s}(V''+V)\|_{L^2}\|\partial_x\phi_L(s)\|_{L^\infty}+\nu\|(V'+V''')\|_{L^2}\big)\mathrm{d}s\\
     &\leq \int_{0}^{t}\big(\langle \nu s^3\rangle^{1/2}\langle s\rangle^{-2}\|\omega_0\|_{H^3}^2+\nu\|\omega_0\|_{H^3}\big)\mathrm{d}s\\
     &\leq C\|\omega_0\|_{H^3}^2+C\nu t\|\omega_0\|_{H^3}\leq C\big(\nu^{1/3}\|\omega_{0}\|_{H^3}+\|\omega_0\|_{H^3}^2\big),
  \end{align*}
and by \eqref{est:uLnawL},
  \begin{align*}
\int_{0}^{t}\| u_L\cdot\nabla\omega_L(s)\|_{L^2}\mathrm{d}s\leq& C\int_{0}^{t}\big(\langle s\rangle^{-1.2}+|\nu s^3|^{1/2}\langle \nu s^3\rangle^{-1}\langle {s}\rangle^{-1}\big)\mathrm{d}s\|\omega_0\|_{H^3}^2\\
\leq &C\big(1+\nu^{1/2}\langle t\rangle^{1/2}\big)\|\omega_0\|_{H^3}^2\leq C\|\omega_0\|_{H^3}^2.
  \end{align*}
  
  Thus, we conclude 
  \begin{align*}
     & \int_{0}^{t}\|Er(s)\|_{L^2}\mathrm{d}s\leq C\big(\nu^{1/3}\|\omega_{0}\|_{H^3}+\|\omega_0\|_{H^3}^2\big)\leq C\nu^{1/3}\|\omega_{0}\|_{H^3}.
  \end{align*}
  This finishes the proof of this lemma.
 \end{proof}

 \begin{proof}[Proof of the Proposition \ref{prop:om-e-short}]
  
  Thanks to $P_0(\mathrm{e}^{-\nu t}V\partial_x(\omega_e+\phi_e))=0$, we obtain 
   \begin{align*}
      &\big\langle \mathrm{e}^{-\nu t}V\partial_x(\omega_e+\phi_e),\omega_e \big\rangle_\star=\big \langle \mathrm{e}^{-\nu t}V\partial_xP_{\neq}(\omega_e+\phi_e),P_{\neq}\omega_e +\Delta^{-1}P_{\neq}\omega_e \big\rangle=0, \\
      &\big\langle -\nu\Delta \omega_e ,\omega_e\big\rangle_\star =\nu\big(\|\nabla P_{\neq}\omega_e\|_{L^2}^2-\|P_{\neq }\omega_e\|_{L^2}^2\big)+\nu\|\partial_yP_0\omega_e\|_{L^2}^2,\\
      &\big\langle u\cdot\nabla\omega_e,\omega_e\big\rangle_\star=\big\langle u\cdot\nabla \omega_e,\omega_e+P_{\neq}\phi_e\big\rangle =\big\langle u\cdot\nabla \omega_e,P_{\neq}\phi_e\big\rangle =-\big\langle u\omega_e,\nabla P_{\neq}\phi_e\big\rangle,\\
      &\big\langle \partial_t\omega_e,\omega_e\big\rangle_\star=\big\langle \partial_tP_{\neq}\omega_e,P_{\neq }\omega_e+P_{\neq}\phi_e\big\rangle+\big\langle\partial_tP_{0}\omega_e,P_{0 }\omega_e\big\rangle\\
      &\quad= \dfrac{1}{2}\dfrac{\mathrm{d}}{\mathrm{d}t}\big(\|P_{\neq }\omega_e\|_{L^2}^2-\|\nabla P_{\neq}\phi_{e}\|_{L^2}^2+\|P_0\omega\|_{L^2}^2\big)= \dfrac{1}{2}\dfrac{\mathrm{d}}{\mathrm{d}t}\|\omega_e\|_{\star}^2.
   \end{align*} 
   Then we get by taking the inner product $\langle ,\rangle_\star$ between  \eqref{eq:ome} and $\omega_e$ that
   \begin{align*}
      &\dfrac{1}{2}\dfrac{\mathrm{d}}{\mathrm{d}t}\|\omega_e\|_{\star}^2+\nu\big(\|\nabla P_{\neq}\omega_e\|_{L^2}^2-\|P_{\neq }\omega_e\|_{L^2}^2\big)+\nu\|\partial_yP_0\omega_e\|_{L^2}^2\\
&=\big\langle u\omega_e,\nabla P_{\neq}\phi_e\big\rangle+\big\langle\mathrm{e}^{-\nu t}(V''+V)\partial_x\phi_e-u_e\cdot\nabla \omega_L-Er ,\omega_e\big\rangle_\star.
   \end{align*}

Thanks to $\|\nabla P_{\neq}\omega_e\|_{L^2}^2-\|P_{\neq }\omega_e\|_{L^2}^2+\|\partial_y P_0\omega_e\|_{L^2}^2\geq (1-(2\pi/\mathfrak{p})^{2})\|\nabla \omega_e\|_{L^2}^2\geq 0$, we conclude
   \begin{align}\nonumber
      \dfrac{1}{2}\dfrac{\mathrm{d}}{\mathrm{d}t}\|\omega_e\|_{\star}^2 \leq& \|u\|_{L^4}\|\omega_e\|_{L^2}\|P_{\neq}\nabla\phi_e\|_{L^4}\\
      &+ \big|\big\langle\mathrm{e}^{-\nu t}(V''+V)\partial_x\phi_e-u_e\cdot\nabla \omega_L-Er ,\omega_e\big\rangle_\star\big|,\nonumber
   \end{align}
from which and  $H^{1/2}\hookrightarrow L^4$, we infer that 
   \begin{align*}
\dfrac{\mathrm{d}}{\mathrm{d}t}\|\omega_e\|_{\star}& 
       \leq C\|u\|_{L^4}\|P_{\neq}\nabla\phi_e\|_{L^4}+\|\mathrm{e}^{-\nu t}(V''+V)\partial_x\phi_e-u_e\cdot\nabla \omega_L-Er \|_{\star} \\
       &\leq  C\|u\|_{L^4}\|P_{\neq }\nabla\phi_e\|_{H^{1/2}}+\|\mathrm{e}^{-\nu t}(V''+V)\partial_x\phi_e-u_e\cdot\nabla \omega_L-Er \|_{L^2}\\
       &\leq C\|u\|_{L^4}\|\omega_e\|_{L^2}+\|(V''+V)\|_{L^\infty}\|\partial_x\phi_e\|_{L^2}+\|u_e\|_{L^2}\|\nabla \omega_L\|_{L^\infty}+\|Er \|_{L^2}.
   \end{align*}
 As $u_e=(-\partial_y,\partial_x)\phi_e$ and $\phi_e=\Delta^{-1}\omega_e$, we have $\|u_e\|_{L^2}=\|\nabla \phi_e\|_{L^2}\leq C\|\omega_e\|_{L^2}$ and $\|\omega_e\|_{L^2}\leq C\|\omega\|_{\star}$. Then we get by Gronwall's inequality that
 \begin{align}\label{est:shortwe}
    \|\omega_e(t)\|_{L^2}&\leq C\|\omega_e(t)\|_{\star}\\
    &\leq C\exp\Big(C\int_{0}^{t}\big(\|u(s)\|_{L^4}+\|(V''+V)\|_{L^\infty} +\|\nabla \omega_L(s)\|_{L^\infty}\big)\mathrm{d}s\Big)\nonumber \\ &\quad\times\left(\|\omega_e(0)\|_{\star}+\int_{0}^{t}\|Er(s)\|_{L^2}\mathrm{d}s\right).\nonumber
 \end{align}
 Now we have
   \begin{align}\label{est:shortwe-in}
      \|\omega_e(0)\|_{\star}&\leq C\|\omega_e(0)\|_{L^2}=C\|\omega_0-\omega_L(0)\|_{L^2}=C\|\mathrm{P}_{\text{high}}\omega_0\|_{L^2}  \leq C\nu^{\f13}\|\partial_x\omega_0\|_{L^2},
   \end{align}
 and by $H^{1/2}\hookrightarrow L^4$,
 \begin{align}
     \int_{0}^{t}\|u(s)\|_{L^4}\mathrm{d}s\leq& \int_{0}^{t}\big(\|u_L(s)\|_{L^4}+\|u_e(s)\|_{L^4}\big)\mathrm{d}s\leq C\int_{0}^{t}\big(\|u_L(s)\|_{H^{1/2}}+\|u_e(s)\|_{H^{1/2}}\big)\mathrm{d}s\nonumber\\
    \leq & C\int_{0}^{t}\big(\|\nabla \omega_L(s)\|_{L^\infty}+\|\omega_e(s)\|_{L^2}\big)\mathrm{d}s \nonumber\\
    \leq& C\int_{0}^{t}\|\nabla \omega_L(s)\|_{L^\infty}\mathrm{d}s +Ct\|\omega_e\|_{L^\infty(0,t;L^2)}.\label{est:usL4-1}
 \end{align}
 By \eqref{est:paromL-inf} in Proposition \ref{prop:app}, we have
 \begin{align*}
    & \int_{0}^{t}(\|(V''+V)\|_{L^\infty}+\|\nabla \omega_L(s)\|_{L^\infty})\mathrm{d}s\leq C\int_{0}^{t}(1+\langle\nu s^3\rangle^{-1/2}\langle s\rangle )\|\omega_0\|_{H^3}\mathrm{d}s\leq C,
 \end{align*}
 which along with Lemma \ref{lem:short time Er}, \eqref{est:shortwe}, \eqref{est:shortwe-in} and \eqref{est:usL4-1} gives
 \begin{align}\label{est:shortwe-0}
    \|\omega_e(t)\|_{L^2}\leq &C\exp\big(C\nu^{-\f16}\|\omega_e\|_{L^\infty(0,t;L^2)}\big)\times\nu^{\f13}\|\omega_0\|_{H^3}.
 \end{align} 

If $\|\omega_e(s)\|_{L^2}\leq \nu^{\f16}$ for $0\leq s\leq t\leq T_0$, then by \eqref{est:shortwe-0}, this estimate can be improved  to $\|\omega_e(s)\|_{L^2}\leq C\nu^{\f13}\|\omega_0\|_{H^3}\leq C\nu^{\f23}\leq \nu^{\f16}/2$ for $0\leq s\leq t\leq T_0$. Thus, by a bootstrap argument, we conclude $\|\omega_e(t)\|_{L^2}\leq \nu^{\f16}$ for $t\leq T_0$. Then \eqref{est:shortwe-0} yields 
\begin{align*}
    & \|\omega_e(t)\|_{L^2}\leq C\nu^{\f13}\|\omega_0\|_{H^3}\quad \text{for}\quad t\leq T_0.
 \end{align*}
  
 This completes the proof for Proposition \ref{prop:om-e-short}.
 \end{proof}

\subsection{Energy estimate in medium-time regime}\label{Energy-middle-regime}

First, we derive the following estimates for the source $G$ in the middle-time regime.
 
 \begin{lemma}\label{lem:G}
For $G$ defined by \eqref{def:G} and $Er_{L3}$ defined by \eqref{def:ErL=ErL123}, there holds
  \begin{align*}
 &\|G-Er_{L3}-\text{curl} ( u\cdot \nabla u_e)\|_{L^1(T_0, T; L^2)}
 \leq C\big(\nu^{-\f13}\|\omega_e\|_{X_I}\|\omega_0\|_{H^3}+
 \nu^{\f13}\|\omega_{0}\|_{H^3} +\|\omega_0\|_{H^3}^2\big),\\
 &\nu^{-\f12}\|u\cdot \nabla u_e\|_{L^2(T_0, T;L^2)}\leq C\nu^{-\f13}\|\omega_e\|_{X_I}\|\omega_0\|_{H^3}+C\nu^{-\f23}\|\omega_e\|_{X_I}^2,\\
 &\||D_x|^{\f12}\nabla Er_{L3}\|_{L^2(T_0, T; L^2)}\leq  C\nu^{1/2}|\ln(\nu)|\|\omega_0\|_{H^3}.
  \end{align*}
 Here $T\in[T_0, T_1],$ $I=[T_0,T]$, and $\text{curl} ( u\cdot \nabla u_e)=\partial_x(u\cdot\nabla u_e^2)-\partial_y(u\cdot\nabla u_e^1). $
\end{lemma}

\begin{proof}
Recalling \eqref{def:G}, we have
\begin{align*}
   G-Er_{L3}-\text{curl} ( u\cdot \nabla u_e)=&{u^x_e\partial_x\omega_L}  +u^y_e(\partial_y+t_*V'\partial_x)\omega_L-\mathrm{e}^{-\nu t}(V+V'')\partial_x\phi_e\\
 \notag&-t_*V'P_{\text{high}}u^y_e\partial_x\omega_L-t_*V'P_{\text{low}}u^y_e\partial_x \widetilde{\omega}_L\\
 \notag&+u\cdot\nabla \omega_e -\text{curl}(u\cdot\nabla u_e)\\
 \notag& +Er-Er_{L3},
\end{align*}

Next we estimate each term of $G-Er_{L3}-\text{curl} ( u_L\cdot \nabla u_e)$ by the following four steps.\smallskip
  
\textit{ Step 1}. Estimate of $ u^x_e\partial_x\omega_L+u^y_e(\partial_y+t_*V'\partial_x)\omega_L$ and 
  $ \mathrm{e}^{-\nu t}(V+V'')\partial_x\phi_e$.\smallskip

   By \eqref{est:paromL-inf} in Proposition \ref{prop:app} and \eqref{ass:U'-b'}, we have
  \begin{align*}
     \|u^y_e(\partial_y+t_*V'\partial_x)\omega_L\|_{L^2}\leq & \|u^y_e\|_{L^2}\|(\partial_y+t_*V'\partial_x)\omega_L\|_{L^\infty}
     \leq C\|u^y_e\|_{L^2}\|\omega_{0}\|_{H^3},\\ \|\mathrm{e}^{-\nu t}(V+V'')\partial_x\phi_e\|_{L^2}\leq & \|\partial_x\phi_e\|_{L^2}\|V+V''\|_{L^\infty}
     \leq C\|\partial_x\phi_e\|_{L^2}\|\omega_{0}\|_{H^3},
   \end{align*}
  from which, we infer that (as $u^y_e=\partial_x\phi_e $, here $L^pL^q=L^p(T_0, T;L^q(\Omega))$)
   \begin{align*}
      &\| u^y_e(\partial_y+t_*V'\partial_x)\omega_L\|_{L^1L^2}+\| \mathrm{e}^{-\nu t}(V+V'')\partial_x\phi_e\|_{L^1L^2}\\
       \leq & C\|\partial_x\phi_e\|_{L^2L^2}T^{1/2}\|\omega_{0}\|_{H^3}\leq CT^{1/2}\|\omega_e\|_{X_I}\|\omega_0\|_{H^3}\leq C\nu^{-\f13}\|\omega_e\|_{X_I}\|\omega_0\|_{H^3}.
   \end{align*}
 By \eqref{est:paromL-inf} in Proposition \ref{prop:app} again, we get
   \begin{align*}
 \| u^x_e\partial_x\omega_L\|_{L^2}\leq  & \|u^x_e\|_{L^2}\|\partial_x\omega_L\|_{L^\infty}
 \leq C \|\omega_e\|_{L^2}\langle\nu t^3\rangle^{-1}\|\omega_{0}\|_{H^3},
   \end{align*}
   which gives
     \begin{align*}
      &\| u^x_e\partial_x\omega_L\|_{L^1L^2}
       \leq C\|\omega_e\|_{L^\infty L^2}\|\langle\nu t^3\rangle^{-1}\|_{L^1(T_0,T)}\|\omega_{0}\|_{H^3}
      \leq  C\nu^{-\f13}\|\omega_e\|_{X_I}\|\omega_0\|_{H^3}.
   \end{align*}

\textit{Step 2}. Estimate of $t_*V'P_{\text{high}}u^y_e\partial_x\omega_L$ and $t_*V'P_{\text{low}}u^y_e\partial_x\widetilde{\omega}_L$. \smallskip

By \eqref{est:t*Vparw} in Proposition \ref{prop:app}, we have
\begin{align*}
     \|t_*V'P_{\text{high}}u^y_e\partial_x\omega_L\|_{L^2}\leq & \|P_{\text{high}}u^y_e\|_{L^2}\|t_*V'\partial_x\omega_L\|_{L^{\infty}}
     \leq C\nu^{\f23}\|u_e^y\|_{H^2}t\langle\nu t^3\rangle^{-1}\|\omega_{0}\|_{H^3},
   \end{align*}
   and
   \begin{align*}
     \|t_*V'P_{\text{low}}u^y_e\partial_x\tilde{\omega}_L\|_{L^2}\leq & \|P_{\text{low}}u^y_e\|_{L^2}\|t_*V'\partial_x\widetilde{\omega}_L\|_{L^{\infty}}
     \leq C\|u^y_e\|_{L^2}\|\omega_{0}\|_{H^3}.
   \end{align*}
 Then we conclude (as $\|u_e^y\|_{H^2}\leq C\|\nabla\omega_e\|_{L^2}$)
   \begin{align*}
      \|t_*V'P_{\text{high}}u^y_e\partial_x\omega_L\|_{L^1L^2}\leq& C\nu^{\f23}\|\nabla\omega_e\|_{L^2L^2}\|t\langle\nu t^3\rangle^{-1}\|_{L^2(T_0,T)}\|\omega_{0}\|_{H^3}\\
      \leq&C\nu^{\f23}\nu^{-\f12}\|\omega_e\|_{X_I}\nu^{-\f12}\|\omega_0\|_{H^3}=  C\nu^{-\f13}\|\omega_e\|_{X_I}\|\omega_0\|_{H^3},
   \end{align*}
   and
   \begin{align*}
      \|t_*V'P_{\text{low}}u^y_e\partial_x\tilde{\omega}_L\|_{L^1L^2}\leq& C\|\partial_x\phi_e\|_{L^2L^{2}}T^{\f12}\|\omega_{0}\|_{H^3}\\
      \leq&C\|\omega_e\|_{X_I}T^{\f12}\|\omega_0\|_{H^3}\leq  C\nu^{-\f13}\|\omega_e\|_{X_I}\|\omega_0\|_{H^3}.
   \end{align*}
   
\textit{Step 3}. Estimate of $u\cdot\nabla \omega_e-\text{curl} ( u_L\cdot \nabla u_e)$.\smallskip

Thanks to $u_e\cdot\nabla \omega_e= \text{curl} ( u_e\cdot \nabla u_e)$, we have
  \begin{align*}
  &u\cdot\nabla \omega_e - \text{curl} ( u\cdot \nabla u_e)\\
  &=u_L\cdot\nabla \omega_e - \text{curl} ( u_L\cdot \nabla u_e) +\big(u_e\cdot\nabla \omega_e - \text{curl} ( u_e\cdot \nabla u_e) \big)\\
    &=  u_L\cdot\nabla \omega_e - \text{curl} ( u_L\cdot \nabla u_e)=
-\partial_xu_L\cdot\nabla u_e^y+\partial_yu_L\cdot\nabla u^x_e.
  \end{align*}
  
  By  \eqref{est:paromL-inf} in Proposition \ref{prop:app}, we get
  \begin{align*}
     \|-\partial_xu_L\cdot\nabla u_e^y(t)+\partial_yu_L\cdot\nabla u^x_e(t)\|_{L^2}&\leq C\|\nabla u_L(t)\|_{L^\infty_y}\|\nabla u_e(t)\|_{L^2}
     \\&\leq C\langle \nu t^3\rangle^{-1}\|\omega_0\|_{H^3}\|\omega_e(t)\|_{L^2},
  \end{align*}
  Then we infer that
  \begin{align}\label{est:unablaome-1}
     \|u\cdot\nabla \omega_e - \text{curl} ( u\cdot \nabla u_e)\|_{L^1L^2}=& \| \partial_xu_e\cdot\nabla u_L^y(t)-\partial_yu_e\cdot\nabla u^x_L(t)\|_{L^1L^2}\\
     \leq &C\|\langle\nu t^3\rangle^{-1}\|_{L^1(T_0,T)}\|\omega_0\|_{H^3}\|\omega_e\|_{L^\infty L^2}\nonumber\\
     \leq &C\nu^{-\f13}\|\omega_0\|_{H^3}\|\omega_e\|_{X_0}.\nonumber
  \end{align}

\textit{Step 4}. Estimate of $Er-Er_{L3}$.\smallskip

  Recalling \eqref{def:Er} and \eqref{def:ErL=ErL123}, we have 
  \beno
 Er-Er_{L3}=Er_{L1}+Er_{L2}-\mathrm{e}^{-\nu t}(V''+V)\partial_x\phi_L+u_L\cdot\nabla\omega_L+\nu\mathrm{e}^{-\nu t}(V'+V''').
 \eeno 
  By Proposition \ref{prop:error}, we have
  \begin{align*}
     & \|Er_{L1}+Er_{L2}\|_{L^1(T_0,T;L^2)}\leq C\|(\nu\langle t\rangle\langle \nu t^3\rangle^{-1}+\nu^{1/3}\langle t\rangle^{-2})\|_{L^1(T_0,T)}\|\omega_0\|_{H^3}\\
   &\leq C\|(\nu^{2/3}\langle \nu t^3\rangle^{-2/3}+\nu^{1/3}\langle t\rangle^{-2})\|_{L^1(T_0,T)}\|\omega_0\|_{H^3}\leq C\nu^{1/3}\|\omega_0\|_{H^3}.
   \end{align*}
By \eqref{est:u2Linf} in Proposition \ref{prop:app} and $u^y_L=\partial_x\phi_L$, we have
\begin{align*}
   &\|\mathrm{e}^{-\nu t}(V''+V)\partial_x\phi_L\|_{L^1(T_0,T;L^2)}\leq C\|V''+V\|_{L^\infty}\|\mathrm{e}^{-\nu t}u_L^y\|_{L^1(T_0,T;L^2)}\\
&\leq C\|\omega_0\|_{H^3}\|\mathrm{e}^{-\nu t}\langle \nu t^3\rangle^{1/2}\langle t\rangle^{-2}\|_{L^1(T_0,T)}\|\omega_{0}\|_{H^3}
\leq C\|\omega_0\|_{H^3}^2.
\end{align*}
By \eqref{est:uLnawL} in Proposition \ref{prop:app}, we have
\begin{align*}
   &\|u_L\cdot\nabla\omega_L\|_{L^1(T_0,T;L^2)}\leq C\|u_L\cdot\nabla\omega_L\|_{L^1(T_0,T;L^\infty)}\\
   &\leq C\|(\langle t\rangle^{-1.2}+| \nu t^3|^{1/2}\langle \nu t^3\rangle^{-1}\langle t\rangle^{-1})\|_{L^1_t}\|\omega_{0}\|_{H^3}^2
\\
   &\leq C\|(\langle t\rangle^{-1.2}+\nu^{1/3}(\nu t^3)^{1/6}\langle \nu t^3\rangle^{-1})\|_{L^1_t}\|\omega_{0}\|_{H^3}^2
   \leq C\|\omega_{0}\|_{H^3}^2.
\end{align*}
We also have
\begin{align*}
    \|\nu\mathrm{e}^{-\nu t}(V'+V''')\|_{L^1(T_0,T;L^2)}\leq& C\|\nu\mathrm{e}^{-\nu t}\|_{L^1(T_0,T)}\|(V'+V''')\|_{L^\infty L_y^2}\\
\leq& C\nu T\|\omega_0\|_{H^3}\leq C\nu^{\f13}\|\omega_0\|_{H^3}.
\end{align*}
 This shows that 
  \begin{align*}
     & \|Er-Er_{L3}\|_{L^1(T_0,T_1;L^2)}\leq C\nu^{\f13}\|\omega_0\|_{H^3}+C\|\omega_0\|_{H^3}^2.
  \end{align*}
  
  Thus, we finish the estimate of $\|G-Er_{L3}\|_{L^1(T_0,T;L^2)}$ by these four steps.\smallskip
  
   Next we estimate $\nu^{-\f12}\|u\cdot \nabla u_e\|_{L^2L^2}$. By \eqref{est:u1Linf} and \eqref{est:u2Linf} in Proposition \ref{prop:app} (for $T_0\leq t\leq T\leq T_1\leq \nu^{-\f12}$), 
   \begin{align*}
     & |u^x_L(t,x,y)|\leq C (\langle t\rangle^{-\f32}+\langle t\rangle^{-1}|V'|)\|\omega_0\|_{H^3},\\
     &\|u^y_L(t)\|_{L^\infty}\leq C\langle \nu t^3\rangle^{1/2}\langle t\rangle^{-2}\|\omega_0\|_{H^3}\leq C \langle t\rangle^{-\f32}\|\omega_0\|_{H^3},\\
     &\|u_L\cdot \nabla u_e\|_{L^2}\leq C(\langle t\rangle^{-\f32}\|\nabla u_e\|_{L^2}+\langle t\rangle^{-1}\|V'\partial_xu_e\|_{L^2})\|\omega_0\|_{H^3}\\ &\quad\leq C(\langle t\rangle^{-\f32}\|\omega_e\|_{L^2}+\langle t\rangle^{-1}\||V'|^{\f12}\partial_xu_e\|_{L^2})\|\omega_0\|_{H^3}
  \end{align*} 
  Then we infer that (recall that $u_e=(-\partial_y,\partial_x)\phi_e=(-\partial_y,\partial_x)\Delta^{-1}\omega_e$)
  \begin{align}\label{est:uLnaue}
     &\nu^{-\f12}\|u_L\cdot \nabla u_e\|_{L^2L^2}\\
     &\leq C\nu^{-\f12}\big(\|\langle t\rangle^{-1}\|_{L^\infty(T_0,T)}\||V'|^{\f12}\partial_xu_e\|_{L^2L^2}+\|\langle t\rangle^{-\f32}\|_{L^2(T_0,T)}\| \omega_e\|_{L^\infty L^2}\big)\|\omega_0\|_{H^3}\nonumber\\
     &\leq C\nu^{-\f12}T_0^{-1}\|\omega_e\|_{X_I}\|\omega_0\|_{H^3}
     \leq C\nu^{-\f13}\|\omega_e\|_{X_I}\|\omega_0\|_{H^3}.\nonumber
  \end{align}
  
   For $u_e\cdot\nabla u_e$, we have
  \begin{align*}
     u_e\cdot\nabla u_e=\mathrm{P}_{\neq} u_e\cdot\nabla u_e+\mathrm{P}_0u^x_e\partial_xu_e. 
  \end{align*}
  By Lemma \ref{lem:norm-X} and $u_e=(-\partial_y,\partial_x)\Delta^{-1}\omega_e$, we get
  \begin{align*}
     &\|\partial_x u_e\|_{L^2L^2}+\|\mathrm{P}_{\neq}u_e\|_{L^2L^\infty}=
     \|\partial_x \nabla\Delta^{-1}\omega_e\|_{L^2L^2}+\|\mathrm{P}_{\neq}\nabla\Delta^{-1}\omega_e\|_{L^2L^\infty}\leq C\nu^{-\f{1}{6}}\|\omega_e\|_{X_I},\\
     &\|\mathrm{P}_0u_e\|_{L^\infty}\leq C\|\mathrm{P}_0u_e\|_{H^1}\leq C\|\omega_e\|_{L^2}.
  \end{align*} 
  Thus, we obtain 
  \begin{align*}
  &\|\mathrm{P}_{\neq} u_e\cdot\nabla u_e\|_{L^2L^2}\leq \|\mathrm{P}_{\neq} u_e\|_{L^2L^\infty}\|\nabla u_e\|_{L^\infty L^2} \leq C\nu^{-\f16}\|\omega_e\|_{X_I}^2,\\
     &\|\mathrm{P}_0u^x_e\partial_xu_e\|_{L^2L^2}\leq \|\mathrm{P}_0u_e
     \|_{L^\infty L^\infty}\|\partial_x u_e\|_{L^2L^2} \leq C\nu^{-\f16}\|\omega_e\|_{X_I}^2.
  \end{align*}
   This shows that
  \begin{align*}
     & \nu^{-\f12}\|u_e\cdot\nabla u_e\|_{L^2L^2}\leq C\nu^{-\f23}\|\omega_e\|_{X_I}^2,
  \end{align*}
  which along with \eqref{est:uLnaue} and the fact $u\cdot\nabla u_e=u_L\cdot\nabla u_e+u_e\cdot\nabla u_e$  gives
  \begin{align}\label{est:unablaome-2}
   \nu^{-\f12} \| u\cdot \nabla u_e\|_{L^2(T_0,T_1;L^2)}
     \leq &C\nu^{-\f13}\|\omega_0\|_{H^3}\|\omega_e\|_{X_I}+ C\nu^{-\f23}\|\omega_e\|_{X_I}^2.
  \end{align}
  
  For $\||D_x|^{\f12}\nabla Er_{L3}(t)\|_{L^2(T_0, T_1; L^2)}$, we get by Proposition \ref{prop:error} that
  \begin{align*}
     &\||D_x|^{\f12}\nabla Er_{L3}\|_{L^2(T_0, T; L^2)}
     \leq  C\|\langle t\rangle^{-2}|\nu t^3|^{1/2}\|_{L^2(T_0,T)}\|\omega_{0}\|_{H^3}\leq C\nu^{\f12}|\ln\nu|\|\omega_0\|_{H^3}.
  \end{align*}

 This completes the proof of this lemma.
\end{proof}

Now we are in a position to prove Proposition \ref{prop:om-star}.

 \begin{proof}[Proof of Proposition \ref{prop:om-star}]
 Recall that $\omega_*$ satisfies
 \begin{align*}
    & \partial_t\omega_{*}-\nu \Delta \omega_{*}+\mathrm{e}^{-\nu t}V\partial_x(\omega_{*}+\Delta^{-1}\omega_{*}) +G=0, \quad\omega_{*}|_{t=T_0}=\omega_{e}|_{t=T_0},
 \end{align*}
 Thanks to $G=(G-Er_{L3}-\text{curl}(u\cdot\nabla u_e))+ \text{curl}(u\cdot\nabla u_e)+ Er_{L3}$, we get by 
 Proposition \ref{prop:LNS-sp1} and Lemma \ref{lem:G} that 
 \begin{align*}
    \|\omega_{*}\|_{X_0}\leq& C\|G-Er_{L3}-\text{curl}(u\cdot\nabla u_e)\|_{L^1(T_0, T;L^2)}+C\nu^{-\f12}\|u\cdot\nabla u_e\|_{L^2(T_0, T;L^2)}\\
    &+C\||D_x|^{\f12}\nabla Er_{L3}\|_{L^2(T_0, T;L^2)}+C\|\omega_*(T_0)\|_{L^2}\\
    \leq & C\big(\nu^{-\f23}\|\omega_e\|_{X_I}^2+\nu^{-\f13}\|\omega_e\|_{X_I}\|\omega_0\|_{H^3}+
 \nu^{\f13}\|\omega_{0}\|_{H^3} +\|\omega_0\|_{H^3}^2\big)\\
 &+C\nu^{\f12}|\ln(\nu)|\|\omega_0\|_{H^3} +C\|\omega_*(T_0)\|_{L^2}.
 \end{align*}
 \end{proof}

\subsection{Energy estimate in large-time regime}\label{Energy-large-regime}

In this subsection, we prove Proposition \ref{prop:omega-e-large}.

\begin{proof}
First of all, we prove the following estimate: for $T\in[T_1, 1/\nu],$ $I=[T_1,T]$, 
     \begin{align}\label{eq:om-large}
       \|\omega\|_{X_I}+\|\mathrm{e}^{\epsilon\nu^{\f12}t}\omega_{\neq}\|_{X_I}\leq & C\big(\|\omega(T_1)\|_{L^2}+\nu^{-\f23}\|\omega\|_{X_I}\|\mathrm{e}^{\epsilon\nu^{\f12}t}\omega_{\neq}\|_{X_I}\big).
     \end{align}
Thanks to \eqref{ass:U-b1} and \eqref{ass:U-b2},  applying Proposition \ref{prop:LNS-sp2} to \eqref{eq:>T1}
({taking $\tilde{\epsilon}=\epsilon $ and using} $u\cdot\nabla \omega=-\partial_y(u\cdot\nabla u^x)+\partial_x(u\cdot\nabla u^y)$), 
we obtain 
\begin{align}
     \label{om1a}  \|\omega\|_{X_I}+\|\mathrm{e}^{\epsilon\nu^{\f12}t}\omega_{\neq}\|_{X_I}\leq & 
       C\big(\|\omega(T_1)\|_{L^2}+\nu^{-\f12}\|\mathrm{e}^{\epsilon\nu^{\f12}t}u\cdot\nabla u\|_{L^2(I;L^2)}\big).
     \end{align}
  By Lemma \ref{lem:norm-X} and $u=(-\partial_y,\partial_x)\Delta^{-1}\omega$, we get
  \begin{align*}
     &\|\mathrm{e}^{\epsilon\nu^{\f12}t}\partial_x u\|_{L^2(I;L^2)}+\|{\mathrm{e}^{\epsilon\nu^{\f12}t}}\mathrm{P}_{\neq}u\|_{L^2(I;L^\infty)}\\&=
     \|\mathrm{e}^{\epsilon\nu^{\f12}t}\partial_x \nabla\Delta^{-1}\omega_{\neq}\|_{L^2(I;L^2)}+\|\mathrm{e}^{\epsilon\nu^{\f12}t}\nabla\Delta^{-1}\omega_{\neq}\|_{L^2(I;L^\infty)}\leq C\nu^{-\f{1}{6}}\|\mathrm{e}^{\epsilon\nu^{\f12}t}\omega_{\neq}\|_{X_I},\\
     &\|\mathrm{P}_0u\|_{L^\infty}\leq C\|\mathrm{P}_0u\|_{H^1}\leq C\|\omega\|_{L^2}\leq C\|\omega\|_{X_I}.
  \end{align*} 
  Thus, we obtain 
  \begin{align*}
  &\|\mathrm{e}^{\epsilon\nu^{\f12}t}\mathrm{P}_{\neq} u\cdot\nabla u\|_{L^2(I;L^2)}\leq \|\mathrm{e}^{\epsilon\nu^{\f12}t}\mathrm{P}_{\neq} u\|_{L^2(I;L^\infty)}\|\nabla u\|_{L^\infty (I;L^2)} \leq C\nu^{-\f16}\|\mathrm{e}^{\epsilon\nu^{\f12}t}\omega_{\neq}\|_{X_I}\|\omega\|_{X_I},\\
     &\|\mathrm{e}^{\epsilon\nu^{\f12}t}\mathrm{P}_0u^x\partial_xu\|_{L^2(I;L^2)}\leq \|\mathrm{P}_0{u}
     \|_{L^\infty(I;L^\infty)}\|\mathrm{e}^{\epsilon\nu^{\f12}t}\partial_x u\|_{L^2(I;L^2)} \leq C\nu^{-\f16}\|\omega\|_{X_I}\|\mathrm{e}^{\epsilon\nu^{\f12}t}\omega_{\neq}\|_{X_I}.
  \end{align*}
   Thanks to $u\cdot\nabla u=\mathrm{P}_{\neq} u\cdot\nabla u+\mathrm{P}_0u^x\partial_xu$, we obtain
  \begin{align}
   \label{om2b}  & \nu^{-\f12}\|\mathrm{e}^{\epsilon\nu^{\f12}t}u\cdot\nabla u\|_{L^2(I;L^2)}\leq C\nu^{-\f23}\|\omega\|_{X_I}\|\mathrm{e}^{\epsilon\nu^{\f12}t}\omega_{\neq}\|_{X_I}.
  \end{align}
  Now \eqref{eq:om-large} follows from \eqref{om1a} and \eqref{om2b}.

Next we use a bootstrap argument. First, we assume that $\|\omega\|_{X_{[T_1,T]}}\leq \varepsilon_0\nu^{\f23}$ for small $\varepsilon_0$ determined later, which holds for some $T>T_1$ by $\|\omega(T_1)\|_{L^2}\leq \epsilon_2\nu^{\f23}$(for $\epsilon_2$ small enough). 
Then \eqref{eq:om-large} implies that 
 \begin{align}\label{est>T1}
   \|\omega\|_{X_{[T_1,T]}}+\|\mathrm{e}^{\epsilon\nu^{\f12}t}\omega_{\neq}\|_{X_{[T_1,T]}}\leq C\|\omega(T_1)\|_{L^2}\leq  C\epsilon_2\nu^{\f23}.
 \end{align}
 Taking $\epsilon_2=\varepsilon_0/(2C)$, we can conclude \eqref{est>T1} for $T=1/\nu$, and the result follows from 
the definition of the $X_{[T_1,T]}$ norm.
\end{proof}

Next, we prove Proposition \ref{prop2}.

\begin{proof}
Thanks to \eqref{ass:U-b1} and \eqref{ass:U-b2},  applying Proposition \ref{Prop1} to \eqref{eq:>T1}
(using $u\cdot\nabla \omega=-\partial_y(u\cdot\nabla u^x)+\partial_x(u\cdot\nabla u^y)$), 
we obtain 
\begin{align*}
&\|\omega_{\neq}\|_{\YI}:=\|\mathrm{e}^{\nu t}\omega_{\neq}\|_{L^{\infty}(I;L^2)}+\nu^{\f16}\|\mathrm{e}^{\nu t/2}\partial_x \nabla\Delta^{-1}\omega\|_{L^2(I;L^2)}
+\nu^{\f16}\|\mathrm{e}^{\nu t/2}\nabla\Delta^{-1}\omega_{\neq}\|_{L^2(I;L^\infty)}\\ &\leq 
C\big(\|\omega_{\neq}|_{t=1/\nu}\|_{L^2}+\nu^{-\f12}\|\mathrm{e}^{\nu t}u\cdot\nabla u\|_{L^2(I;L^2)}\big),\\
&\|\mathrm{e}^{\nu t/2}\omega\|_{L^{\infty}(I;L^2)}\leq C\big(\|\omega|_{t=1/\nu}\|_{L^2}+\nu^{-\f12}\|\mathrm{e}^{\nu t}u\cdot\nabla u\|_{L^2(I;L^2)}\big).
\end{align*}
Thanks to $u\cdot\nabla u=\mathrm{P}_{\neq} u\cdot\nabla u+\mathrm{P}_0u^x\partial_xu$, and $u=(-\partial_y,\partial_x)\Delta^{-1}\omega$, we obtain
  \begin{align*}
     & \|\mathrm{e}^{\nu t}u\cdot\nabla u\|_{L^2(I;L^2)}\\ \leq&  
   \|\mathrm{e}^{\nu t/2}\mathrm{P}_{\neq}u\|_{L^2(I;L^{\infty})}\|\mathrm{e}^{\nu t/2}\nabla u\|_{L^{\infty}(I;L^{2})} +
   \|\mathrm{e}^{\nu t/2}\mathrm{P}_0u^x\|_{L^{\infty}(I;L^{\infty})}\|\mathrm{e}^{\nu t/2}\partial_xu\|_{L^{\infty}(I;L^{2})} \\ \leq&\|\mathrm{e}^{\nu t/2}\nabla\Delta^{-1}\omega_{\neq}\|_{L^2(I;L^{\infty})}\|\mathrm{e}^{\nu t/2}\omega\|_{L^{\infty}(I;L^{2})}\\ &+\|\mathrm{e}^{\nu t/2}\omega\|_{L^{\infty}(I;L^{\infty})}\|\mathrm{e}^{\nu t/2}\partial_x\nabla\Delta^{-1}\omega\|_{L^{\infty}(I;L^{2})}\leq \nu^{-\f16}\|\mathrm{e}^{\nu t/2}\omega\|_{L^{\infty}(I;L^2)}\|\omega_{\neq}\|_{\YI}.
  \end{align*}
  Thus,
  \begin{align}\label{eq:om1}
&\|\omega_{\neq}\|_{\YI}\leq 
C\big(\|\omega_{\neq}|_{t=1/\nu}\|_{L^2}+\nu^{-\f23}\|\mathrm{e}^{\nu t/2}\omega\|_{L^{\infty}(I;L^2)}\|\omega_{\neq}\|_{\YI}\big),\\
\label{eq:om2}&\|\mathrm{e}^{\nu t/2}\omega\|_{L^{\infty}(I;L^2)}\leq C\big(\|\omega|_{t=1/\nu}\|_{L^2}+\nu^{-\f23}\|\mathrm{e}^{\nu t/2}\omega\|_{L^{\infty}(I;L^2)}\|\omega_{\neq}\|_{\YI}\big).
\end{align}
 Next we use a bootstrap argument. First, we assume that $\|\omega_{\neq}\|_{\YI}+\|\mathrm{e}^{\nu t/2}\omega\|_{L^{\infty}(I;L^2)}\leq \varepsilon_0\nu^{\f23}$ for small $\varepsilon_0$ determined later, which holds for some $T>1/\nu$ by $\|\omega|_{t=1/\nu}\|_{L^2}\leq \epsilon_3\nu^{\f23}$(for $\epsilon_3$ small enough). 
Then \eqref{eq:om1}, \eqref{eq:om2} implies that 
 \begin{align}\label{est>1/nu}
   \|\omega_{\neq}\|_{\YI}\leq 
C\|\omega_{\neq}|_{t=1/\nu}\|_{L^2}\leq C\epsilon_3\nu^{\f23},\quad \|\mathrm{e}^{\nu t/2}\omega\|_{L^{\infty}(I;L^2)}\leq C\|\omega|_{t=1/\nu}\|_{L^2}\leq  C\epsilon_3\nu^{\f23}.
 \end{align}
 Taking $\epsilon_3=\varepsilon_0/(3C)$, we can conclude \eqref{est>1/nu} for $T=+\infty$, and the result follows from 
the definition of the $Y_{I}$ norm. 
\end{proof}

\section{Control of the reaction term}\label{sec:reactionpart}
In this section, we control the most difficult reaction part $\om_{re}$, which solves 
\begin{align}\label{eq:reaction-sec}
   & \partial_t\omega_{re}-\nu \Delta \omega_{re}+\mathrm{e}^{-\nu t}V\partial_x(\omega_{re}+\Delta^{-1}\omega_{re}) -t_*V'P_{\text{low}}u^y_e\partial_x\omega_L^* =0, \quad\omega_{re}|_{t=T_0}=0.
\end{align}
For fixed $T\in[T_0,T_1]$, let $N=\inf(\Z\cap[(T-T_0)\nu^{\f13},+\infty))$, $t_{N}=T$, $t_j=T_0+j\nu^{-\f13}$, for $j\in\Z\cap[0,N)$.
We introduce the following decomposition to the time interval
\beno
[T_0,T]=\cup_{j=1}^NI_j,\quad I_j:=[t_{j-1},t_{j}].
\eeno
\subsection{Plancherel's type lemma}

Let $\Psi$ be a nonnegative function, which is supported on the interval $1-1/7\leq y\leq 2$, and equals to 1 on the smaller interval $1\leq y\leq 2-2/7$, and satisfies $\Psi(y)+\Psi(y/2)=1$ for $1\leq y\leq 4-4/7$. Then we have
\begin{align*}
\sum_{m\in \mathbb{Z}}\Psi(2^{m}y) =1\quad \text{for}\,\, y>0.
\end{align*}

Let $V'(y_1)=V'(y_2)=0$ with $|y_1|<1/10$ and $|y_2-\pi|<1/10$. Then $\pi/2-y_1,\ y_1+\pi/2,\ y_2-\pi/2,\ 3\pi/2-y_2\in (1,2-2/7)$.
We denote
\begin{align*}
  &\Psi_0^+(y)=\Psi(y-y_1)\quad y\in(y_1,\pi/2],\quad \Psi_0^+(y)=\Psi(y_2-y)\quad y\in[\pi/2,y_2),\\
  &\Psi_m^+(y)=\Psi(2^m(y-y_1)),\quad \Psi_{-m}^+(y)=\Psi(2^m(y_2-y))\quad y\in(y_1,y_2),\ m\in\Z_+,\\
 &\Psi_0^-(y)=\Psi(y_1-y),\quad y\in[-\pi/2,y_1),\quad \Psi_0^-(y)=\Psi(y+2\pi-y_2)\quad y\in(y_2-2\pi,-\pi/2]),\\
  &\Psi_m^-(y)=\Psi(2^m(y_1-y)),\quad \Psi_{-m}^-(y)=\Psi(2^m(y+2\pi-y_2))\quad y\in(y_2-2\pi,y_1),\ m\in\Z_+. 
\end{align*}
We then have
\begin{align}\label{eq:Psisum1}
   &\sum_{m\in\Z}\Psi^{+}_m(y)=1\ \forall\ y\in(y_1,y_2),\quad \sum_{m\in\Z}\Psi^{-}_m(y)=1\ \forall\ y\in(y_2-2\pi,y_1),\\
  \label{eq:Psisum2}
   &\f12\leq\sum_{m\in\Z}\big(\Psi^{+}_m(y)\big)^2\leq1\ \forall\ y\in(y_1,y_2),\quad \f12\leq\sum_{m\in\Z}\big(\Psi^{-}_m(y)\big)^2\leq1\ 
   \forall\ y\in(y_2-2\pi,y_1).
\end{align}
We define $ \Psi^{+}_m(y)=0$ for $y\in(y_2-2\pi,y_1)$, $\Psi^{-}_m(y)=0$ for $y\in(y_1,y_2)$, and then extend $\Psi^{\pm}_m$ to a $2\pi$-periodic function (i.e., a function on $\T_{2\pi}$). 
We define
  \begin{align}
      \widetilde{g}^{+}_m(t,\eta)&= \dfrac{1}{\sqrt{2\pi}}\int_{y_1}^{y_2}g(t,y)\Psi_{m}^{+}(y)\mathrm{e}^{-\mathrm{i} V(y)\eta}\mathrm{d}y,\quad\forall\ m\in \mathbb{Z},\\
      \widetilde{g}^{-}_m(t,\eta)&= \dfrac{1}{\sqrt{2\pi}}\int_{y_2-2\pi}^{y_1}g(t,y)\Psi_{m}^{-}(y)\mathrm{e}^{-\mathrm{i} V(y)\eta}\mathrm{d}y,\quad\forall\ m\in \mathbb{Z},\\
  \label{def:tildegmpm}    \widetilde{g}^{\pm}_m(t,\eta)&= \dfrac{1}{\sqrt{2\pi}}\int_{y_2-2\pi}^{y_2}g(t,y)\Psi_{m}^{\pm}(y)\mathrm{e}^{-\mathrm{i} V(y)\eta}\mathrm{d}y,\quad\forall\ m\in \mathbb{Z}.
  \end{align}
  
 As $V:[y_1,y_2]\to\R$ is strictly decreasing, we denote its inverse function by \\
$V^+: [V(y_2),V(y_1)]\to[y_1,y_2]$. As $V:[y_2-2\pi,y_1]\to\R$ is strictly increasing, we denote its inverse function by
$V^-:[V(y_2),V(y_1)]\to[y_2-2\pi,y_1]$.  Then we have $V(V^{-}(z))=V(V^{+}(z))=z $ for $z\in [V(y_2),V(y_1)]$.
       
\begin{lemma}\label{lem:plancherel}
 It holds that for $\gamma\in\mathbb{R}$, $m\in\mathbb{Z}$,
  \begin{align}
     &\|\widetilde{g}^{\pm}_m(t,\eta)\|_{L^2_\eta}=
     \||V'|^{-\f12}\Psi_m^{\pm}(y)g(t,y)\|_{L^2_y}
     \leq C2^{(\f12-\gamma)|m|}\||V'|^{-\gamma}\Psi_m^{\pm}(y)g(t,y)\|_{L^2_y},\label{est:gmppm}\\
  \label{gm1}  &\int_{\R}\widetilde{g}^{\pm}_m(t,\eta)\overline{\widetilde{h}^{\pm}_{m'}(t,\eta)}\mathrm{d}\eta=
    \int_{y_2-2\pi}^{y_2}|V'(y)|^{-1}\Psi_m^{\pm}(y)\Psi_{m'}^{\pm}(y)g(t,y)\overline{h(t,y)}\mathrm{d}y, \\
 & \||\eta+ls|\widetilde{g}^{\pm}_m(t,\eta)\|_{L^2_\eta}
     \leq C2^{\f {3|m|}{2}}\|\Psi^{\pm}_m(y)(\partial_y+\mathrm{i}lV's)g\|_{L^2_y}\label{est:etagmppm}\\
     &\qquad+C2^{(\f52-\gamma)|m|}\| |V'|^{-\gamma}\Psi^{\pm}_m(y)g\|_{L^2_y}+C\sum_{|m_1-m|\leq 1}2^{(\f 52-\gamma)|m|}\||V'|^{-\gamma}\Psi_{m_1}^{\pm}(y)g\|_{L^2_y}.\nonumber
  \end{align}
\end{lemma}
\begin{proof}
  Thanks to the definition of $V^+$, $V^-$, we have
   \begin{align*}
      \widetilde{g}^{-}_m(t,\eta)&= \dfrac{1}{\sqrt{2\pi}}\int_{y_2-2\pi}^{y_1}g(t,y)\Psi^-_m(y)\mathrm{e}^{-\mathrm{i} V(y)\eta}\mathrm{d}y\\
          &=\dfrac{1}{\sqrt{2\pi}}\int_{V(y_2)}^{V(y_1)}g(t,V^{-}(z)) \Psi^-_m(V^{-}(z))
      \mathrm{e}^{-\mathrm{i} z\eta}\mathrm{d}V^{-}(z)\\
      &=\dfrac{1}{\sqrt{2\pi}}\int_{\mathbb{R}}g(t,V^{-}(z)) \Psi^-_m(V^{-}(z))
      \partial_zV^-(z)\mathrm{e}^{-\mathrm{i} z\eta}\mathbf{1}_{[V(y_2),V(y_1)]}(z)\mathrm{d}z,
   \end{align*}
   and
      \begin{align*}
      \widetilde{g}^{+}_m(t,\eta)&= \dfrac{1}{\sqrt{2\pi}}\int_{y_1}^{y_2}g(t,y)\Psi^+_m(y)\mathrm{e}^{-\mathrm{i} V(y)\eta}\mathrm{d}y,\\
           &=-\dfrac{1}{\sqrt{2\pi}}\int_{V(y_2)}^{V(y_1)}g(t,V^{+}(z))\Psi_m^+(V^{+}(z))\mathrm{e}^{-\mathrm{i} z\eta}\mathrm{d}V^{+}(z)\\
      &=-\dfrac{1}{\sqrt{2\pi}}\int_{\mathbb{R}}
      g(t,V^{+}(z))\Psi_m^+(V^{+}(z))\partial_zV^{+}(z)\mathrm{e}^{-\mathrm{i} z\eta}\mathbf{1}_{[V(y_2),V(y_1)]}(z)\mathrm{d}z.
   \end{align*}
  Notice that
   \begin{align}\label{eq:parVVm}
   &\partial_zV^{-}(z)V'(V^{-}(z))=\partial_zV^{+}(z)V'(V^{+}(z))=1, \quad z\in [V(y_2),V(y_1)].
    \end{align}
   By Plancherel's formula and \eqref{eq:parVVm}, we get
   \begin{align*}
       \|\tilde{g}_{m}^{\pm}(t,\eta)\|_{L^2_\eta}^2&= \big\|g(t,V^{\pm}(z)) \Psi^{\pm}_m(V^{\pm}(z))
      \partial_zV^{\pm}(z)\big\|_{L^2_z}^2\\
      &= \big\|g(t,y) \Psi^{\pm}_m(y)
      |\partial_zV^{\pm}(z)|^{\f12}\big\|_{L^2_y}^2\\
      &=\big\|g(t,y)\Psi^{\pm}_{m}(y)|V'(y)|^{-\f12}\big\|_{L^2_y}^2,\qquad m\in\mathbb{Z},
   \end{align*}
which gives \eqref{est:gmppm} by noticing $ |V'(y)|\sim 2^{-|m|} $ on the support of $\Psi^{\pm}_m$. The proof of  \eqref{gm1} is similar.

  For $m\in\mathbb{Z}$, $ m_1\in\{m-1,m,m+1\}$, let
  \begin{align*}
     & I_m^{\pm}(t,\eta):= \dfrac{1}{\sqrt{2\pi}}\int_{-\pi}^{\pi} \Psi^{\pm}_m(y)(\partial_y+\mathrm{i}lV't_*)\big(g(t,y)/ V'(y)\big)\mathrm{e}^{-\mathrm{i} V(y)\eta}\mathrm{d}y,\\
     &II_{m_1}^{\pm}(t,\eta):= \dfrac{1}{\sqrt{2\pi}}\int_{-\pi}^{\pi} \Psi_{m_1}^{\pm}(y)\partial_y\Psi^{\pm}_m(y){(g(t,y)/ V'(y))}\mathrm{e}^{-\mathrm{i} V(y)\eta}\mathrm{d}y.
  \end{align*}
  We get by \eqref{est:gmppm} that
  \begin{align}\label{est:Ipmm}
\| I_m^{\pm}(t,\eta)\|_{L^2_\eta}\leq& C\||V'|^{-\f12}\Psi^{\pm}_m(y)(\partial_y+\mathrm{i}lV't_*)\big(g(t,y)/V'(y)\big)\|_{L_y^2}\\
     \leq &C\||V'|^{-\f32}\Psi^{\pm}_m(y)(\partial_y+\mathrm{i}lV't_*)g\|_{L_y^2}+ C\||V'|^{-\f52}V''\Psi^{\pm}_m(y)g\|_{L_y^2}\nonumber\\
     \leq &C2^{\f {3|m|}{2}}\|\Psi^{\pm}_m(y)(\partial_y+\mathrm{i}lV't_*)g\|_{L_y^2}+C2^{(\f52-\gamma)|m|}\| |V'|^{-\gamma}\Psi^{\pm}_m(y) g\|_{L_y^2},\nonumber
  \end{align}
  and
  \begin{align}\label{est:IIpmm}
     \|II_{m_1}^{\pm}(t,\eta)\|_{L^2_\eta}\leq &C \||V'|^{-\f32}\Psi_{m_1}^{\pm}(y)\partial_y\Psi^{\pm}_m(y)g\|_{L_y^2} \\
     \leq& C2^{(\f 52-\gamma)|m|}\||V'|^{-\gamma}\Psi_{m_1}^{\pm}(y)g\|_{L_y^2}, \nonumber
  \end{align}
  here we used $|\partial_y\Psi^{\pm}_m(y)|\leq C2^{|m|}$.
  
 On one hand, we have 
  \begin{align*}
     &\dfrac{1}{\sqrt{2\pi}}\int_{y_2-2\pi}^{y_2} (\partial_y+\mathrm{i}lV's)\big(\Psi^{\pm}_m(y)g(t,y)/V'(y)\big)\mathrm{e}^{-\mathrm{i} V(y)\eta}\mathrm{d}y\\
     =&\dfrac{1}{\sqrt{2\pi}}\int_{y_2-2\pi}^{y_2} \big[\Psi^{\pm}_m(y)(\partial_y+\mathrm{i}lV's)\big(g(t,y)/V'(y)\big)\mathrm{e}^{-\mathrm{i} V(y)\eta} +\partial_y\Psi^{\pm}_m({y})g\mathrm{e}^{-\mathrm{i} V(y)\eta}\big]\mathrm{d}y\\
     =&I_m^{\pm}(t,\eta) +\dfrac{1}{\sqrt{2\pi}}\sum_{|m_1-m|\leq 1}\int_{-\pi}^{\pi} \Psi_{m_1}^{\pm}(y) \partial_y\Psi^{\pm}_m(y){(g(t,y)/ V'(y))}\mathrm{e}^{-\mathrm{i} V(y)\eta}\mathrm{d}y\\
     =&I_m^{\pm}(t,\eta) +\sum_{|m_1-m|\leq 1}II_{m_1}^{\pm}(t,\eta),
  \end{align*}
  here we used the fact 
  \begin{align*}
     \partial_y\Psi^{\pm}_m(y)=& \big(\Psi_{m-1}^{\pm}(y)+\Psi_{m}^{\pm}(y)+\Psi_{m+1}^{\pm}(y)\big) \partial_y\Psi^{\pm}_m(y).
  \end{align*}
  On the other hand, we {have}
     \begin{align*}
     &\dfrac{1}{\sqrt{2\pi}}\int_{y_2-2\pi}^{y_2} (\partial_y+\mathrm{i}lV's)\big(\Psi^{\pm}_m(y)g(t,y)/V'(y)\big)\mathrm{e}^{-\mathrm{i} V(y)\eta}\mathrm{d}y \\
     =&\dfrac{\mathrm{i}(\eta+ls)}{\sqrt{2\pi}}\int_{y_2-2\pi}^{y_2} 
     \Psi^{\pm}_m(y)g(t,y)\mathrm{e}^{-\mathrm{i} V(y)\eta}\mathrm{d}y \\
     =&\mathrm{i}(\eta+ls)\widetilde{g}^{\pm}_{m}(t,\eta).
  \end{align*}
  Thus, we conclude 
  \begin{align*}
     & \mathrm{i}(\eta+ls)\widetilde{g}^{\pm}_{m}(t,\eta)= I_m^{\pm} (t,\eta)+\sum_{|m_1-m|\leq 1}II_{m_1}^{\pm}(t,\eta),
  \end{align*}
  which along with \eqref{est:Ipmm} and \eqref{est:IIpmm} shows 
  \begin{align*}
     &\||\eta+ls|\widetilde{g}^{\pm}_m(t,\eta)\|_{L^2_\eta}
     \leq C2^{\f {3|m|}{2}}\|\Psi^{\pm}_m(y) (\partial_y+\mathrm{i}lV's)g\|_{L_y^2}+C2^{(\f52-\gamma)|m|}\| |V'|^{-\gamma}\Psi^{\pm}_m(y) g\|_{L_y^2}\\
     &\qquad+C\sum_{|m_1-m|\leq 1}2^{(\f 52-\gamma)|m|}\||V'|^{-\gamma}\Psi_{m_1}^{\pm}(y)g\|_{L_y^2}.
  \end{align*}
This proves \eqref{est:etagmppm}. 
 \end{proof}

\subsection{Inviscid damping estimates}

\begin{lemma}\label{lem5.7}
Let $\om(t,y)$ solve 
\beno
\partial_t{\om}+\mathrm{i}k\mathrm{e}^{-\nu t}V {\om}+f=0
\eeno
for $t\in I_j\subseteq[0,\nu^{-1}]$ and $y\in \mathbb{T}_{2\pi}$, $\om(t_{j-1},y)=0$. 
Let $\psi=\Delta_k^{-1}\omega $. Assume that $k,l\in\alpha\Z$, $k\neq l$, $|l|\leq \nu^{-1/3}$, $|k-l|\leq \nu^{-1/3}$. Then it holds that
\begin{align}\label{9.2}
&\||V'|^{-\f12}\om\|_{L^{\infty}L^2}^2+|k|^2\|\big(\partial_y,k)\psi\|_{L^{2}L^2}^2\\  \notag &\qquad\leq  C(\gamma)|k-l|^{-1}(\|\partial_y(\mathrm{e}^{\mathrm{i}lt_*V}f)\|_{L^2L^2}^2+\||V'|^{-\gamma}f\|_{L^2L^2}^2\big),\ \forall\ \gamma>2,
\\
 \label{9.3} &\|V'\om\|_{L^{\infty}L^2}^2\leq C|k-l|^{-2}\big(\|V'\partial_y(\mathrm{e}^{\mathrm{i}lt_*V}f)\|_{L^2L^2}^2+|k-l|^{2}\|f\|_{L^2L^2}^2\big),
\end{align}
and
\begin{align*}
   & \|\omega\|_{L^\infty L^2}^2+|k|^2\|(\partial_y,k)\psi\|_{L^{2}L^2}^2 +\nu \|(\partial_y,k)\omega\|_{L^2L^2}^2\\
   &\leq  C\big(1+|l|^2j^2|k-l|^{-2}\big)
\big(\|V'\partial_y(\mathrm{e}^{\mathrm{i}lt_*V}f)\|_{L^2L^2}^2+|k-l|^{2}\|f\|_{L^2L^2}^2\big)\\
&\qquad+C\big(\gamma)|k-l|^{-1}(\|\partial_y(\mathrm{e}^{\mathrm{i}lt_*V}f)\|_{L^2L^2}^2+
\||V'|^{-\gamma}f\|_{L^2L^2}^2\big),\ \forall\ \gamma>2.
\end{align*}
Here $L^pL^q$ denotes $L^p(I_j;L^q(\mathbb{T}_{2\pi}))$ and $C$ only depends on $\gamma $.
\end{lemma}

\begin{proof}
\textit{Step 1.} Estimate of $\||V'|^{-\f12}\omega\|_{L^\infty L^2}$.\smallskip

We introduce 
\begin{align*}
   & \widetilde{\omega}_m^{\pm}(t,\eta)=\dfrac{1}{\sqrt{2\pi}}\int_{y_2-2\pi}^{y_2}\omega(t,y)\Psi_m^{\pm} (y)\mathrm{e}^{-\mathrm{i} V(y)\eta}\mathrm{d}y,\\
   & \widetilde{f}_m^{\pm}(t,\eta)=\dfrac{1}{\sqrt{2\pi}}
   \int_{y_2-2\pi}^{y_2}f(t,y)\Psi^{\pm}_m(y)\mathrm{e}^{-\mathrm{i} V(y)\eta}\mathrm{d}y,\qquad m\in\mathbb{Z}.
\end{align*}
 Notice that
 \begin{align*}
    & \partial_t\left(\mathrm{e}^{\mathrm{i}kVt_*}\Psi^{\pm}_m\omega\right) +\mathrm{e}^{\mathrm{i}kVt_*}\Psi^{\pm}_mf=0,
  \end{align*}
  and $\omega(t_{j-1},y)=0$. We have
  \begin{align*}
     & \dfrac{\mathrm{d}}{\mathrm{d}t}\widetilde{\omega}_m^{\pm}(t,\eta-k t_*)+\widetilde{f}_m^{\pm}(t,\eta-kt_*)=0,\quad \tilde{\omega}_m^{\pm}(t_{j-1},\eta-kt_*(t_{j-1}))=0.
  \end{align*}
  
  For fixed $\eta\in\mathbb{R},$ let $W_m^{\pm}(\eta)=\text{sup}_{t\in I_j}|\tilde{\omega}_m^{\pm}(t,\eta-kt_*)|,\ T\leq \nu^{-1}.$ Then we have
  \begin{align}\label{est:Wmpm}
    &W_m^{\pm}(\eta) \leq \int_{I_j}|\tilde{f}^{\pm}_m(t,\eta-kt_*(t))|\mathrm{d}t\\
    &\leq  \left(\int_{I_j}( 2^{-(\gamma+1)|m|} |\eta-kt_*+lt_*|^2+a^22^{(\gamma+1)|m|})|\widetilde{f}_m^{\pm}(t,\eta-kt_*)|^2\mathrm{d}t\right)^{\f12} \nonumber \\
    &\quad\times\left(\int_{T_0}^{T}\dfrac{\mathrm{d}t}{ 2^{-(\gamma+1)|m|} |\eta-kt_*+lt_*|^2+a^22^{(\gamma+1)|m|}}\right)^{\f12},\ \forall\ \gamma\in\R,\ a>0.\nonumber
  \end{align}
  Notice that$(s=(k-l)t_*(t)-\eta)$
  \begin{align}\label{est:intT0T}
    & \int_{T_0}^{T}\dfrac{\mathrm{d}t}{ 2^{-(\gamma+1)|m|} |\eta-kt_*+lt_*|^2+a^22^{(\gamma+1)|m|}}\\  &=\dfrac{1}{k-l}\int_{(k-l)t_*(T_0)-\eta}^{(k-l)t_*(T)-\eta}\dfrac{\mathrm{e}^{\nu t}\mathrm{d}s}{2^{-(\gamma+1)|m|}|s|^2+a^22^{(\gamma+1)|m|}}\nonumber\\
&\leq\dfrac{1}{k-l}\int_{(k-l)t_*(T_0)-\eta}^{(k-l)t_*(T)-\eta}\dfrac{\mathrm{e}^{\nu T}\mathrm{d}s}{2^{-(\gamma+1)|m|}|s|^2+a^22^{(\gamma+1)|m|}}\nonumber\\
&\leq \dfrac{\mathrm{e}^{\nu T}}{|k-l|}\int_{\mathbb{R}}\dfrac{\mathrm{d}s}{2^{-(\gamma+1)|m|}|s|^2+a^22^{(\gamma+1)|m|}}\nonumber\\
&=\dfrac{\pi\mathrm{e}^{\nu T}}{|k-l|}\dfrac{1}{a}\leq\dfrac{\pi\mathrm{e}}{|k-l|a},\ \forall\ \gamma\in\R,\ a>0,\nonumber
  \end{align}
 By \eqref{est:etagmppm} in Lemma \ref{lem:plancherel} (for $s=t_*$), we have 
 \begin{align*}
    & 2^{-3|m|}\||\eta+lt_*|\widetilde{f}^{\pm}_m(t,\eta)\|_{L^2_\eta}^2
     \leq C\|\Psi^{\pm}_m(y)(\partial_y+\mathrm{i}lV't_*)f\|^2_{L^2_y}\\
     &\quad+C2^{2(1-\gamma)|m|}\| |V'|^{-\gamma}\Psi^{\pm}_m(y)f\|_{L^2_y}^2+C\sum_{|m_1-m|\leq 1}2^{2(1-\gamma)|m|}\||V'|^{-\gamma}\Psi_{m_1}^{\pm}(y)f\|_{L^2_y}^2,\ \forall\ \gamma\in\R,
 \end{align*}
 and
\begin{align*}
  2^{(2\gamma-1)|m|}\|\widetilde{f}^{\pm}_m(t,\eta)\|^2_{L^2_\eta}\leq C\||V'|^{-\gamma} \Psi_m^{\pm}(y)f(t,y)\|_{L^2_y}^2,\ \forall\ \gamma\in\R.
\end{align*}
Thus, we conclude that for all $\gamma\geq1,\ a\geq 1$,
\begin{align*}
   & \int_{\mathbb{R}}( 2^{-(\gamma+1)|m|} |\eta-kt_*+lt_*|^2+a^22^{(\gamma+1)|m|})|\widetilde{f}_m^{\pm}(t,\eta-kt_*)|^2\mathrm{d}\eta\\
   &\leq 2^{-(\gamma+1)|m|} \||\eta+lt_*|\widetilde{f}^{\pm}_m(t,\eta)\|_{L^2_\eta}^2+a^22^{(\gamma+1)|m|} \|\widetilde{f}^{\pm}_m(t,\eta)\|_{L^2_\eta}^2\\
   &= 2^{-(\gamma-2)|m|} \big(2^{-3|m|}\||\eta+lt_*|\widetilde{f}^{\pm}_m(t,\eta)\|^2_{L^2_\eta}+ a^2 2^{(2\gamma-1)|m|}\|\widetilde{f}^{\pm}_m(t,\eta)\|_{L^2_\eta}^2\big)\\
   &\leq C2^{-(\gamma-2)|m|}  \Big(\|\Psi^{\pm}_m(y)(\partial_y+\mathrm{i}lV't_*)f\|^2_{L^2_y}+2^{2(1-\gamma)|m|}\| |V'|^{-\gamma}\Psi^{\pm}_m(y) f\|_{L^2_y}^2\\
     &\qquad+\sum_{|m_1-m|\leq 1}2^{2(1-\gamma)|m|}\||V'|^{-\gamma}\Psi_{m_1}^{\pm}(y)f\|_{L^2_y}^2+a^2\||V'|^{-\gamma} \Psi_m^{\pm}(y)f(t,y)\|_{L^2_y}^2 \Big)\\
     &\leq C2^{-(\gamma-2)|m|} \Big(\|\Psi^{\pm}_m(y)(\partial_y+\mathrm{i}lV't_*)f\|^2_{L^2_y}+\sum_{|m_1-m|\leq 1}\|a|V'|^{-\gamma}\Psi_{m_1}^{\pm}(y)f\|_{L^2_y}^2 \Big),
\end{align*}
 which along with \eqref{est:Wmpm} and  \eqref{est:intT0T} shows 
 \begin{align}\label{Wm1}
    & \gamma\geq1,\ a\geq 1\Rightarrow\int_{\mathbb{R}}|W_m^{\pm}(\eta)|^2\mathrm{d}\eta\\
 \notag   &\quad\leq \dfrac{C 2^{-(\gamma-2)|m|}}{|k-l|a} \Big(\|\Psi^{\pm}_m(y)(\partial_y+\mathrm{i}lV't_*)f\|^2_{L^2L^2_y}+\sum_{|m_1-m|\leq 1}\|a|V'|^{-\gamma}\Psi_{m_1}^{\pm}(y)f\|_{L^2L^2_y}^2 \Big).
 \end{align}
 Then by \eqref{est:gmppm} in Lemma \ref{lem:plancherel}, \eqref{eq:Psisum2} and \eqref{Wm1} (for $a=1$), we infer that for $t\in I_j$,
 \begin{align*}
     &\||V'|^{-\f12}\omega(t)\|_{L^2}^2\leq 2\sum_{m\in\Z}\||V'|^{-\f12}\Psi^{-}_m(y)\omega(t)\|_{ L^2}^2+ 2 \sum_{m\in\Z}\||V'|^{-\f12}\Psi^{+}_m(y)\omega(t)\|_{ L^2}^2\\
      = &2\sum_{m\in\Z}\|\tilde{\omega}^-_m(t,\eta)\|_{ L^2}^2+ 2\sum_{m\in\Z}\|\tilde{\omega}^-_m(t,\eta)\|_{ L^2}^2\\
      =&2\sum_{m\in\Z}\int_{\mathbb{R}}|\tilde{\omega}^-_m(t,\eta-kt_*)|^2\mathrm{d}\eta+  2\sum_{m\in\Z}\int_{\mathbb{R}}|\tilde{\omega}^+_m(t,\eta-kt_*)|^2\mathrm{d}\eta\\
      \leq &2\sum_{m\in\Z}\int_{\mathbb{R}}|W_m^{-}(\eta)|^2\mathrm{d}\eta+ 2\sum_{m\in\Z}\int_{\mathbb{R}}|W_m^{+}(\eta)|^2\mathrm{d}\eta\\
      \leq &\sum_{m\in\Z} \dfrac{C 2^{-(\gamma-2)|m|}}{|k-l|}  \Big(\|\Psi^{-}_m(y) (\partial_y+\mathrm{i}lV't_*)f\|^2_{L^2L^2_y}+\sum_{|m_1-m|\leq 1}\||V'|^{-\gamma}\Psi_{m_1}^{-}(y)f\|_{L^2L^2_y}^2 \Big)\\
     &+\sum_{m\in\Z} \dfrac{C 2^{-(\gamma-2)|m|}}{|k-l|}  \Big(\|\Psi^{+}_m(y) (\partial_y+\mathrm{i}lV't_*)f\|^2_{L^2L^2_y}+\sum_{|m_1-m|\leq 1}\||V'|^{-\gamma}\Psi_{m_1}^{+}(y)f\|_{L^2L^2_y}^2 \Big)\\
     \leq &C|k-l|^{-1}\big(\| (\partial_y+\mathrm{i}lV't_*)f\|^2_{L^2L^2}+\||V'|^{-\gamma}f\|_{L^2L^2}^2 \big)\\
     =&C|k-l|^{-1}\big(\|\partial_y(\mathrm{e}^{\mathrm{i}lt_*V}f)\|_{L^2L^2}^2+\||V'|^{-\gamma}f\|_{L^2L^2}^2 \big),\quad \forall\ \gamma\geq 2.
 \end{align*}
 
 This shows the estimates for $\||V'|^{-\f12}\omega\|_{L^\infty L^2}$.\smallskip
 
 \textit{Step 2.} Estimate of $\|V'\omega\|_{L^\infty L^2}$. \smallskip
 
   By \eqref{est:gmppm} in Lemma \ref{lem:plancherel}, \eqref{eq:Psisum2}, \eqref{Wm1} (for $\gamma=1$, $a=|k-l|$) and $ |V'(y)|\sim 2^{-|m|} $ on the support of $\Psi^{\pm}_m$, we infer that for $t\in I_j$, 
 \begin{align*}
     &\|V'\omega(t)\|_{L^2}^2\leq 2\sum_{m\in\Z}\|V'\Psi^{-}_m(y)\omega(t)\|_{ L^2}^2+ 2\sum_{m\in\Z}\|V'\Psi^{+}_m(y)\omega(t)\|_{ L^2}^2\\
     &\leq C\sum_{m\in\Z}2^{-3|m|}\||V'|^{-\f12}\Psi^{-}_m(y)\omega(t)\|_{ L^2}^2+ C\sum_{m\in\Z}2^{-3|m|}\||V'|^{-\f12}\Psi^{+}_m(y)\omega(t)\|_{ L^2}^2\\
      &= C\sum_{m\in\Z}2^{-3|m|}\|\tilde{\omega}^-_m(t,\eta)\|_{ L^2}^2+ C\sum_{m\in\Z}2^{-3|m|}\|\tilde{\omega}^-_m(t,\eta)\|_{ L^2}^2\\
      &=C\sum_{m\in\Z}2^{-3|m|}\int_{\mathbb{R}}|\tilde{\omega}^-_m(t,\eta-kt_*)|^2\mathrm{d}\eta+  C\sum_{m\in\Z}2^{-3|m|}\int_{\mathbb{R}}|\tilde{\omega}^+_m(t,\eta-kt_*)|^2\mathrm{d}\eta\\
      &\leq C\sum_{m\in\Z}2^{-3|m|}\int_{\mathbb{R}}|W_m^{-}(\eta)|^2\mathrm{d}\eta+ C\sum_{m\in\Z}2^{-3|m|}\int_{\mathbb{R}}|W_m^{+}(\eta)|^2\mathrm{d}\eta\\
      &\leq \sum_{m\in\Z} \dfrac{C 2^{-3|m|}2^{|m|}}{|k-l|^2}  \Big(\|\Psi^{-}_m(y) (\partial_y+\mathrm{i}lV't_*)f\|^2_{L^2L^2_y}+\sum_{|m_1-m|\leq 1}\||k-l||V'|^{-1}\Psi_{m_1}^{-}(y)f\|_{L^2L^2_y}^2 \Big)\\
     &\quad+\sum_{m\in\Z} \dfrac{C 2^{-3|m|}2^{|m|}}{|k-l|^2}  \Big(\|\Psi^{+}_m(y) (\partial_y+\mathrm{i}lV't_*)f\|^2_{L^2L^2_y}+\sum_{|m_1-m|\leq 1}\||k-l||V'|^{-1}\Psi_{m_1}^{+}(y)f\|_{L^2L^2_y}^2 \Big)\\
     &\leq C|k-l|^{-2}\big(\|V' (\partial_y+\mathrm{i}lV't_*)f\|^2_{L^2L^2}+|k-l|^2\|f\|_{L^2L^2}^2 \big)\\
     &=C|k-l|^{-2}\big(\|V'\partial_y(\mathrm{e}^{\mathrm{i}lt_*V}f)\|_{L^2L^2}^2+|k-l|^2\|f\|_{L^2L^2}^2 \big).
 \end{align*}
This completes the proof of \eqref{9.3}.\smallskip

\textit{Step 3.} Estimate of $|k|\|(\partial_y,k)\psi\|_{L^2 L^2}$. \smallskip

We first claim that  
\begin{align}\label{am}
   & |k|\|(\partial_y,k)\psi(t)\|_{L^2 }\leq C\sum_{m\in\Z}(a_m^{+}(t)+a_m^{-}(t)),
 \end{align}
where
\begin{align}\label{ampm}
&a_m^{\pm}(t):=\left|\int_{\R}\dfrac{|k|^22^{-|m|}|\widetilde{\omega}_m^{\pm}(t,\eta)|^2}{(2^{-|m|}\eta)^2+k^2}\mathrm{d}\eta
   +|k||\widetilde{\omega}_m^{\pm}(t,0)|^2\right|^{1/2}.
\end{align}

 By using the change of variables $s=kt_*(t)-\eta$ or $s=kt_*(t)$, and \eqref{Wm1} for $a=1$, we obtain
 \begin{align*}
 \int_{I_j}|a_m^{\pm}(t)|^2\mathrm{d}t=&\int_{I_j}
 \int_{\R}\dfrac{|k|^22^{-|m|}|\widetilde{\omega}_m^{\pm}(t,\eta-kt_*)|^2}{2^{-2|m|}|\eta-kt_*|^2+k^2}\mathrm{d}\eta
   +|k|\int_{I_j}|\tilde{\omega}_m^{\pm}(t,\eta-kt_*)|^2\big|_{\eta=kt_*}\mathrm{d}t\\
   \le & \int_{T_0}^T\int_{\mathbb{R}}\dfrac{|k|^22^{-|m|}|W_m^{\pm}(\eta)|^2}{2^{-2|m|}|\eta-kt_*|^2+k^2}
    \mathrm{d}\eta\mathrm{d}t+|k|\int_{T_0}^T|W_m^{\pm}(kt_*)|^2\mathrm{d}t\\ \leq &
    \int_{\mathbb{R}}|W_m^{\pm}(\eta)|^2\int_{kt_*(T_0)-\eta}^{kt_*(T)-\eta}
    \dfrac{k2^{-|m|}\mathrm{e}^{\nu t}\mathrm{d}s}{ 2^{-2|m|}|s|^2+k^2}\mathrm{d}\eta+
    \mathrm{sgn}(k)\int_{kt_*(T_0)}^{kt_*(T)}\mathrm{e}^{\nu t}|W_m^{\pm}(s)|^2\mathrm{d}s\\ \leq &
    \int_{\mathbb{R}}|W_m^{\pm}(\eta)|^2\int_{\R}
    \dfrac{|k|2^{-|m|}\mathrm{e}^{\nu T}\mathrm{d}s}{ 2^{-2|m|}|s|^2+k^2}\mathrm{d}\eta+
    \int_{\R}\mathrm{e}^{\nu T}|W_m^{\pm}(s)|^2\mathrm{d}s\\
    =&\mathrm{e}^{\nu T}(1/\pi+1)\int_{\mathbb{R}}|W_m^{\pm}(\eta)|^2\mathrm{d}\eta\\
    \leq& \dfrac{C 2^{-(\gamma-2)|m|}}{|k-l|} \Big(\|\Psi^{\pm}_m(y)(\partial_y+\mathrm{i}lV't_*)f\|^2_{L^2L^2_y}+\sum_{|m_1-m|\leq 1}\||V'|^{-\gamma}\Psi_{m_1}^{\pm}(y)f\|_{L^2L^2_y}^2 \Big)\\
    \leq& \dfrac{C 2^{-(\gamma-2)|m|}}{|k-l|} \Big(\|(\partial_y+\mathrm{i}lV't_*)f\|^2_{L^2L^2_y}+\||V'|^{-\gamma}f\|_{L^2L^2_y}^2 \Big)\\
    =&C2^{-(\gamma-2)|m|}|k-l|^{-1}\big(\|\partial_y(\mathrm{e}^{\mathrm{i}lt_*V}f)\|_{L^2L^2}^2+\||V'|^{-\gamma}f\|_{L^2L^2}^2 \big),
 \end{align*}
 which along with \eqref{am} gives(for fixed $\gamma>2$)
 \begin{align}\label{9.2a}
   & |k|\|(\partial_y,k)\psi\|_{L^2L^2 }\leq C\sum_{m\in\Z}\|a_m^{+}(t)+a_m^{-}(t)\|_{L^2(I_j)}\\ 
 \notag  &\leq C\sum_{m\in\Z}2^{-(\gamma-2)|m|/2}|k-l|^{-1/2}\big(\|\partial_y(\mathrm{e}^{\mathrm{i}lt_*V}f)\|_{L^2L^2}+\||V'|^{-\gamma}f\|_{L^2L^2} \big)
   \\ \notag &\leq C|k-l|^{-1/2}\big(\|\partial_y(\mathrm{e}^{\mathrm{i}lt_*V}f)\|_{L^2L^2}+\||V'|^{-\gamma}f\|_{L^2L^2} \big).
\end{align}

 It remains to prove \eqref{am}. As \eqref{am} is clearly true for $k=0$, we assume $k\neq 0$. As $\psi=\Delta_k^{-1}\omega $, 
we get  by \eqref{eq:Psisum1}  that
\begin{align*}
   & \|(\partial_y,k)\psi(t)\|_{L^2}^2=-\int_{y_2-2\pi}^{y_2}\omega(t,y)\overline{\psi(t,y)}\mathrm{d}y=
   -\sum_{m\in\Z}\int_{y_2-2\pi}^{y_2}(\Psi^{+}_m(y)+\Psi^{-}_m(y))\omega(t,y)\overline{\psi(t,y)}\mathrm{d}y.
\end{align*} 
Thus, it is enough to prove that
\begin{align}\label{k1}
   & \left|\int_{y_2-2\pi}^{y_2}\Psi^{\pm}_m(y)\omega(t,y)\overline{\psi(t,y)}\mathrm{d}y\right|\leq C|k|^{-1}a_{m}^{\pm}(t)\|(\partial_y,k)\psi(t)\|_{L^2}.
\end{align} 

{\it Case 1.} If $m\in\Z$, $2^{|m|}\leq |k|$, then we get by using \eqref{gm1} in Lemma \ref{lem:plancherel} that
\begin{align*}
   & \int_{y_2-2\pi}^{y_2}\Psi^{\pm}_m(y)\omega(t,y)\overline{\psi(t,y)}\mathrm{d}y=
   \sum_{|m_1-m|\leq 1}\int_{y_2-2\pi}^{y_2}\Psi^{\pm}_m(y)\Psi^{\pm}_{m_1}(y)\omega(t,y)\psi(t,y)\mathrm{d}y\\&=
   \sum_{|m_1-m|\leq 1}\int_{\R}\widetilde{\omega}^{\pm}_m(t,\eta)\overline{{h}^{\pm}_{m_1}(t,\eta)}\mathrm{d}\eta,
\end{align*}
where  
\begin{align*}
 h_m^{\pm}(t,\eta):=\dfrac{1}{\sqrt{2\pi}}\int_{y_2-2\pi}^{y_2}|V'(y)|\psi(t,y)\Psi_m^{\pm} (y)\mathrm{e}^{-\mathrm{i} V(y)\eta}\mathrm{d}y.
\end{align*}
Thus,
\begin{align*}
   & \left|\int_{y_2-2\pi}^{y_2}\Psi^{\pm}_m(y)\omega(t,y)\overline{\psi(t,y)}\mathrm{d}y\right|\\ &\leq
   \sum_{|m_1-m|\leq 1}\left|\int_{\R}\dfrac{|\widetilde{\omega}_m^{\pm}(t,\eta)|^2}{(2^{-|m|}\eta)^2+k^2}\mathrm{d}\eta\right|^{\f12}
   \left|\int_{\R}((2^{-|m|}\eta)^2+k^2)|{{h}^{\pm}_{m_1}(t,\eta)}|^2\mathrm{d}\eta\right|^{\f12}\\ &\leq
   \sum_{|m_1-m|\leq 1}a_m^{\pm}(t)2^{|m|/2}|k|^{-1}
   \left|\int_{\R}((2^{-|m|}\eta)^2+k^2)|{{h}^{\pm}_{m_1}(t,\eta)}|^2\mathrm{d}\eta\right|^{\f12}.
\end{align*}
Then, it is enough to prove that
\begin{align}\label{k2}
   & 2^{|m|}\int_{\R}((2^{-|m|}\eta)^2+k^2)|\overline{{h}^{\pm}_{m}(t,\eta)}|^2\mathrm{d}\eta\leq C\|(\partial_y,k)\psi(t)\|_{L^2}^2,\quad \ \forall\ m\in\Z,\ 2^{|m|}\leq 2|k|.
\end{align}
By \eqref{est:gmppm} in Lemma \ref{lem:plancherel}, we have
\begin{align}\label{km1}
   &\int_{\R}|{{h}^{\pm}_{m}(t,\eta)}|^2\mathrm{d}\eta=\||V'|^{-\f12}|V'|\Psi^{\pm}_m(y)\psi(t)\|_{ L^2}^2\leq 
   C2^{-|m|}\|\Psi^{\pm}_m(y)\psi(t)\|_{ L^2}^2,
\end{align}
and by \eqref{est:etagmppm} in Lemma \ref{lem:plancherel} (for $s=0$, $ \gamma=1$), 
 \begin{align}\label{km2}
    & \int_{\R}\eta^2|\overline{{h}^{\pm}_{m}(t,\eta)}|^2\mathrm{d}\eta
     \leq C2^{3|m|}\|\Psi^{\pm}_m(y)\partial_y(|V'|\psi)\|^2_{L^2_y}\\
   \notag  &\qquad+C2^{3|m|}\| \Psi^{\pm}_m(y)\psi\|_{L^2_y}^2+C\sum_{|m_1-m|\leq 1}2^{3|m|}\|\Psi_{m_1}^{\pm}(y)\psi\|_{L^2_y}^2\\
   \notag   &\leq C2^{3|m|}\|\Psi^{\pm}_m(y)|V'|\partial_y\psi\|^2_{L^2_y}+C2^{3|m|}\| \psi\|_{L^2_y}^2\\
   \notag   &\leq C2^{|m|}\|\partial_y\psi\|^2_{L^2_y}+C2^{|m|}|k|^2\| \psi\|_{L^2_y}^2= C2^{|m|}\|(\partial_y\psi,k)\|^2_{L^2_y}.
 \end{align}
 Now, \eqref{k2} follows from \eqref{km1} and \eqref{km2}, and \eqref{k2} implies \eqref{k1} for {\it Case 1}.\smallskip
 
 {\it Case 2.} If $m\in\Z$, $m>0$, $2^{|m|}> |k|$, then we get by using \eqref{gm1} in Lemma \ref{lem:plancherel} that 
 \begin{align*}
   & \int_{y_2-2\pi}^{y_2}\Psi^{\pm}_m(y)\omega(t,y)\overline{\psi(t,y)}\mathrm{d}y\\&=
   \int_{y_2-2\pi}^{y_2}\Psi^{\pm}_m(y)\omega(t,y)\overline{\psi(t,y)-\psi(t,y_1)}\mathrm{d}y+
   \overline{\psi(t,y_1)}\int_{y_2-2\pi}^{y_2}\Psi^{\pm}_m(y)\omega(t,y)\mathrm{d}y\\&=
   \sum_{|m_1-m|\leq 1}\int_{y_2-2\pi}^{y_2}\Psi^{\pm}_m(y)\Psi^{\pm}_{m_1}(y)\omega(t,y)\overline{\psi(t,y)-\psi(t,y_1)}\mathrm{d}y+
   \overline{\psi(t,y_1)}\widetilde{\omega}_m^{\pm}(t,0)\\&=
   \sum_{|m_1-m|\leq 1}\int_{\R}\widetilde{\omega}^{\pm}_m(t,\eta)\overline{{h}^{\pm1}_{m_1}(t,\eta)}\mathrm{d}\eta+
   \overline{\psi(t,y_1)}\widetilde{\omega}_m^{\pm}(t,0),
\end{align*}
where 
\begin{align*}
 h_m^{\pm1}(t,\eta):=\dfrac{1}{\sqrt{2\pi}}\int_{y_2-2\pi}^{y_2}|V'(y)|(\psi(t,y)-\psi(t,y_1))\Psi_m^{\pm} (y)\mathrm{e}^{-\mathrm{i} V(y)\eta}\mathrm{d}y.
\end{align*}
Then we have
\begin{align*}
   & \left|\int_{y_2-2\pi}^{y_2}\Psi^{\pm}_m(y)\omega(t,y)\overline{\psi(t,y)}\mathrm{d}y\right|\\ \leq&
   \sum_{|m_1-m|\leq 1}\left|\int_{\R}\dfrac{|\widetilde{\omega}_m^{\pm}(t,\eta)|^2}{(2^{-|m|}\eta)^2+k^2}\mathrm{d}\eta\right|^{\f12}
   \left|\int_{\R}((2^{-|m|}\eta)^2+k^2)|\overline{{h}^{\pm1}_{m_1}(t,\eta)}|^2\mathrm{d}\eta\right|^{\f12}+
   \|\psi(t)\|_{L^{\infty}}|\tilde{\omega}_m^{\pm}(t,0)|\\ \leq&
   \sum_{|m_1-m|\leq 1}a_m^{\pm}(t)2^{|m|/2}|k|^{-1}
   \left|\int_{\R}((2^{-|m|}\eta)^2+k^2)|\overline{{h}^{\pm1}_{m_1}(t,\eta)}|^2\mathrm{d}\eta\right|^{\f12}+
   \|\psi(t)\|_{L^{\infty}}|k|^{-\f12}a_m^{\pm}(t).
\end{align*}
Thus, it is enough to prove that (note that $ \|\psi(t)\|_{L^{\infty}}\leq C|k|^{-\f12}\|(\partial_y,k)\psi(t)\|_{L^2}$)\begin{align}\label{k3}
   & 2^{|m|}\int_{\R}((2^{-|m|}\eta)^2+k^2)|\overline{{h}^{\pm1}_{m}(t,\eta)}|^2\mathrm{d}\eta\leq C\|(\partial_y,k)\psi(t)\|_{L^2}^2,\quad \ \forall\ m\in\Z,\ 2^m>|k|/2.
\end{align} 
 By \eqref{est:gmppm} in Lemma \ref{lem:plancherel}, we have
\begin{align}\label{km3}
   \int_{\R}|{{h}^{\pm1}_{m}(t,\eta)}|^2\mathrm{d}\eta=&\||V'|^{-\f12}|V'|\Psi^{\pm}_m(y)(\psi(t,y)-\psi(t,y_1))\|_{ L^2}^2\\
 \nonumber  \leq &C\||V'|^{\f12}\Psi^{\pm}_m(y)\|_{ L^2}^2\|\psi(t)\|_{L^{\infty}}^2 \leq  C2^{-2|m|}\|\psi(t)\|_{L^{\infty}}^2\\  \nonumber
  \leq& C 2^{-2|m|}|k|^{-1}\|(\partial_y,k)\psi(t)\|_{L^2}^2 \leq C 2^{-|m|}|k|^{-2}\|(\partial_y,k)\psi(t)\|_{L^2}^2.
\end{align}
By \eqref{est:etagmppm} in Lemma \ref{lem:plancherel} (for $s=0$, $ \gamma=1$), and Hardy's inequality, we have 
 \begin{align}\label{km4}
    &\int_{\R}\eta^2|\overline{{h}^{\pm1}_{m}(t,\eta)}|^2\mathrm{d}\eta
     \leq C2^{3|m|}\|\Psi^{\pm}_m(y)\partial_y[|V'(y)|(\psi(t,y)-\psi(t,y_1))]\|^2_{L^2_y}\\
   \notag  &\quad+C2^{3|m|}\| \Psi^{\pm}_m(y)(\psi(t,y)-\psi(t,y_1))\|_{L^2_y}^2\\
   \notag  &\quad+C\sum_{|m_1-m|\leq 1}2^{3|m|}\|\Psi_{m_1}^{\pm}(y)(\psi(t,y)-\psi(t,y_1))\|_{L^2_y}^2\\
   \notag   &\leq C2^{3|m|}\|\Psi^{\pm}_m(y)|V'|\partial_y\psi\|^2_{L^2_y}+
   C2^{3|m|}\sum_{|m_1-m|\leq 1}\|\Psi_{m_1}^{\pm}(y)(\psi(t,y)-\psi(t,y_1))\|_{L^2_y}^2\\
   \notag   &\leq C2^{|m|}\|\partial_y\psi\|^2_{L^2_y}+C2^{|m|}\| (\psi(t,y)-\psi(t,y_1))/(y-y_1)\|_{L^2_y}^2
   \leq C2^{|m|}\|\partial_y\psi\|^2_{L^2_y}.
 \end{align} 
 Now \eqref{k3} follows from \eqref{km3} and \eqref{km4}, and \eqref{k3} implies \eqref{k1} for {\it Case 2}.\smallskip
 
{\it Case 3.} If $m\in\Z$, $m<0$, $2^{|m|}> |k|$, the proof of \eqref{k1} is similar to {\it Case 2} by using
 \begin{align*}
   & \int_{y_2-2\pi}^{y_2}\Psi^{\pm}_m(y)\omega(t,y)\overline{\psi(t,y)}\mathrm{d}y\\&=
   \int_{y_2-2\pi}^{y_2}\Psi^{\pm}_m(y)\omega(t,y)\overline{\psi(t,y)-\psi(t,y_2)}\mathrm{d}y+
   \overline{\psi(t,y_2)}\int_{y_2-2\pi}^{y_2}\Psi^{\pm}_m(y)\omega(t,y)\mathrm{d}y.
\end{align*}

Now we have completed the proof of \eqref{k1}, hence  \eqref{am} and  \eqref{9.2a}. Moreover \eqref{9.2} follows from {\it Step 1} and  {\it Step 3}.

Finally,  we use \eqref{9.2} and \eqref{9.3} to finish the proof of  this lemma. Notice that 
\begin{align*}
&\partial_t(\mathrm{e}^{\mathrm{i}lt_*V}\om)+\mathrm{i}(k-l)\mathrm{e}^{-\nu t}V(\mathrm{e}^{\mathrm{i}lt_*V}\om)+\mathrm{e}^{\mathrm{i}lt_*V}f=0,\\
&(\partial_t+\mathrm{i}(k-l)\mathrm{e}^{-\nu t}V)\partial_y(\mathrm{e}^{\mathrm{i}lt_*V}\om)+\mathrm{i}(k-l)\mathrm{e}^{-\nu t}V'\mathrm{e}^{\mathrm{i}lt_*V}\om
+\partial_y(\mathrm{e}^{\mathrm{i}lt_*V}f)=0.
\end{align*}
Thanks to $\om|_{t=t_{j-1}}=0,\ \partial_y(\mathrm{e}^{\mathrm{i}lt_*V}\om)|_{t=t_{j-1}}=0$, we infer that for $t\in I_j,$
\begin{align*}
\|\partial_y(\mathrm{e}^{\mathrm{i}lt_*V}\om(t))\|_{L^2}&\leq\int_{t_{j-1}}^t\|\mathrm{i}(k-l)\mathrm{e}^{-\nu s}V'\mathrm{e}^{\mathrm{i}ls_*V}\om(s)
+\partial_y(\mathrm{e}^{\mathrm{i}ls_*V}f(s))\|_{L^2}\mathrm{d}s
\\ &\leq\|k-l\|_{L^1(t_{j-1},t)}\|V'\om\|_{L^{\infty}L^2}+\|1\|_{L^2(t_{j-1},t)}\|\partial_y(\mathrm{e}^{\mathrm{i}lt_*V}f)\|_{L^2L^2}
\\&\leq|k-l|\nu^{-\f13}\|V'\om\|_{L^{\infty}L^2}+\nu^{-\f16}\|\partial_y(\mathrm{e}^{\mathrm{i}lt_*V}f)\|_{L^2L^2}.
\end{align*}
Thanks to $\partial_y{\om}=\mathrm{e}^{-\mathrm{i}lt_*V}\partial_y(\mathrm{e}^{\mathrm{i}t_*lV}\om)-\mathrm{i}lt_*V'\om$, we have 
(also for $t\in I_j$)
\begin{align*}
\|\partial_y\om(t)\|_{L^2}&\leq\|\partial_y(\mathrm{e}^{\mathrm{i}lt_*V}\om(t))\|_{L^2}+|lt_*|\|V'\om(t)\|_{L^2}
\\&\leq(|k-l|+|l|j)\nu^{-\f13}\|V'\om\|_{L^{\infty}L^2}+\nu^{-\f16}\|\partial_y(\mathrm{e}^{\mathrm{i}lt_*V}f)\|_{L^2L^2}.
\end{align*}
This shows that 
\begin{align*}
  &\nu  \|\partial_y\om\|^2_{L^2L^2}+\nu |k|^2\|\om\|^2_{L^2L^2}\\ &\leq 2\nu^{\f13}(|k-l|+|l|j)^2|I_j|\|V'\om\|_{L^{\infty}L^2}^2+
  2\nu^{\f23}|I_j|\|\partial_y(\mathrm{e}^{\mathrm{i}lt_*V}f)\|_{L^2L^2}^2+\nu |k|^2|I_j|\|\om\|_{L^{\infty}L^2}^2\\ &\leq 2(|k-l|+|l|j)^2\|V'\om\|_{L^{\infty}L^2}^2+
  2\nu^{\f13}\|\partial_y(\mathrm{e}^{\mathrm{i}lt_*V}f)\|_{L^2L^2}^2+{\nu^{\f23} |k|^2\|\om\|_{L^{\infty}L^2}^2}. 
  \end{align*}
Thus, we get by using $|k|\leq |k-l|+|l|\leq 2\nu^{-1/3}$, \eqref{9.2} and \eqref{9.3} that
\begin{align*}
   & \|\omega\|_{L^\infty L^2}^2+|k|^2\|(\partial_y,k)\psi\|_{L^{2}L^2}^2 +\nu \|(\partial_y,k)\omega\|_{L^2L^2}^2\\
   &\leq  2(|k-l|+|l|j)^2\|V'\om\|_{L^{\infty}L^2}^2+
  2\nu^{\f13}\|\partial_y(\mathrm{e}^{\mathrm{i}lt_*V}f)\|_{L^2L^2}^2+C\|\om\|_{L^{\infty}L^2}^2+|k|^2\|(\partial_y,k)\psi\|_{L^{2}L^2}^2 \\
  &\leq C(|k-l|+|l|j)^2|k-l|^{-2}
\big(\|V'\partial_y(\mathrm{e}^{\mathrm{i}lt_*V}f)\|_{L^2L^2}^2+|k-l|^{2}\|f\|_{L^2L^2}^2\big)\\
&\quad+2\nu^{\f13}\|\partial_y(\mathrm{e}^{\mathrm{i}lt_*V}f)\|_{L^2L^2}^2+C|k-l|^{-1}\big(\|\partial_y(\mathrm{e}^{\mathrm{i}lt_*V}f)\|_{L^2L^2}^2+
\||V'|^{-\gamma}f\|_{L^2L^2}^2\big)\\
&\leq C(1+|l|^2j^2|k-l|^{-2})
\big(\|V'\partial_y(\mathrm{e}^{\mathrm{i}lt_*V}f)\|_{L^2L^2}^2+|k-l|^{2}\|f\|_{L^2L^2}^2\big)\\
&\quad+C|k-l|^{-1}\big(\|\partial_y(\mathrm{e}^{\mathrm{i}lt_*V}f)\|_{L^2L^2}^2+
\||V'|^{-\gamma}f\|_{L^2L^2}^2\big),
\end{align*}

This completes the proof of the lemma.
\end{proof}

\subsection{Proof of Proposition \ref{prop:reaction}}

\begin{lemma}\label{lem:f_{1j}}
Let $f_{1,j}$ solve
\begin{align*}
   & \partial_tf_{1,j}+\mathrm{e}^{-\nu t}V\partial_xf_{1,j}-t_*V'P_{\text{low}}u^y_e\partial_x\omega_L^*\chi_{I_j}(t) =0, \quad f_{1,j}|_{t=T_0}=0.
\end{align*}
Then we have 
\beno
\|f_{1,j}\|_{X_{I_j}}\leq C\nu^{-\f13}j^{-1.1}\|\omega_{e}\|_{X_I}\|\om_0\|_{H^3}.
\eeno
Here  $I_j=[t_{j-1},t_j]$, $I=[T_0,T]$ and $I_j\subset I$.
\end{lemma}
\begin{proof}

We write
\beno
P_{\text{low}}u_{e}^y(t,x,y)=\sum_{k\in \Lambda_{*}}u_{k}^y(t,y)\mathrm{e}^{\mathrm{i}kx},\quad  
\Lambda_{*}:=[-\nu^{-1/3},\nu^{-1/3}]\cap\alpha\Z\setminus\{0\},\quad \alpha=2\pi/\mathfrak{p}, 
\eeno
and then
\begin{align*}
&u_{e}^y\partial_x \om_{L}^*(t,x,y)=\sum_{k,l\in\Lambda_{*}}(u_{k}^y\mathrm{i}lw_{l*})(t,y)\mathrm{e}^{\mathrm{i}(k+l)x}
=\sum_{k-l,l\in\Lambda_{*}}(u_{k-l}^y\mathrm{i}lw_{l*})(t,y)\mathrm{e}^{\mathrm{i}kx}.
\end{align*}
Here $w_{l*}(t,y)=(1-\eta(\sqrt{t}V'(y)))w_{l,2}(t,y)$ {is} defined in \eqref{def:wk*}.  Let $ w_{j,k,l}$ solve
\begin{align*}
(\partial_t+\mathrm{i}k\mathrm{e}^{-\nu t}V) w_{j,k,l}(t,y)=t_*V'u^{y}_{k-l}(t,y)\mathrm{i}lw_{l*}(t,y)=F_{k,l}(t,y)\quad\text{for}\ t\in I_j,
\end{align*}
with the initial data $w_{j,k,l}(t_{j-1},y)=0 $. Then we have
\begin{align*}
&f_{1,j}(t,x,y)=\sum_{k-l,l\in\Lambda_{*}}w_{j,k,l}(t,y)\mathrm{e}^{\mathrm{i}kx}=\sum_{k\in\Lambda_+}\widetilde{w}_{j,k}(t,y)
\mathrm{e}^{\mathrm{i}kx},\quad \widetilde{w}_{j,k}=\sum_{l\in\Lambda_*\cap(k-\Lambda_{*})}w_{j,k,l},
\end{align*}for $t\in I_j $. Here $\Lambda_+:=[-2\nu^{-1/3},2\nu^{-1/3}]\cap\alpha\Z$.

We denote (here $L^pL^q=L^p(I;L^q)$, $I=[T_0,T]$)
\beno
E_k=\|(|k|^{-\f12}+|V'|^{\f12})(\partial_y,k)u_{k}^y\|_{L^2L^2}
\eeno
 and
\begin{align*}
&\|f\|_{Y_{k,j}}^2=\|f\|_{L^{\infty}(I_j;L^2)}^2+
|k|^2\|(\partial_y,k)\Delta_{k}^{-1}f\|_{L^{2}(I_j;L^2)}^2+\nu\|(\partial_y,k)f\|_{L^{2}(I_j;L^2)}^2.
\end{align*}
By using $u_{e}^y=\partial_x\Delta^{-1}\om_e$, $M_k=\|(\partial_y,k)^3\omega_{0,k}\|_{L^2}$, and the definition of $X_I$ in \eqref{def:XT}, we obtain
\begin{align*}
&\sum_{k\in\Lambda_*}E_k^2\lesssim \||D_x|^{-\f12}\nabla u_{e}^y\|_{L^2L^2}^2+\||V'|^{\f12}\nabla u_{e}^y\|_{L^2L^2}^2\le\|\om_e\|_{X_I}^2,\\
&\|f_{1,j}\|_{X_{I_j}}^2\lesssim\|f_{1,j}\|_{L^\infty(I_j;L^2)}^{2}+\nu\|\nabla f_{1,j}\|_{L^2(I_j;L^2)}^{2}+ \|\partial_x\nabla \Delta^{-1}f_{1,j}\|_{L^2(I_j;L^2)}^{2}\lesssim\sum_{k\in\Lambda_+}\|\widetilde{w}_{j,k}\|_{Y_{k,j}}^2,\\
&\sum_{k\in\Lambda_*}E_k^2\sum_{l\in\Lambda_*}M_l^2
\lesssim \|\om_e\|_{X_I}^2\|\om_{0}\|_{H^3}^2.
\end{align*}
Thus, it is enough to show that
\begin{align*}
&\sum_{k\in\Lambda_+}\|\widetilde{w}_{j,k}\|_{Y_{j,k}}^2\lesssim \nu^{-\frac{2}{3}}j^{-2.2}\sum_{k\in\Lambda_*}E_k^2\sum_{l\in\Lambda_*}M_l^2
=\nu^{-\frac{2}{3}}j^{-2.2}\sum_{k-l,l\in\Lambda_*}E_{k-l}^2M_l^2,
\end{align*}
which is further reduced to the following estimate
\begin{align}\label{wYk}
&\|\widetilde{w}_{j,k}\|_{Y_{k,j}}^2\lesssim \nu^{-\frac{2}{3}}j^{-2.2}\sum_{l\in\Lambda_*\cap(k-\Lambda_*)}
E_{k-l}^2M_l^2,\quad \forall\ k\in \Lambda_+,\ j\in(0,N]\cap\Z.
\end{align}

Thanks to $\widetilde{w}_{j,k}=\sum_{l\in\Lambda_*\cap(k-\Lambda_*)}w_{j,k,l}$, we only need to estimate $\|w_{j,k,l}\|_{Y_{k,j}} $. Recall that
\beno
(\partial_t+\mathrm{i}k\mathrm{e}^{-\nu t}V) w_{j,k,l}(t,y)=F_{k,l}(t,y),\quad w_{k,l}(t_{j-1},y)=0.
\eeno
Then  we get by Lemma \ref{lem5.7} that for $k\neq l$,
\begin{align}\label{est:wjklY-0}
\|w_{j,k,l}\|_{Y_{k,j}}^2\lesssim&(1+|l|^2j^2|k-l|^{-2})
\big(\|V'\partial_y(\mathrm{e}^{\mathrm{i}lt_*V}F_{k,l})\|_{L^2(I_j;L^2)}^2+
|k-l|^{2}\|F_{k,l}\|_{L^2(I_j;L^2)}^2\big)\nonumber\\&
+|k-l|^{-1}\big(\|\partial_y(\mathrm{e}^{\mathrm{i}lt_*V}F_{k,l})\|_{L^2(I_j;L^2)}^2+ \||V'|^{-2.1}F_{k,l}\|_{L^2(I_j;L^2)}^2\big)\\
\lesssim&(1+|l|^2j^2|k-l|^{-2})
\big(\|\partial_y(\mathrm{e}^{\mathrm{i}lt_*V}F_{k,l})\|_{L^2(I_j;L^2)}^2+
|k-l|^{2}\|F_{k,l}\|_{L^2(I_j;L^2)}^2\big)\nonumber\\&
+|k-l|^{-1}\||V'|^{-2.1}F_{k,l}\|_{L^2(I_j;L^2)}^2\nonumber.
\end{align}

Next we estimate each term on the right hand side. 
Recall\\ $F_{k,l}=t_*V'(y)u^{y}_{k-l}(t,y)\mathrm{i}lw_{l*}(t,y)$.
First of all, we have
\begin{align*}
&\|u^{y}_{k}\|_{L^2L^{\infty}}\leq C\||k|^{-\f12}(\partial_y,k)u_{k}^y\|_{L^2L^2}\leq CE_k,
\end{align*}
We consider 
\begin{align*}
\||V'|^{-2.1}F_{k,l}\|_{L^2}= t_*|l|\||V'|^{-1.1}u^{y}_{k-l}w_{l*}\|_{L^2}\leq t_*|l|\|u^{y}_{k-l}\|_{L^{\infty}}\||V'|^{-1.1}w_{l*}\|_{L^2}.
\end{align*}
Thanks to Lemma \ref{lem:wk*}, we have ${|w_{l*}(t,y)|\leq C\min(|V'|^{2}|l|^{-\f12},|l|^{-\f52})\mathrm{e}^{-\nu l^2\gamma_1(t)|V'(y)|^2}}M_l$, then
\begin{align*}
&\||V'|^{-1.1}w_{l*}\|_{L^2}\leq CM_l\|\min(|V'|^{0.9}|l|^{-\f12},|V'|^{-1.1}|l|^{-\f52})\mathrm{e}^{-\nu l^2\gamma_1(t)|V'(y)|^2}\|_{L^2}\\
\leq& CM_l\min\big(\||V'|^{0.9}|l|^{-\f12}\mathrm{e}^{-\nu l^2\gamma_1(t)|V'(y)|^2}\|_{L^2},\|\min(|V'|^{0.9}|l|^{-\f12},|V'|^{-1.1}|l|^{-\f52})\|_{L^2}\big)\\
\leq& CM_l\min(|l|^{-\f12}|\nu t^3 l^2|^{-0.7},|l|^{-1.9})\leq C|l|^{-1.9}M_l\langle\nu  t^3\rangle^{-0.7}.
\end{align*}
As $t_*\leq t\leq T_0+\nu^{-\f13}j\leq 2\nu^{-\f13}j$, $1+\nu t^3\sim j^3$ for $t\in I_j$, we obtain
\begin{align}
&\||V'|^{-2.1}F_{k,l}\|_{L^2}\leq Ct|l|^{-0.9}\|u^{y}_{k-l}\|_{L^{\infty}}M_l\langle\nu t^3\rangle^{-0.7}\leq C\nu^{-\f13}|l|^{-0.9}\|u^{y}_{k-l}\|_{L^{\infty}}M_lj^{-1.1},\nonumber
\end{align}
which yields
\begin{align}\label{est:wjklY-1}
&\||V'|^{-2.1}F_{k,l}\|_{L^2(I_j;L^2)}\leq C\nu^{-\f13}|l|^{-0.9}\|u^{y}_{k-l}\|_{L^2L^{\infty}}M_lj^{-1.1}\leq C\nu^{-\f13}|l|^{-0.9}E_{k-l}M_lj^{-1.1}.
\end{align}

By Lemma \ref{lem:wk*}, we have 
\begin{align}\label{est:V'wl*}
\||V'|^{\f12}w_{l*}\|_{L^{\infty}} &\leq M_l\|\min(|V'|^{\f52}|l|^{-\f12},|V'|^{\f12}|l|^{-\f52})\mathrm{e}^{-\nu l^2\gamma_1(t)|V'(y)|^2}\|_{L^{\infty}}\\
&\leq C M_l\min(|l|^{-\f12}|\nu t^3 l^2|^{-\f52},|l|^{-\f52})\leq C\langle\nu t^3\rangle^{-\f54}|l|^{-\f52} M_l.\nonumber
\end{align}
Together with the fact that 
\begin{align*}
&\|F_{k,l}\|_{L^2}= t_*|l|\||V'|u^{y}_{k-l}w_{l*}\|_{L^2}\leq Ct|l|\||V'|^{\f12}u^{y}_{k-l}\|_{L^{2}}\||V'|^{\f12}w_{l*}\|_{L^{\infty}},
\end{align*}
we deduce that for $t\in I_j$,
\begin{align}\label{est:wjklY-11}
&\|F_{k,l}\|_{L^2}\leq Ct|l|^{-1}\||V'|^{\f12}u^{y}_{k-l}\|_{L^{2}}M_l\langle \nu t^3\rangle^{-\f54}
\leq C\nu^{-\f13}|l|^{-1}\||V'|^{\f12}u^{y}_{k-l}\|_{L^{2}}M_lj^{-2.75},\nonumber\\
&\|F_{k,l}\|_{L^2(I_j;L^2)}\leq C\nu^{-\f13}|l|^{-1}\||V'|^{\f12}u^{y}_{k-l}\|_{L^2L^{2}}M_lj^{-2.75}\\
&\qquad\quad\qquad\quad\leq C\nu^{-\f13}|l|^{-1}|k-l|^{-1}E_{k-l}M_lj^{-2.75}.\nonumber
\end{align}

Thanks to the fact
\begin{align*}
\partial_y(\mathrm{e}^{\mathrm{i}lt_*V}F_{k,l})=\mathrm{i}t_*l\partial_yu^{y}_{k-l} V'\mathrm{e}^{\mathrm{i}lt_*V}w_{l*}+
\mathrm{i}t_*lu^{y}_{k-l}\partial_y(V'\mathrm{e}^{\mathrm{i}lt_*V}w_{l*}),
\end{align*}
we obtain
\begin{align*}
\|\partial_y(\mathrm{e}^{\mathrm{i}lt_*V}F_{k,l})\|_{L^2}\leq t_*|l|\||V'|^{\f12}\partial_yu^{y}_{k-l}\|_{L^{2}}\||V'|^{\f12}w_{l*}\|_{L^{\infty}}+
t_*|l|\|u^{y}_{k-l}\|_{L^{\infty}}\|\partial_y(V'\mathrm{e}^{\mathrm{i}lt_*V}w_{l*})\|_{L^2}.
\end{align*}
By Lemma \ref{lem:wk*}, we have 
\begin{align*}
&\|\partial_y(V'\mathrm{e}^{\mathrm{i}lt_*V}w_{l*})\|_{L^2}\leq C|l|^{-\f52}\langle \nu t^3\rangle^{-\f54}
M_l,
\end{align*}
and then by \eqref{est:V'wl*}, 
\begin{align*}
\|\partial_y(\mathrm{e}^{\mathrm{i}lt_*V}F_{k,l})\|_{L^2}
&\leq Ct|l|^{-1}\||V'|^{\f12}\partial_yu^{y}_{k-l}\|_{L^{2}}\langle\nu t^3\rangle^{-\f54} M_l+
Ct|l|^{-1}\|u^{y}_{k-l}\|_{L^{\infty}}\langle\nu t^3\rangle^{-\f54}M_l\\
&\leq Ct|l|^{-1}\big(\||V'|^{\f12}\partial_yu^{y}_{k-l}\|_{L^{2}}+\|u^{y}_{k-l}\|_{L^{\infty}}\big)\langle\nu t^3\rangle^{-\f54}M_l\\
&\leq C\nu^{-\f13}|l|^{-1}\big(\||V'|^{\f12}\partial_yu^{y}_{k-l}\|_{L^{2}}+\|u^{y}_{k-l}\|_{L^{\infty}}\big)M_l j^{-2.75},
\end{align*}
which gives
\begin{align}\label{est:wjklY-2}
\|\partial_y(\mathrm{e}^{\mathrm{i}lt_*V}F_{k,l})\|_{L^2(I_j;L^2)}
&\leq C\nu^{-\f13}|l|^{-1}\big(\||V'|^{\f12}\partial_yu^{y}_{k-l}\|_{L^2L^{2}}+\|u^{y}_{k-l}\|_{L^2L^{\infty}}\big)M_lj^{-2.75}\nonumber\\
&\leq C\nu^{-\f13}|l|^{-1}E_{k-l}M_lj^{-2.75}.
\end{align}

Summing up \eqref{est:wjklY-0}, \eqref{est:wjklY-1}, \eqref{est:wjklY-11} and \eqref{est:wjklY-2}, we conclude 
\begin{align*}
&\|w_{j,k,l}\|_{Y_{k,j}}^2\\
&\lesssim (1+|l|^2j^2|k-l|^{-2})\nu^{-\f23}|l|^{-2}(E_{k-l}M_l)^2j^{-5.5}
+|k-l|^{-1}\nu^{-\f23}|l|^{-1.8}(E_{k-l}M_l)^2j^{-2.2}\\
&\lesssim (|l|^{-2}+|k-l|^{-2})\nu^{-\f23}(E_{k-l}M_l)^2j^{-2.2},
\end{align*}
which gives
\begin{align*}
\|\widetilde{w}_{j,k}\|_{Y_{k,j}}&\lesssim \sum_{l\in\Lambda_*\cap(k-\Lambda_*)}\|w_{j,k,l}\|_{Y_{k,j}}
\lesssim\sum_{l\in\Lambda_*\cap(k-\Lambda_*)}(|l|^{-1}+|k-l|^{-1})\nu^{-\f13}E_{k-l}M_lj^{-1.1},\\
\|\widetilde{w}_{j,k}\|_{Y_{k,j}}^2&\lesssim\sum_{l\in\Lambda_*\cap(k-\Lambda_*)}
(|l|^{-2}+|k-l|^{-2})
\sum_{l\in\Lambda_*\cap(k-\Lambda_*)}\nu^{-\f23}(E_{k-l}M_l)^2j^{-2.2}\\ &\lesssim
\sum_{l\in\Lambda_*\cap(k-\Lambda_*)}\nu^{-\f23}(E_{k-l}M_l)^2j^{-2.2}.
\end{align*}
This implies \eqref{wYk}. 
\end{proof}

Now we prove Proposition \ref{prop:reaction}. 

\begin{proof}
In each interval $I_j$, let $\omega_{re,j}$ solve 
\begin{align*}
   & \partial_t\omega_{re,j}-\nu \Delta \omega_{re,j}+\mathrm{e}^{-\nu t}V\partial_x(\omega_{re,j}+\Delta^{-1}\omega_{re,j}) -t_*V'P_{\text{low}}u^2_e\partial_x\omega_L^*\chi_{I_j}(t) =0, \quad\omega_{re,j}|_{t=T_0}=0.
\end{align*} 
Then we have $\omega_{re}=\sum_{j=1}^N\omega_{re,j}$ and
\begin{align*}
   & \partial_t\omega_{re,j}-\nu \Delta \omega_{re,j}+\mathrm{e}^{-\nu t}V\partial_x(\omega_{re,j}+\Delta^{-1}\omega_{re,j}) =0, \quad\forall t\in[T_0,t_{j-1}]\cup[t_j,T].
\end{align*} 
Thus, $\omega_{re,j}=0$ for $t\in[T_0,t_{j-1}]$,   
$\|\omega_{re,j}\|_{X_{[t_{j},T]}}\leq C\|\omega_{re,j}(t_j)\|_{L^2}$ (by Proposition \ref{prop:LNS-sp1}), and 
\begin{align*}
   &  \|\omega_{re,j}\|_{X_{I}}=\|\omega_{re,j}\|_{X_{[T_0,T]}}\leq 
   \|\omega_{re,j}\|_{X_{[t_{j-1},t_j]}}+\|\omega_{re,j}\|_{X_{[t_{j},T]}}\\
   &\leq \|\omega_{re,j}\|_{X_{[t_{j-1},t_j]}}+C\|\omega_{re,j}(t_j)\|_{L^2}\leq C\|\omega_{re,j}\|_{X_{[t_{j-1},t_j]}}=
   C\|\omega_{re,j}\|_{X_{I_j}}.
\end{align*}
Here $I_j=[t_{j-1},t_j]$, $I=[T_0,T]$. Let $f_{2,j}=\omega_{re,j}-f_{1,j}$, which solves
\begin{align*}
   & \partial_tf_{2,j}-\nu \Delta f_{2,j}+\mathrm{e}^{-\nu t}V\partial_x(f_{2,j}+\Delta^{-1}f_{2,j})=\nu \Delta f_{1,j}-
   \mathrm{e}^{-\nu t}V\partial_x\Delta^{-1} f_{1,j}, \quad f_{2,j}|_{t=t_{j-1}}=0.
\end{align*}
By Proposition \ref{prop:LNS-sp1} again, we get
\begin{align*}
    \|f_{2,j}\|_{X_{I_j}}\leq  &C\nu^{-\f12}\|\nu\nabla f_{1,j}\|_{L^2(I_j;L^2)}+C\||D_x|^{-\f12}\nabla(
   \mathrm{e}^{-\nu t}V\partial_x\Delta^{-1} f_{1,j})\|_{L^2(I_j;L^2)}\\
   \leq & C\nu^{\f12}\|\nabla f_{1,j}\|_{L^2(I_j;L^2)}+ C\||V'|\partial_x\Delta^{-1}f_{1,j} \|_{L^2(I_j;L^2)}
   +C{\||D_x|^{\f12}\nabla\Delta^{-1}f_{1,j}\|_{L^2(I_j;L^2)}}\\
   \leq &C\|f_{1,j}\|_{X_{I_j}}.
\end{align*}
Then we have
\begin{align*}
   &  C^{-1}\|\omega_{re,j}\|_{X_{I}}\leq\|\omega_{re,j}\|_{X_{I_j}}\leq \|f_{2,j}\|_{X_{I_j}}+\|f_{1,j}\|_{X_{I_j}}\leq  C\|f_{1,j}\|_{X_{I_j}}.
  \end{align*}
Thus, we infer from Lemma \ref{lem:f_{1j}} that 
\begin{align*}
   \|\omega_{re}\|_{X_I}&\leq \sum_{j=1}^N\|\omega_{re,j}\|_{X_{I}}\leq C\sum_{j=1}^N\|f_{1,j}\|_{X_{I_j}}
    \leq C\sum_{j=1}^N\nu^{-\f13}j^{-1.1}\|\omega_{e}\|_{X_I}\|\om_0\|_{H^3}\\
    &\leq C\nu^{-\f13}\|\omega_e\|_{X_I}\|\omega_0\|_{H^3}.
\end{align*}

  This completes the proof.
\end{proof}

\appendix

\section{{Morse} type lemma}

In this section, we prove the following {Morse} type lemma.

\begin{lemma}\label{lem:genV}
For general (real valued) $V$ satisfying $\|V-b\|_{H^4}\leq\epsilon_0$ with $\epsilon_0>0$ sufficiently small,  there exist constants $a$,  $d$ and a function $\theta(y):\T_{2\pi}\to \T_{2\pi}$ such that $V(y)=ab(\theta(y))+d$, $b(y)=\cos y$, with
     \begin{align*}
{|a-1|+|d|+}\|\theta(y)-y\|_{H^3}\leq C\|V-b\|_{H^4}. 
 \end{align*}
  \end{lemma}

We need the following lemma.
\if0\begin{lemma}\label{lem1.2}
Let $If(y)=\int_0^yf(z)\mathrm{d}z$, $a_1,a_2>0$,  then we have
     \begin{align*}
 \|y^{-3}I(y^2f)\|_{H^1(-a_1,a_2)}\leq C\|f\|_{H^1(-a_1,a_2)}. 
 \end{align*}
  \end{lemma}$f\in C^3$, $y^2\leq f(y)\leq 2y^2$, $g(y)={sgn}(y)\sqrt{f(y)}$, then $g\in C^2$.
  \begin{proof}
  Note that $y^{-3}I(y^2f)(y)=\int_0^1s^2f(sy)\mathrm{d}s,$ $(y^{-3}I(y^2f))'(y)=\int_0^1s^3f'(sy)\mathrm{d}s. $ Then the result follows from Minkowski inequality.
  \end{proof}\fi
 \begin{lemma}\label{lem1.1}
 Assume that $1/4\leq f_i''(y)\leq 1$, $f_i(0)=f_i'(0)=0$,  $g_i(y)=\mathrm{sgn}(y)\sqrt{f_i(y)}$ for $y\in (-a_1,a_2)$, $a_1,a_2\in[1/10,1]$, $\|f_i\|_{H^4(-a_1,a_2)}\leq 10 $, $i=1,2$. Then we have
     \begin{align*}
 \|g_i\|_{H^3(-a_1,a_2)}\leq C\|f_i\|_{H^4(-a_1,a_2)},\quad \|g_1-g_2\|_{H^3(-a_1,a_2)}\leq C\|f_1-f_2\|_{H^4(-a_1,a_2)}. 
 \end{align*}
  \end{lemma}
  \begin{Remark}\label{rem1}
By the translation, if $1/4\leq f_i''(y)\leq 1$, $f_i(a_0)=f_i'(a_0)=0$,  $g_i(y)=\mathrm{sgn}(y-a_0)\sqrt{f_i(y)}$, $\forall\ y\in (-a_1,a_1)$, $a_1\in[1/5,1/2]$, $|a_0|\leq 1/10$, $\|f_i\|_{H^4(-a_1,a_1)}\leq 10 $, $i=1,2$ then we also have
     \begin{align*}
 \|g_i\|_{H^3(-a_1,a_1)}\leq C,\quad \|g_1-g_2\|_{H^3(-a_1,a_1)}\leq C\|f_1-f_2\|_{H^4(-a_1,a_1)}. 
 \end{align*}
\end{Remark} 

 \begin{proof}[Proof of Lemma \ref{lem1.1}]
As $f_i(0)=f_i'(0)=0$, we get by Taylor's formula  that
\begin{align*}
 &f_i(y)=y^2\tilde {f}_i(y),\quad \tilde {f}_i(y):=\int_0^1f_i''(sy)(1-s)\mathrm{d}s,\quad 
 \|\tilde {f}_i\|_{H^2(-a_1,a_2)}\leq C\| {f}_i\|_{H^4(-a_1,a_2)}\leq C,\\
 &f_i'(y)=y{f}_i^*(y),\quad  {f}_i^*(y):=\int_0^1f_i''(sy)\mathrm{d}s,\quad 
 \|{f}_i^*\|_{H^2(-a_1,a_2)}\leq C\| {f}_i\|_{H^4(-a_1,a_2)}\leq C,\\
 &\tilde {f}_1(y)-\tilde {f}_2(y)=\int_0^1(f_1-f_2)''(sy)(1-s)\mathrm{d}s,\quad 
 \|\tilde {f}_1-\tilde {f}_2\|_{H^2(-a_1,a_2)}\leq C\| {f}_1-f_2\|_{H^4(-a_1,a_2)},\\
 &{f}_1^*(y)-{f}_2^*(y)=\int_0^1(f_1-f_2)''(sy)\mathrm{d}s,\quad 
 \| {f}_1^*-{f}_2^*\|_{H^2(-a_1,a_2)}\leq C\| {f}_1-f_2\|_{H^4(-a_1,a_2)}. 
 \end{align*}
As $1/4\leq f_i''(y)\leq 1$, we have $1/8\leq \tilde {f}_i\leq 1/2$. 
Let $\tilde {g}_i:=\tilde {f}_i^{1/2}$. Then
\beno
g_i(y)=y\tilde {g}_i(y),\quad  \|\tilde {g}_i\|_{H^2(-a_1,a_2)}\leq C,\quad  \|\tilde {g}_i^{-1}\|_{H^2(-a_1,a_2)}\leq C,\quad  \|{g}_i\|_{H^2(-a_1,a_2)}\leq C.
\eeno 
Notice that  $f_i=g_i^2$, $y{f}_i^*=f_i'=2g_ig_i'=2y\tilde {g}_ig_i'$, $g_i'=\tilde {g}_i^{-1}{f}_i^*/2$, and then
\begin{align*}
&\| {g}_i'\|_{H^2(-a_1,a_2)}\leq C\|\tilde {g}_i^{-1}\|_{H^2(-a_1,a_2)}\|f_i^*\|_{H^2(-a_1,a_2)}\leq C,\\
&\|g_i\|_{H^3(-a_1,a_2)}\leq C\big(\|g_i\|_{H^2(-a_1,a_2)}+\|g_i'\|_{H^2(-a_1,a_2)}\big)\leq C.\end{align*}

Notice that 
\begin{align*}
g_1'-g_2'=(\tilde {g}_1^{-1}-\tilde {g}_2^{-1}){f}_1^*/2+\tilde {g}_2^{-1}({f}_1^*-{f}_2^*)/2.
\end{align*}
Then by $\tilde {g}_i=\tilde {f}_i^{1/2}$, $1/8\leq \tilde {f}_i\leq 1/2$, we have
\begin{align*}
&\|g_1'-g_2'\|_{H^2(-a_1,a_2)}\leq C\|\tilde {g}_1^{-1}-\tilde {g}_2^{-1}\|_{H^2(-a_1,a_2)}\|{f}_1^*\|_{H^2(-a_1,a_2)}\\
&\qquad+C\|\tilde {g}_2^{-1}\|_{H^2(-a_1,a_2)}\|{f}_1^*-{f}_2^*\|_{H^2(-a_1,a_2)}\\
&\leq C\|\tilde {g}_1^{-1}-\tilde {g}_2^{-1}\|_{H^2(-a_1,a_2)}+C\|{f}_1^*-{f}_2^*\|_{H^2(-a_1,a_2)}\\
&=C\|\tilde {f}_1^{-1/2}-\tilde {f}_2^{-1/2}\|_{H^2(-a_1,a_2)}+C\|{f}_1^*-{f}_2^*\|_{H^2(-a_1,a_2)}\\
&\leq C\|\tilde {f}_1-\tilde {f}_2\|_{H^2(-a_1,a_2)}+C\|{f}_1^*-{f}_2^*\|_{H^2(-a_1,a_2)}\leq C\|f_1-f_2\|_{H^4(-a_1,a_1)}.
\end{align*}
Moreover, we have
\begin{align*}
&\|g_1-g_2\|_{H^2(-a_1,a_2)}=\|y(\tilde {g}_1-\tilde {g}_2)\|_{H^2(-a_1,a_2)}\leq C\|\tilde {g}_1-\tilde {g}_2\|_{H^2(-a_1,a_2)}\\
&=C\|\tilde {f}_1^{1/2}-\tilde {f}_2^{1/2}\|_{H^2(-a_1,a_2)}\leq C\|\tilde {f}_1-\tilde {f}_2\|_{H^2(-a_1,a_2)}\leq 
C\|{f}_1- {f}_2\|_{H^4(-a_1,a_2)}.
\end{align*}
Then we infer that 
\begin{align*}
\|g_1-g_2\|_{H^3(-a_1,a_2)}\leq C(\|g_1-g_2\|_{H^2(-a_1,a_2)}+\|g_1'-g_2'\|_{H^2(-a_1,a_2)})\leq C\|{f}_1- {f}_2\|_{H^4(-a_1,a_2)}.
\end{align*}

This completes the proof.
\end{proof}

Now we prove Lemma  \ref{lem:genV}.

  \begin{proof}

    Thanks to $b''(y)=-\cos(y) \in [-1,-\cos(1/10)]$, it holds that if $y\in[-1/10,1/10]$, then
    \begin{align*}
       & V''(y)\leq \max\{b''(y):y\in[-1/10,1/10]\}+\|V''-b''\|_{L^\infty}\leq -\cos(1/4)+C\epsilon_0<0.
    \end{align*}
  It also holds that
    \begin{align*}
       &V'(-1/10)\geq b'(-1/10)-\|V'-b'\|_{L^\infty}\geq \sin(1/10)-C\epsilon_0>0,\\
       &V'(1/10)\leq b'(1/10)+\|V'-b'\|_{L^\infty}\leq -\sin(1/10)+C\epsilon_0<0.
    \end{align*}
    Then there exists a point $y_1\in [-1/10,1/10]$ such that $V'(y_1)=0$; $y_1$ is the maximum point in $[-1/10,1/10]$; $V'(y)> 0$ if $y\in[-1/10,y_1)$ and $V'(y)< 0$ if $y\in(y_1,1/10].$ Similar argument ensures that there exists a point $y_2\in [-1/10+\pi,1/10+\pi]$ such that $V'(y_2)=0$; $y_2$ is the minimum point in 
    $[-1/10+\pi,1/10+\pi]$; $V'(y)< 0$ if $y\in[-1/10+\pi,y_2)$ and $V'(y)> 0$ if $y\in(y_2,1/10+\pi]$.
        
    For $y\in[1/10,-1/10+\pi]$, we have 
    \beno
    V'(y)\leq \min\{b'(y):y\in[1/10,-1/10+\pi]\}+\|V'-b'\|_{L^\infty}\leq -\sin(1/10)+C\epsilon_0<0.
    \eeno 
    For $y\in[1/10-\pi,-1/10]$, we have 
    \beno 
    V'(y)\geq \min\{b'(y):y\in[1/10-\pi,-1/10]\}-\|V'-b'\|_{L^\infty}\geq \sin(1/10)-C\epsilon_0>0.
    \eeno
    
 Thus, we can conclude that $y_1$ is the maximum point of $V$ and $y_2$ is the minimum point of $V$; $V(y)$ is decreasing when $y\in[y_1,y_2]$ and $V(y)$ is increasing when $y\in [y_2-2\pi,y_1]$.
    
    Let $V_{max}:=V(y_1)$ and $V_{min}:=V(y_2)$. We denote
    \begin{align}\label{choice:ab}
       & a=(V_{max}-V_{min})/2,\quad d=(V_{max}+V_{min})/2,\quad \widetilde{V}(y)=(V(y)-d)/a.
    \end{align}    
    Then we have $\widetilde{V}(y)\in[-1,1]$, $\widetilde{V}(y_1)=1,\ \widetilde{V}(y_2)=-1$.
    
   Next  we define $\theta(y)$ as follows 
    \begin{equation}\label{def:theta-0}
      \theta(y)=\left\{\begin{aligned}
                          & \arccos (\widetilde{V}(y)),\qquad\quad \text{if}\ y\in[y_1,y_2],\\
       &-\arccos(\widetilde{V}(y)),\quad \text{if}\ y\in[y_2-2\pi,y_1].
                       \end{aligned}
      \right.
    \end{equation}
 We extend $\theta(y+2\pi k)=\theta(y)+2\pi k$, $\forall\ y\in[y_2-2\pi,y_1]$, $k\in\Z$.
    
    By \eqref{choice:ab} and \eqref{def:theta-0}, we have
    \begin{align}\label{eq:V=ab+d}
       & V(y)=a\cos(\theta(y))+d=ab(\theta(y))+d.
    \end{align}
    
    By the definition of $y_1$ and $y_1\in[-1/10,1/10]$, we have
    \begin{align*}
       &|b'(0)-b'(y_1)|=|\sin(0)-\sin(y_1)|\geq |y_1-0|/\cos(1/10)\geq C^{-1}|y_1|,
    \end{align*}
    then by using $V'(y_1)=0$,
    \begin{align*}
       &|y_1|\leq C|b'(0)-b'(y_1)|= C|V'(y_1)-b'(y_1)|\leq C\|V'-b'\|_{L^\infty}\leq C\|V-b\|_{H^4}.
    \end{align*}
    Similarly, we have $|y_2-\pi|\leq C\|V-b\|_{H^4}$.
   Then
    \begin{align*}
        |V_{max}-1|=&|V(y_1)-b(0)|\leq |V(y_1)-V(0)|+|V(0)-b(0)|\\
       \leq & |y_1|\|V'\|_{L^\infty}+\|V-b\|_{L^\infty}\leq C\|V-b\|_{H^4},\\
       |V_{min}+1|=&|V(y_2)-b(\pi)|\leq |V(y_2)-V(\pi)|+|V(\pi)-b(\pi)|\\
       \leq & |y_2-\pi|\|V'\|_{L^\infty}+\|V-b\|_{L^\infty}\leq C\|V-b\|_{H^4},
    \end{align*}
    hence, 
    \begin{align}\label{est:a-1+d}
       & |a-1|+|d|\leq |V_{max}-1|+|V_{min}+1|\leq C\|V-b\|_{H^4}.
    \end{align}
By \eqref{est:a-1+d}, we get
    \begin{align*}
\|\widetilde{V}-b\|_{H^4}=&\|(V-d)/a-b\|_{H^4}\leq \|d/a\|_{H^4}+\|(V-b)/a\|_{H^4}+\|(1-a)b/a\|_{H^4}\\
       \leq& C\|V-b\|_{H^4}\leq C\epsilon_0.
    \end{align*}
        
For $y\in(1/10,\pi-1/10)$, we have   
\begin{align*}
       \theta(y)-y=&\arccos (\widetilde{V}(y))-\arccos (b(y))=\int_{\widetilde{V}(y)}^{b(y)}\f{\mathrm{d}z}{\sqrt{1-z^2}}
       \\=&\int_{0}^{1}\f{(b(y)-\widetilde{V}(y))\mathrm{d}z}{[1-(z\widetilde{V}(y)+(1-z)b(y))^2]^{1/2}}.
        \end{align*}   
From the following facts that 
 \begin{align*}
      z\in[0,1]\Rightarrow&|z\widetilde{V}(y)+(1-z)b(y)|\leq |b(y)|+|V(y)-b(y)|\leq \cos(1/10)+C\epsilon_0<0.999,\\
       &\|z\widetilde{V}(y)+(1-z)b(y)\|_{H_y^3}\leq \|\widetilde{V}-b\|_{H^3}+\|b\|_{H^3}\le C,\quad \forall\ z\in[0,1],\\
       &\varphi(z):=(1-z^2)^{-1/2}\in C^{\infty}([-0.999,0.999]),\\
       &\|\varphi(z\widetilde{V}(y)+(1-z)b(y))\|_{H_y^3(1/10,\pi-1/10)}\leq C,\quad \forall\ z\in[0,1],
    \end{align*}   
we infer that    
   \begin{align*}
       &\|\theta(y)-y\|_{H^3(1/10,\pi-1/10)}\leq\int_{0}^{1}
       \bigg\|\f{(b(y)-\widetilde{V}(y))}{[1-(z\widetilde{V}(y)+(1-z)b(y))^2]^{1/2}}\bigg\|_{H_y^3(1/10,\pi-1/10)}\mathrm{d}z\\
       &\leq C\int_{0}^{1}
      \|b(y)-\widetilde{V}(y)\|_{H^3}\|\varphi(z\widetilde{V}(y)+(1-z)b(y))\|_{H_y^3(1/10,\pi-1/10)}\mathrm{d}z\\
      &\leq C\|b(y)-\widetilde{V}(y)\|_{H^3}\leq C\|V-b\|_{H^4}.
    \end{align*}
     Similarly, we have 
     \begin{align*}
       &\|\theta(y)-y\|_{H^3(-\pi+1/10,-1/10)}\leq C\|V-b\|_{H^4}.
    \end{align*}  
    
Let $V_1(y)=\sin\f{\theta(y)}{2}$, $b_1(y)=\sin\f{y-y_1}{2}$, $\widetilde{b}(y)=\cos(y-y_1)=b(y-y_1)$. Then for\\ $y\in[-1/5,1/5]$, we have
\begin{align*}
      &V_1(y)=\mathrm{sgn}(y-y_1)\sqrt{(1-\widetilde{V}(y))/2},\quad b_1(y)=\mathrm{sgn}(y-y_1)\sqrt{(1-\widetilde{b}(y))/2},\\
      &|y_1|\leq C\epsilon_0\leq 1/10,\quad (1-\widetilde{V}(y_1))/2=-\widetilde{V}'(y_1)/2=0,\quad 
      (1-\widetilde{b}(y_1))/2=-\widetilde{b}'(y_1)/2=0,\\
      &-\widetilde{b}''(y)/2=\cos(y-y_1)/2\geq\cos(2/5)/2\geq 1/4,\quad -\widetilde{b}''(y)/2=\cos(y-y_1)/2\leq 1/2,\\
      &-\widetilde{V}''(y)/2\geq-b''(y)/2-\|\widetilde{V}''-b''\|_{L^{\infty}}\geq \cos(y)/2-C\|\widetilde{V}-b\|_{H^4}\geq \cos(0.2)/2-C\epsilon_0\geq 1/4,\\
      &-\widetilde{V}''(y)/2\leq-b''(y)/2+\|\widetilde{V}''-b''\|_{L^{\infty}}\leq \cos(y)/2+C\|\widetilde{V}-b\|_{H^4}{\leq } 1/2+C\epsilon_0\leq 1,\\
      &\|(1-\widetilde{b}(y))/2\|_{H^4}=3\sqrt{\pi/2}\leq 10,\\& \|(1-\widetilde{V}(y))/2\|_{H^4}\leq \|(1-b(y))/2\|_{H^4}+\|\widetilde{V}-b\|_{H^4}\leq 3\sqrt{\pi/2}+C\epsilon_0\leq 10.
    \end{align*}
Then we get by Remark \ref{rem1} that 
\begin{align*}
      &\|V_1-b_1\|_{H^3(-1/5,1/5)}\leq C\|(1-\widetilde{V})/2-(1-\widetilde{b})/2\|_{H^4(-1/5,1/5)} 
       \leq C\|\widetilde{V}-\widetilde{b}\|_{H^4}\\ &\leq 
       C\|\widetilde{V}-{b}\|_{H^4}+C\|\widetilde{b}-{b}\|_{H^4}\leq 
       C\|\widetilde{V}-{b}\|_{H^4}+C|y_1|\leq 
       C\|{V}-{b}\|_{H^4}\leq C\epsilon_0.
    \end{align*}
    
For $y\in(-1/5,1/5)$, we have $\theta(y)=2\arcsin V_1(y)$  and
    \begin{align*}
       \theta(y)-y=&2\arcsin ({V}_1(y))-2\arcsin (b_1(y))-y_1=2\int_{{b}_1(y)}^{V_1(y)}\f{\mathrm{d}z}{\sqrt{1-z^2}}-y_1
       \\=&\int_{0}^{1}\f{2({V}_1(y)-{b}_1(y))\mathrm{d}z}{[1-(z{V}_1(y)+(1-z)b_1(y))^2]^{1/2}}-y_1.
       \end{align*} 
 From the following facts that
    \begin{align*}
      z\in[0,1]\Rightarrow&|z{V}_1(y)+(1-z)b_1(y)|\leq |b_1(y)|+|V_1(y)-b_1(y)|\leq \sin(1/5)+C\epsilon_0<1/2,\\
       &\|z{V}_1(y)+(1-z)b_1(y)\|_{H_y^3(-1/5,1/5)}\leq \|\widetilde{V}\|_{H^3}+\|b\|_{H^3}\le C,\quad \forall\ z\in[0,1],\\
       &\varphi(z):=(1-z^2)^{-1/2}\in C^{\infty}([-1/2,1/2]),\\
       &\|\varphi(z{V}_1(y)+(1-z)b_1(y))\|_{H_y^3(-1/5,1/5)}\leq C,\quad \forall\ z\in[0,1],
    \end{align*}     
 we can deduce that 
    \begin{align*}
&\|\theta(y)-y\|_{H^3(-1/5,1/5)}\\ &\leq \int_{0}^{1}
       \bigg\|\f{2({V}_1(y)-{b}_1(y))\mathrm{d}z}{[1-(z{V}_1(y)+(1-z)b_1(y))^2]^{1/2}}\bigg\|_{H_y^3(-1/5,1/5)}\mathrm{d}z+C|y_1|\\
       &\leq C\int_{0}^{1}
      \|{V}_1(y)-{b}_1(y)\|_{H^3(-1/5,1/5)}\|\varphi(z\widetilde{V}(y)+(1-z)b(y))\|_{H_y^3(-1/5,1/5)}\mathrm{d}z+C|y_1|\\
      &\leq C\|{V}_1(y)-{b}_1(y)\|_{H^3(-1/5,1/5)}+C|y_1|\leq C\|V-b\|_{H^4}.
    \end{align*} 
    Similarly, we have
     \begin{align*}
       &\|\theta(y)-y\|_{H^3(\pi-1/5,\pi+1/5)}\leq C\|V-b\|_{H^4}.
    \end{align*}   
       
  Combining the estimates in the 4 intervals $(1/10,\pi-1/10),\ (-\pi+1/10,-1/10),\\ (-1/5,1/5),\ (\pi-1/5,\pi+1/5)$, we arrive at 
 \begin{align*}
       &\|\theta(y)-y\|_{H^3}\leq C\|V-b\|_{H^4}.
    \end{align*}  
    
    This completes the proof of the lemma.        
  \end{proof}
  
\section{Sobolev estimates}

\begin{lemma}\label{lem:GN-1}
If $\mathrm{P}_0f=0$, then for any  $0<\delta<1$, there exists $C>0$ independent of $\delta$ such that  
  \begin{align*}
     & \|f\|_{L^{\infty}}\leq C\delta^{-\f18}\||D_x|^{\f12}f\|_{L^2}^{\f34}\|f\|_{H^2}^{\f14}+C\delta\|f\|_{H^2}.
  \end{align*}
 \end{lemma}
 
\begin{proof}
  Let $f_k(y)=\f{1}{2\pi}\int_{\mathbb{T}}f(x,y)e^{-\mathrm{i}kx}dx$ for $k\in\mathbb{Z}$. Then $f(x,y)=\sum_{k\in\mathbb{Z}\setminus\{0\}}f_k(y)e^{\mathrm{i}kx}$ and
  \begin{align*}
     &\||D_x|^{\f12}f\|_{L^2(\Omega)}^2=2\pi \sum_{k\in\mathbb{Z}\setminus\{0\}}\||k|^{\f12}f_k\|_{L^2(\mathbb{T}_{\mathfrak{p}})}^2,\quad \|\nabla^2f\|_{L^2(\Omega)}^2=2\pi \sum_{k\in\mathbb{Z}\setminus\{0\}}\|(\partial_y,k)^2f_k\|_{L^2(\mathbb{T}_{\mathfrak{p}})}^2.
  \end{align*}
  By Gagliardo-Nirenberg inequality and H\"older inequality, we obtain 
  \begin{align*}
    \|f\|_{L^{\infty}(\Omega)}&\leq \sum_{k\in\mathbb{Z}\setminus\{0\}}\|f_k\|_{L^\infty(\mathbb{T}_{\mathfrak{p}})}\leq C\sum_{k\in\mathbb{Z}\setminus\{0\}}\|f_k\|_{L^2(\mathbb{T}_{\mathfrak{p}})}^{\f34}\|f_k\|_{H^2(\mathbb{T}_{\mathfrak{p}})}^{\f14}\\
    &\leq C\sum_{k\in\mathbb{Z}\setminus\{0\}}(1+|\delta k|)^{-\f98}|k|^{-\f38}\|| k|^{\f12}(1+|\delta k|)^{\f32}f_k\|_{L^2(\mathbb{T}_{\mathfrak{p}})}^{\f34}\|f_k\|_{H^2(\mathbb{T}_{\mathfrak{p}})}^{\f14},\\
     &\leq C\Big(\sum_{k\in\mathbb{Z}\setminus\{0\}}(1+|\delta k|)^{-\f94}|k|^{-\f34}\Big)^{\f12}\Big(\sum_{k\in\mathbb{Z}\setminus\{0\}}
     \|| k|^{\f12}(1+|\delta k|)^{\f32}f_k\|_{L^2(\mathbb{T}_{\mathfrak{p}})}^2\Big)^{\f38}\\
     &\quad\times\Big(\sum_{k\in\mathbb{Z}\setminus\{0\}}\|(\partial_y,k)^2f_k\|_{L^2(\mathbb{T}_{\mathfrak{p}})}^2\Big)^{\f18}\\
    &\leq C\delta^{-\f18}\big(\||D_x|^{\f12}f\|_{L^2}+\delta^{\f32}\|\partial_x^2f\|_{L^2}\big)^{\f34} \|f\|_{H^2}^{\f14}\\
   &\leq C\delta^{-\f18}\||D_x|^{\f12}f\|_{L^2}^{\f34}\|f\|_{H^2}^{\f14}+C\delta\|f\|_{H^2}.
  \end{align*}
  
  This completes the proof.
\end{proof}

  \begin{lemma}\label{lem:comm-1}
  It holds that
     \begin{align*}
     &\|(\partial_y,k)\Delta_k^{-1}h\|_{L^\infty}\leq C\|h\|_{L^1},\quad {\|(\partial_y,k)^2\Delta_k^{-1}h\|_{L^\infty}\leq C\|h\|_{L^\infty}},\\
      &\|(\partial_y,k)[{\Delta}_k^{-1}(fg)-g{\Delta}_k^{-1}f]\|_{L^{\infty}}\leq C\|{\Delta}_k^{-1}f\|_{L^{\infty}}\|g''\|_{L^{1}} ,\\
       &\|\partial_y{\Delta}_k^{-1}(fg)\|_{L^{\infty}}\leq C\big(\|\partial_y{\Delta}_k^{-1}f\|_{L^{\infty}}\|g\|_{L^{\infty}}+
       \|{\Delta}_k^{-1}f\|_{L^{\infty}}\|g''\|_{L^{1}}\big),\\
       &\|{\Delta}_k^{-1}(fg)\|_{L^{\infty}}\leq C\|{\Delta}_k^{-1}f\|_{L^{\infty}}\big(\|g\|_{L^{\infty}}+|k|^{-1}\|g''\|_{L^{1}}\big).
       \end{align*}
      Here $C$ is a constant independent of $k$.
  \end{lemma}
  
  \begin{proof}
    Let $\hat{h}(l)=\dfrac{1}{2\pi}\int_{\mathbb{T}} e^{-\mathrm{i}l y}h(y)\mathrm{d}y$. Then $\|h\|_{L^\infty}\leq C\|\hat{h}\|_{\ell^1_{l}}, \|\hat{h}\|_{\ell^\infty_{l}}\leq C\|h\|_{L^1}$, and 
    \begin{align}\label{est:Del-1h}
       &\|\Delta_k^{-1}h\|_{L^\infty}\leq C\|(k^2+l^2)^{-1}\hat{h}\|_{\ell^1_{l}}\leq C\|(k^2+l^2)^{-1}\|_{\ell^1_{l}}\|\hat{h}\|_{\ell^\infty_{l}}\leq C|k|^{-1}\|h\|_{L^1}.
    \end{align}
Without lose of generality, we assume that $h$ is real valued. Notice that
    \begin{align*}
       &\partial_y\left[(\partial_y\Delta_k^{-1}h)^2-k^2(\Delta_k^{-1}h)^2\right] =2h\partial_y\Delta_k^{-1}h,
    \end{align*}
    and
    \begin{align*}
       \min\{\big|(\partial_y\Delta_k^{-1}h)^2-k^2(\Delta_k^{-1}h)^2\big|\}\leq & \Big(\dfrac{1}{2\pi}\int_{\mathbb{T}} \partial_y\Delta_k^{-1}h\mathrm{d}y\Big)^2+k^2\|\Delta_k^{-1}h\|_{L^\infty}^2\\
       =& k^2\|\Delta_k^{-1}h\|_{L^\infty}^2.
    \end{align*}
  Then we infer that
    \begin{align*}
       \left\|(\partial_y\Delta_k^{-1}h)^2 -k^2(\Delta_k^{-1}h)^2\right\|_{L^\infty}&\leq \min\{\big|(\partial_y\Delta_k^{-1}h)^2-k^2(\Delta_k^{-1}h)^2\big|\}+ 2\|h\partial_y\Delta_k^{-1}h\|_{L^1}\\
       &\leq k^2\|\Delta_k^{-1}h\|_{L^\infty}^2+2\|h\|_{L^1}\|\partial_y\Delta_k^{-1}h\|_{L^\infty},
    \end{align*}
   which along with \eqref{est:Del-1h} gives
    \begin{align*}
       & \|\partial_y\Delta_k^{-1}h\|_{L^\infty}^2 \leq C k^2\|\Delta_k^{-1}h\|_{L^\infty}^2+ 2\|h\|_{L^1}\|\partial_y\Delta_k^{-1}h\|_{L^\infty}\leq C\|h\|_{L^1}^2+2\|h\|_{L^1}\|\partial_y\Delta_k^{-1}h\|_{L^\infty}.
    \end{align*}
  This implies that
    \begin{align}\label{est:hLin-L1}
       &\|(\partial_y,k)\Delta_k^{-1}h\|_{L^\infty}\leq C\|h\|_{L^1}.
    \end{align}
As $ (\partial_y^2-k^2)\Delta_k^{-1}h=h$, we get by the maximum principle that
$$ k^2\|\Delta_k^{-1}h\|_{L^\infty}\leq\|h\|_{L^\infty},$$ and then
\beno
\|\partial_y^2\Delta_k^{-1}h\|_{L^\infty}\leq k^2\|\Delta_k^{-1}h\|_{L^\infty}+\|h\|_{L^\infty}\leq2\|h\|_{L^\infty}.
\eeno  
By Gagliardo-Nirenberg inequality, we have
\beno 
\|\partial_y\Delta_k^{-1}h\|_{L^\infty}\leq
C\|\partial_y^2\Delta_k^{-1}h\|_{L^\infty}^{1/2}\|h\|_{L^\infty}^{1/2}\leq C|k|^{-1}\|h\|_{L^\infty}.
\eeno
Summing up, we arrive at
   \begin{align}\label{est:parD-1hLp}
      \|(\partial_y,k)^2\Delta_k^{-1}h\|_{L^\infty}\leq C\|h\|_{L^\infty}.
    \end{align}
    
    Thanks to $\Delta_k(g\Delta_k^{-1}f)=fg-\partial_y^2g\Delta_k^{-1}f+2\partial_y[\partial_yg\Delta^{-1}_kf]$, we get
  \begin{align*}
     & \Delta_k^{-1}(fg)-g\Delta_k^{-1}f =\Delta_k^{-1}[\partial_y^2g\Delta_k^{-1}f]-2\Delta_k^{-1}\partial_y[\partial_yg\Delta^{-1}_kf].
  \end{align*}
  By \eqref{est:hLin-L1}, we have
    \begin{align*}
       &\|(\partial_y,k)\Delta_k^{-1}[\partial_y^2g\Delta_k^{-1}f]\|_{L^\infty}\leq C\|\partial_y^2g\Delta_k^{-1}f\|_{L^1} \leq C\|\Delta_k^{-1}f\|_{L^\infty}\|g''\|_{L^1},
    \end{align*}
    and by  \eqref{est:parD-1hLp} and Poinc\'are inequality for $ \partial_yg$, we get
    \begin{align*}
       \|(\partial_y,k)\Delta_k^{-1}\partial_y[\partial_yg\Delta^{-1}_kf]\|_{L^\infty}\leq &C\|\partial_yg\Delta_k^{-1}f\|_{L^\infty}
       \leq  C\|\partial_yg\|_{L^\infty}\|\Delta_k^{-1}f\|_{L^\infty} \\ \leq &C\|g''\|_{L^1}\|\Delta_k^{-1}f\|_{L^\infty} .
    \end{align*}
    Then we conclude
    \begin{align}\label{est:parD-1comm}
      &\|(\partial_y,k)[{\Delta}_k^{-1}(fg)-g{\Delta}_k^{-1}f]\|_{L^{\infty}}\leq C\|{\Delta}_k^{-1}f\|_{L^{\infty}}\|g''\|_{L^{1}} .
       \end{align}
    
 By \eqref{est:parD-1comm}, we have 
  \begin{equation}\label{est:payD-1fg}
    \begin{aligned}
       &\|\partial_y\Delta_k^{-1}(fg)\|_{L^\infty}\leq \|(\partial_y,k)[{\Delta}_k^{-1}(fg)-g{\Delta}_k^{-1}f]\|_{L^{\infty}} +\|\partial_y(g\Delta_k^{-1}f)\|_{L^\infty}\\
        &\leq C\|{\Delta}_k^{-1}f\|_{L^{\infty}}\|g''\|_{L^{1}}  +\|g\|_{L^\infty}\|\partial_y\Delta_k^{-1}f\|_{L^\infty} +\|\partial_yg\|_{L^\infty}\|\Delta_k^{-1}f\|_{L^\infty}\\
         &\leq C\|{\Delta}_k^{-1}f\|_{L^{\infty}}\|g''\|_{L^{1}}  +\|g\|_{L^\infty}\|\partial_y\Delta_k^{-1}f\|_{L^\infty} .
    \end{aligned}
  \end{equation}

  By \eqref{est:parD-1comm} and  \eqref{est:payD-1fg}, we get
  \begin{equation}\label{est:D-1fg}
    \begin{aligned}
       &\|\Delta_k^{-1}(fg)\|_{L^\infty}\leq |k|^{-1}\|(\partial_y,k)[{\Delta}_k^{-1}(fg)-g{\Delta}_k^{-1}f]\|_{L^{\infty}} +\|g\Delta_k^{-1}f\|_{L^\infty}\\
        &\leq C|k|^{-1}\|{\Delta}_k^{-1}f\|_{L^{\infty}}\|g''\|_{L^{1}}  +\|g\|_{L^\infty}\|\Delta_k^{-1}f\|_{L^\infty} \\
         &\leq C\|{\Delta}_k^{-1}f\|_{L^{\infty}}(|k|^{-1}\|g''\|_{L^{1}} +\|g\|_{L^\infty}) .
    \end{aligned}
  \end{equation}
     
 The lemma follows from \eqref{est:hLin-L1}, \eqref{est:parD-1hLp}, \eqref{est:parD-1comm}, \eqref{est:payD-1fg} and \eqref{est:D-1fg}.
  \end{proof}

  \section{ODE lemma}
   In this subsection, we prove the ODE Lemma \ref{lem5a} via the Green{'s} function method. First, we construct two linearly {independent} solutions to $(L-\mathrm{i}tb)f=0$, where $L=(\partial_y+\mathrm{i}tb')\partial_y$.
    
    \begin{lemma}\label{lemf1} For fixed $t\geq k^2$, let $f_1=\mathrm{e}^{-\mathrm{i}tb}$, and
    \beno
   && f_{1+}(y)=f_1(y)\int_y^{1/|k|}\mathrm{e}^{-\mathrm{i}tb}f_1^{-2}(y')\mathrm{d}y'=
   \mathrm{e}^{-\mathrm{i}tb(y)}\int_y^{1/|k|}\mathrm{e}^{\mathrm{i}tb(y')}\mathrm{d}y',\\
   &&f_{1-}(y)=f_{1+}(-y),\quad  W_1=\int_{-1/|k|}^{1/|k|}\mathrm{e}^{\mathrm{i}tb(y')}\mathrm{d}y'.
     \eeno
  Then for $y\in I=[-1/|k|,1/|k|]$, we have 
  \begin{align*}
      &Lf_{1+}=\mathrm{i}tbf_{1+},\quad Lf_{1-}=\mathrm{i}tbf_{1-},\\
      &W(f_{1+},f_{1-})=f_{1+}f_{1-}'-f_{1+}'f_{1-}=f_{1+}+f_{1-}=f_1W_1,
      \end{align*}
      and 
      \begin{align*}
     &|f_{1+}(y)|\leq Ct^{-1}(t^{-1/2}+y_+)^{-1},\quad |f_{1-}(y)|\leq Ct^{-1}(t^{-1/2}+y_-)^{-1},\\
      & |W_1|\geq C^{-1}t^{-1/2},\quad \|f_{1+}+f_{1-}\|_{L^2(I)}\geq C^{-1}|tk|^{-1/2},\\
      &\|f_{1+}-f_{1-}\|_{L^2(I)}\geq C^{-1}|tk|^{-1/2}.\end{align*}      
      Here  $y_+=\max(y,0)$ and $y_-=\max(-y,0)$.
      \end{lemma}

    \begin{proof}
{\it Step 1.}  As $f_{1+}=\mathrm{e}^{-\mathrm{i}tb(y)}\int_y^{1/|k|}\mathrm{e}^{\mathrm{i}tb(y')}\mathrm{d}y'$, we have $(\partial_y+\mathrm{i}tb')f_{1+}=-1$, $f_{1+}(1/|k|)=0$ and \begin{align*}
      &Lf_{1+}=(\partial_y+\mathrm{i}tb')\partial_yf_{1+}=\partial_y(\partial_y+\mathrm{i}tb')f_{1+}-\mathrm{i}tb''f_{1+}=-\mathrm{i}tb''f_{1+}
      =\mathrm{i}tbf_{1+}.
 \end{align*}  
 As $f_{1-}(y)=f_{1+}(-y)$, $b(y)=b(-y)$, we have $f_{1-}(y)=\mathrm{e}^{-\mathrm{i}tb(y)}\int_{-1/|k|}^y\mathrm{e}^{\mathrm{i}tb(y')}\mathrm{d}y'$, and 
 \beno
 (\partial_y+\mathrm{i}tb')f_{1-}={1},\quad  Lf_{1-}=\mathrm{i}tbf_{1-}, 
 \eeno
 and
 \begin{align*}
      W(f_{1+},f_{1-})&=f_{1+}f_{1-}'-f_{1+}'f_{1-}=f_{1+}(\partial_y+\mathrm{i}tb')f_{1-}-(\partial_y+\mathrm{i}tb')f_{1+}f_{1-}=f_{1+}+f_{1-}\\
      &=\mathrm{e}^{-\mathrm{i}tb(y)}\int_{-1/|k|}^{1/|k|}\mathrm{e}^{\mathrm{i}tb(y')}\mathrm{d}y'=\mathrm{e}^{-\mathrm{i}tb(y)}W_1=f_1W_1.
 \end{align*}

 {\it Step 2.} As $t\geq k^2$, for $t^{-1/2}\leq y\leq 1/|k|$, we get by using $b'(y)=-\sin y<0$, $b''(y)=-\cos y<0$ that 
 \begin{align*}
     &f_{1+}(y)\mathrm{e}^{\mathrm{i}tb(y)}=\int_y^{1/|k|}\mathrm{e}^{\mathrm{i}tb(y')}\mathrm{d}y'
     =\frac{\mathrm{e}^{\mathrm{i}tb}}{\mathrm{i}tb'}\Big|_{y}^{1/|k|}+
     \int_y^{1/|k|}\frac{\mathrm{e}^{\mathrm{i}tb(y')}b''(y')}{\mathrm{i}tb'(y')^2}\mathrm{d}y',
     \end{align*}
 which gives
 \begin{align*}
    |f_{1+}(y)|=&|f_{1+}(y)\mathrm{e}^{\mathrm{i}tb(y)}|\le \frac{1}{t|b'(y)|}+\frac{1}{t|b'(1/|k|)|}+\int_y^{1/|k|}\frac{|b''(y')|}{tb'(y')^2}\mathrm{d}y'\\
     =&\frac{1}{t|b'(y)|}+\frac{1}{t|b'(1/|k|)|}+\frac{1}{t|b'(y)|}-\frac{1}{t|b'(1/|k|)|}\\=&\frac{2}{t|b'(y)|}=\frac{2}{t\sin y}\leq \frac{C}{ty}\leq\frac{C}{t(t^{-1/2}+y_+)}.
 \end{align*}

 For $0\leq y\leq t^{-1/2}$, we have
 \begin{align*}
     (f_{1+}\mathrm{e}^{\mathrm{i}tb})\big|_{y}^{t^{-1/2}}=\int_y^{t^{-1/2}}\mathrm{e}^{\mathrm{i}tb(y')}\mathrm{d}y',\quad \left|(f_{1+}\mathrm{e}^{\mathrm{i}tb})\big|_{y}^{t^{-1/2}}\right|\leq \int_y^{t^{-1/2}}\mathrm{d}y'=t^{-1/2}-y\leq t^{-1/2}, 
     \end{align*}
 which gives
  \begin{align*}
    |f_{1+}(y)|=&|f_{1+}(y)\mathrm{e}^{\mathrm{i}tb(y)}|
     \le |(f_{1+}\mathrm{e}^{\mathrm{i}tb})|_{t^{-1/2}}|+|(f_{1+}\mathrm{e}^{\mathrm{i}tb})|_{y}^{t^{-1/2}}|\\
     \leq&\frac{C}{tt^{-1/2}}+t^{-1/2}\leq{C}{t^{-1/2}}\le\frac{C}{t(t^{-1/2}+y_+)}.
 \end{align*} 
Thus, $|f_{1+}(y)|\leq {C}{t^{-1}(t^{-1/2}+y_+)^{-1}}\leq Ct^{-1/2}$ for $0\leq y\leq 1/|k|$. 

As  $W_1=\int_{-1/|k|}^{1/|k|}\mathrm{e}^{\mathrm{i}tb(y')}\mathrm{d}y'=2\int_{0}^{1/|k|}\mathrm{e}^{\mathrm{i}tb(y')}\mathrm{d}y'$, we have $|W_1|=2|f_{1+}(0)|\leq Ct^{-1/2}$. Thanks to $ f_{1+}+f_{1-}=\mathrm{e}^{-\mathrm{i}tb(y)}W_1$, we obtain
 \begin{align*}
     &|f_{1-}(y)|\leq |W_1|+|f_{1+}(y)|\leq Ct^{-1/2}= Ct^{-1}(t^{-1/2}+y_-)^{-1},\quad \forall\ 0\leq y\leq 1/|k|.
 \end{align*}
 As $f_{1-}(y)=f_{1+}(-y)$, for $ -1/|k|\leq y\leq 0$, we have $0\leq -y\leq 1/|k|$ and 
 \begin{align*}
     &|f_{1+}(y)|=|f_{1-}(-y)|\leq Ct^{-1/2}= Ct^{-1}(t^{-1/2}+y_+)^{-1},\\ 
     &|f_{1-}(y)|=|f_{1+}(-y)|\leq Ct^{-1}(t^{-1/2}-y)^{-1}= Ct^{-1}(t^{-1/2}+y_-)^{-1}.
 \end{align*}
 
 In summary, we arrive at 
 \begin{align*}
     &|f_{1+}(y)|\leq Ct^{-1}(t^{-1/2}+y_+)^{-1},\quad |f_{1-}(y)|\leq Ct^{-1}(t^{-1/2}+y_-)^{-1},\quad \forall\ y\in I=[-1/|k|,1/|k|].
 \end{align*}
 
 {\it Step 3.} 
  As $Lf_{1+}=(\partial_y+\mathrm{i}tb')\partial_yf_{1+}=\mathrm{i}tbf_{1+}$, $ b''=-b$, $f_{1+}(1/|k|)=0$, $b'(0)=0$, we have
 \begin{align*}
      &tb|f_{1+}|^2=\mathrm{Im}\big[(\partial_y+\mathrm{i}tb')\partial_yf_{1+}\overline{f_{1+}}\big]
      =\mathrm{Im}\partial_y(\partial_yf_{1+}\overline{f_{1+}})+tb'\mathrm{Re}(\partial_yf_{1+}\overline{f_{1+}})\\
      &=\partial_y\mathrm{Im}(\partial_yf_{1+}\overline{f_{1+}})+tb'\partial_y|f_{1+}|^2/2=
      \partial_y\mathrm{Im}(\partial_yf_{1+}\overline{f_{1+}})+\partial_y(tb'|f_{1+}|^2/2)-tb''|f_{1+}|^2/2,\\
      &\partial_y\big[\mathrm{Im}(\partial_yf_{1+}\overline{f_{1+}})+tb'|f_{1+}|^2/2\big]=tb|f_{1+}|^2+tb''|f_{1+}|^2/2=tb|f_{1+}|^2/2, 
      \end{align*}
 which gives
 \begin{align*}
     \int_0^{1/|k|}[tb(y)|f_{1+}(y)|^2/2]\mathrm{d}y=&\big[\mathrm{Im}(\partial_yf_{1+}\overline{f_{1+}})+tb'|f_{1+}|^2/2\big]|_0^{1/|k|}
      =-\mathrm{Im}(\partial_yf_{1+}\overline{f_{1+}})(0)\\
      \leq&|f_{1+}'(0)||f_{1+}(0)|=|f_{1+}(0)|.
 \end{align*}
Here we used  $f_{1+}'(0)=-1$. Due to $b(y)=\cos y>1/2$ for $0\leq y\leq 1/|k|\leq 1<\pi/3$, we obtain
\begin{align*}
      &t\int_0^{1/|k|}|f_{1+}(y)|^2\mathrm{d}y\le 2\int_0^{1/|k|}tb(y)|f_{1+}(y)|^2\mathrm{d}y\leq 4|f_{1+}(0)|.
 \end{align*}
 
 Recall that $(\partial_y+\mathrm{i}tb')f_{1+}=-1$, we have $\partial_y(f_{1+}+y)=-\mathrm{i}tb'f_{1+}$. Thus, for $0< y\leq 1/|k|$,
 we have $|b'(y)|=|\sin y|\leq|y|$ and 
  \begin{align*}
      &f_{1+}(y)+y-f_{1+}(0)=-\mathrm{i}t\int_0^yb'f_{1+}(y')\mathrm{d}y',\\
      &|f_{1+}(y)+y-f_{1+}(0)|^2\leq t^2\int_0^y|b'(y')|^2\mathrm{d}y'\int_0^y|f_{1+}(y')|^2\mathrm{d}y'\\ &\qquad\leq  t^2\int_0^y|y'|^2\mathrm{d}y'\int_0^{1/|k|}|f_{1+}(y')|^2\mathrm{d}y'\leq t^2(y^3/3)\cdot4|f_{1+}(0)|=(4/3)ty^3|f_{1+}(0)|,\\
      &|f_{1+}(y)+y-f_{1+}(0)|\leq 2t^{1/2}y^{3/2}|f_{1+}(0)|^{1/2},
       \end{align*}
 which shows
  \begin{align*}
     y^2/2=&\int_0^yy'\mathrm{d}y'\leq \int_0^y(|f_{1+}(y')+y'-f_{1+}(0)|+|f_{1+}({y'})|+|f_{1+}(0)|)\mathrm{d}y'\\ \leq& \int_0^y(2t^{1/2}y'^{3/2}|f_{1+}(0)|^{1/2}+|f_{1+}(0)|)\mathrm{d}y'+y^{1/2}\bigg|\int_0^{1/|k|}|f_{1+}(y')|^2\mathrm{d}y'\bigg|^{1/2}\\ \leq& t^{1/2}y^{5/2}|f_{1+}(0)|^{1/2}+y|f_{1+}(0)|+y^{1/2}(4|f_{1+}(0)|/t)^{1/2}. 
     \end{align*} 
 Thus, we obtain
  \begin{align*}
    1\leq 2t^{1/2}y^{1/2}|f_{1+}(0)|^{1/2}+2y^{-1}|f_{1+}(0)|+4t^{-1/2}y^{-3/2}|f_{1+}(0)|^{1/2},
 \end{align*}  
 which gives by taking $y=t^{-1/2}\leq 1/|k|$ that 
 \begin{align*}
      &1\leq 2t^{1/4}|f_{1+}(0)|^{1/2}+2t^{1/2}|f_{1+}(0)|+4t^{1/4}|f_{1+}(0)|^{1/2}=6(t^{1/2}|f_{1+}(0)|)^{1/2}+2t^{1/2}|f_{1+}(0)|.
 \end{align*} 
 Thus, $t^{1/2}|f_{1+}(0)|\geq C^{-1}$. Recall that $|W_1|=2|f_{1+}(0)|$,  we have $|W_1|\geq C^{-1}t^{-1/2}$. \smallskip
 
 \def\Ca{C_3}\def\A{a}
 {\it Step 4.}  Thanks to $ f_{1+}+f_{1-}=\mathrm{e}^{-\mathrm{i}tb(y)}W_1$, $I=[-1/|k|,1/|k|]$, we get 
 \begin{align*}
      &\|f_{1+}+f_{1-}\|_{L^2(I)}=\|W_1\|_{L^2(I)}=|2/k|^{1/2}|W_1|\geq  C^{-1}|k|^{-1/2}t^{-1/2}=C^{-1}|tk|^{-1/2}.\end{align*}
  As   $|f_{1+}(y)|\leq Ct^{-1}(t^{-1/2}+y_+)^{-1}$, we have $|f_{1+}(y)|\leq Ct^{-1}|k|$ for $y\in I_+:=[1/|2k|,1/|k|]$, and then 
  \begin{align*}
      \|f_{1+}-f_{1-}\|_{L^2(I)}\geq \|f_{1+}-f_{1-}\|_{L^2(I_+)}\geq& \|f_{1+}+f_{1-}\|_{L^2(I_+)}-2\|f_{1+}\|_{L^2(I_+)}\\ \geq& \|W_1\|_{L^2(I_+)}-C\|t^{-1}|k|\|_{L^2(I_+)}\\=&|2k|^{-1/2}|W_1|-C|2k|^{-1/2}t^{-1}|k|\\
      \geq&  \Ca^{-1}|tk|^{-1/2}-\Ca t^{-1}|k|^{1/2}.
      \end{align*}
   Thus, for $t\geq 4\Ca^4k^2$, then we have 
   \beno
   \|f_{1+}-f_{1-}\|_{L^2(I)}\geq  (2\Ca)^{-1}|tk|^{-1/2}.
   \eeno
    For $k^2\leq t\leq 4\Ca^4k^2$, recall that $f_{1+}(y)=
   \mathrm{e}^{-\mathrm{i}tb(y)}\int_y^{1/|k|}\mathrm{e}^{\mathrm{i}tb(y')}\mathrm{d}y'$, we have $|f_{1+}(y)|\leq 1/|k|-y$, and then
    for $\A\in(0,1/|k|]$,
    \begin{align*}
      \|f_{1+}-f_{1-}\|_{L^2(I)}\geq& \|f_{1+}-f_{1-}\|_{L^2(1/|k|-\A,1/|k|)}\\ \geq& \|f_{1+}+f_{1-}\|_{L^2(1/|k|-\A,1/|k|)}-2\|f_{1+}\|_{L^2(1/|k|-\A,1/|k|)}\\ \geq& \|W_1\|_{L^2(1/|k|-\A,1/|k|)}-2\|1/|k|-y\|_{L^2(1/|k|-\A,1/|k|)}\\=&\A^{1/2}|W_1|-2\A^{3/2}/3^{1/2}.
      \end{align*}
      Now we take $\A=\min(|W_1|/3,1/|k|)$. Due to $|W_1|\geq C^{-1}t^{-1/2}$ and $t\leq 4\Ca^4k^2$, we get 
    \begin{align*}
      &\A\geq \min(C^{-1}t^{-1/2},1/|k|)\geq 1/ |Ck|,\quad |W_1|\ge 3\A\geq4\A^{3/2}/3^{1/2},\\
      &\|f_{1+}-f_{1-}\|_{L^2(I)}\geq \A^{1/2}|W_1|/2=|Ck|^{-1/2}C^{-1}t^{-1/2}=C^{-1}|tk|^{-1/2}.\end{align*}
     
     {Combining two cases}, we conclude $\|f_{1+}-f_{1-}\|_{L^2(I)}\geq C^{-1}|tk|^{-1/2}$.
      \end{proof}

      Let  $f_2=b-1+1/(2\mathrm{i}t)$ and  $h=Lf_2/f_2$.  Next, we construct two linearly {independent} solutions to $(L-h)f=0$.

\begin{lemma}\label{lemf2}
For fixed $t\geq k^2$, let  
\beno
&&f_{2+}(y)=f_2(y)\int_y^{1/|k|}\mathrm{e}^{-\mathrm{i}tb}f_2^{-2}(y')\mathrm{d}y',\\
 &&f_{2-}(y)=f_{2+}(-y),\quad 
 W_2=\int_{-1/|k|}^{1/|k|}\mathrm{e}^{-\mathrm{i}tb}f_2^{-2}(y')\mathrm{d}y'.
  \eeno
   Then for $y\in I=[-1/|k|,1/|k|]$, we have
   \begin{align*}
      &Lf_{2+}=hf_{2+},\quad Lf_{2-}=hf_{2-},\\
      & W(f_{2+},f_{2-})=f_{2+}f_{2-}'-f_{2+}'f_{2-}=\mathrm{e}^{-\mathrm{i}tb(y)}W_2,
      \end{align*}
      and 
       \begin{align*}
    &|f_{2+}(y)|\leq C(t^{-1/2}+y_-)^{2}(t^{-1/2}+y_+)^{-3},\\ &|f_{2-}(y)|\leq C(t^{-1/2}+y_+)^{2}(t^{-1/2}+y_-)^{-3},\\
      &|W_2|\geq C^{-1}t^{3/2},\quad \|f_{2+}+f_{2-}\|_{L^2(I)}\geq C^{-1}t^{3/2}|k|^{-5/2},\\& 
      \|f_{2+}-f_{2-}\|_{L^2(I)}\geq C^{-1}t^{3/2}|k|^{-5/2}.\end{align*}

      \end{lemma}

\begin{proof}{\it Step 1.} It is easy to verify that 
\begin{align*}
      &(f_{2+}/f_2)'=-\mathrm{e}^{-\mathrm{i}tb}f_2^{-2},\quad f_{2+}'f_2-f_{2+}f_2'=-\mathrm{e}^{-\mathrm{i}tb},\quad f_{2+}''f_2-f_{2+}f_2''=\mathrm{i}tb'\mathrm{e}^{-\mathrm{i}tb},\\
      &0=f_{2+}''f_2-f_{2+}f_2''+\mathrm{i}tb'(f_{2+}'f_2-f_{2+}f_2')=Lf_{2+}f_2-f_{2+}Lf_2,
      \end{align*}
   which gives $Lf_{2+}=f_{2+}Lf_2/f_2=hf_{2+}$. Thanks to  $f_{2-}(y)=f_{2+}(-y)$, $b(y)=b(-y)$ and $f_2(y)=f_2(-y)$, we obtain
 \beno
f_{2-}(y)=f_2(y)\int_{-1/|k|}^y\mathrm{e}^{-\mathrm{i}tb}f_2^{-2}(y')\mathrm{d}y',\quad  f_{2-}'f_2-f_{2-}f_2'=\mathrm{e}^{-\mathrm{i}tb}, \quad  Lf_{2-}=hf_{2-},
\eeno 
and 
\begin{align*}
&f_{2+}(y)+f_{2-}(y)=f_2(y)\int_{-1/|k|}^{1/|k|}\mathrm{e}^{-\mathrm{i}tb}f_2^{-2}(y')\mathrm{d}y'=f_2(y)W_2,\\
     W(f_{2+},f_{2-})&=f_{2+}f_{2-}'-f_{2+}'f_{2-}=\big[f_{2+}(f_{2-}'f_2-f_{2-}f_2')-(f_{2+}'f_2-f_{2+}f_2')f_{2-}\big]/f_2
      \\&=(f_{2+}+f_{2-})\mathrm{e}^{-\mathrm{i}tb}/f_2=\mathrm{e}^{-\mathrm{i}tb(y)}W_2.
      \end{align*}

{\it Step 2.}  For $t\geq k^2$, and $t^{-1/2}\leq y\leq 1/|k|$, we have $f_2'(y)=b'(y)=-\sin y<0$, $b''(y)=-\cos y<0$, 
 $|f_2|\geq |b-1|=1-b$, and then
 \begin{align*}
     &f_{2+}(y)/f_2(y)=\int_y^{1/|k|}\mathrm{e}^{\mathrm{i}tb(y')}f_2^{-2}(y')\mathrm{d}y'
     =\frac{\mathrm{e}^{\mathrm{i}tb}}{\mathrm{i}tb'f_2^2}\bigg|_{y}^{1/|k|}+
     \int_y^{1/|k|}\frac{\mathrm{e}^{\mathrm{i}tb}(b''f_2+2b'f_2')}{\mathrm{i}tb'^2f_2^3}\mathrm{d}y',\\
     &|f_{2+}(y)/f_2(y)|\leq \frac{1}{t|b'f_2^2(y)|}+\frac{1}{t|b'f_2^2(1/k)|}+
     \int_y^{1/|k|}\frac{|b''f_2+2b'^2|}{t|b'^2f_2^3|}\mathrm{d}y'\\
     &\leq \frac{1}{t|b'f_2^2(y)|}+\frac{1}{t|b'f_2^2(1/k)|}+
     \int_y^{1/|k|}\left(\frac{|b''|}{t|b'^2f_2^2|}+\frac{2}{t|f_2^3|}\right)\mathrm{d}y'\\
     &\leq \frac{1}{t|b'(1-b)^2(y)|}+\frac{1}{t|b'(1-b)^2(1/k)|}+
     \int_y^{1/|k|}\left(\frac{-b''}{tb'^2(1-b)^2}+\frac{2}{t(1-b)^3}\right)\mathrm{d}y'\\
     &=\frac{1}{t|b'(1-b)^2(y)|}+\frac{1}{t|b'(1-b)^2(1/k)|}+\frac{1}{tb'(1-b)^2}\bigg|_{y}^{1/|k|}=\frac{2}{t|b'(1-b)^2(y)|}\\
     &=\frac{2}{t\sin y(1-\cos y)^2}\leq \frac{C}{ty^5}\leq \frac{C}{t(t^{-1/2}+y_+)^5}.
 \end{align*} 
 
 For $0\leq y\leq t^{-1/2}$, we have $|f_2(y)|\geq 1/(2t) $ and
 \begin{align*}
     &(f_{2+}/f_2)\big|_{y}^{t^{-1/2}}=\int_y^{t^{-1/2}}\mathrm{e}^{\mathrm{i}tb(y')}f_2^{-2}(y')\mathrm{d}y',\\ 
     &\left|(f_{2+}/f_2)\big|_{y}^{t^{-1/2}}\right|\leq \int_y^{t^{-1/2}}|f_2(y')|^{-2}\mathrm{d}y'\leq \int_y^{t^{-1/2}}(2t)^{2}\mathrm{d}y'=(t^{-1/2}-y)(2t)^{2}\leq 4t^{3/2},\\
     &|f_{2+}(y)/f_2(y)|
     \le |(f_{2+}/f_2)|_{t^{-1/2}}|+|(f_{2+}/f_2)|_{y}^{t^{-1/2}}|\\
     &\leq \frac{C}{t(t^{-1/2})^5}+4t^{3/2}\leq{C}{t^{3/2}}\le\frac{C}{t(t^{-1/2}+y_+)^5}.
 \end{align*}   
 Thus, $|f_{2+}/f_2(y)|\leq {C}{t^{-1}(t^{-1/2}+y_+)^{-5}}\leq Ct^{3/2}$ for $0\leq y\leq 1/|k|$. 
 As\\ $W_2=2\int_{0}^{1/|k|}\mathrm{e}^{\mathrm{i}tb(y')}f_2^{-2}(y')\mathrm{d}y'$, we have
 $|W_2|=2|f_{2+}(0)/f_2(0)|\leq Ct^{3/2}.$
 Due to\\ $ f_{2+}+f_{2-}=f_2W_2$, we get
 \begin{align*}
     &|f_{2-}/f_2(y)|\leq |W_2|+|f_{2+}/f_2(y)|\leq Ct^{3/2}= Ct^{-1}(t^{-1/2}+y_-)^{-5},\quad \forall\ 0\leq y\leq 1/|k|.
 \end{align*}
 As $f_{2-}(y)=f_{2+}(-y)$, $f_{2}(y)=f_{2}(-y)$, for $ -1/|k|\leq t\leq 0$, we have $0\leq -y\leq 1/|k|$ and 
 \begin{align*}
     &|f_{2+}/f_2(y)|=|f_{2-}/f_2(-y)|\leq Ct^{3/2}= Ct^{-1}(t^{-1/2}+y_+)^{-5},\\ 
     &|f_{2-}/f_2(y)|=|f_{2+}/f_2(-y)|\leq Ct^{-1}(t^{-1/2}-y)^{-5}= Ct^{-1}(t^{-1/2}+y_-)^{-5}.
 \end{align*}
 
 In summary, for $y\in I=[-1/|k|,1/|k|]$, there holds
 \begin{align*}
     &|f_{2+}/f_2(y)|\leq Ct^{-1}(t^{-1/2}+y_+)^{-5},\quad |f_{2-}/f_2(y)|\leq Ct^{-1}(t^{-1/2}+y_-)^{-5}.
 \end{align*} As $f_2=b-1+1/(2\mathrm{i}t)$, we have $|f_2|\sim|b-1|+1/(2t)\sim|\cos y-1|+1/t\sim y^2+1/t\sim(t^{-1/2}+|y|)^{2}
 =t(t^{-1/2}+y_+)^{2}(t^{-1/2}+y_-)^{2} $, and then for $y\in I=[-1/|k|,1/|k|]$,
 \begin{align*}
     |f_{2+}(y)|&\leq Ct^{-1}(t^{-1/2}+y_+)^{-5}|f_2(y)|\leq Ct^{-1}(t^{-1/2}+y_+)^{-5}(t^{-1/2}+|y|)^{2}\\
     &= C(t^{-1/2}+y_-)^{2}(t^{-1/2}+y_+)^{-3},\\
     |f_{2-}(y)|&\leq Ct^{-1}(t^{-1/2}+y_-)^{-5}|f_2(y)|\leq Ct^{-1}(t^{-1/2}+y_-)^{-5}(t^{-1/2}+|y|)^{2}\\
     &= C(t^{-1/2}+y_+)^{2}(t^{-1/2}+y_-)^{-3}.
 \end{align*}

 {\it Step 3.}
  As $Lf_{2+}=(\partial_y+\mathrm{i}tb')\partial_yf_{2+}=hf_{2+}$, $f_{2+}(1/|k|)=0$, $b'(0)=0$, we have
 \begin{align*}
      &\mathrm{Im}(h)|f_{2+}|^2=\mathrm{Im}\big[(\partial_y+\mathrm{i}tb')\partial_yf_{2+}\overline{f_{2+}}\big]
      =\mathrm{Im}\partial_y(\partial_yf_{2+}\overline{f_{2+}})+tb'\mathrm{Re}(\partial_yf_{2+}\overline{f_{2+}})\\
      &=\partial_y\mathrm{Im}(\partial_yf_{2+}\overline{f_{2+}})+tb'\partial_y|f_{2+}|^2/2=
      \partial_y\mathrm{Im}(\partial_yf_{2+}\overline{f_{2+}})+\partial_y(tb'|f_{2+}|^2/2)-tb''|f_{2+}|^2/2,\\
      &\partial_y\big[\mathrm{Im}(\partial_yf_{2+}\overline{f_{2+}})+tb'|f_{2+}|^2/2\big]=\mathrm{Im}(h)|f_{2+}|^2+tb''|f_{2+}|^2/2,
 \end{align*}
 which gives
  \begin{align*}
      \int_0^{1/|k|}[-\mathrm{Im}(h)-tb''/2]|f_{2+}|^2\mathrm{d}y&=-\big[\mathrm{Im}(\partial_yf_{2+}\overline{f_{2+}})+tb'|f_{2+}|^2/2\big]\big|_0^{1/|k|}
     \\ &=\mathrm{Im}(\partial_yf_{2+}\overline{f_{2+}})(0)\leq|f_{2+}'(0)||f_{2+}(0)|=2t|f_{2+}(0)|.
 \end{align*} 
 Here we used $f_{2+}'(0)=-\mathrm{e}^{-\mathrm{i}tb}f_2(0)^{-1}=-2\mathrm{i}t\mathrm{e}^{-\mathrm{i}t}$. 
 
 Recall that $ (L+\mathrm{i}t(b+1))f_2=(1-b)/2$, $h=Lf_2/f_2$, we have $h=-\mathrm{i}t(b+1)+(1-b)/(2f_2)$, $|f_2|\geq 1/(2t)$ and\begin{align*}
      &\mathrm{Im}(h)\leq -t(b+1)+(1-b)/|2f_2|\leq -t(b+1)+(1-b)t=-2tb,\\ & -\mathrm{Im}(h)-tb''/2=-\mathrm{Im}(h)+tb/2\geq5tb/2.
 \end{align*}
 As $b(y)=\cos y>1/2>2/5$ for $0\leq y\leq 1/|k|\leq 1<\pi/3$, we obtain 
 \begin{align*}
      \int_0^{1/|k|}|f_{2+}(y)|^2\mathrm{d}y&\le {t}^{-1}\int_0^{1/|k|}(5tb/2)|f_{2+}|^2\mathrm{d}y\le {t}^{-1}\int_0^{1/|k|}[-\mathrm{Im}(h)-tb''/2]|f_{2+}|^2\mathrm{d}y\\
      &\leq {t}^{-1}\cdot 2t|f_{2+}(0)|=2|f_{2+}(0)|.
 \end{align*}
 Recall that $f_{2+}'f_2-f_{2+}f_2'=-\mathrm{e}^{-\mathrm{i}tb}$, we have 
 $\partial_y(f_{2+}f_2\mathrm{e}^{\mathrm{i}tb})-f_{2+}(2f_2'+\mathrm{i}tb'f_2)\mathrm{e}^{\mathrm{i}tb}=-1$, and
 \begin{align*}
      &\partial_y(f_{2+}f_2\mathrm{e}^{\mathrm{i}tb}+y)=f_{2+}(2f_2'+\mathrm{i}tb'f_2)\mathrm{e}^{\mathrm{i}tb},\\
      &2f_2'+\mathrm{i}tb'f_2=2b'+\mathrm{i}tb'(b-1+1/(2\mathrm{i}t))=5b'/2+\mathrm{i}tb'(b-1),\\
      &f_{2+}f_2\mathrm{e}^{\mathrm{i}tb}|_0^y+y=\int_0^yf_{2+}(2f_2'+\mathrm{i}tb'f_2)\mathrm{e}^{\mathrm{i}tb}\mathrm{d}y'.
      \end{align*}
 Thus, we infer that
      \begin{align*}
      &|f_{2+}f_2\mathrm{e}^{\mathrm{i}tb}|_0^y+y|^2\leq\int_0^y|f_{2+}|^2\mathrm{d}y'
      \int_0^y|2f_2'+\mathrm{i}tb'f_2|^2\mathrm{d}y'\\ &\leq \int_0^{1/|k|}|f_{2+}|^2\mathrm{d}y'
      \int_0^{y}|b'(y')|^2((5/2)^2+t^2|b(y')-1|^2)\mathrm{d}y'\\ &\leq 2|f_{2+}(0)|
      \int_0^{y}|y'|^2((5/2)^2+t^2|y'^2/2|^2)\mathrm{d}y'=|f_{2+}(0)|(25y^3/6+t^2y^7/14),
      \end{align*}      
      which gives $|f_{2+}f_2\mathrm{e}^{\mathrm{i}tb}|_0^y+y|\leq|f_{2+}(0)|^{1/2}(5y^{3/2}/2+ty^{7/2}).$
       Then by $f_2(0)=1/(2\mathrm{i}t)$ and $|f_2(y)|^2=|b(y)-1|^2+1/(2t)^2{\leq}|y^2/2|^2+1/(2t)^2$, we obtain
       \begin{align*}
      &y^2/2=\int_0^yy'\mathrm{d}y'\leq \int_0^y\big(|f_{2+}f_2\mathrm{e}^{\mathrm{i}tb}|_0^{y'}+y'|+|f_{2+}f_2(y')|+|f_{2+}f_2(0)|\big)\mathrm{d}y'\\ &\leq \int_0^y\big(|f_{2+}(0)|^{1/2}(5y'^{3/2}/2+ty'^{7/2})+|f_{2+}(0)|/(2t)\big)\mathrm{d}y'\\
      &\qquad+\bigg|\int_0^{y}|f_{2+}(y')|^2\mathrm{d}y'\bigg|^{1/2}\bigg|\int_0^{y}|f_{2}(y')|^2\mathrm{d}y'\bigg|^{1/2}\\ &\leq |f_{2+}(0)|^{1/2}(y^{5/2}+ty^{9/2})+|f_{2+}(0)|y/t\\
      &\quad+\bigg|\int_0^{1/|k|}|f_{2+}(y')|^2\mathrm{d}y'\bigg|^{1/2}\bigg|\int_0^{y}(|y'^2/2|^2+1/(2t)^2)\mathrm{d}y'\bigg|^{1/2}\\ &\leq |f_{2+}(0)|^{1/2}(y^{5/2}+ty^{9/2})+|f_{2+}(0)|y/t+\big|2|f_{2+}(0)|\big|^{1/2}\big|y^5/20+y/(2t)^2\big|^{1/2}\\ &\leq |f_{2+}(0)|^{1/2}(y^{5/2}+ty^{9/2})+|f_{2+}(0)|y/t+|f_{2+}(0)|^{1/2}(y^{5/2}+t^{-1}y^{1/2}),
       \end{align*}
 which gives
 \beno
 1/2\le|f_{2+}(0)|^{1/2}(2y^{1/2}+ty^{5/2}+t^{-1}y^{-3/2})+|f_{2+}(0)|y^{-1}t^{-1}.
  \eeno
Letting $y=t^{-1/2}\leq 1/|k|$, we obtain
 \begin{align*}
   1/2\leq 4t^{-1/4}|f_{2+}(0)|^{1/2}+t^{-1/2}|f_{1+}(0)|=4(t^{-1/2}|f_{1+}(0)|)^{1/2}+t^{-1/2}|f_{1+}(0)|.
 \end{align*} 
 Thus, $t^{-1/2}|f_{2+}(0)|\geq C^{-1}$.  Thanks to $|W_2|=2|f_{2+}(0)/f_2(0)|$ and $|f_2(0)|=1/(2t)$, we arrive at $|W_2|=4t|f_{2+}(0)|\geq C^{-1}t^{3/2}$. \smallskip
 
 \def\Cb{C_4}\def\Cc{C_5}\def\A{a}
 {\it Step 4.}  As $ f_{2+}+f_{2-}=f_2W_2$, $I=[-1/|k|,1/|k|]$, $|f_2(y)|\sim y^2+1/t$, let $I_+:=[1/|2k|,1/|k|]$, then we have
 \begin{align*}
 &\|f_{2}\|_{L^2(I)}\geq \|f_{2}\|_{L^2(I_+)}\geq C^{-1}\|y^2+1/t\|_{L^2(I_+)}\geq C^{-1}|k|^{-5/2},\\
      &\|f_{2+}+f_{2-}\|_{L^2(I)}\geq\|f_{2+}+f_{2-}\|_{L^2(I_+)}=|W_2|\|f_2\|_{L^2(I_+)}\geq  C^{-1}t^{3/2}|k|^{-5/2}.
      \end{align*}
  As    $|f_{2+}(y)|\leq C(t^{-1/2}+y_-)^{2}(t^{-1/2}+y_+)^{-3}$, we have $|f_{2+}(y)|\leq Ct^{-1}|k|^{3}$ for $y\in I_+$ and
  \begin{align*}
      &\|f_{2+}-f_{2-}\|_{L^2(I)}\geq \|f_{2+}-f_{2-}\|_{L^2(I_+)}\geq \|f_{2+}+f_{2-}\|_{L^2(I_+)}-2\|f_{2+}\|_{L^2(I_+)}\\ &\geq C^{-1}t^{3/2}|k|^{-5/2}-C\|t^{-1}|k|^{3}\|_{L^2(I_+)}=C^{-1}t^{3/2}|k|^{-5/2}-C|2k|^{-1/2}t^{-1}|k|^{3}
      \\ &\geq  \Cb^{-1}t^{3/2}|k|^{-5/2}-\Cb t^{-1}|k|^{5/2}.\end{align*}
 Thus, for $t\geq 2\Cb^{4/5}k^2$ then $\|f_{2+}-f_{2-}\|_{L^2(I)}\geq  (2\Cb)^{-1}t^{3/2}|k|^{-5/2}$. For $k^2\leq t\leq 2\Cb^{4/5}k^2$, recall that $f_{2+}(y)=f_2(y)\int_y^{1/|k|}\mathrm{e}^{-\mathrm{i}tb}f_2^{-2}(y')\mathrm{d}y'$, 
      $|f_2(y)|\sim y^2+1/t\sim |k|^{-2}$ for $y\in I_+$, we have $|f_{2+}(y)|\leq C|k|^{-2}\int_y^{1/|k|}|k|^{4}\mathrm{d}y'=C|k|^{2}(1/|k|-y)$ for $y\in I_+$. Then for $\A\in(0,1/|2k|]$, 
    \begin{align*}
      &\|f_{2+}-f_{2-}\|_{L^2(I)}\geq \|f_{2+}-f_{2-}\|_{L^2(1/|k|-\A,1/|k|)}\\ &\geq \|f_{2+}+f_{2-}\|_{L^2(1/|k|-\A,1/|k|)}-2\|f_{2+}\|_{L^2(1/|k|-\A,1/|k|)}\\ &\geq |W_2|\|f_2\|_{L^2(1/|k|-\A,1/|k|)}-C|k|^{2}\|1/|k|-y\|_{L^2(1/|k|-\A,1/|k|)}\\ &\geq C^{-1}t^{3/2}\||k|^{-2}\|_{L^2(1/|k|-\A,1/|k|)}-C|k|^{2}\A^{3/2}=\Cc^{-1}t^{3/2}|k|^{-2}\A^{1/2}-\Cc|k|^{2}\A^{3/2}.\end{align*}
   Taking $\A=[|2k|(\Cc^2+1)]^{-1}$, we infer that  for $t\geq k^2$, $\A\sim |1/k|$,
    \begin{align*}
      &\Cc^{-1}t^{3/2}|k|^{-2}\geq \Cc^{-1}|k|^{3}|k|^{-2}=2\Cc^{-1}|k|^{2}(\Cc^2+1)\A\geq 2|k|^{2}\Cc\A,
      \end{align*}
    which gives 
        \begin{align*}
          &\|f_{2+}-f_{2-}\|_{L^2(I)}\geq (2\Cc)^{-1}t^{3/2}|k|^{-2}\A^{1/2}\geq C^{-1}t^{3/2}|k|^{-2}|1/k|^{1/2}=C^{-1}t^{3/2}|k|^{-5/2}.
       \end{align*}
    
{Combining two cases}, we conclude   
   $\|f_{2+}-f_{2-}\|_{L^2(I)}\geq C^{-1}t^{3/2}|k|^{-5/2}$. 
\end{proof}
      
Now we are in a position to prove Lemma \ref{lem5a}.
      
\begin{proof}
{\it Step 1.}
As $Lf_{1+}=\mathrm{i}tbf_{1+},$ $Lf_{1-}=\mathrm{i}tbf_{1-}$, $W(f_{1+},f_{1-})=f_1W_1=\mathrm{e}^{-\mathrm{i}tb}W_1\neq0$, the solution to
 $(L-\mathrm{i}tb)\phi=F_1$ can be given by      
 \begin{align*}
      \phi(y)=&\frac{f_{1+}(y)}{-W_1}\int_{-1/|k|}^y\mathrm{e}^{\mathrm{i}tb}F_1f_{1-}(y')\mathrm{d}y'+
      \frac{f_{1-}(y)}{-W_1}\int_y^{1/|k|}\mathrm{e}^{\mathrm{i}tb}F_1f_{1+}(y')\mathrm{d}y'\\&+c_1f_{1+}(y)+c_2f_{1-}(y)\\
      =&\phi_1(y)+c_1f_{1+}(y)+c_2f_{1-}(y),
      \end{align*}
 where $c_1$, $c_2$ are constants and 
 \beno
 &&\phi_1(y)=\int_{-1/|k|}^{1/|k|}\mathrm{e}^{\mathrm{i}tb}F_1(y')K_{1}(y,y')\mathrm{d}y',\\
 &&K_1(y,y')=-W_1^{-1}{f_{1+}(y)f_{1-}(y')}{}\mathbf{1}_{y>y'}-W_1^{-1}{f_{1+}(y')f_{1-}(y)}\mathbf{1}_{y<y'}. 
  \eeno
 By Lemma \ref{lemf1}, we have(using $y_+\geq y'_+$ for $y>y'$ and $y_-\geq y'_-$ for $y<y'$)
      \begin{align*}
     |K_1(y,y')|\leq&|W_1|^{-1}{|f_{1+}(y)f_{1-}(y')|}{}\mathbf{1}_{y>y'}+|W_1|^{-1}{|f_{1+}(y')f_{1-}(y)|}\mathbf{1}_{y<y'}\\
      \leq& Ct^{1/2}t^{-1}(t^{-1/2}+y_+)^{-1}t^{-1}(t^{-1/2}+y_-')^{-1}\mathbf{1}_{y>y'}\\&+
      Ct^{1/2}t^{-1}(t^{-1/2}+y_+')^{-1}t^{-1}(t^{-1/2}+y_-)^{-1}\mathbf{1}_{y<y'}\\
      \leq& Ct^{-3/2}(t^{-1/2}+y'_+)^{-1}(t^{-1/2}+y'_-)^{-1}
      =Ct^{-1}(t^{-1/2}+|y'|)^{-1}.\end{align*}
       Thus, for $y\in I=[-1/|k|,1/|k|]$,
       \begin{align*}
     |\phi_1(y)|\leq\int_{-1/|k|}^{1/|k|}|F_1(y')K_1(y,y')|\mathrm{d}y'
      =\int_{-1/|k|}^{1/|k|}Ct^{-1}(t^{-1/2}+|y'|)^{-1}|F_1(y')|\mathrm{d}y'.\end{align*}
 Notice that
 $f_{1+}'(0)=-1$, $f_{1-}'(0)=1$, and  \begin{align*}
\phi_1'(y)=\frac{f_{1+}'(y)}{-W_1}\int_{-1/|k|}^y\mathrm{e}^{\mathrm{i}tb}F_1f_{1-}(y')\mathrm{d}y'+
      \frac{f_{1-}'(y)}{-W_1}\int_y^{1/|k|}\mathrm{e}^{\mathrm{i}tb}F_1f_{1+}(y')\mathrm{d}y',\end{align*}
  We get by  Lemma \ref{lemf1} and $f_{1+}'(0)=-1$, $f_{1-}'(0)=1$  that
  \begin{align*}
&|\phi_1'(0)|\le|W_1|^{-1}\int_{-1/|k|}^0|F_1f_{1-}(y')|\mathrm{d}y'+
      |W_1|^{-1}\int_0^{1/|k|}|F_1f_{1+}(y')|\mathrm{d}y'\\
      &\leq Ct^{1/2}\int_{-1/|k|}^0|F_1(y')|t^{-1}(t^{-1/2}+y_-')^{-1}\mathrm{d}y'+
      Ct^{1/2}\int_0^{1/|k|}|F_1(y')|t^{-1}(t^{-1/2}+y_+')^{-1}\mathrm{d}y'\\
      &=Ct^{-1/2}\int_{-1/|k|}^{1/|k|}|F_1(y')|(t^{-1/2}+|y'|)^{-1}\mathrm{d}y'.\end{align*}
Summing up, we arrive at 
\begin{align*}
      &t\|\phi_1\|_{L^{\infty}(I)}+t^{1/2}|\phi_1'(0)|\leq C\int_{-1/|k|}^{1/|k|}|F_1(y')|(t^{-1/2}+|y'|)^{-1}\mathrm{d}y'
      =C\|F_1/(t^{-1/2}+|y|)\|_{L^1(I)}.
      \end{align*}
   Recall that   $\phi(y)=\phi_1(y)+c_1f_{1+}(y)+c_2f_{1-}(y)$ and $I=[-1/|k|,1/|k|]$, we have 
   \begin{align*}
      \|c_1f_{1+}+c_2f_{1-}\|_{L^{2}(I)}&=\|\phi-\phi_1\|_{L^{2}(I)}\leq\|\phi\|_{L^{2}(I)}+\|\phi_1\|_{L^{2}(I)}\\ &\leq
      \|\phi\|_{L^{2}(I)}+|2/k|^{1/2}\|\phi_1\|_{L^{\infty}(I)}\\&\leq \|\phi\|_{L^{2}(I)}+C|k|^{-1/2}t^{-1}\|F_1/(t^{-1/2}+|y|)\|_{L^1(I)}.
      \end{align*}
      As $f_{1-}(y)=f_{1+}(-y)$, we have $\|f_{1+}\|_{L^{2}(I)}=\|f_{1-}\|_{L^{2}(I)} $. Then we get  by Lemma \ref{lemf1} that
     \begin{align*}
      &\|c_1f_{1+}+c_2f_{1-}\|_{L^{2}(I)}^2=\|(c_1+c_2)(f_{1+}+f_{1-})/2+(c_1-c_2)(f_{1+}-f_{1-})/2\|_{L^{2}(I)}^2\\
      &=|(c_1+c_2)/2|^2\|f_{1+}+f_{1-}\|_{L^{2}(I)}^2+|(c_1-c_2)/2|^2\|f_{1+}-f_{1-}\|_{L^{2}(I)}^2\\ &\geq  
      |(c_1+c_2)/2|^2C^{-1}|tk|^{-1}+|(c_1-c_2)/2|^2C^{-1}|tk|^{-1},
      \end{align*}
      which implies 
      \begin{align*}
      |c_1|+ |c_2|&\leq2(|(c_1+c_2)/2|+|(c_1-c_2)/2|)\leq C|tk|^{1/2}\|c_1f_{1+}+c_2f_{1-}\|_{L^{2}(I)}\\ &\leq C|tk|^{1/2}\|\phi\|_{L^{2}(I)}+Ct^{-1/2}\|F_1/(t^{-1/2}+|y|)\|_{L^1(I)}.
      \end{align*}
By Lemma \ref{lemf1} again, we have $\|f_{1+}\|_{L^{\infty}(I)}+\|f_{1-}\|_{L^{\infty}(I)}\leq Ct^{-1}(t^{-1/2})^{-1}=Ct^{-1/2}$. Then 
(recall $f_{1+}'(0)=-1, f_{1-}'(0)=1$)
\begin{align*}
      &t\|\phi\|_{L^{\infty}(I)}+t^{1/2}|\phi'(0)|\leq t\|\phi_1\|_{L^{\infty}(I)}+t^{1/2}|\phi_1'(0)|\\&\qquad+|c_1|\big(t\|f_{1+}\|_{L^{\infty}(I)}+t^{1/2}|f_{1+}'(0)|\big)
      +|c_2|\big(t\|f_{1-}\|_{L^{\infty}(I)}+t^{1/2}|f_{1-}'(0)|\big)\\
      &\leq C\|F_1/(t^{-1/2}+|y|)\|_{L^1(I)}+Ct^{1/2}(|c_1|+ |c_2|)\\
      &\leq C\|F_1/(t^{-1/2}+|y|)\|_{L^1(I)}+Ct^{1/2}\big(|tk|^{1/2}\|\phi\|_{L^{2}(I)}+t^{-1/2}\|F_1/(t^{-1/2}+|y|)\|_{L^1(I)}\big)\\
      &\leq C\|F_1/(t^{-1/2}+|y|)\|_{L^1(I)}+Ct|k|^{1/2}\|\phi\|_{L^{2}(I)}.
      \end{align*}
      This proves the first part of the lemma. \smallskip

      {\it Step 2.}  As $Lf_{2+}=hf_{2+},$ $Lf_{2-}=hf_{2-}$, $W(f_{2+},f_{2-})=\mathrm{e}^{-\mathrm{i}tb}W_2\neq0$, the solution to
 $(L-h)\phi=F_2$ can be given by   
 \begin{align*}
      \phi(y)&=\frac{f_{2+}(y)}{-W_2}\int_{-1/|k|}^y\mathrm{e}^{\mathrm{i}tb}F_2f_{2-}(y')\mathrm{d}y'+
      \frac{f_{2-}(y)}{-W_2}\int_y^{1/|k|}\mathrm{e}^{\mathrm{i}tb}F_2f_{2+}(y')\mathrm{d}y'\\&
      \qquad+c_1'f_{2+}(y)+c_2'f_{2-}(y)\\
      &=\phi_2(y)+c_1'f_{2+}(y)+c_2'f_{2-}(y),
      \end{align*} 
    where  $c_1'$, $c_2'$ are constants and 
      \begin{align*}
     &\phi_2(y)=\int_{-1/|k|}^{1/|k|}\mathrm{e}^{\mathrm{i}tb}F_2(y')K_2(y,y')\mathrm{d}y',\\& K_2(y,y')=-W_2^{-1}{f_{2+}(y)f_{2-}(y')}{}\mathbf{1}_{y>y'}-W_2^{-1}{f_{2+}(y')f_{2-}(y)}\mathbf{1}_{y<y'}.
      \end{align*}      
      
      By Lemma \ref{lemf2}, we have(using $y_+\geq y'_+$ for $y>y'$ and $y_-\geq y'_-$ for $y<y'$)
      \begin{align*}
     |K_2(y,y')|\leq&|W_2|^{-1}{|f_{2+}(y)f_{2-}(y')|}{}\mathbf{1}_{y>y'}+|W_2|^{-1}{|f_{2+}(y')f_{2-}(y)|}\mathbf{1}_{y<y'}\\
      \leq& Ct^{-3/2}(t^{-1/2}+y_-)^{2}(t^{-1/2}+y_+)^{-3}(t^{-1/2}+y_+')^{2}(t^{-1/2}+y_-')^{-3}\mathbf{1}_{y>y'}\\&+
      Ct^{-3/2}(t^{-1/2}+y_-')^{2}(t^{-1/2}+y_+')^{-3}(t^{-1/2}+y_+)^{2}(t^{-1/2}+y_-)^{-3}\mathbf{1}_{y<y'}\\
      \leq& Ct^{-3/2}(t^{-1/2}+y_-')^{2}(t^{-1/2}+y_+')^{-3}(t^{-1/2}+y_+')^{2}(t^{-1/2}+y_-')^{-3}\\
      =&Ct^{-1}(t^{-1/2}+|y'|)^{-1}.\end{align*}Here we  also used $(t^{-1/2}+y_-')(t^{-1/2}+y_+')=t^{-1/2}(t^{-1/2}+|y'|)$. 
      Thus, for $y\in I$,
      \begin{align*}
     |\phi_2(y)|\leq\int_{-1/|k|}^{1/|k|}|F_2(y')K_2(y,y')|\mathrm{d}y'
      =\int_{-1/|k|}^{1/|k|}Ct^{-1}(t^{-1/2}+|y'|)^{-1}|F_2(y')|\mathrm{d}y'.\end{align*}
   Notice that  $f_2=b-1+1/(2\mathrm{i}t)$, $f_{2+}(y)=f_2(y)\int_y^{1/|k|}\mathrm{e}^{-\mathrm{i}tb}f_2^{-2}(y')\mathrm{d}y'$,
   $f_{2-}(y)=f_{2+}(-y)$, then $f_2(0)=1/(2\mathrm{i}t)$, $f_2'(0)=b'(0)=0$, $f_{2+}'(0)=-\mathrm{e}^{-\mathrm{i}tb}f_2^{-1}(0)=-2\mathrm{i}t\mathrm{e}^{-\mathrm{i}t}$, $f_{2-}'(0)=2\mathrm{i}t\mathrm{e}^{-\mathrm{i}t}$, and 
  \begin{align*}
\phi_2'(y)=\frac{f_{2+}'(y)}{-W_2}\int_{-1/|k|}^y\mathrm{e}^{\mathrm{i}tb}F_2f_{2-}(y')\mathrm{d}y'+
      \frac{f_{2-}'(y)}{-W_2}\int_y^{1/|k|}\mathrm{e}^{\mathrm{i}tb}F_2f_{2+}(y')\mathrm{d}y'.
      \end{align*}
   Then we get by  Lemma \ref{lemf2} that
   \begin{align*}
&|\phi_2'(0)|=2t|W_2|^{-1}\int_{-1/|k|}^0|F_2f_{2-}(y')|\mathrm{d}y'+
      2t|W_2|^{-1}\int_0^{1/|k|}|F_2f_{2+}(y')|\mathrm{d}y'\\
      &\leq Ctt^{-3/2}\int_{-1/|k|}^0|F_2(y')|t^{-1}(t^{-1/2}+y_-')^{-3}\mathrm{d}y'+
      Ctt^{-3/2}\int_0^{1/|k|}|F_2(y')|t^{-1}(t^{-1/2}+y_+')^{-3}\mathrm{d}y'\\
      &= Ct^{-3/2}\int_{-1/|k|}^{1/|k|}|F_{2}(y')|(t^{-1/2}+|y'|)^{-3}\mathrm{d}y'\le Ct^{-1/2}\int_{-1/|k|}^{1/|k|}|F_{2}(y')|(t^{-1/2}+|y'|)^{-1}\mathrm{d}y'.
      \end{align*}
      
       Summing up, we conclude 
\begin{align*}
      &t\|\phi_2\|_{L^{\infty}(I)}+t^{1/2}|\phi_2'(0)|\leq C\int_{-1/|k|}^{1/|k|}|F_2(y')|(t^{-1/2}+|y'|)^{-1}\mathrm{d}y'
      =C\|F_2/(t^{-1/2}+|y|)\|_{L^1(I)}.
      \end{align*}
   Recall that   $\phi(y)=\phi_2(y)+c_1'f_{2+}(y)+c_2'f_{2-}(y)$, $I=[-1/|k|,1/|k|]$, we have 
   \begin{align*}
      &\|c_1'f_{2+}+c_2'f_{2-}\|_{L^{2}(I)}=\|\phi-\phi_2\|_{L^{2}(I)}\leq\|\phi\|_{L^{2}(I)}+\|\phi_2\|_{L^{2}(I)}\\ &\leq
      \|\phi\|_{L^{2}(I)}+|2/k|^{1/2}\|\phi_2\|_{L^{\infty}(I)}\leq \|\phi\|_{L^{2}(I)}+C|k|^{-1/2}t^{-1}\|F_2/(t^{-1/2}+|y|)\|_{L^1(I)}.
      \end{align*}
      As $f_{2-}(y)=f_{2+}(-y)$, we have $\|f_{2+}\|_{L^{2}(I)}=\|f_{2-}\|_{L^{2}(I)}$, and then by Lemma \ref{lemf2},
      \begin{align*}
      &\|c_1'f_{2+}+c_2'f_{2-}\|_{L^{2}(I)}^2=\|(c_1'+c_2')(f_{2+}+f_{2-})/2+(c_1'-c_2')(f_{2+}-f_{2-})/2\|_{L^{2}(I)}^2\\
      &=|(c_1'+c_2')/2|^2\|f_{2+}+f_{2-}\|_{L^{2}(I)}^2+|(c_1'-c_2')/2|^2\|f_{2+}-f_{2-}\|_{L^{2}(I)}^2\\ &\geq  
      |(c_1'+c_2')/2|^2C^{-1}t^{3}|k|^{-5}+|(c_1'-c_2')/2|^2C^{-1}t^{3}|k|^{-5},
      \end{align*}
      which implies 
      \begin{align*}
      |c_1'|+ |c_2'|&\leq2(|(c_1'+c_2')/2|+|(c_1'-c_2')/2|)\leq Ct^{-3/2}|k|^{5/2}\|c_1'f_{2+}+c_2'f_{2-}\|_{L^{2}(I)}\\ &\leq Ct^{-3/2}|k|^{5/2}\|\phi\|_{L^{2}(I)}+Ct^{-5/2}|k|^{2}\|F_2/(t^{-1/2}+|y|)\|_{L^1(I)}.
      \end{align*}
By Lemma \ref{lemf2}, we have 
\beno
\|f_{2+}\|_{L^{\infty}(I)}+\|f_{2-}\|_{L^{\infty}(I)}\leq C(t^{-1/2}+1/|k|)^2(t^{-1/2})^{-3}\leq Ct^{3/2}|k|^{-2}.
\eeno 
Recalling $t\geq k^2$, $f_{2+}'(0)=-2\mathrm{i}t\mathrm{e}^{-\mathrm{i}t}$, $f_{2-}'(0)=2\mathrm{i}t\mathrm{e}^{-\mathrm{i}t}$, 
$t^{5/2}|k|^{-2}\geq t^{3/2}$, we obtain
\begin{align*}
      &t\|\phi\|_{L^{\infty}(I)}+t^{1/2}|\phi'(0)|\leq t\|\phi_2\|_{L^{\infty}(I)}+t^{1/2}|\phi_2'(0)|\\&\qquad+|c_1'|\big(t\|f_{2+}\|_{L^{\infty}(I)}+t^{1/2}|f_{2+}'(0)|\big)
      +|c_2'|\big(t\|f_{2-}\|_{L^{\infty}(I)}+t^{1/2}|f_{2-}'(0)|\big)\\
      &\leq C\|F_2/(t^{-1/2}+|y|)\|_{L^1(I)}+C(t^{5/2}|k|^{-2}+t^{3/2})(|c_1'|+ |c_2'|)\\
      &\leq C\|F_2/(t^{-1/2}+|y|)\|_{L^1(I)}\\&+Ct^{5/2}|k|^{-2}\big(t^{-3/2}|k|^{5/2}\|\phi\|_{L^{2}(I)}+t^{-5/2}|k|^{2}\|F_2/(t^{-1/2}+|y|)\|_{L^1(I)}\big)\\
      &\leq C\|F_2/(t^{-1/2}+|y|)\|_{L^1(I)}+Ct|k|^{1/2}\|\phi\|_{L^{2}(I)}.
      \end{align*}
      
      This proves the second part of the lemma.  
       \end{proof}

  \section*{Acknowledgments}
  
Q. Chen is supported by NSF of China under Grant12288201. H. Jia is supported in part by NSF DMS-2453270 and NSF DMS-2245021. Z. Zhang is partially supported by NSF of China under Grant 12288101.

\end{document}